%% file: HGTDormancyMutation1905.tex
              \def\version{25 January 2023}	        	%

\documentclass[reqno,11pt]{amsart} 

\usepackage[utf8]{inputenc}
\usepackage[english]{babel}
\usepackage[T1]{fontenc}

\usepackage{amsthm,amsmath,amsfonts,amssymb,mathtools, graphicx, mathrsfs,wasysym,mathtools}
\usepackage{caption}
\usepackage{enumitem, nicefrac}
\usepackage{physics}
\usepackage{setspace}
\usepackage{comment}

\excludecomment{exclude}

\usepackage{hyperref}
\hypersetup{
	colorlinks,
	citecolor=black,
	filecolor=black,
	linkcolor=black,
	urlcolor=black,
	linktoc=all
}
\hypersetup{colorlinks=true}

\usepackage{tikz,subfig}
\usetikzlibrary{arrows,arrows.meta,decorations.pathmorphing,backgrounds,positioning,fit,petri,calc,intersections,through,decorations.markings,shapes}

\usepackage[mathscr]{eucal}
\usepackage{srcltx} 
\usepackage{color}
\usepackage{lscape}
\usepackage{graphicx}
\usepackage{tabularx}

\numberwithin{equation}{section}
 
\newfam\Bbbfam 
\font\tenBbb=msbm10 
\font\sevenBbb=msbm7 
\font\fiveBbb=msbm5 
\textfont\Bbbfam=\tenBbb 
\scriptfont\Bbbfam=\sevenBbb 
\scriptscriptfont\Bbbfam=\fiveBbb

\def\2{\mathbf 2}
\def\la{\langle}

\newcommand{\R}     {\mathbb{R}} 
\newcommand{\Z}     {\mathbb{Z}} 
\newcommand{\N}     {\mathbb{N}} 
\renewcommand{\P}   {\mathbb{P}} 
 
\newcommand{\E}     {\mathbb{E}} 
\newcommand{\Q}     {\mathbb{Q}}

\newcommand{\calA}{\mathcal A}

\newcommand{\calX}{\mathcal X}

\renewcommand{\la}{\lambda}

\newcommand{\veps}{\varepsilon}
\newcommand{\vphi}{\varphi}

\newcommand{\sse}{\subseteq}
\newcommand{\epsi}{\veps > 0}	
\newcommand{\reci}[1]{\frac{1}{#1}}

\newcommand{\sra}{\rightarrow}

\newcommand{\Llra}{\Longleftrightarrow}

\newcommand{\downto}{\downarrow}

\newcommand{\xdownarrow}[1]{%
	{\left\downarrow\vbox to #1{}\right.\kern-\nulldelimiterspace}
}

\newcommand{\xuparrow}[1]{%
	{\left\uparrow\vbox to #1{}\right.\kern-\nulldelimiterspace}}

\newcommand{\lr}[1]{\left(#1\right)}
\newcommand{\lrset}[1]{\left\{#1\right\}}

\newcommand{\dt}{\ \mathrm{d}t}

\newcommand{\ds}{\ \mathrm{d}s}

\newcommand{\w}{\omega}

\newcommand{\SP}[2]{\left\langle #1 , #2 \right\rangle}
\renewcommand{\tilde}{\widetilde}
\renewcommand{\hat}{\widehat}

\renewcommand{\P}{\mathbb{P}}
\renewcommand{\E}{\mathbb{E}}
\newcommand{\V}{\mathbb{V}}
\newcommand{\1}{\mathbb{1}}

\def\1{{\mathchoice {1\mskip-4mu\mathrm l}      
{1\mskip-4mu\mathrm l} 
{1\mskip-4.5mu\mathrm l} {1\mskip-5mu\mathrm l}}} 
 
\def\comment#1{} 
\newtheoremstyle{thm}{2ex}{2ex}{\itshape\rmfamily}{} 
{\bfseries\rmfamily}{}{1.7ex}{} 
 
\newtheoremstyle{rem}{1.3ex}{1.3ex}{\rmfamily}{} 
{\itshape\rmfamily}{}{1.5ex}{}

\newtheorem{definition}{Definition}[section]
\newtheorem{theorem}[definition]{Theorem}
\newtheorem{corollary}[definition]{Corollary}
\newtheorem{lemma}[definition]{Lemma}
\newtheorem{notation}[definition]{Notation}
\newtheorem{proposition}[definition]{Proposition}
\theoremstyle{definition}
\newtheorem{remark}[definition]{Remark}
\newtheorem{example}[definition]{Example}

\definecolor{Red}{rgb}{1,0,0}

\begin{document} 
 
\title[A stochastic adaptive dynamics model with mutation, dormancy 
and transfer]
{A stochastic adaptive dynamics model for bacterial populations with 
	mutation, dormancy and transfer}
\author[Jochen Blath, Tobias Paul and András Tóbiás]{}
\maketitle
\thispagestyle{empty}
\vspace{-0.5cm}

\centerline{\sc Jochen Blath{\footnote{Goethe-Universität Frankfurt, Robert-Mayer-Straße 10, 60325 Frankfurt am Main, {\tt blath@math.uni-frankfurt.de}}}, Tobias Paul{\footnote{HU Berlin, Rudower Chaussee 25, 12489 Berlin, {\tt t.paul@math.hu-berlin.de}}} and András Tóbiás{\footnote{Budapest University of Technology and Economics, M\H{u}egyetem rkp. 3., H-1111 Budapest and Alfréd Rényi Institute of Mathematics, Reáltanoda utca 13-15., 1053 Budapest, Hungary,  {\tt tobias@cs.bme.hu}}}}
\renewcommand{\thefootnote}{}
\vspace{0.5cm}

\bigskip

\centerline{\small(\version)} 
\vspace{.5cm} 
 
\begin{quote} 
{\small {\bf Abstract:}} This paper introduces a stochastic adaptive dynamics model for the interplay of several crucial traits and mechanisms in bacterial evolution, namely dormancy, horizontal gene transfer (HGT), mutation and competition. In particular, it combines the recent model of Champagnat, M\'el\'eard and Tran (2021) involving HGT with the model for competition-induced dormancy of Blath and Tóbiás (2020). 
%

Our main result is a convergence theorem which describes the evolution of the different traits in the population on a `doubly logarithmic scale' as piece-wise affine functions. Interestingly, even for a relatively small trait space, the limiting process exhibits a non-monotone dependence of the success of the dormancy trait on the dormancy initiation probability. Further, the model establishes a new `approximate coexistence regime' for multiple traits that has not been observed in previous literature.

\end{quote}



\setcounter{tocdepth}{3}


\setcounter{section}{0}

\begin{comment}{
This is not visible.}
\end{comment}

\input{Model}

\input{Examples}
\input{Proof}
\appendix
\input{BitypeProcesses}

\input{Competition}

\subsection*{Acknowledgements}
JB was supported by DFG Priority Programme 1590 ``Probabilistic Structures in Evolution’’, project 1105/5-1, and by Berlin Mathematics Research Center MATH+, project EF 4-7. TP was supported by Berlin Mathematics Research Center MATH+, project EF 4-7. AT was supported by DFG Priority Programme 1590 ``Probabilistic Structures in Evolution’’.

\bibliography{Literature}
\bibliographystyle{alpha}

\end{document}

%% file: Model.tex
\section{Introduction and Biological Motivation}

\subsection{Motivation and Previous Work}

The stochastic individual based modelling and analysis of the dynamics and evolution of bacterial populations has attracted significant interest in recent years (see e.g.\ \cite{champagnat2005microscopic,FM04,BILLIARD201648, BCFMT18, LFL17,BB18}). This can on the one hand be motivated externally by the relevance of bacterial population dynamics in biology, medicine and industry, and on the other hand internally by the presence of interesting and distinctive features which invite new modelling approaches and  lead to new patterns and results. Two of these distinct features, which have only rather recently been incorporated in population genetic/dynamic models in a systematic way, are {\em horizontal gene transfer} and {\em dormancy}. 

The first feature, horizontal gene transfer (HGT), can in an abstract sense be understood as the ability of individuals to transfer parts of their genome (resp.\ the corresponding traits) to other living individuals, for example via exchange of plasmids during bacterial conjugation \cite{LT46}. This is in contrast to the hereditary `vertical transfer', where genes are copied from parent to daughter cell during binary fission. Essentially, HGT may thus be interpreted as an evolutionary strategy to increase the production of (one's own) favourable traits. HGT comes in several different forms, but for the assumptions of this paper, we will only consider a mechanism that can be motivated from transfer via conjugation. However, it is known that carrying a large quantity of plasmids slows down cell division and as such reduces the reproduction rate (cf.~\cite{BALTRUS2013489}). Such a trade-off leads to interesting questions about the optimality of HGT strategies. HGT has received increasing attention from the modelling side in the last decades, and is now considered as an additional and relatively novel major evolutionary force in bacterial populations (see e.g.\ \cite{BP14, KW12, SL77}).

A second common feature in microbial population dynamics is the wide-spread ability of individuals to enter a reversible state of low/vanishing metabolic activity. Such a dormancy trait comes in many guises, but the general feature seems to be that it allows individuals to survive (e.g.\ in the form of an endospore or cyst) during adverse conditions. It can be triggered by environmental cues (responsive switching), but may also happen spontaneously (stochastic bet hedging) see \cite{LJ11, LdHWB20} for recent overviews. Again, as for HGT, such a trait comes with a significant reproductive trade-off, since the maintenance of a dormancy trait requires a substantial machinery, and thus consumes resources which are unavailable for reproduction.

Interestingly, both mechanisms (HGT and dormancy) also play a crucial role in the context of antibiotic resistance, though in very different ways. While the exchange of resistance genes via horizontal transfer can lead to multi-resistant microbial populations (see e.g.~\cite{multires}), dormancy in the form of {\em persister cells} can be the cause of chronic infections, since these dormant cells with their vanishing metabolism seem to be  protected from antibiotic treatment (\cite{Le10}).

However, HGT and dormancy are of course not the only features of bacterial population dynamics, and interact with classical mechanisms such as reproduction (and hereditary effects), mutation, selection, and competition. Only recently, the joint effects resulting from these mechanisms seem to have moved into the focus of mathematical modellers.
However, given the complexity of bacterial dynamics and the underlying mechanisms, and in view of the sheer number of different evolutionary forces involved in such communities, it is clear that mathematical modelling has to start with simple, idealized scenarios in order to begin to understand basic patterns emerging from such complex interactions.  This process has been initiated  in the last decade.

Indeed, the papers of Billiard et al \cite{BILLIARD201648, BCFMT18} have investigated the consequences of a simple directional HGT mechanism in stochastic individual based models with a focus on its interplay with competition, mutation, and the maintenance of polymorphic variability. In \cite{champagnat2019stochastic}, the approach is transferred and extended into an adaptive dynamics setting with moderately large mutation rates (as previously considered in \cite{DM11}, see also \cite{coquille2020stochastic}), providing a rather new and sophisticated mathematical machinery that leads to interesting scaling limits and emergent behaviour on a `doubly logarithmic scale'. It is shown that HGT can have major consequences for the long-term behaviour of the affected systems, including coexistence, evolutionary suicide and evolutionary cyclic behaviour, depending on the strength of the transfer rate.

Regarding dormancy (and the resulting seed banks), this feature has now been well established as an evolutionary force in population genetics, starting with \cite{kaj_krone_lascoux_2001}, and become a topic of investigation in coalesence theory (cf.~\cite{Blath,Genetics,BGKW18}). In ecology, dormancy and seed banks have been investigated for several decades, starting with Cohen \cite{Co66}, and this lead to a rich (mostly deterministic) theory, see e.g.\ \cite{LdHWB20} for many further references. Traditional seed bank theory is complemented by quantitative research on phenotypic switches in microbial communities, cf. e.g.\  \cite{KL05}. However, the mathematical  analysis of dormancy in stochastic individual based models,  in particular in the framework of adaptive dynamics, seems to be still in its infancy. Yet, several building blocks are already available. The interplay with competition has been investigated in \cite{blath2020invasion}, where it is shown that dormancy traits responding to competitive pressure  can invade and fixate in a resident population despite a substantial reproductive trade-off. One step further, the interplay of dormancy with competition and directional HGT has been investigated in \cite{blath2020interplay}, where coexistence regimes of HGT and dormancy traits are being established.

\subsection{Overview of the Present Paper}

In the present paper, we are attempting to combine the evolutionary forces of mutation, selection, competition, HGT and dormancy within the adaptive dynamics framework of \cite{champagnat2019stochastic}. In particular, we aim to obtain an analogue of their key convergence result, and to investigate the resulting macroscopic behaviour in dependence of the strength of a `dormancy initiation parameter'.

Let us briefly sketch some of the aspects of our model. We will consider a finite set of possible traits $\mathcal{X}$ where each trait reproduces randomly. The trait space is the intersection of a constant multiple of the integer grid $\Z^2$ with the square $[0,4]^2$. The first coordinate $x$ of the trait $(x,y)$ expresses the strength of dormancy (increasing with $x$), and the second coordinate $y$ corresponds to the strength of HGT (increasing with $y$), as we will explain below. To incorporate reproductive trade-offs, the birth rate of an individual of trait $(x,y)$ is strictly decreasing both in $x$ and in $y$. Further we consider natural death at a fixed rate 1 for any active individual, which may be thought of as death by age. We also involve `death by competition' for active individuals. This gives the death rates a dependence on the current population size. Now, traits $(x,y)$ can become dormant instead of dying by competition with probability proportional to $x$. The dormant individuals are not competing for resources and hence do not contribute towards nor are affected by death by competition. Dormant individuals will also not take part in reproduction nor horizontal transfer.  The dormant individuals will switch back to their active state at a fixed rate and have only a natural death rate, which usually is less than the one for active individuals. For horizontal transfer, we will assume that at a population size dependent rate, any given two active individuals meet. In this event, the individual with the `stronger' HGT trait, ie.\ with the higher $y$-coordinate, transfers its trait to the other individual. Lastly, mutations occur randomly at birth with a power law with respect to the carrying capacity $K$. More precisely, the probability of a mutation at birth is $K^{-\alpha}$ for some $\alpha \in (0,1)$. The mutations will either increase the $x$-coordinate or the $y$-coordinate, to the next possible value. In particular, we assume that it is not possible for both the ability to become dormant and the ability to perform horizontal transfer to be improved by one mutation. 

We are interested in the dynamics of our model on the $\log K$ time-scale as $K\to\infty$. Our main result Theorem~\ref{Theorem: Main Theorem} describes convergence properties as in \cite[Theorem 2.1]{champagnat2019stochastic} or \cite[Theorem 2.2]{coquille2020stochastic}. However, in its proof the auxiliary processes that we have to consider are now mostly bi-type (with one component representing the active individuals of a trait and the other component representing the dormant ones), which goes beyond their frameworks. Regarding our bi-type setting, some invasion properties have been studied in \cite{blath2020invasion}, where the form of HGT is slightly different.

Here, the mutation rate scales like $K^{-\alpha}$ for some power $\alpha \in (0,1)$. Consequently, mutants relevant for the evolution of the population are not separated from each other in time. This is a major difference from the classical `Champagnat scaling' discussed in \cite{champagnat2005microscopic}, where mutations are less frequent and cannot influence each other. In the polynomial mutation regime, under suitable assumptions, the logarithm of the size of any trait (with base $K$) converges to a piecewise linear function on the $\log K$ time scale as $K\to\infty$, as we will discuss below. In a population genetic framework, such a mutation regime was studied in \cite{DM11} in a model with clonal interference. In the adaptive dynamics literature, this scaling of mutations occurred before in \cite{Charline2017,BCS19}. From a mathematical point of view, the main novelty of the paper \cite{champagnat2019stochastic} is the systematic study of logistic birth-and-death processes with non-constant immigration, as it was also noted in \cite{coquille2020stochastic}.

In our analysis, we will assume that the population is always of the same order as the carrying capacity, which already poses significant technical challenges, as the length of the present manuscript indicates. In particular, behaviours such as evolutionary suicide are not included in our analysis. In Section \ref{Section: Examples}, we will explore the limiting dynamics for a couple of fixed parameters. We are able to recover some cyclic behaviour already observed in \cite{champagnat2019stochastic}. In addition, the introduction of dormancy seems to allow for the system to be driven towards a state of coexistence in the following sense: At no point in time there are more than two traits with size of order $K$, but on the $\log K$ timescale, there exists a finite time $T_1<\infty$ such that for all $\veps>0$ there exists a time $T_0<T_1$ such that on the time interval $[T_0,T_1]$ at least three traits are of order at least $K^{1-\veps}$, which means that at least three traits are simultaneously macroscopic on a suitable interval. This behaviour has been found previously by \cite{coquille2020stochastic} in the case of asymmetric competition without HGT. In the model studied in \cite{DM11}, the set of points where the limiting piecewise linear process changes slopes may also have a finite accumulation point, see Lemma 1 therein.

The remainder of this paper is organized as follows. In Section~\ref{Section: Model} below, we present our model and our main result. Section~\ref{Section: Examples} contains numerical results regarding some fixed choices of parameters for our model. The proof of our main convergence result, Theorem \ref{Theorem: Main Theorem} will be carried out in Section \ref{Section: Proof}.

In preparation of proving the convergence properties for our model, we analyse bi-type branching processes in Appendix \ref{Section: Appendix B}. We will see that similar properties hold for bi-type processes as they have been shown in \cite[Appendix B]{champagnat2019stochastic} for one-type processes. However, the addition of a second component to the considered processes is sufficient to only allow the ideas of the proof to carry over. The details of the proofs, in particular Theorem \ref{firstconvergencetheorem}, are more involved and require significant amounts of preparation.

In Appendix \ref{Section: Appendix C}, we consider several properties of logistic branching processes. Here, we can also make use of the ideas from \cite{blath2020invasion}, since we are interested in showing that after some time an initially resident trait is driven towards a small population size, while an invasive species becomes resident. As there are many cases of this competition to be distinguished, we also make use of the ideas in \cite{blath2020interplay} in the case of competition between a bi-type process and a single-type process.

\section{Presentation of the Model and Main Result}\label{Section: Model}
We construct a continuous time Markov jump process as follows:  Let $K\in\N$ be a number, which controls the population size and is referred to as the \emph{carrying capacity}. Further we consider the \emph{trait space} $\calX\coloneqq\lrset{0,\delta,\ldots,L\delta}^2=([0,4]\cap\delta\N)^2$, where $\delta>0$ is a fixed real number and $L\coloneqq\lfloor\tfrac{4}{\delta}\rfloor$. Here, the choice of the number 4 is arbitrary, it follows the paper~\cite{champagnat2019stochastic}. As already anticipated, the first coordinate $x$ of the trait $(x,y)$ of an individual expresses the strength of dormancy of the individual, and the second coordinate $y$ of its trait expresses its strength of HGT. For each trait $(x,y)$ we may have \emph{active} or \emph{dormant} individuals (in fact, if $x=0$, then individuals cannot be dormant). We use the notation $N^{K,a}_{m,n}(t)$ and $N^{K,d}_{m,n}(t)$ to refer to the active and dormant population size respectively of trait $(m\delta ,n\delta)$ at time $t\geq 0$.

 \begin{itemize}
	\item Active individuals of trait $(x,y)$ give birth to another individual at rate 
	\[ b(x,y)=4-\frac{x+y}{2}. \]
Fixing $\alpha\in(0,1)$, the child carries the trait $(x+\delta,y)$ with probability $\tfrac{K^{-\alpha}}{2}$, and with the same probability it carries the trait $(x,y+\delta)$. Otherwise the offspring has trait $(x,y)$. Also, if a mutated trait would not belong to $\calX$ anymore, the offspring does not mutate and carries the parental trait $(x,y)$. The decreasing birth rate as $x$ and $y$ increase reflects the trade-off between high reproduction and other survival mechanisms.
	\item There is competition over resources between active individuals, which we incorporate into the death rate. Let $C>0$ and $p\in(0,\tfrac{1}{4})$ be fixed. Active individuals of trait $(x,y)\in\calX$ die at rate 
	\[ d((x,y),N^{K,a})=1+\frac{C(1-px)N^{K,a}}{K}, \] where $N^{K,a}$ denotes the entire active population size $\textstyle
	N^{K,a}=\sum_{m,n=0}^{L}N^{K,a}_{m,n}$.
	\item Active individuals of trait $(x,y)$ can become dormant at rate 
	\[ c((x,y),N^{K,a})=\frac{CpxN^{K,a}}{K}. \] In particular, we are interested in 'competition induced switching', where due to competition from other individuals a part of the population becomes dormant. Individuals with a high value in the first trait component $x$ are thus able to efficiently avoid death in favour of dormancy.
	\item Dormant individuals of any trait die at a natural rate $\kappa\geq 0$ and become active again at rate $\sigma>0$. Usually $\kappa$ will be a small rate, significantly less than $1$, so that dormant individuals are less likely to die than active individuals. This reflects the immunity of dormant individuals to external pressures.
	\item An active individual of trait $(x,y)$ can transfer its trait to a given active individual with trait $(\tilde{x},\tilde{y})$ at rate \[
	\tau((x,y),(\tilde{x},\tilde{y}),N^{K,a})=\frac{\tau}{N^{K,a}}\1_{y>\tilde{y}}.
	\]
	Note that dormant individuals are neither affected by nor are able to perform transfer. Here, traits with a large second component $y$ are advantageous.

\end{itemize}

\begin{figure}[htbp]
	\begin{tikzpicture}
	\draw[->,thick](0,0)--(4,0);
	\draw[->,thick](0,0)--(0,4);
	\draw[dotted, step=.6cm](0,0) grid (3.6,3.6);
	\node at (-0.5,-0.2) {$(0,0)$};
	\node at (0.5,-0.4) {$(\delta,0)$};
	\node at (4.1,-0.3) {Dormancy};
	\node at (-0.6,0.6) {$(0,\delta)$};
	\node at (-0.5,4.1) {HGT};
	\node at (-0.7,3.6) {$(0,L\delta)$};
	\foreach \x in {0,0.6,1.2,1.8,2.4,3,3.6}{
		\foreach \y in {0,0.6,1.2,1.8,2.4,3,3.6}
		{\node[circle, inner sep=0pt, minimum size=4pt, fill=black] at (\x,\y) {};
	}};
	\end{tikzpicture}
	\caption{A visualization of the trait space $\calX$. The strength of dormancy in a trait increases as the first component increases and the strength of HGT increases with the second component.}
\end{figure}
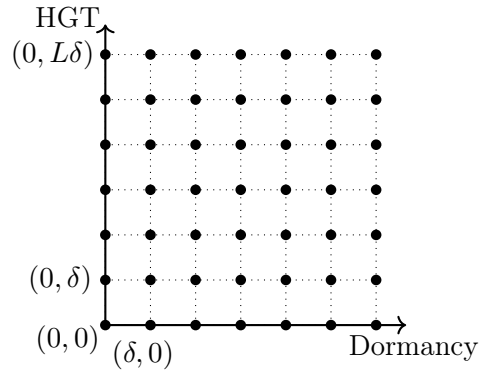

Comparing with Theorem \ref{Theorem: Main Convergence}, we are concerned with the total size of each trait $(m\delta,n\delta)\in\calX$, which we will denote by $N_{m,n}^K(t)\coloneqq N_{m,n}^{K,a}(t)+N_{m,n}^{K,d}(t)$, and the corresponding exponents \begin{align}
N_{m,n}^K(t\log K)=K^{\beta_{m,n}^K(t)}-1\quad\Llra\quad\beta_{m,n}^K(t)\coloneqq\frac{\log(1+N_{m,n}^K(t\log K))}{\log K}.\label{Definition: beta}
\end{align}
We are interested in the behaviour of $\beta_{m,n}^K$ as $K\to\infty$, that is, we want to understand the evolution of the population sizes on the $\log K$ timescale. Since our death and dormancy rate are dependent on the population size, there may be two cases: Either there is a single trait $(x,y)$, which has a population size of order $K$, in which case we refer to the trait $(x,y)$ as \emph{resident}; or the entire population is of size $o(K)$, in which case we refer to the trait with the largest population size as \emph{dominant}. For our purposes, we will only consider the case where there is always one resident trait.

Now, assume that the trait $(x,y)=(m\delta,n\delta)$ is resident. Then, for large $K$, we can approximate the dynamics of $(N_{m,n}^{K,a}(t),N_{m,n}^{K,d}(t))$ as $K(z^a(t),z^d(t))$, where $(z^a(t),z^d(t))$ solves the ordinary differential equation 
\begin{align}
\begin{aligned}
\dot{z}^a(t)&=\lr{3-\frac{x+y}{2}-Cz^a(t)}z^a(t)+\sigma z^d(t)\\
\dot{z}^d(t)&=Cpx(z^a(t))^2-(\kappa+\sigma)z^d(t).
\end{aligned}\label{approximateODE}
\end{align}
Indeed, this approximation follows from \cite[Theorem 11.2.1]{ethier2009markov}. We want to calculate a stable equilibrium of this system, which has already been done in \cite[Section 2.2]{blath2020invasion}. There it is shown that the only coordinate-wise non-negative asymptotically stable equilibrium of the system (\ref{approximateODE}) is given for $3-\tfrac{x+y}{2}>0$ as $(\bar{z}^a_{m,n},\bar{z}^d_{m,n})$, where \begin{align}
\bar{z}^a_{m,n}=\frac{(3-\frac{x+y}{2})(\kappa+\sigma)}{C(\kappa+(1-px)\sigma)}\quad\text{and}\quad \bar{z}^d_{m,n}=\frac{px(3-\frac{x+y}{2})^2(\kappa+\sigma)}{C(\kappa+(1-px)\sigma)^2}.\label{equilibrium}
\end{align}
Observe that this also holds true in the case where $x=0$, in which case the equilibrium size of the dormant population is $0$, and the active population size is $\tfrac{1}{C}(3-\tfrac{x+y}{2})$, which corresponds to the equilibrium of the differential equation \[
\dot{z}(t)=\lr{3-\frac{x+y}{2}-Cz(t)}z(t).
\]
If $3-\tfrac{x+y}{2}<0$, then there is no positive equilibrium and the fixed point $(0,0)$ becomes asymptotically stable. This can be seen from linearizing the system (\ref{approximateODE}), which yields the Jacobian \[
A(0,0)=\begin{pmatrix}
3-\tfrac{x+y}{2}& \sigma\\
0 & -(\kappa+\sigma)
\end{pmatrix},
\]
whose determinant is positive and trace is negative. Hence both eigenvalues must be negative, showing that in this case $(0,0)$ indeed is a stable equilibrium.

In order to have a well-defined process, we also need to introduce a starting condition. Initially, we assume the trait $(0,0)$ to be close to its equilibrium, which is of size \begin{align}
N_{0,0}^{K,a}(0)=\left\lfloor\frac{3K}{C}\right\rfloor.\label{startingcondition1}\end{align}
Since the effective mutation rate in a population of order $K^c$ is $K^{c-\alpha}$, we choose all other starting conditions to be \begin{align}N_{0,n}^{K,a}(0)=\lfloor K^{1-n\alpha}\rfloor \quad\text{ and }\quad (N_{m,n}^{K,a}(0),N_{m,n}^{K,d}(0))=\lfloor (K^{(1-(m+n)\alpha)},K^{(1-(m+n)\alpha)})\rfloor\label{startingcondition2}\end{align}
if $n\alpha< 1$ and $(m+n)\alpha< 1$ respectively and $0$ otherwise.
Indeed, this choice is consistent with Lemma \ref{Lemma: Strong mutation}, which would suggest that on the $\log K$ timescale we otherwise would immediately obtain a population of our chosen initial size. In addition, this choice shows that \[
\beta_{m,n}^K(0)\xrightarrow{K\to\infty}(1-(m+n)\alpha)\vee 0.
\] 

Our next goal is to define the \emph{invasion fitness} -- also known as the initial rate of growth -- $S((\tilde{x},\tilde{y}),(x,y))$ of a single individual of trait $(\tilde{x},\tilde{y})$ in a population, where the trait $(x,y)$ is resident, i.e.~at its equilibrium size. Hence, we consider the active population given by $K\bar{z}^a$ from (\ref{equilibrium}). In particular we assume $\tfrac{x+y}{2}<3$. We distinguish two cases:\begin{description}
	\item[Case 1: $\tilde{x}=0$] In this case, the population size of trait $(0,\tilde{y})$ follows the dynamics of a usual one-dimensional birth and death process. Hence, we define the initial growth rate $S$ as the asymptotic difference of birth and death rate, where we need to take into account the horizontal transfer as additional births or deaths as follows \begin{align*}
	S((0,\tilde{y}),(x,y))&\coloneqq \lim\limits_{K\to\infty}b(0,\tilde{y})-d((0,\tilde{y}),K\bar{z}^a)+\frac{K\bar{z}^a}{K\bar{z}^a+1}\tau\1_{y<\tilde{y}}-\frac{K\bar{z}^a}{K\bar{z}^a+1}\tau\1_{y>\tilde{y}}\\&=3-\frac{\tilde{y}}{2}-\frac{(3-\frac{x+y}{2})(\kappa+\sigma)}{\kappa+(1-px)\sigma}+\tau\operatorname{sign}(\tilde{y}-y).
	\end{align*}
	\item[Case 2: $\tilde{x}>0$] Here we have transfer between the active and dormant populations. Hence the growth rate corresponds to that of a bi-type branching process. Being consistent with the definition thereof in Appendix \ref{Section: Appendix B}, we define the components of (\ref{Eigenvalue}) asymptotically in accordance with Notation \ref{Notation: Growth}. We set\begin{align*}
	r_1&\coloneqq \lim\limits_{K\to\infty}b(\tilde{x},\tilde{y})-d(K\bar{z}^a)+\frac{K\bar{z}^a}{K\bar{z}^a+1}\tau\1_{y<\tilde{y}}-\frac{K\bar{z}^a}{K\bar{z}^a+1}\tau\1_{\tilde{y}<y}\\
	&=3-\frac{\tilde{x}+\tilde{y}}{2}-\frac{(3-\frac{x+y}{2})(\kappa+\sigma)}{\kappa+(1-px)\sigma}+\tau\operatorname{sign}(\tilde{y}-y),\\
	r_2&\coloneqq 0-\kappa-\sigma=-(\kappa+\sigma),\\
	\sigma_1&\coloneqq \lim\limits_{K\to\infty}\frac{Cp\tilde{x}(K\bar{z}^a+1)}{K}=\frac{p\tilde{x}(3-\frac{x+y}{2})(\kappa+\sigma)}{\kappa+(1-px)\sigma}, \qquad
	\sigma_2 \coloneqq \sigma.
	\end{align*}
	Then the invasion fitness is defined by \[
	S((\tilde{x},\tilde{y}),(x,y))\coloneqq\frac{r_1+r_2+\sqrt{(r_1-r_2)^2+4\sigma_1\sigma_2}}{2}.
	\]
	This number is the largest eigenvalue of the mean matrix of the corresponding approximating bi-type branching process, which is given by \[
	\begin{pmatrix}
	r_1 & \sigma_1\\
	\sigma_2 & r_2
	\end{pmatrix}
	\]
	We refer to Appendix \ref{Section: Appendix B} for details on the derivation of the initial growth rate of bi-type branching processes.
\end{description}
Note that distinguishing these two cases is necessary: If we were to model the behaviour of individuals of traits $(0,\tilde{y})$ as bi-type branching processes without switching into the dormant state, we would have -- using the definition from the second case with $\sigma_1=0$ -- that \[S((\tilde{x},\tilde{y}),(x,y))=\max\lrset{r_1,r_2}\geq-(\kappa+\sigma).\]
In particular, for bi-type processes the invasion fitness is bounded from below by the total rate at which individuals exit the dormancy component. This lower bound is not reasonable for individuals which cannot become dormant.

\begin{example}\label{Example: No Competition}
We are not able to exclude the possibility of long-term coexistence in the sense that $\operatorname{sign}(S((\tilde{x},\tilde{y}),(x,y)))=-\operatorname{sign}(S((x,y),(\tilde{x},\tilde{y})))$. As an example we may choose $C=1$, $\tau=1.3$, $\delta=0.9$, $\kappa=0$, $\sigma=1$ and $p=0.23$. Then an explicit computation shows \[
S((2\delta,4\delta),(0,2\delta))\approx 0.22\quad\text{ and }\quad S((0,2\delta),(2\delta,4\delta))\approx 0.29. 
\]
In these cases, an invasion would lead to coexistence, which we will exclude from our main theorem.
\end{example}

Using the above definitions of the invasion fitness, we can state our convergence result, which is very similar to \cite[Theorem 2.1]{champagnat2019stochastic}.

\begin{theorem}\label{Theorem: Main Theorem}
	Let $\alpha\in(0,1)$, $\delta\in(0,4)$, $\tau\geq 0$, $p\in(0,\tfrac{1}{4})$, $\kappa\geq 0$ and $\sigma>0$ such that 
	$S((\tilde{x},\tilde{y}),(x,y))\neq 0$ for all $(x,y),(\tilde{x},\tilde{y})\in\calX$ with $(x,y)\neq(\tilde{x},\tilde{y})$. Further assume that the transitions are as in the beginning of this section and that the initial conditions \emph{(\ref{startingcondition1})} and \emph{(\ref{startingcondition2})} are satisfied.
	\begin{enumerate}[label=\emph{(\roman*)}]
		\item Then there exists a time $T_0>0$ such that the sequences $\beta_{m,n}^K(t)$ from (\ref{Definition: beta}) converge as $K\to\infty$ in probability in $L^\infty([0,T])$ for all $T<T_0$ towards a deterministic piecewise affine continuous function $t\mapsto\beta_{m,n}(t)$ such that $\beta_{m,n}(0)=(1-(m+n)\alpha)\vee 0$, which is characterized as follows.
		\item We define the sequence $s_k\geq 0$ and $(m_k^*,n_k^*)\in\{0,\ldots,L\}^2$ inductively: Set $s_0=0$ and $(m_1^*,n_1^*)=(0,0)$. Assume that for $k\geq 1$ we have constructed $s_{k-1}<T_0$ and $(m_k^*,n_k^*)$ and assume that $\beta_{m,n}(s_{k-1})\neq 0$ for some $(m\delta,n\delta)\in\calX$. Then we define \[
		s_k\coloneqq\inf\lrset{t>s_{k-1}\mid \exists(m,n)\neq (m_k^*,n_k^*),\ \beta_{m,n}(t)=\beta_{m_k^*,n_k^*}(t)}
		\]
		Using this definition, we can distinguish three cases:
		\begin{enumerate}[label=\emph{(\alph*)}]
			\item If $\beta_{m_k^*,n_k^*}(s_k)>0$ define \[
			(m_{k+1}^*,n_{k+1}^*)=\arg\max\limits_{(m,n)\neq (m_k^*,n_k^*)}\beta_{m,n}(s_k)
			\]
			if the argmax is unique. Otherwise we stop the induction and set $T_0=s_k$.
			\item If in case \emph{(a)} we have \[S((m_k^*\delta, n_k^*\delta),(m_{k+1}^*\delta,n_{k+1}^*\delta))<0\quad\text{ and }\quad S((m_{k+1}^*\delta,n_{k+1}^*\delta),(m_k^*\delta,n_k^*\delta))>0,\] and $(m_{k+1}^*+n_{k+1}^*)\delta< 6$, then we continue our induction. Otherwise set $T_0=s_k$ and stop the induction.
			\item If there exists some $(m,n)\in\{0,\ldots,L\}^2\setminus\{(m_k^*,n_k^*)\}$ such that $\beta_{m,n}(s_k)=0$ and $\beta_{m,n}(s_k-\veps)>0$ for all $\veps>0$ sufficiently small, then we also stop the induction and set $T_0=s_k$.
		\end{enumerate}
	\item The function $\beta_{m,n}$ is defined for $t\in[s_{k-1},s_k]$ as \[
	\beta_{0,0}(t)=\left[\1_{\beta_{0,0}(s_{k-1})>0}\lr{\beta_{0,0}(s_{k-1})+\int_{s_{k-1}}^{t}S((0,0),(m_k^*\delta,n_k^*\delta))\ds}\right]\vee 0
	\]
	and for $m\neq 0 $ or $n\neq 0$ we have \begin{align*}
	\beta_{m,n}(t)&=\lr{\beta_{m,n}(s_{k-1})+\int_{t_{(m,n),k}\wedge t}^{t}S((m\delta,n\delta),(m_k^*\delta,n_k^*\delta))\ds}\\
	&\quad \vee(\beta_{m-1,n}(t)-\alpha)\vee (\beta_{m,n-1}(t)-\alpha) \vee  0,
	\end{align*}
	where $\beta_{-1,n}\equiv \beta_{m,-1}\equiv 0$	and the time $t_{(m,n),k}$ is defined by \[
	t_{(m,n),k}\coloneqq\begin{cases}
	\inf\lrset{t\geq s_{k-1}\mid \beta_{m-1,n}(t)=\alpha\ \text{or } \beta_{m,n-1}(t)=\alpha},&\quad\text{if }\beta_{m,n}(s_{k-1})=0\\
	s_{k-1}&\quad\text{otherwise}.
	\end{cases}
	\]
	\end{enumerate}
\end{theorem}

The proof of this theorem will be discussed in Section \ref{Section: Proof}. In light of the convergence theorems derived in Appendix \ref{Section: Appendix B}, this result is not very surprising. The defined fitness function determines the rate of growth of the corresponding branching process in the same way that the largest eigenvalue of the mean matrix of a single or bi-type branching process does. 

Also, note that the fitness functions are constant on each time interval, so we may replace the integral by multiplying with the length of the integrated interval. We have chosen this representation to allow for a more direct comparison with \cite[Theorem 2.1]{champagnat2019stochastic}.

\begin{remark}
	We will shortly discuss the conditions, listed in part (ii) in our theorem, which lead to an end of the induction. \begin{enumerate}
		\item[(a)] At the time $s_k$ at least one new trait, other than the previously resident trait, becomes of order $K$ in the population. Hence, we want to ensure that the resulting competition between the different traits only occurs between two traits, so that we can apply our results from Appendix \ref{Section: Appendix C}. This condition requires at most two traits to be of size of order $K$ simultaneously.
		\item[(b)] As we have seen in Example \ref{Example: No Competition}, there is not necessarily competitive exclusion. The first requirement in this case ensures that the invading trait becomes resident, while the initially resident trait declines in size, so there is no coexistence. The second condition $(m_{k+1}^*+n_{k+1}^*)\delta<6$ is needed for the invading trait to have a positive equilibrium size.
		\item[(c)] If there is a trait, which is almost, but not fully, extinct at the time at which there is a change in the resident trait, we are not able to determine against which of the two traits of size $K$ there is competition. Therefore, we want to ensure each trait with small population size to be fully extinct at the time when a change in the resident trait occurs.
	\end{enumerate}
\end{remark}

%% file: Examples.tex
\section{Examples}\label{Section: Examples}
\subsection{Examples for limiting Functions in Theorem \ref{Theorem: Main Theorem}}
In this section we will consider specific, arbitrary choices of parameters for our model to find some range of resulting behaviours for the limiting functions $\beta_{m,n}$ established in Theorem \ref{Theorem: Main Theorem}. As we will see, the dynamics are already quite complicated in the case of very few traits. In particular, a full analysis in the case of $2\delta<4<3\delta$ as in \cite[Section 3]{champagnat2019stochastic} is not feasible. The main problem from our simulation appears to be the non-periodicity of our systems.

For all of the upcoming examples we choose $\delta=1.51$, $\tau=1.3$, $\kappa=0$, $\sigma=1$ and $\alpha=0.5$. We will vary the dormancy parameter $p$ which will show us plenty of qualitatively different results.
\vspace{-0.1cm}
\begin{example}[$p=0.21$]\label{Example: p=0.21}
For now we let $p=0.21$. Then we can plot the limiting function and obtain the following graphics.
\vspace{-0.2cm}
\begin{figure}[htbp]
	\centering
	\begin{minipage}[b]{0.49\linewidth}
		\includegraphics[width=1\textwidth]{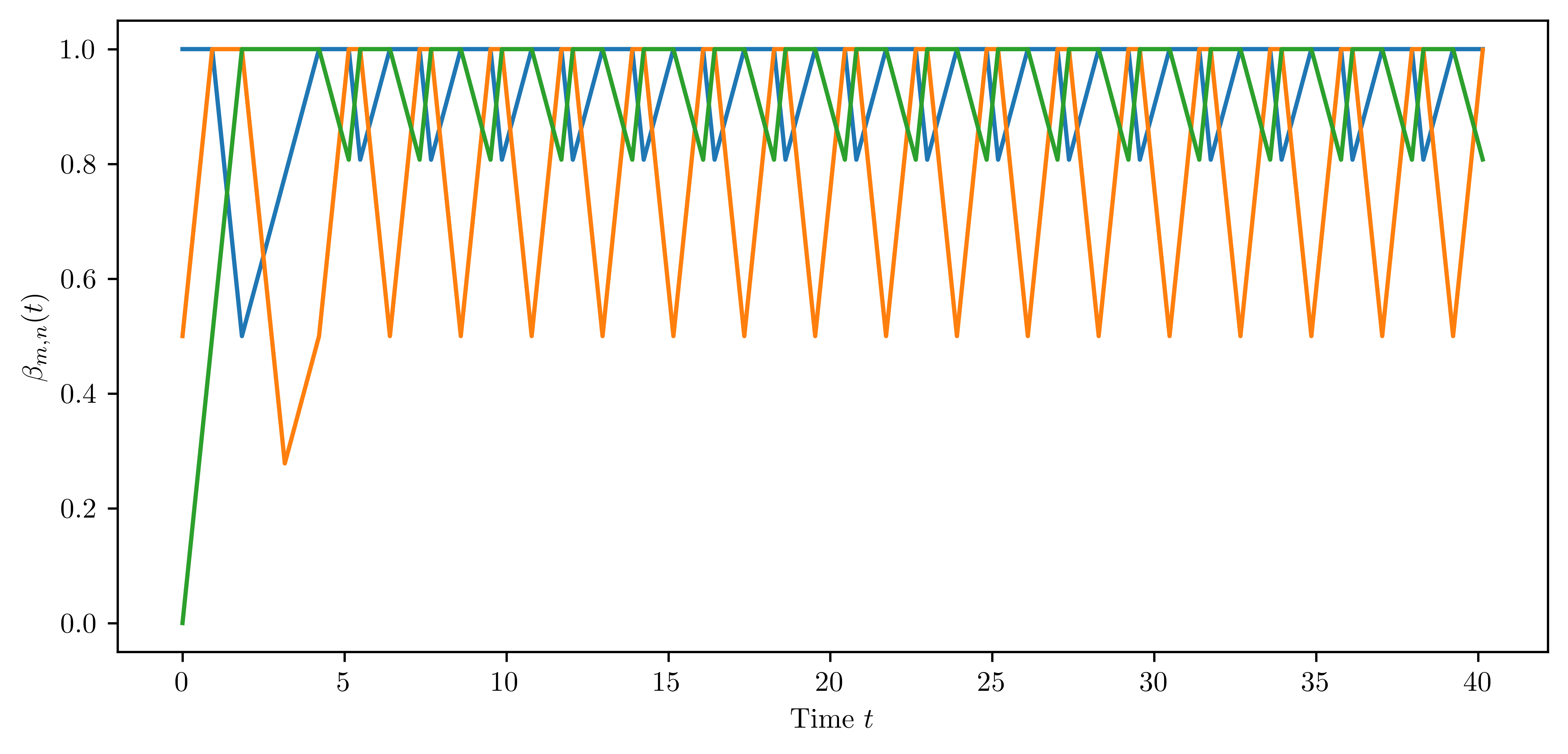}
	\end{minipage}
	\begin{minipage}[b]{0.49\linewidth}
		\includegraphics[width=1\textwidth]{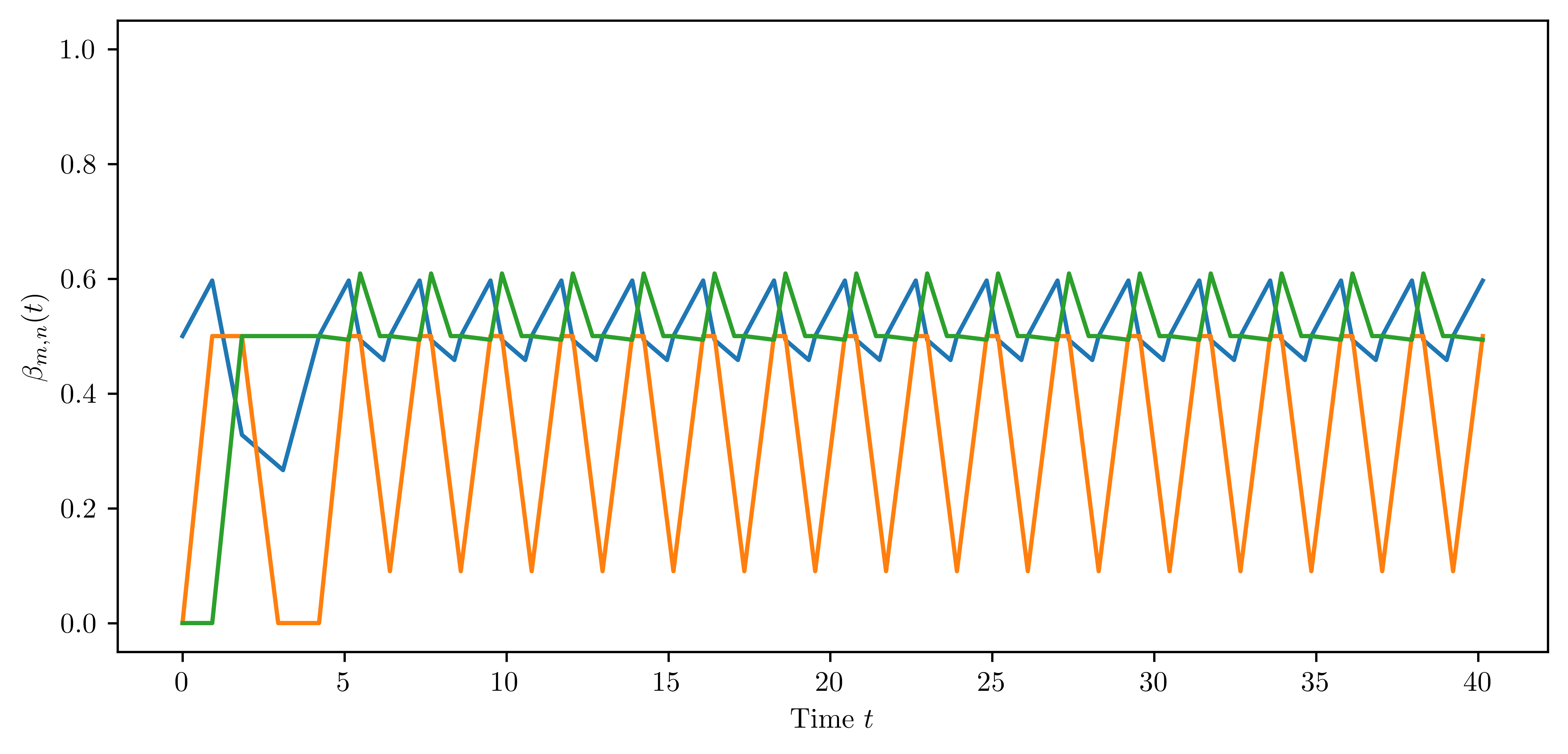}
	\end{minipage}
	\includegraphics[width=0.49\textwidth]{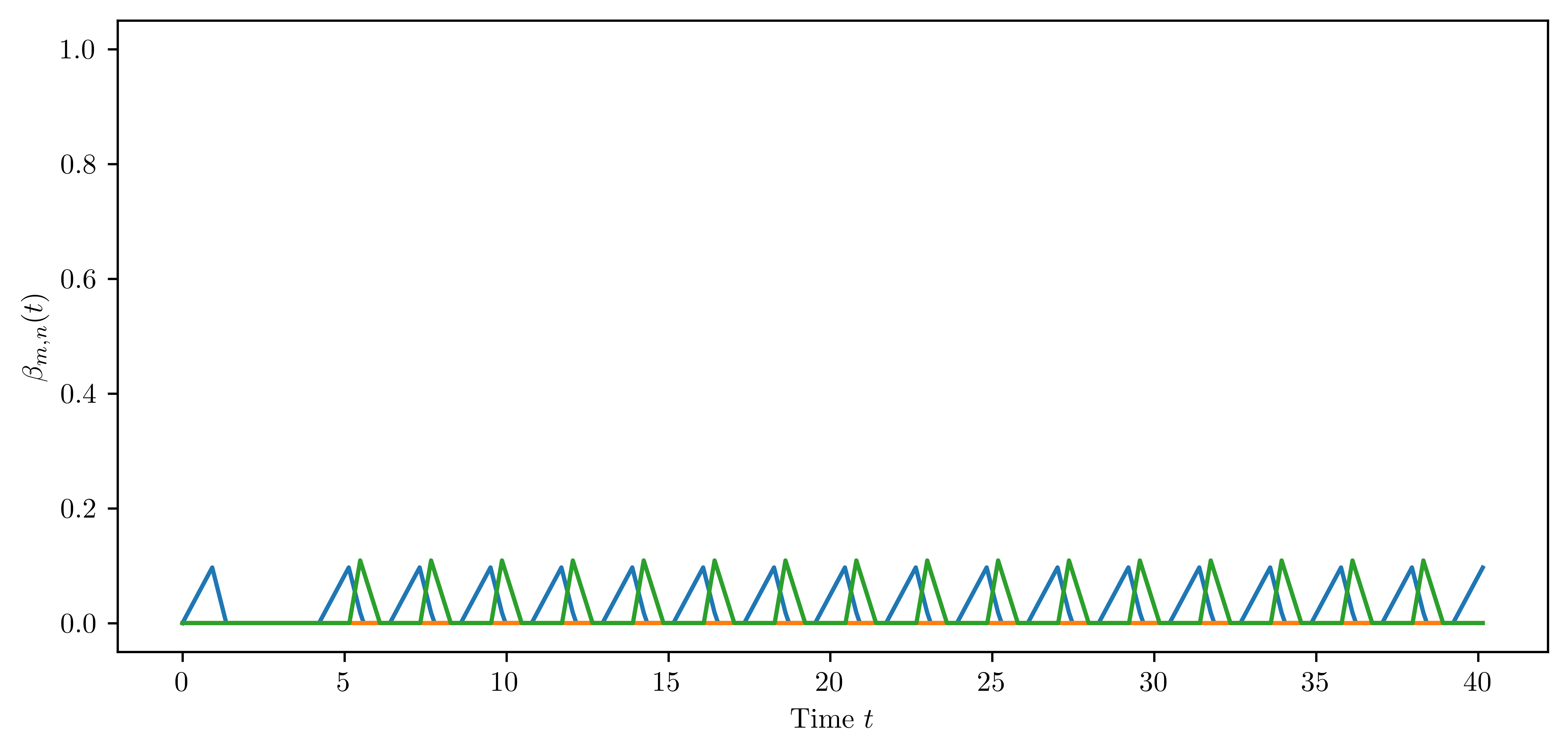}
	\caption{($p=0.21$) Top left: The dynamics of $\beta_{0,0}$ (blue), $\beta_{0,1}$ (orange) and $\beta_{0,2}$ (green). Top right: The dynamics of $\beta_{1,0}$ (blue), $\beta_{1,1}$ (orange) and $\beta_{1,2}$ (green). Bottom: The dynamics of $\beta_{2,0}$ (blue), $\beta_{2,1}$ (orange) and $\beta_{2,2}$ (green).}
\end{figure}

In this case we see that a similar behaviour as in \cite{champagnat2019stochastic} is recovered: all traits exhibit (almost) periodicity. This stems from the fact that the traits with a dormancy component are not sufficiently fit. While the trait $(\delta,0)$ is fit against the trait $(0,0)$ and the trait $(\delta,2\delta)$ is fit against the trait $(0,\delta)$, especially during the times when the trait $(0,2\delta)$ is resident, all dormancy traits have a negative fitness and are only kept alive through the incoming migration. Hence, the essential components of the dynamics can be reduced to the case without dormancy. 
\end{example}

\begin{example}[$p=0.22$]\label{Example: Coexistence}
In this case, the resulting dynamics are given in the figure below.
\vspace{-0.2cm}
\begin{figure}[h!]
	\centering
	\begin{minipage}[b]{0.49\linewidth}
		\includegraphics[width=1\textwidth]{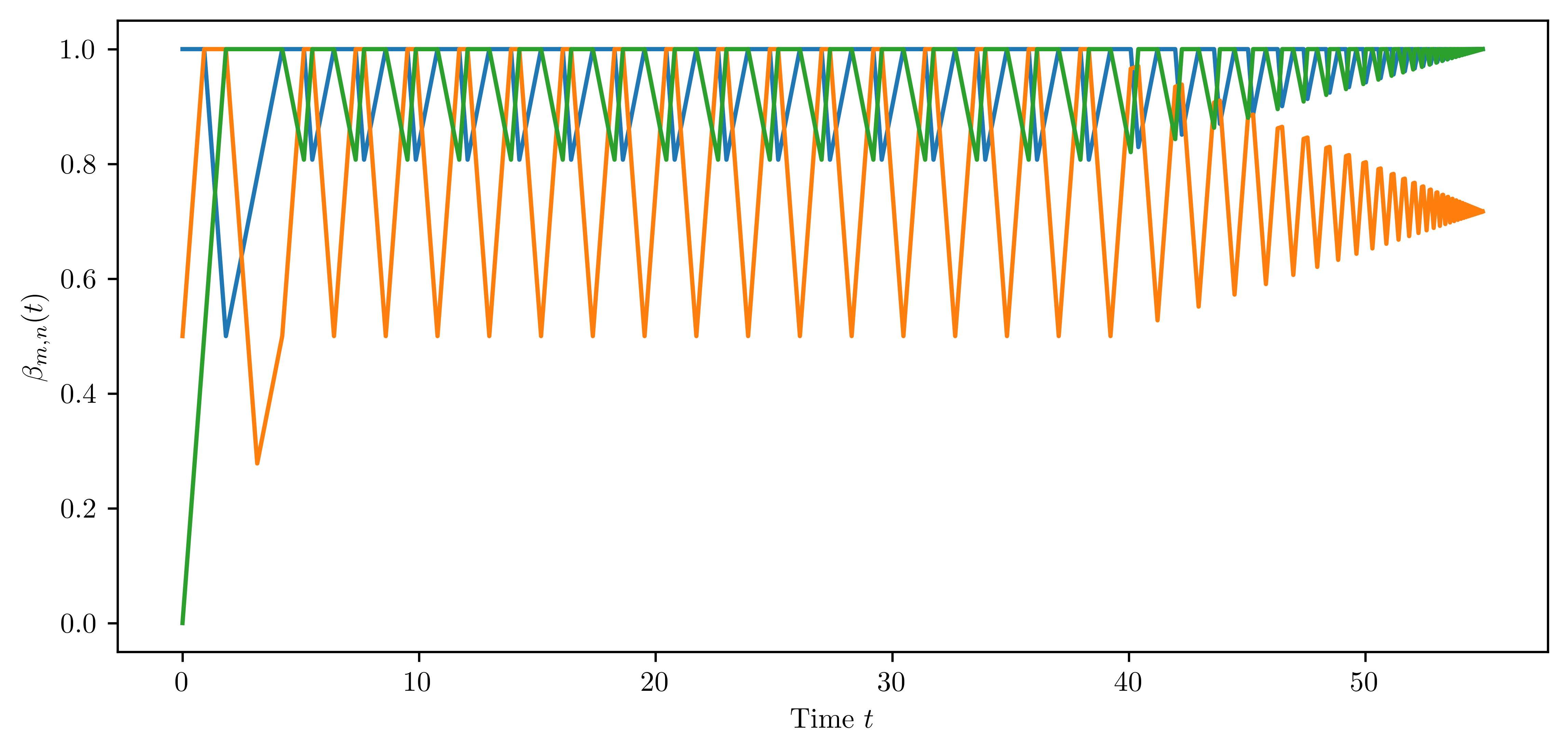}
	\end{minipage}
	\begin{minipage}[b]{0.49\linewidth}
		\includegraphics[width=1\textwidth]{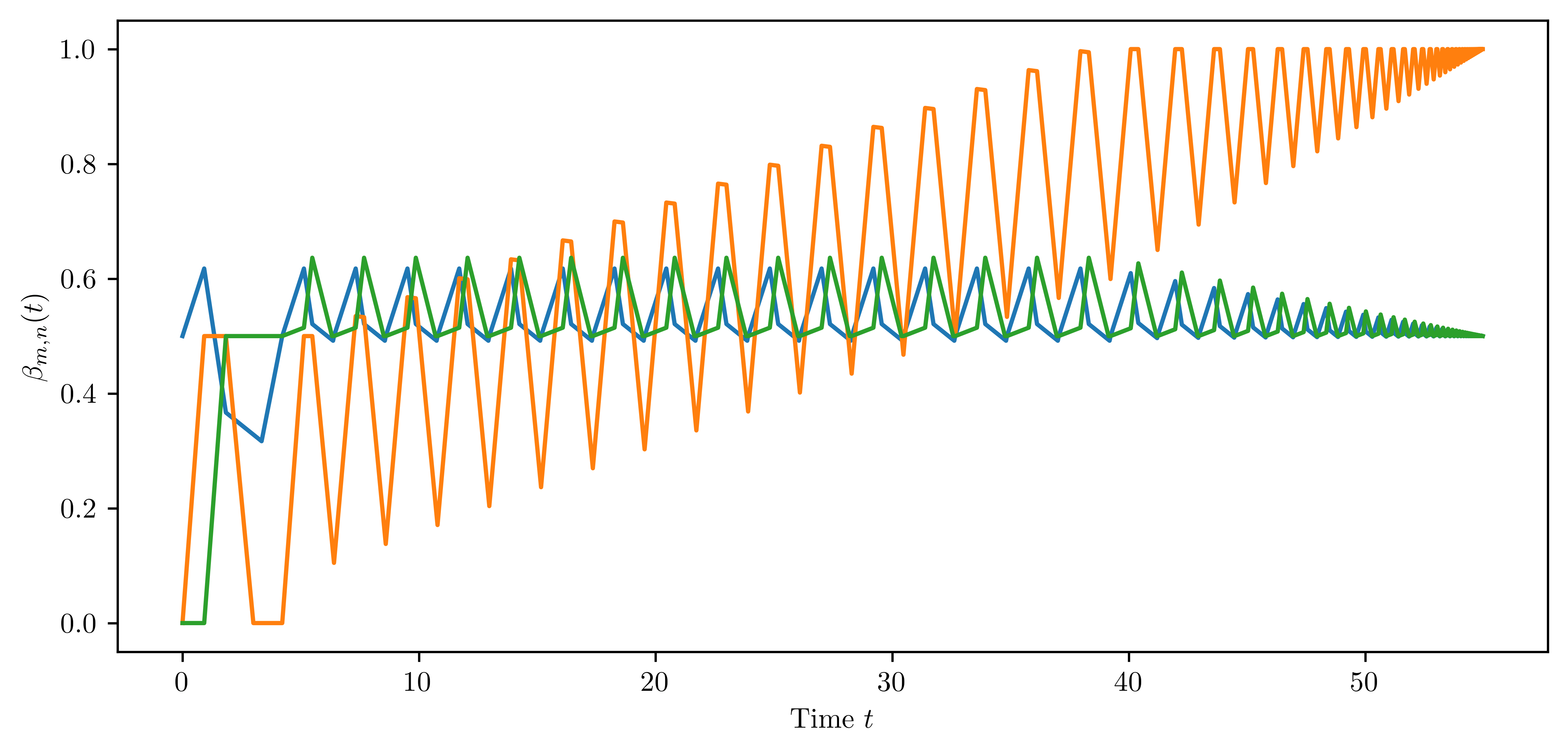}
	\end{minipage}
	\includegraphics[width=0.49\textwidth]{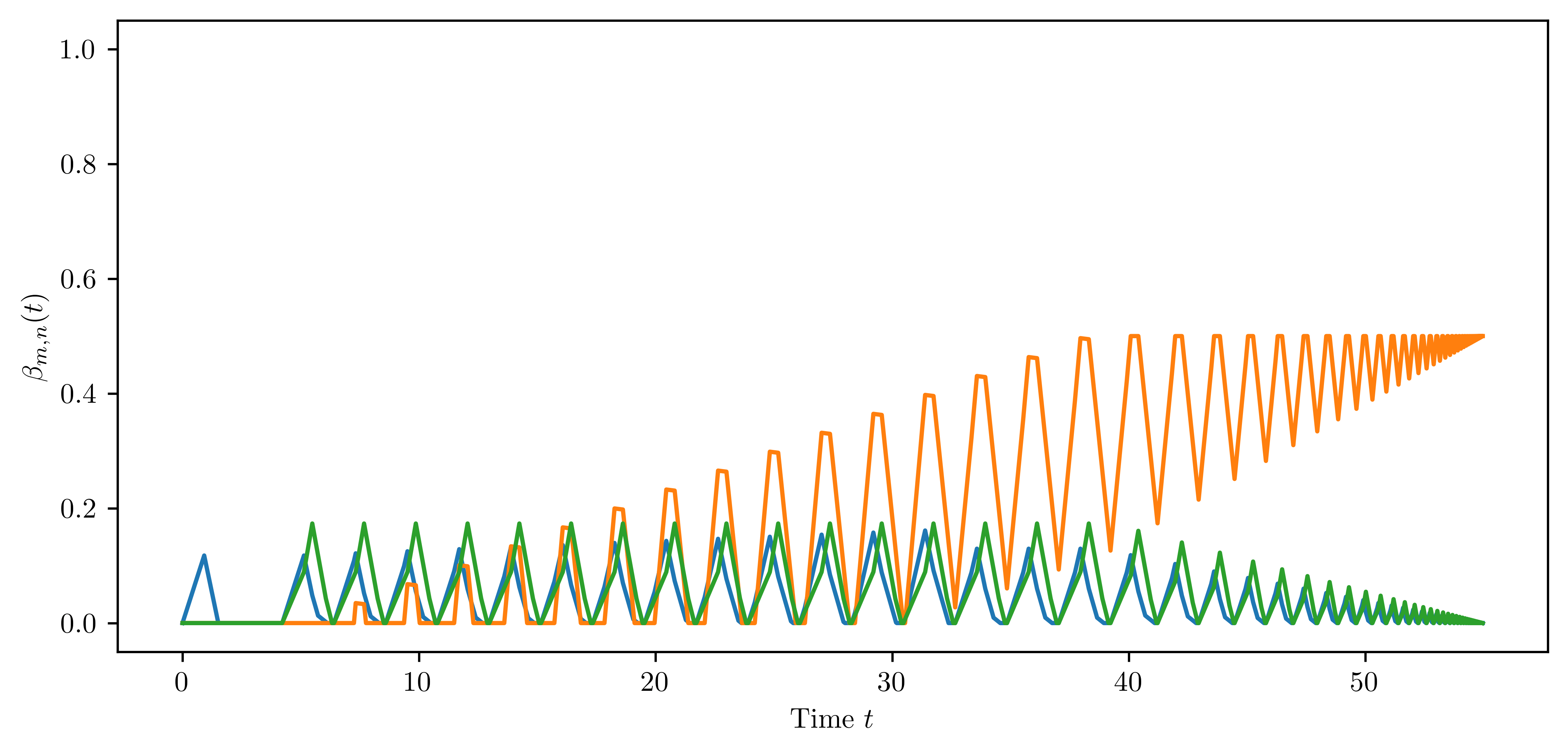}
	\caption{($p=0.22$) Top left: The dynamics of $\beta_{0,0}$ (blue), $\beta_{0,1}$ (orange) and $\beta_{0,2}$ (green). Top right: The dynamics of $\beta_{1,0}$ (blue), $\beta_{1,1}$ (orange) and $\beta_{1,2}$ (green). Bottom: The dynamics of $\beta_{2,0}$ (blue), $\beta_{2,1}$ (orange) and $\beta_{2,2}$ (green).}
\end{figure}

Here, two phases are to be distinguished: At first, we observe a very similar behaviour as in the case $p=0.21$. In fact, for the traits $(0,\ell\delta)$, $\ell\in\{0,1,2\}$, the functions are at first identical to the previous case. However, the trait $(\delta,\delta)$ is now sufficiently fit that its population size overall increases with each cycle until at one point it becomes resident. From this point onwards, we see that the functions are approaching a coexistence limit in the sense that for all $k$ we have $s_{k+1}>s_k$ with $\textstyle\lim_{k\to\infty}s_k<\infty$ and \[
\lim\limits_{k\to\infty}\beta_{0,0}(s_k)=\lim\limits_{k\to\infty}\beta_{0,2}(s_k)=\lim\limits_{k\to\infty}\beta_{1,1}(s_k)=1.
\]
We will prove this claim below. Thus, although we have excluded the possibility of coexistence of any traits in the formulation of our Theorem \ref{Theorem: Main Theorem} by demanding that the fitness functions need to have opposite signs, the system converges to an equilibrium. The reason behind this is the fact that we have demanded opposite signs, but the absolute values of the relative fitnesses of two traits are not necessarily, and often will not be, the same. This allows traits with dormancy to experience a large growth while they are not resident and fit against the dominant trait, but only a slow decline in population size when they are unfit against the dominant trait. In \cite{champagnat2019stochastic} the fitness functions are antisymmetric functions in the sense that $S(x,y)=-S(y,x)$ for traits $x,y$ and therefore such behaviour cannot be observed. The traits $(2\delta,\ell\delta)$ are again only driven by immigration through mutations.

We will now show inductively that the sequence $(s_k)_{k\in\N}$ converges by considering the system where there are only the traits $(0,0)$, $(\delta,\delta)$ and $(0,2\delta)$. This reduction is justified by our simulations above, since all other traits become of order $o(K)$ after time $40$. Further we assume the initial condition of our reduced system to be \[
\beta_{0,0}(0)=1,\quad \beta_{1,1}(0)=1,\quad\text{and}\quad \beta_{0,2}(0)=x_0\in(0,1),
\]
that is, we assume that at the starting point of the system, the trait $(\delta,\delta)$ has just become resident in the population which is only possible, if the trait $(0,0)$ has been previously resident. In particular, the trait $(0,2\delta)$ is unfit against the trait $(0,0)$ and therefore must be of order $o(K)$. We will now construct a sequence of intermediate times until a similar configuration with $\beta_{0,0}(t)=\beta_{1,1}(t)=1$ and $\beta_{0,2}(t)=x_1>x_0$ is reached as is displayed in Figure \ref{Fig: Illustration}. We calculate the individual fitnesses as determined by the fitness function. We obtain \begin{alignat*}{4}
 &S((\delta,\delta),(0,0))&&=\frac{-\delta+\tau-\sigma+\sqrt{(\tau-\delta+\sigma)^2+12p\delta}}{2}\quad && S((0,2\delta),(0,0))&&=\tau-\delta\\
	&S((0,0),(\delta,\delta))&&=3-\frac{3-\delta}{1-p\delta}-\tau\quad && S((0,2\delta),(\delta,\delta))&&=3-\delta-\frac{3-\delta}{1-p\delta}+\tau\\
	&S((\delta,\delta),(0,2\delta))&&=\frac{-\tau-\sigma+\sqrt{(\sigma-\tau)^2+4p\delta(3-\delta)}}{2}\quad && S((0,0),(0,2\delta))&&=\delta-\tau.
\end{alignat*}

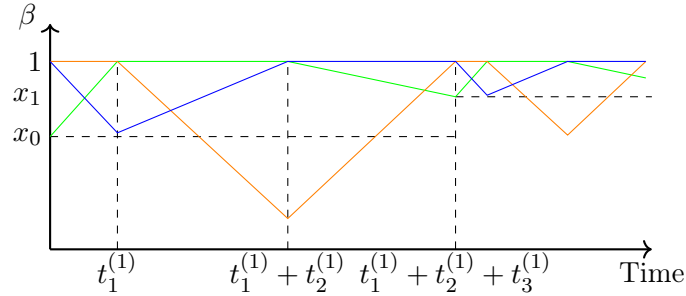
\begin{figure}[htbp]
	\begin{tikzpicture}
	\draw[->,thick](0,1.5)--(8,1.5);
	\draw[->,thick](0,1.5)--(0,4.5);
	\node at (-0.3,3) {$x_0$};
	\node at (-0.3,3.53) {$x_1$};
	\node at (8,1.2) {Time};
	\node at (-0.3,4.6) {$\beta$};
	\node at (-0.2,4) {$1$};
	\node at (0.895,1.2) {$t_1^{(1)}$};
	\node at (3.16,1.2) {$t_1^{(1)}+t_2^{(1)}$};
	\node at (5.39,1.2) {$t_1^{(1)}+t_2^{(1)}+t_3^{(1)}$};
	\draw[dashed] (0,3) -- (5.39,3);
	\draw[dashed] (0.895,1.5) -- (0.895,4);
	\draw[dashed] (3.16,1.5) -- (3.16,4);
	\draw[dashed] (5.39,1.5) -- (5.39,4);
	\draw[dashed] (5.39,3.53) -- (8,3.53);
	\draw[green] (0,3) -- (0.895,4) -- (3.16,4) -- (5.39,3.53) -- (5.81,4) -- (6.88,4) -- (7.92,3.78);
	\draw[orange] (0,4) -- (0.895,4) -- (3.16,1.91) -- (5.39,4) -- (5.81,4)-- (6.88,3.02) -- (7.92, 4);
	\draw[blue] (0,4) -- (0.895,3.05) -- (3.16,4) -- (5.39,4) -- (5.81, 3.55)-- (6.88,4) -- (7.92,4);
	\end{tikzpicture}
	\caption{Illustration for the successive construction of the times $t_i^{(1)}$. Blue represents $\beta_{0,0}$, orange is $\beta_{1,1}$ and green is $\beta_{0,2}$.}
	\label{Fig: Illustration}
\end{figure}
We can therefore explicitly calculate that in this system the trait $(0,2\delta)$ becomes resident after time \[
t_1^{(1)}\coloneqq\frac{1-x_0}{S((0,2\delta),(\delta,\delta))}= C_1(1-x_0).
\]
At this time, we have the sizes \[
\beta_{0,0}(t_1^{(1)})=1+t_1^{(1)}\cdot S((0,0),(\delta,\delta)),\quad \beta_{1,1}(t_1^{(1)})=1\quad\text{ and }\quad \beta_{0,2}(t_1^{(1)})=1.
\]
In the next step, the traits are competing with $(0,2\delta)$. Therefore the trait $(0,0)$ becomes resident after time \[
t_2^{(1)}\coloneqq\frac{1-(1+t_1^{(1)}\cdot S((0,0),(\delta,\delta)))}{S((0,0),(0,2\delta))}=-\frac{S((0,0),(\delta,\delta))}{S((0,0),(0,2\delta))}\cdot t_1^{(1)}= C_2(1-x_0).
\]
We obtain \[
\beta_{0,0}(t_1^{(1)}+t_2^{(1)})=1, \quad \beta_{1,1}(t_1^{(1)}+t_2^{(1)})=1+t_2^{(1)}\cdot S((\delta,\delta),(0,2\delta))\quad\text{ and }\quad \beta_{0,2}(t_1^{(1)}+t_2^{(1)})=1.
\]
The third phase of this system consists of competition of the other traits with $(0,0)$. In this case, the trait $(\delta,\delta)$ becomes resident again after time \[
t_3^{(1)}\coloneqq\frac{1-(1+t_2^{(1)}\cdot S((\delta,\delta),(0,2\delta)))}{S((\delta,\delta),(0,0))}=-\frac{S((\delta,\delta),(0,2\delta))}{S((\delta,\delta),(0,0))}\cdot t_2^{(1)}= C_3(1-x_0).
\]
We can calculate that \[
\beta_{0,0}(t_1^{(1)}+t_2^{(1)}+t_3^{(1)})=1,\quad\beta_{1,1}(t_1^{(1)}+t_2^{(1)}+t_3^{(1)})=1\quad\text{and}\quad\beta_{0,2}(t_1^{(1)}+t_2^{(1)}+t_3^{(1)})=1+t_3^{(1)}\cdot S((0,2\delta),(0,0)).
\]
In particular, we recover our starting condition after time $t_1^{(1)}+t_2^{(1)}+t_3^{(1)}$ where $S((0,2\delta),(0,0))<0$ implies \begin{align*}
\beta_{0,2}(t_1^{(1)}+t_2^{(1)}+t_3^{(1)})=1-C(1-x_0)=:x_1,
\end{align*}
with $C\in(0,1)$. Repeating this process inductively shows that after the $n$-th such cycle, we obtain the condition \[
\beta_{0,0}\lr{\sum_{k=1}^{n}t_1^{(k)}+t_2^{(k)}+t_3^{(k)}}=1,\quad \beta_{1,1}\lr{\sum_{k=1}^{n}t_1^{(k)}+t_2^{(k)}+t_3^{(k)}}=1\]
and \[ \beta_{0,2}\lr{\sum_{k=1}^{n}t_1^{(k)}+t_2^{(k)}+t_3^{(k)}}=1-(1-x_0)C^n=:x_n.
\]
Thus, as $n\to\infty$, the functions converge to $1$ at the endpoint of each cycle. It remains to show, that the time steps are summable. Indeed, we find \begin{align*}
\sum_{k=1}^{n}t_1^{(k)}+t_2^{(k)}+t_3^{(k)}&=\sum_{k=1}^{n}C_1\cdot (1-x_k)+C_2\cdot (1-x_k)+C_3\cdot (1-x_k)\\
&=(C_1+C_2+C_3)\sum_{k=1}^{n}(1-x_0)C^k,
\end{align*}
which converges as $n\to\infty$. Therefore, we have proven that our choice of parameters leads to coexistence after finite time.

\end{example}

\begin{example}[$p=0.23$]
	We obtain the following functions.
	\begin{figure}[htbp]
		\centering
		\begin{minipage}[b]{0.49\linewidth}
			\includegraphics[width=1\textwidth]{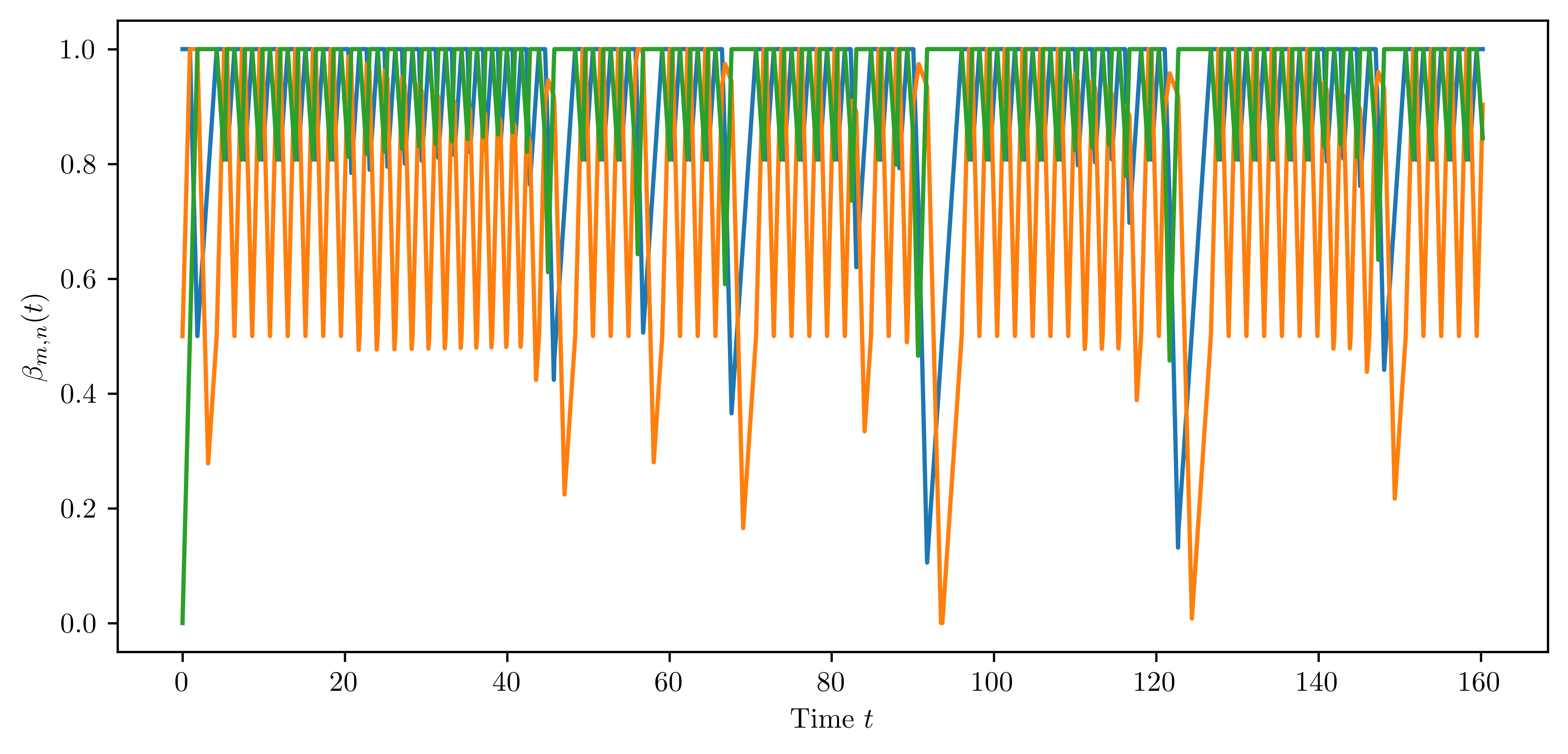}
		\end{minipage}
		\begin{minipage}[b]{0.49\linewidth}
			\includegraphics[width=1\textwidth]{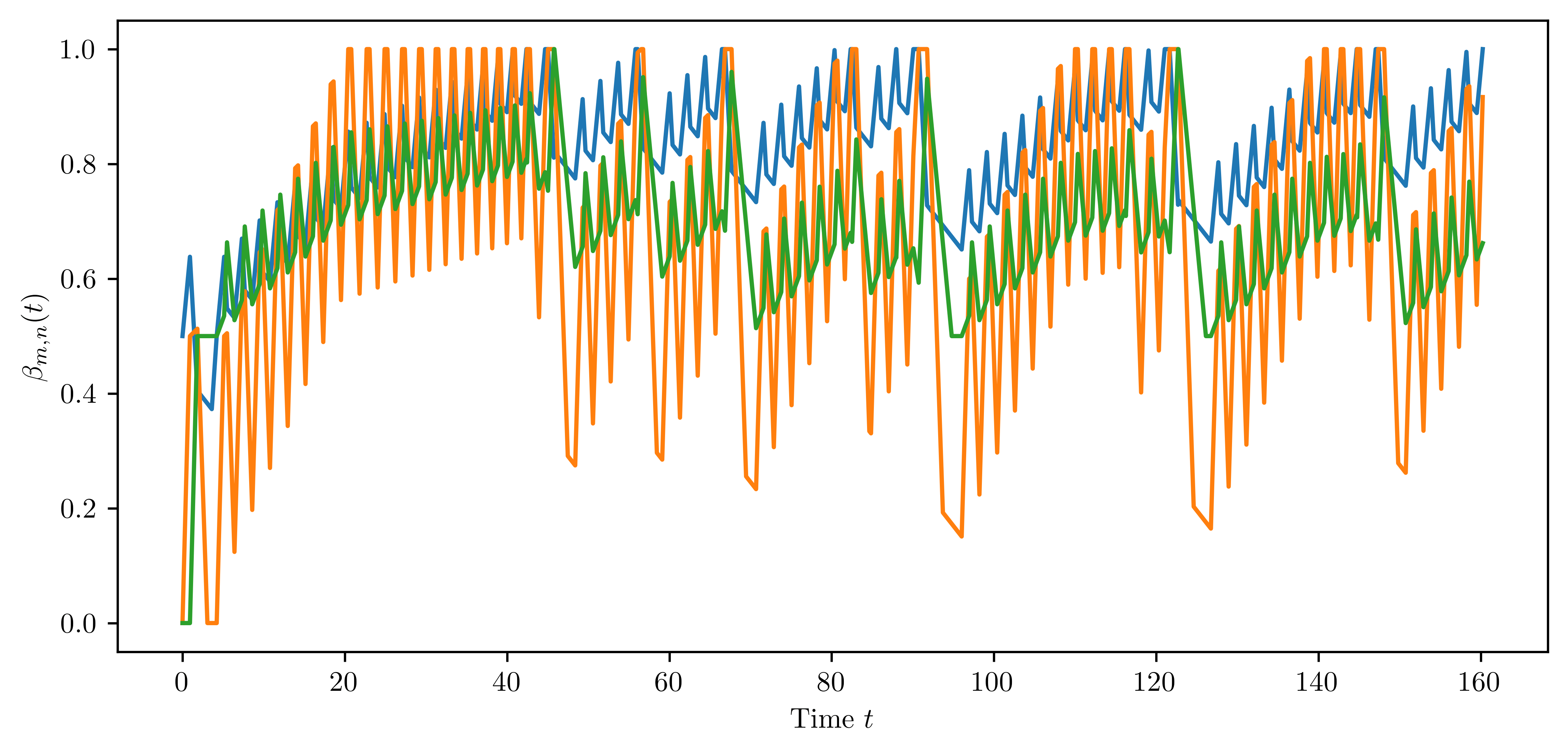}
		\end{minipage}
		\includegraphics[width=0.49\textwidth]{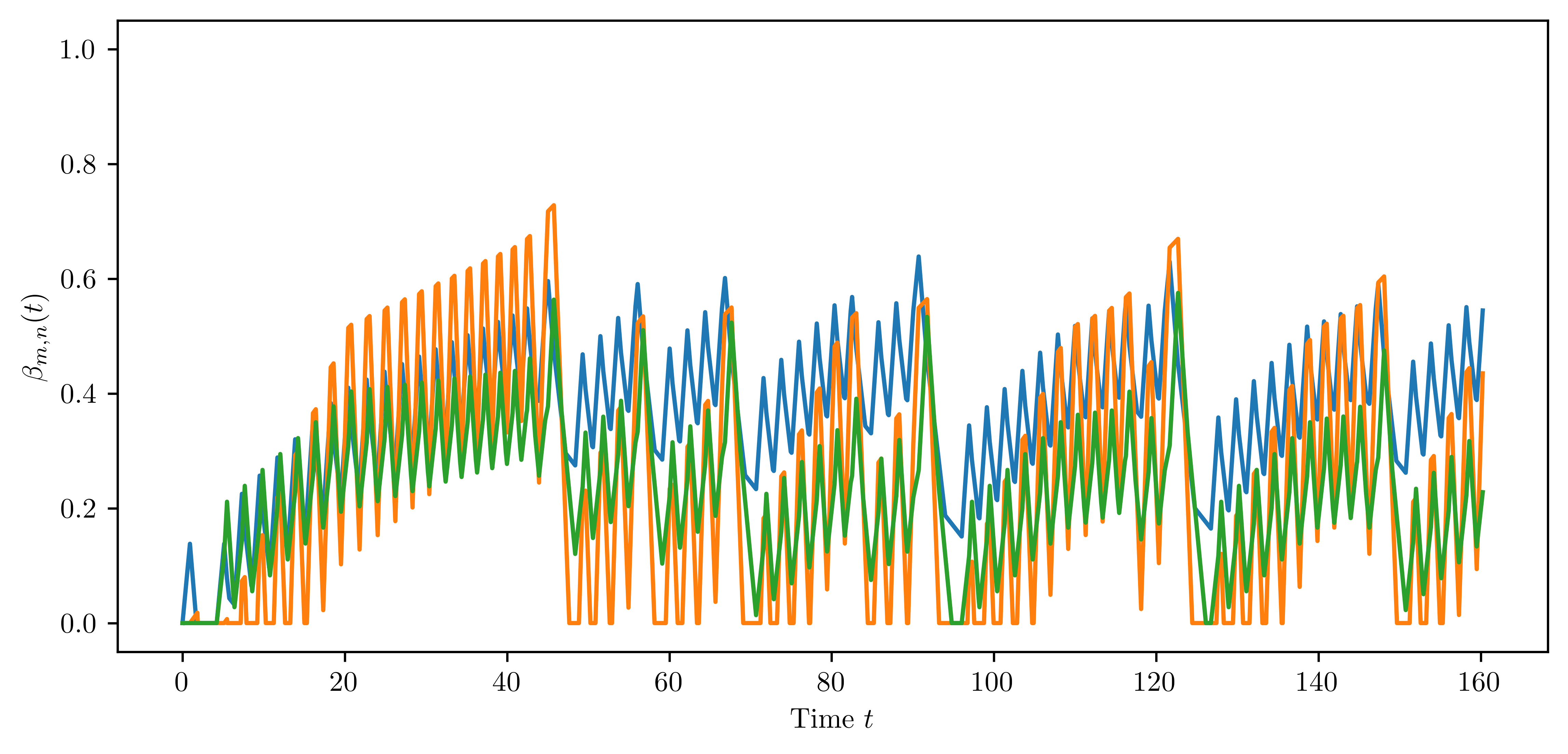}
		\caption{($p=0.23$) Top left: The dynamics of $\beta_{0,0}$ (blue), $\beta_{0,1}$ (orange) and $\beta_{0,2}$ (green). Top right: The dynamics of $\beta_{1,0}$ (blue), $\beta_{1,1}$ (orange) and $\beta_{1,2}$ (green). Bottom: The dynamics of $\beta_{2,0}$ (blue), $\beta_{2,1}$ (orange) and $\beta_{2,2}$ (green).}
	\end{figure}

Here, there are even more phases to distinguish: At first we have the growth phase of the trait $(\delta,\delta)$ until it becomes resident for the first time shortly after time $20$. Note that, due to the increased value of $p$, the traits $(\delta,0)$ and $(\delta,2\delta)$ have an increased fitness as well and are slightly increasing in size. Now the larger value of $p$ increases the equilibrium population size of trait $(\delta,\delta)$, which in turn implies that while $(\delta,\delta)$ is resident, the traits $(0,0)$ and $(0,2\delta)$ have a lower fitness. Thus the times for which the traits $(0,0)$ and $(0,2\delta)$ are resident will be prolonged slightly. This is sufficient for the trait $(\delta,0)$ (which only has a positive fitness while the trait $(0,0)$ is resident) to become resident for the first time around time $43$. Since the advantage from dormancy is not large enough to give the traits $(0,\ell\delta)$ an overall negative fitness, we then have alternating times during which the traits $(\delta,\ell\delta)$ are growing and the traits $(0,\ell\delta)$ are cyclically resident followed by a short phase where one or more of the traits $(\delta,\ell\delta)$ become resident and the traits $(0,\ell\delta)$ experience a short but sharp decline in size. We do not know if the functions become periodic eventually, however from simulations we conjecture that this is not necessarily the case.
\end{example}

\begin{example}[$p=0.234$]
	Here, we observe an interesting change in the dynamics. Due to the increased fitness of the dormant individuals, it takes a shorter amount of time until one of the traits with dormancy becomes resident. In addition, they are able to stay resident for longer periods. Since the trait $(0,0)$ is unfit against both $(\delta,0)$ and $(\delta,\delta)$, it has an overall lower fitness. Other than in the case $p=0.23$, the trait $(\delta,2\delta)$ does not become resident fast enough to prevent $(0,0)$ from becoming extinct. Once $(0,0)$ is extinct, it cannot be resurrected since there are no incoming mutations. Hence we see a significant change.  The dormant traits are not yet strong enough to prevent the trait $(0,2\delta)$ from becoming resident. After $(0,2\delta)$ is resident, all traits with dormancy become extinct or are only kept alive due to incoming mutations.
	
	\begin{figure}[htbp]
		\centering
		\begin{minipage}[b]{0.49\linewidth}
			\includegraphics[width=1\textwidth]{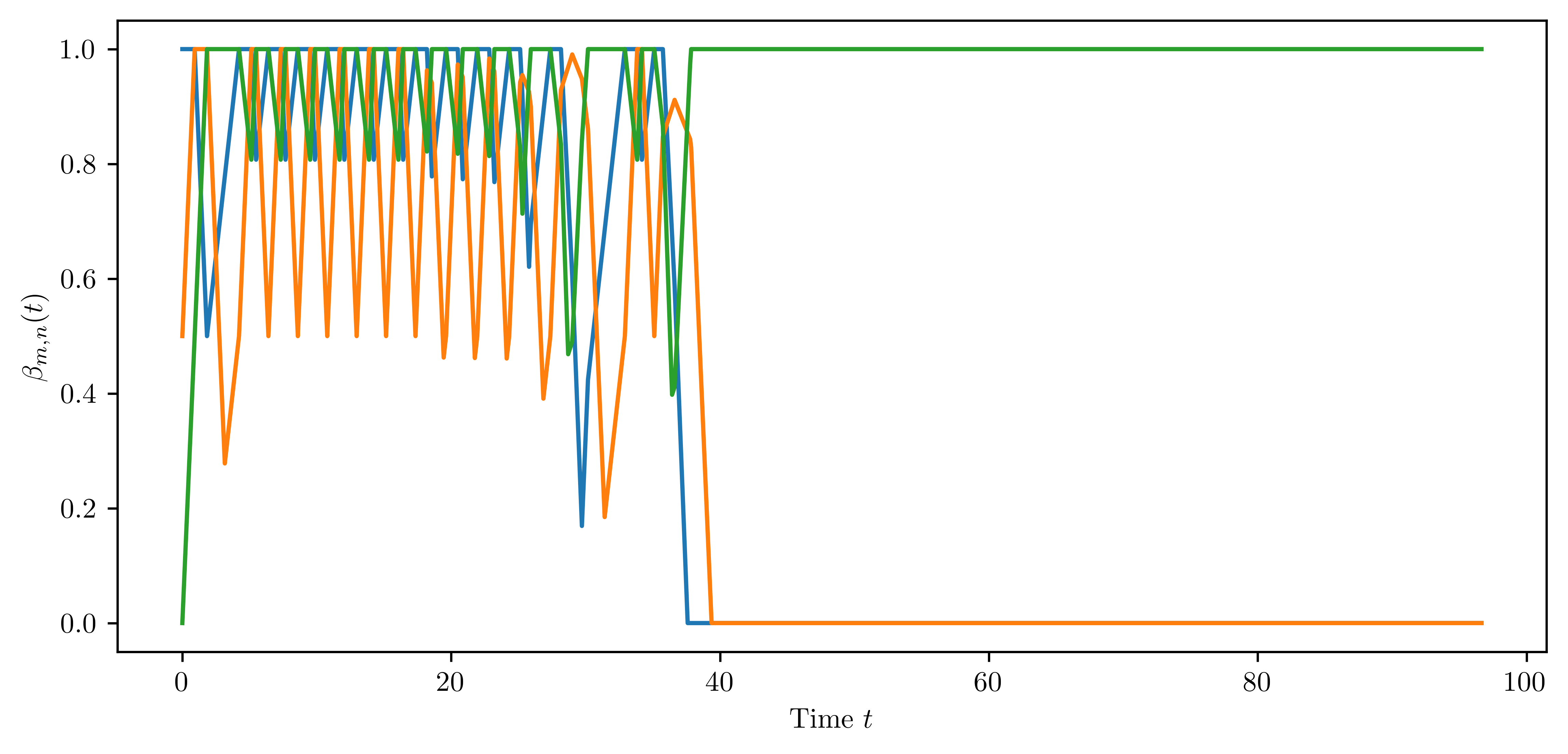}
		\end{minipage}
		\begin{minipage}[b]{0.49\linewidth}
			\includegraphics[width=1\textwidth]{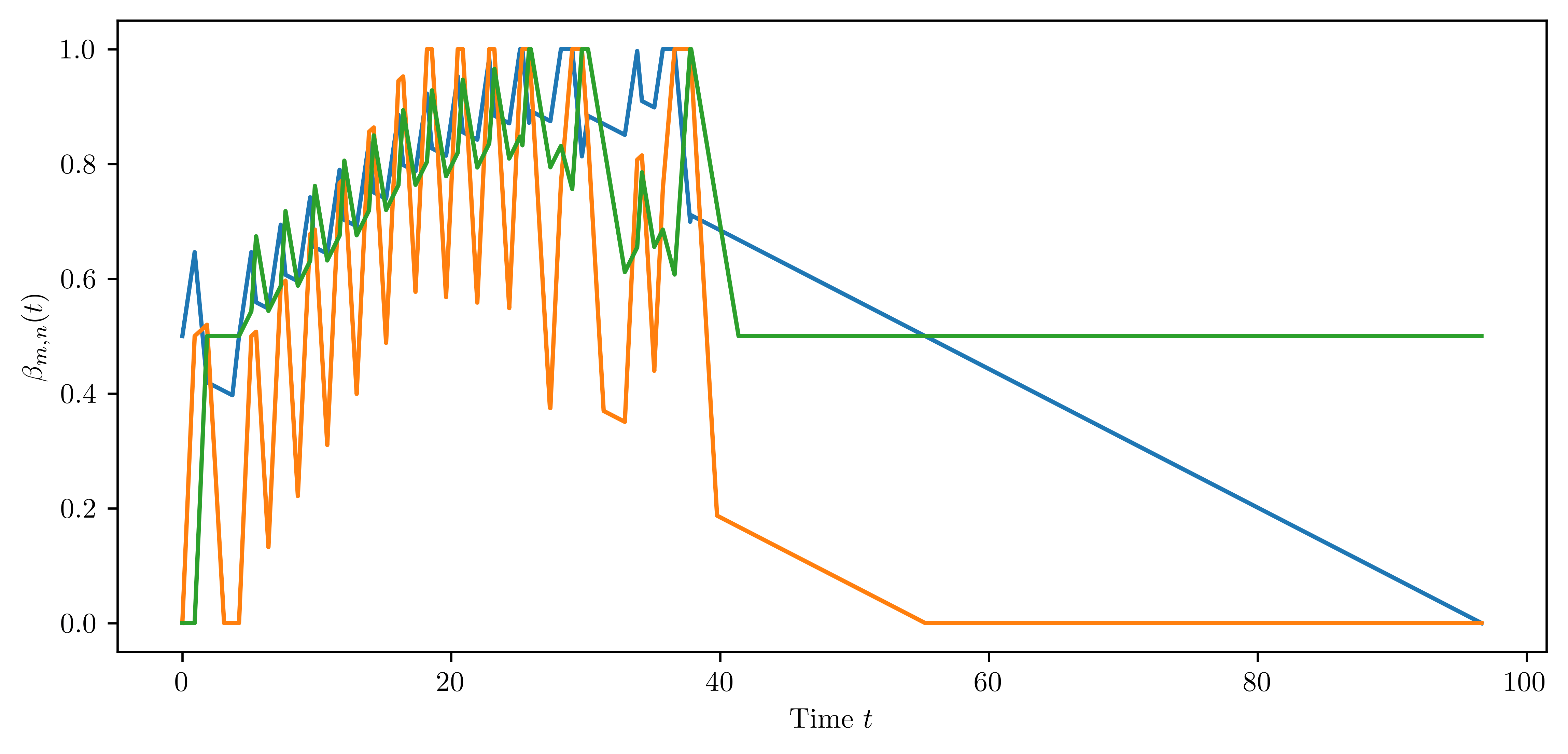}
		\end{minipage}
		\includegraphics[width=0.49\textwidth]{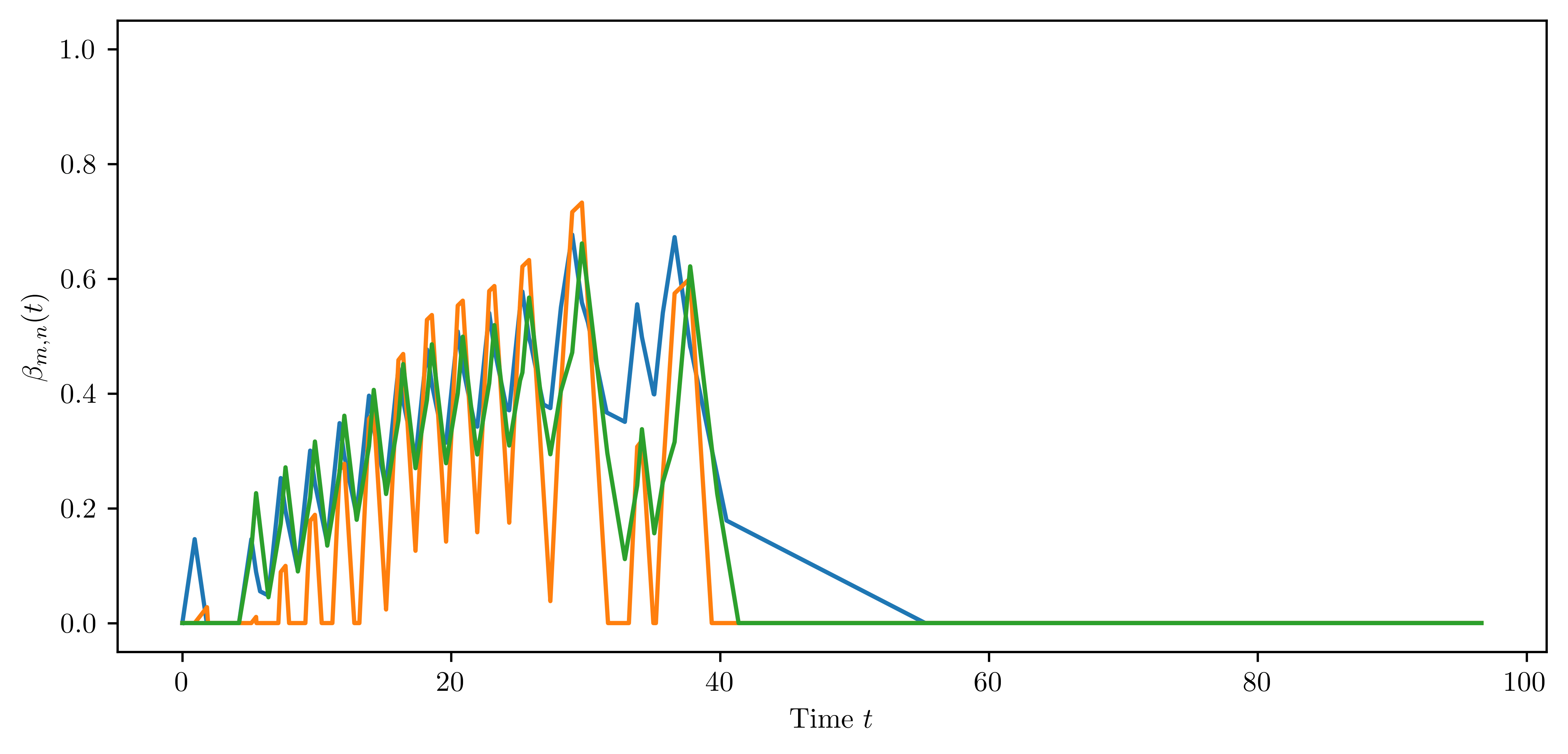}
		\caption{($p=0.234$) Top left: The dynamics of $\beta_{0,0}$ (blue), $\beta_{0,1}$ (orange) and $\beta_{0,2}$ (green). Top right: The dynamics of $\beta_{1,0}$ (blue), $\beta_{1,1}$ (orange) and $\beta_{1,2}$ (green). Bottom: The dynamics of $\beta_{2,0}$ (blue), $\beta_{2,1}$ (orange) and $\beta_{2,2}$ (green).}
	\end{figure}

\end{example}

\begin{example}[$p=0.24$]
	In this case, the dormancy is sufficiently strong such that the overall fitness of the non-dormant traits is negative when the traits $(\delta,\ell\delta)$ become resident. Thus we are getting again coexistence as in Example \ref{Example: Coexistence}, but now between the three traits $(\delta,\ell\delta)$. 
	\begin{figure}[h]
		\centering
		\begin{minipage}[b]{0.49\linewidth}
			\includegraphics[width=1\textwidth]{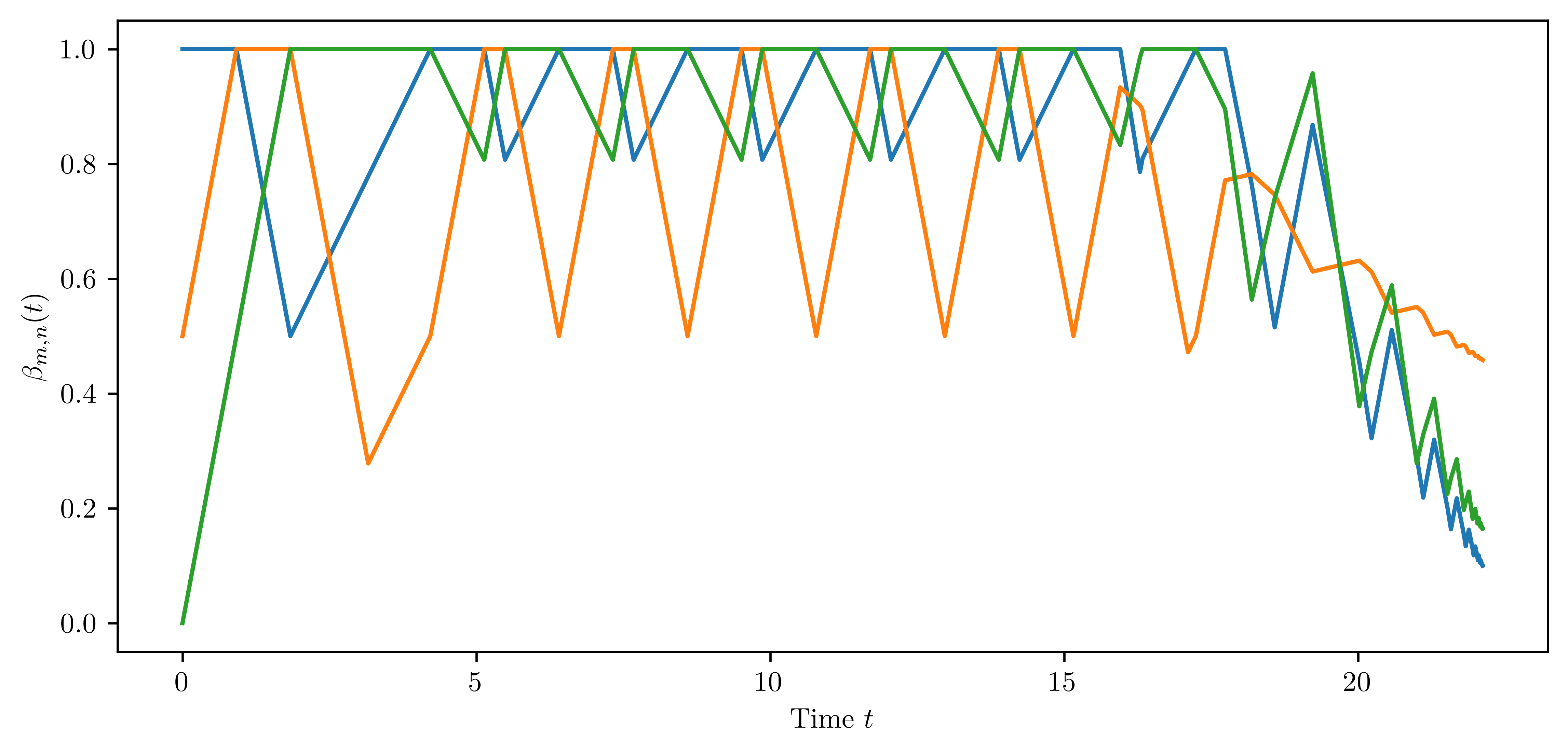}
		\end{minipage}
		\begin{minipage}[b]{0.49\linewidth}
			\includegraphics[width=1\textwidth]{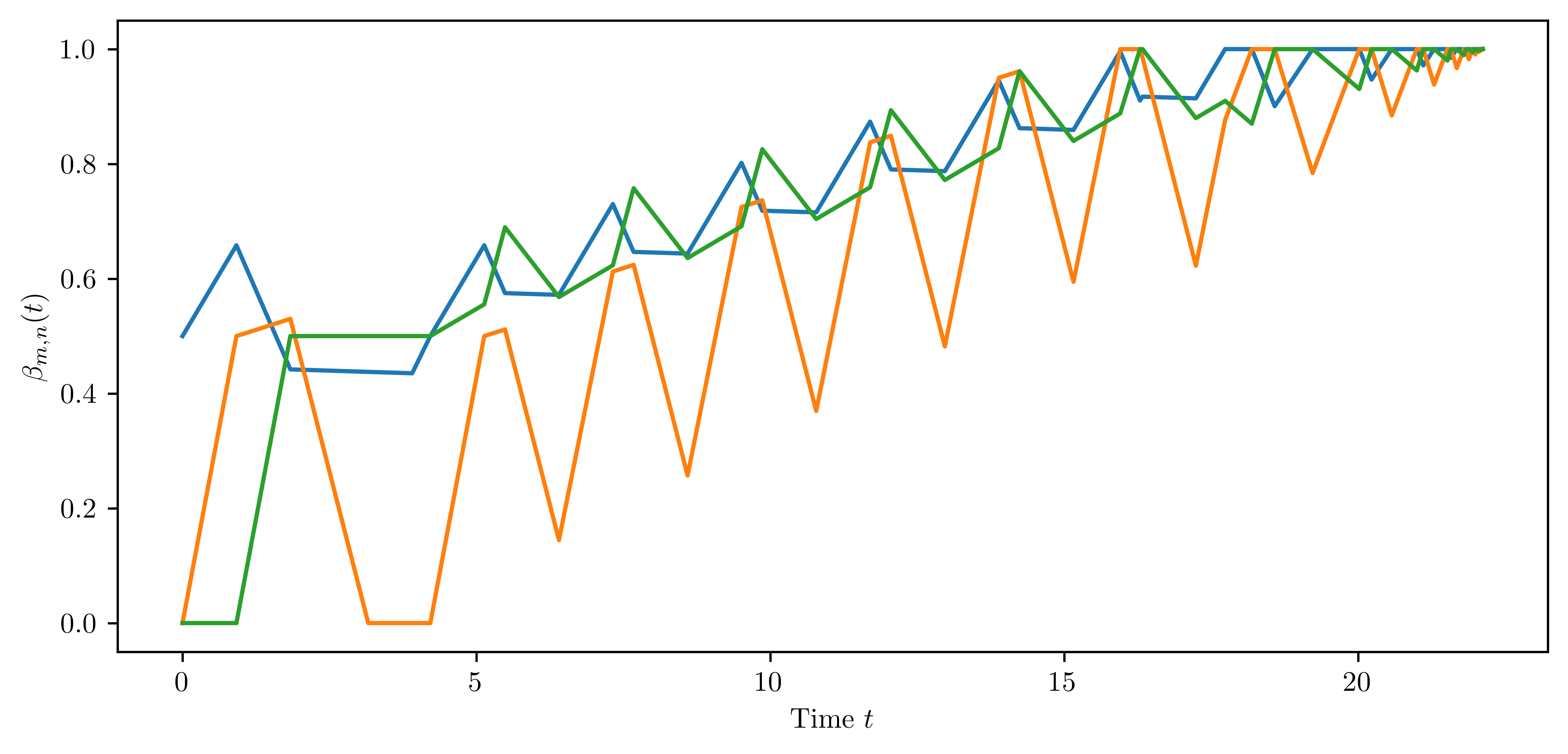}
		\end{minipage}
		\includegraphics[width=0.49\textwidth]{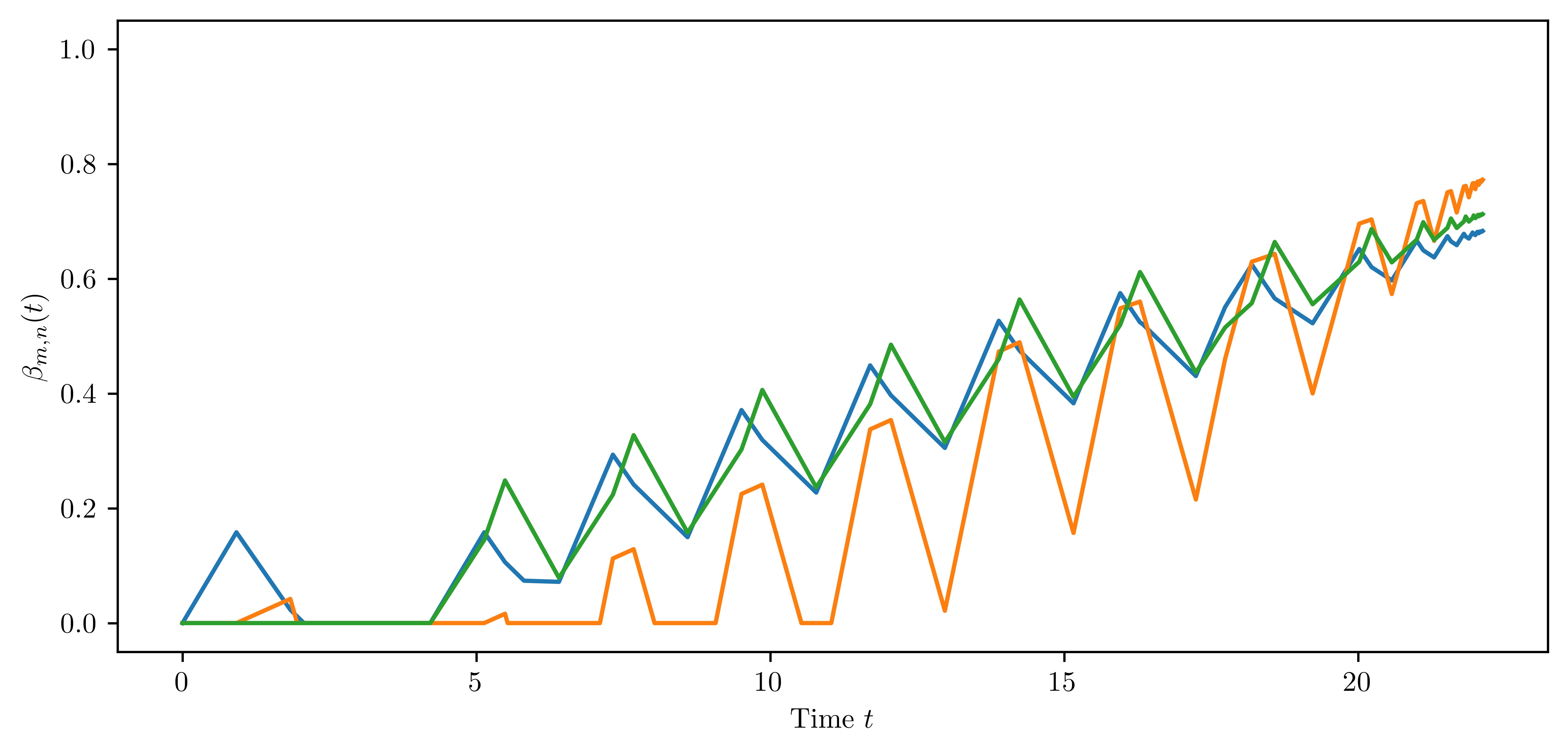}
		\caption{($p=0.24$) Top left: The dynamics of $\beta_{0,0}$ (blue), $\beta_{0,1}$ (orange) and $\beta_{0,2}$ (green). Top right: The dynamics of $\beta_{1,0}$ (blue), $\beta_{1,1}$ (orange) and $\beta_{1,2}$ (green). Bottom: The dynamics of $\beta_{2,0}$ (blue), $\beta_{2,1}$ (orange) and $\beta_{2,2}$ (green).}
	\end{figure}
	
\end{example}
\subsection{Extending Theorem \ref{Theorem: Main Theorem}}

One may ask the question of how the dynamics change as we alter the remaining parameters. Note that Theorem \ref{Theorem: Main Theorem} does not cover the case where a trait $(x,y)\in\calX$ with $\tfrac{x+y}{2}>3$ becomes dominant. Although we have not 
treated this case formally, the convergence claimed in Theorem \ref{Theorem: Main Theorem} should extend in a natural way: If the trait $(x,y)\in\calX$ becomes dominant but is unfit on its own, then the entire population size drops to $o(K)$ immediately on the $\log K$ timescale. Therefore, we can define the fitness functions in these cases as before, but omit all factors which are scaled by $K$, that is we set the death rate to $1$ and the switching rate from active to dormant to $0$. Then, we obtain for $\tilde{x}>0$ the fitness \[
S((\tilde{x},\tilde{y}),(x,y))\coloneqq \frac{r_1+r_2+\sqrt{(r_1-r_2)^2+4\sigma_1\sigma_2}}{2}=\max\lrset{r_1,r_2}
\]
where we use the notation $r_1\coloneqq 3-\tfrac{\tilde{x}+\tilde{y}}{2}+\tau\operatorname{sign}(\tilde{y}-y)$, $r_2\coloneqq -(\kappa+\sigma)$, $\sigma_1\coloneqq0$ and $\sigma_2\coloneqq \sigma$. For $\tilde{x}=0$ we set \[
S((0,\tilde{y}),(x,y))\coloneqq r_1=3-\frac{\tilde{y}}{2}+\tau\operatorname{sign}(\tilde{y}-y).
\]
Also, we need to use this definition of the fitness function when the population size is of the order $o(K)$ but the dominant trait has not reached a size of order $K$. With these extensions to the fitness function, the limiting functions $\beta_{m,n}$ should satisfy the formula stated in Theorem \ref{Theorem: Main Theorem} (iii). Note that the fitness of the traits with dormancy is bounded from below by $r_2=-(\kappa+\sigma)$. Hence it may happen that at a point where there are exactly two dominant traits and normally a change in the dominant trait would occur both traits have the same negative slope. In these cases it is not obvious how to continue. Indeed, using the definition of the times $s_k$ we would obtain $s_{k+1}=s_k$ and we cannot proceed any further.

\begin{example}
	Let $\delta=1.85$, $\tau=1.3$, $p=0.248$ , $\kappa=0$, $\sigma=1$ and $\alpha=0.5$. In this case we get the functions displayed in Figure \ref{Fig:d185p0248}.
	\begin{figure}[htbp]
		\centering
		
		\begin{minipage}[b]{0.49\linewidth}
			\includegraphics[width=1\textwidth]{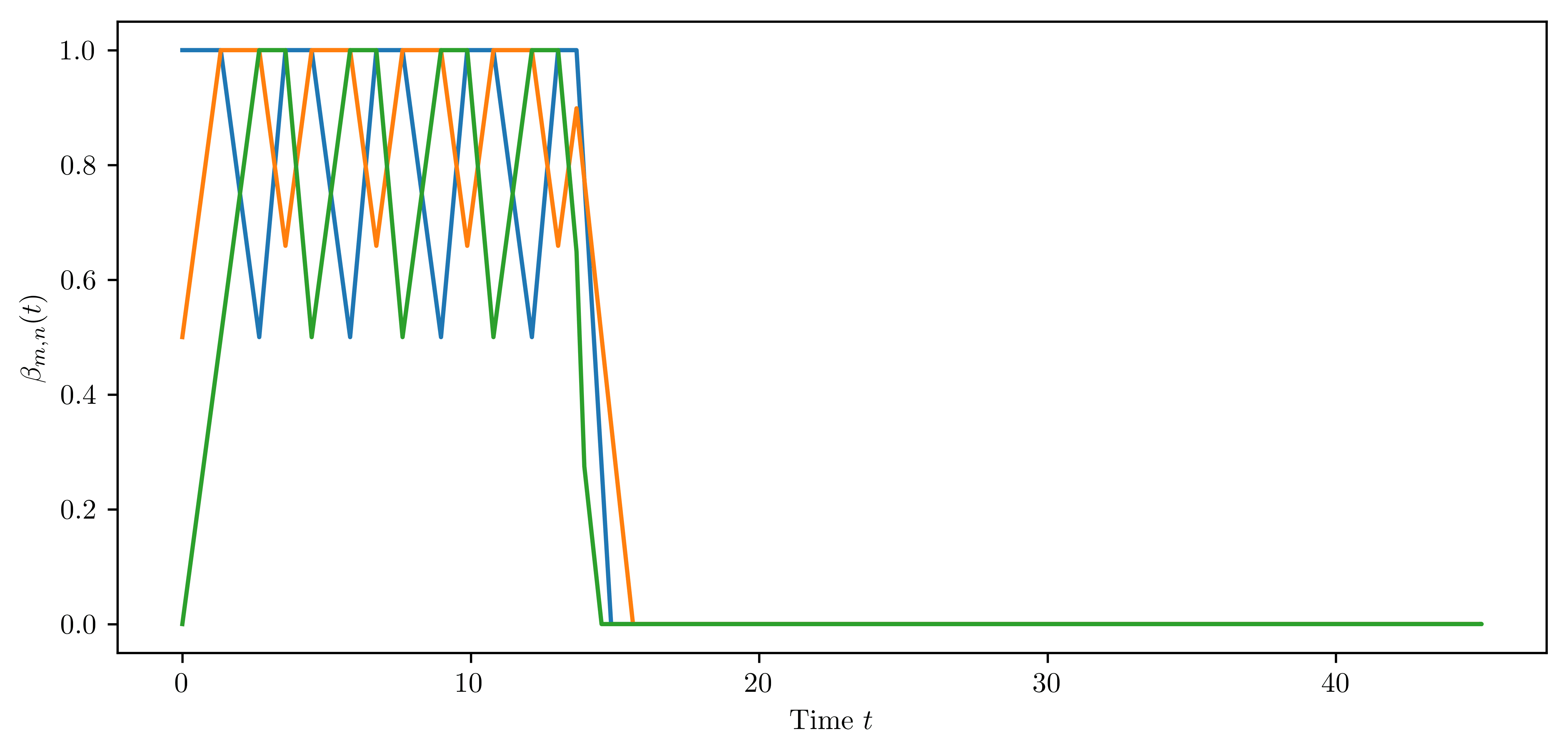}
		\end{minipage}
		\begin{minipage}[b]{0.49\linewidth}
			\includegraphics[width=1\textwidth]{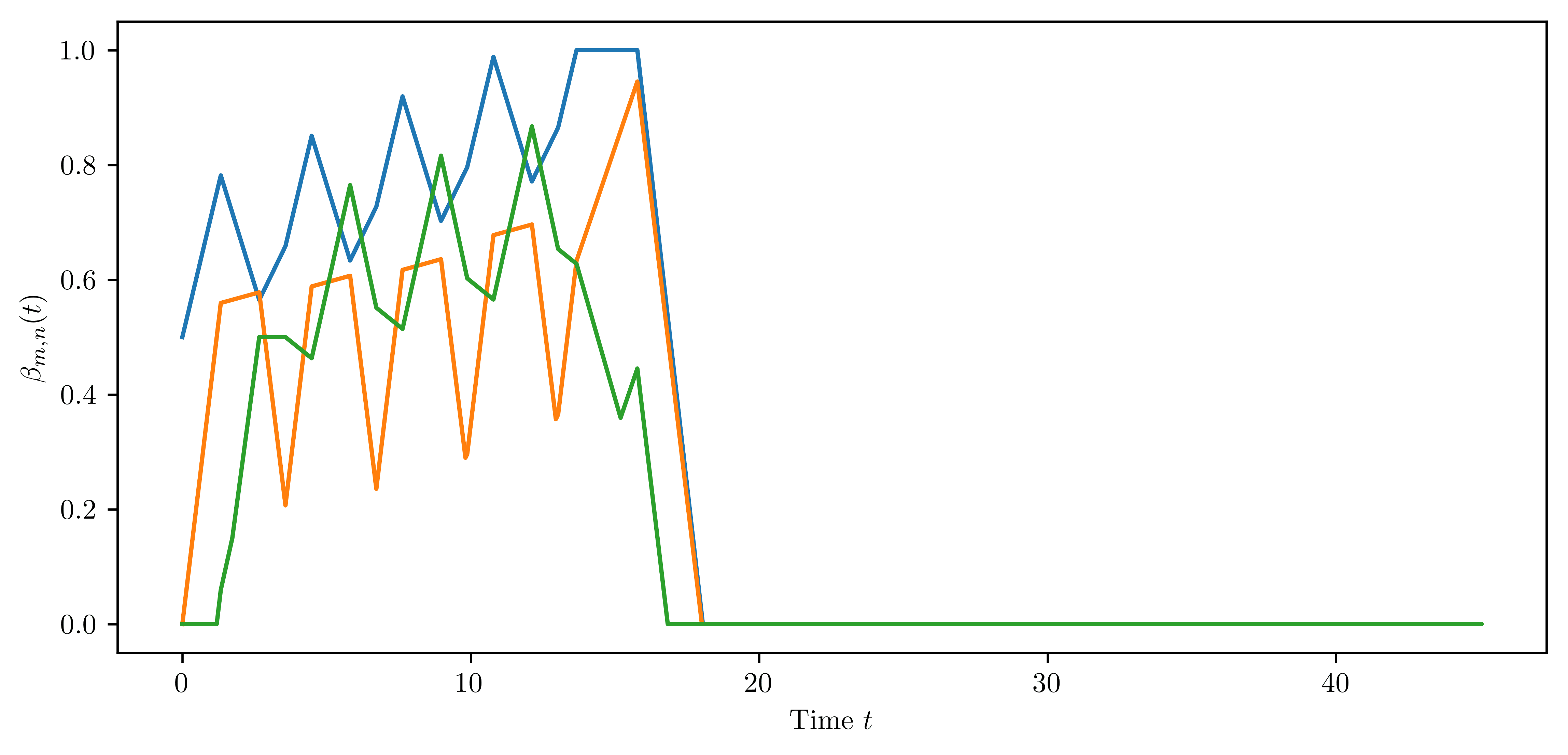}
		\end{minipage}
		\includegraphics[width=0.49\textwidth]{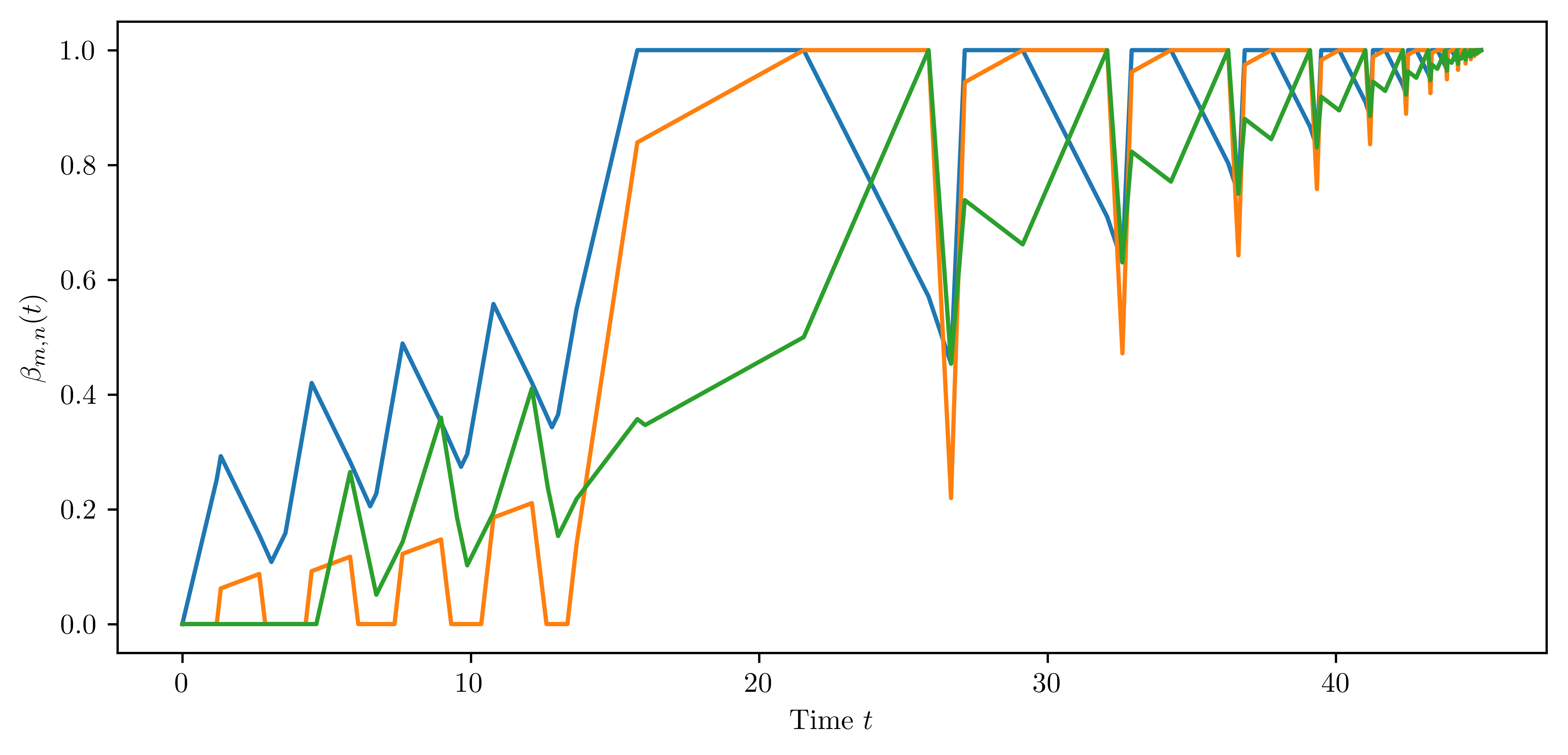}
		\caption{Functions for Example 3.6. Top left: The dynamics of $\beta_{0,0}$ (blue), $\beta_{0,1}$ (orange) and $\beta_{0,2}$ (green). Top right: The dynamics of $\beta_{1,0}$ (blue), $\beta_{1,1}$ (orange) and $\beta_{1,2}$ (green). Bottom: The dynamics of $\beta_{2,0}$ (blue), $\beta_{2,1}$ (orange) and $\beta_{2,2}$ (green).}
		\label{Fig:d185p0248}
	\end{figure}

Interestingly, here we have a finite time horizon $T_0$ and convergence of $\beta_{2,\ell}(s_k)\to 1$ as $k\to\infty$ for $\ell=0,1,2$ although there are repeatedly short periods of macroscopic extinction where the entire population size is of order $o(K)$.
\end{example}
\begin{example}
	If we set $\delta=1.92$, $\tau=1.3$, $p=0.248$ , $\kappa=0$, $\sigma=1$ and $\alpha=0.5$, that is all parameters the same as in the previous example but for $\delta$, then a similar but simultaneously new behaviour emerges.
	
	\begin{figure}[htbp]
		\centering
		\begin{minipage}[b]{0.49\linewidth}
			\includegraphics[width=1\textwidth]{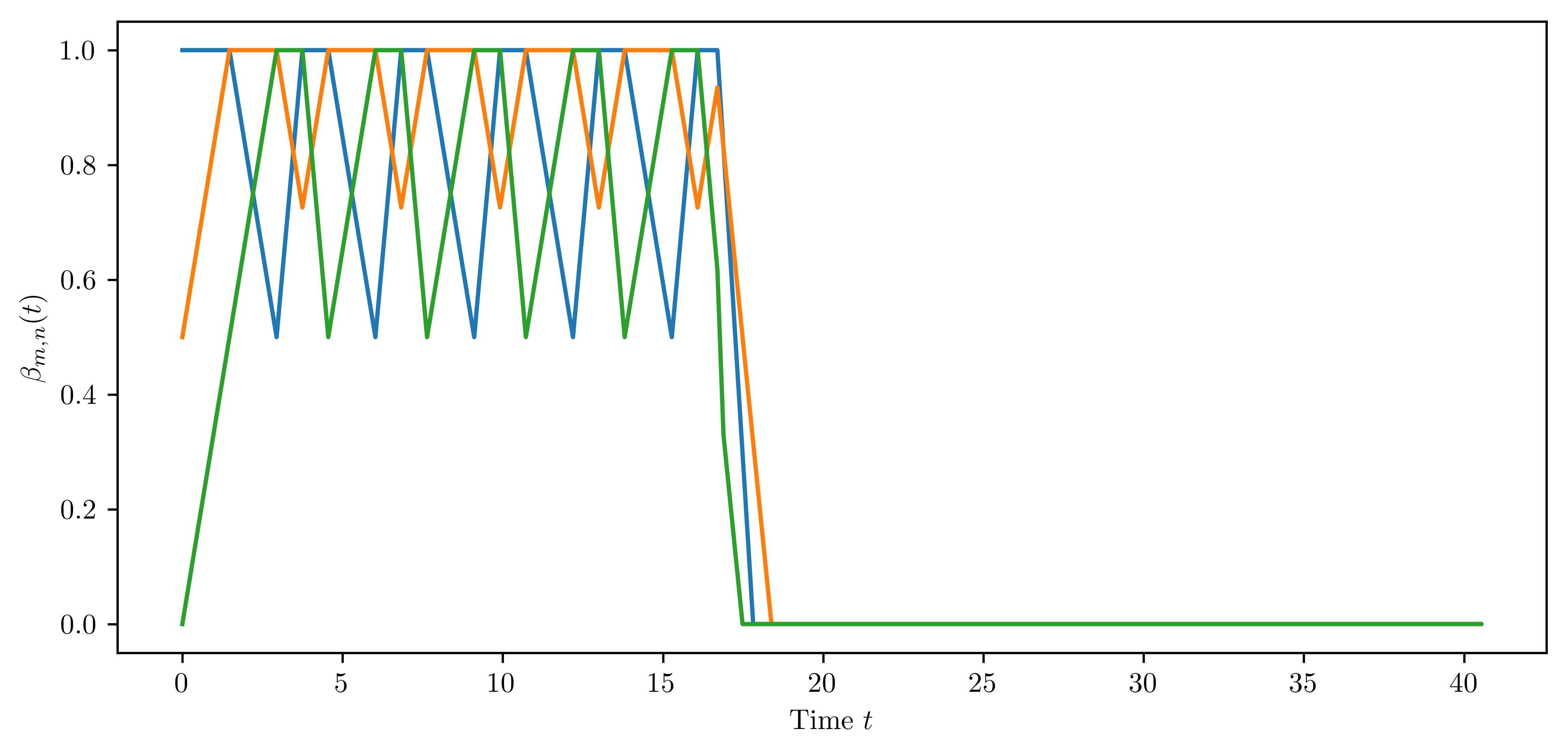}
		\end{minipage}
		\begin{minipage}[b]{0.49\linewidth}
			\includegraphics[width=1\textwidth]{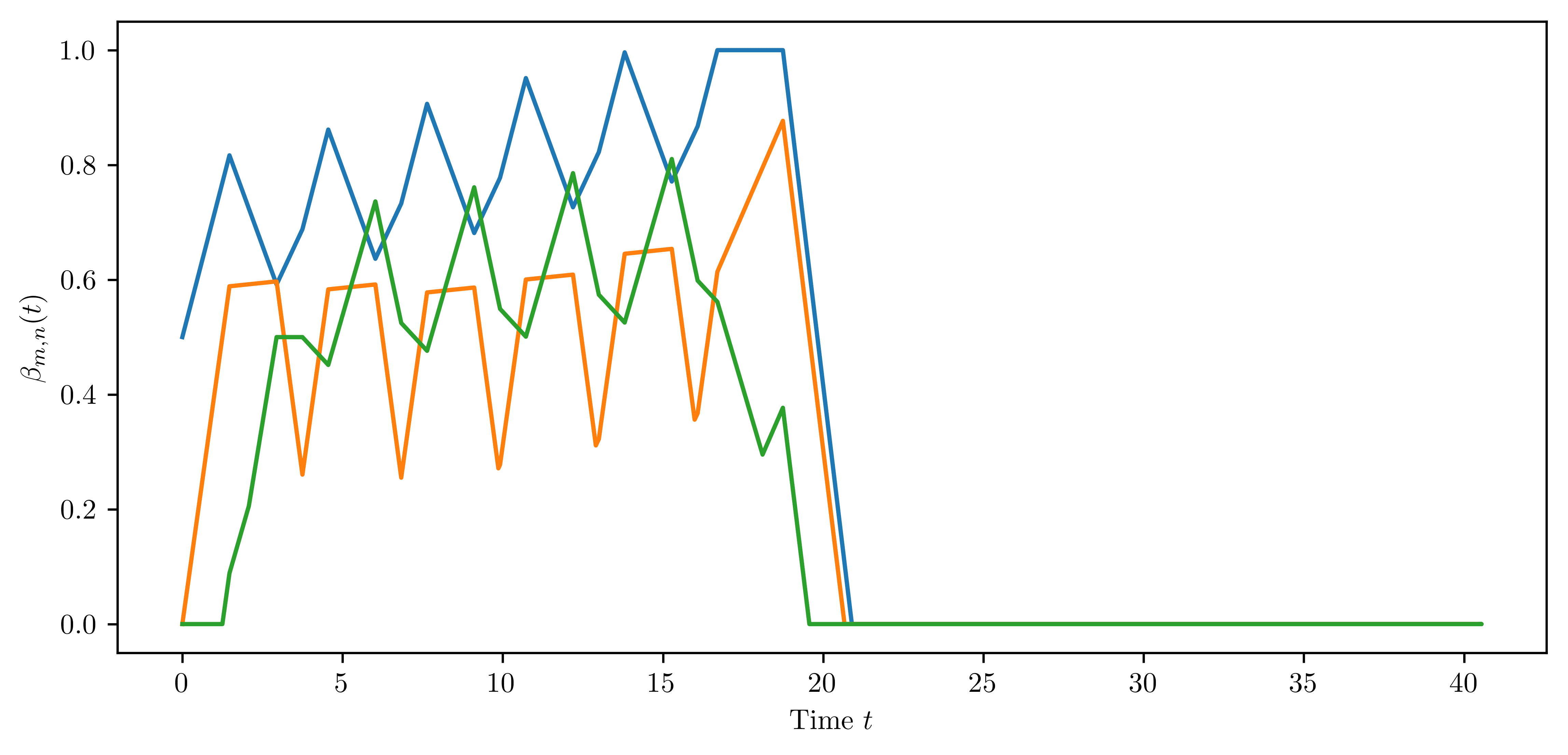}
		\end{minipage}
		\includegraphics[width=0.49\textwidth]{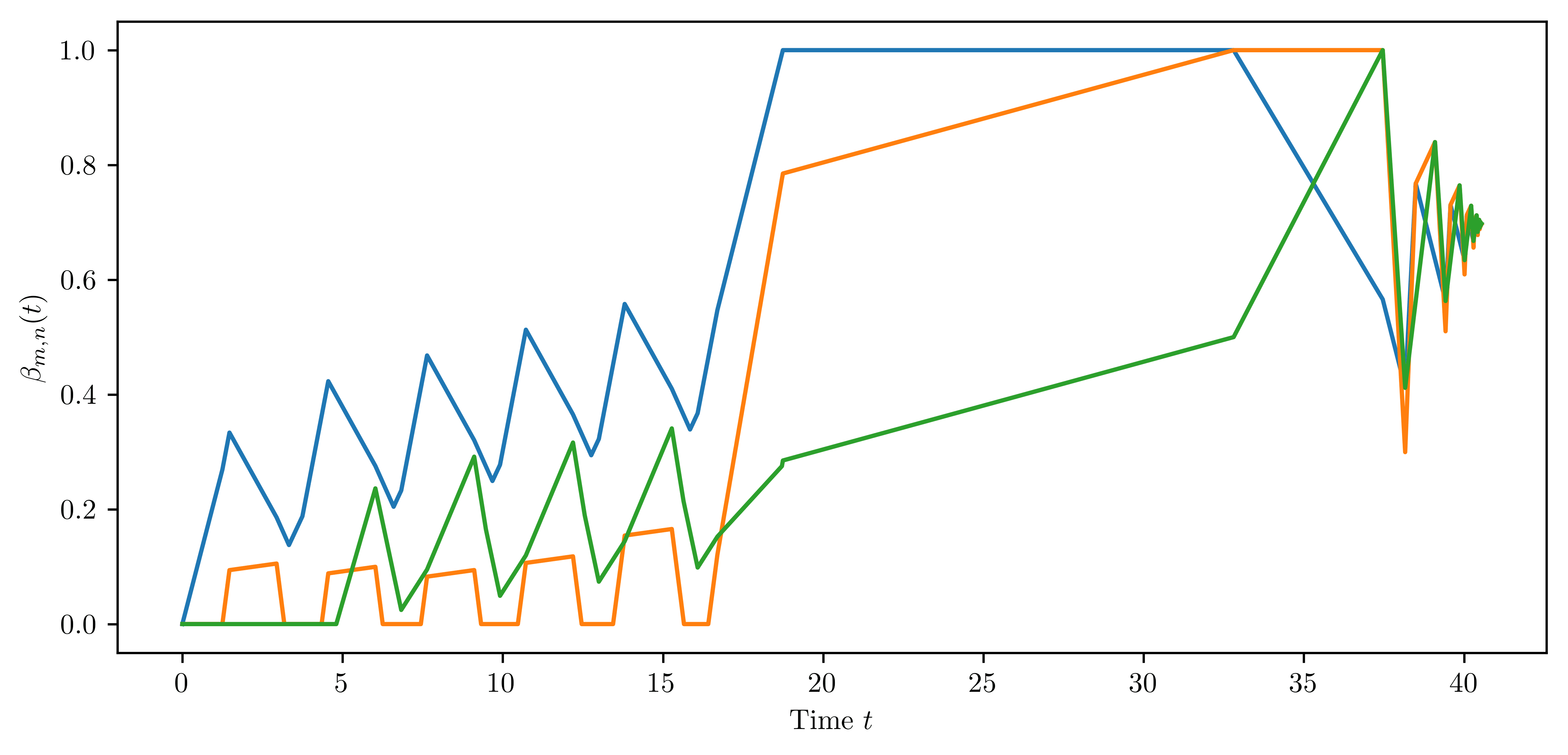}
		\caption{Functions for Example 3.7. Top left: The dynamics of $\beta_{0,0}$ (blue), $\beta_{0,1}$ (orange) and $\beta_{0,2}$ (green). Top right: The dynamics of $\beta_{1,0}$ (blue), $\beta_{1,1}$ (orange) and $\beta_{1,2}$ (green). Bottom: The dynamics of $\beta_{2,0}$ (blue), $\beta_{2,1}$ (orange) and $\beta_{2,2}$ (green).}
	\end{figure}

	The new aspect here is the finite time horizon $T_0$ where we have again convergence of $\beta_{2,\ell}(s_k)$ as $k\to\infty$ for $\ell=0,1,2$, but now $\textstyle\lim_{k\to\infty}\beta_{2,\ell}(s_k)<1$. We may say that in this case the system of individuals is generally unfit, since it is not able to remain of order $K$ at least periodically.
\end{example}

  In all of our simulations where an unfit trait becomes dominant, we have observed either one of the mentioned convergences or two traits with the same negative slope. In particular, we have not been able to observe evolutionary suicide and conjecture that due to the introduction of dormancy, evolutionary suicide is not possible. The reason for our conjecture lies in our fundamental modelling assumptions: only traits which can become dormant can also be unfit. Furthermore, assuming $\delta,C,\alpha,\tau,\kappa$ and $\sigma$ fixed, due to the continuity of the functions $\beta_{m,n}$, we conjecture that the qualitative behaviours observed (cyclic, driving towards coexistence, alternating but not periodic patterns) can be categorized into values of $p$ coming from open intervals $I\sse(0,\tfrac{1}{4})$ and as such it would be interesting to explicitly calculate these threshold values.
  
\subsection{Simulations}
Another point of interest is the size of the carrying capacity $K$. We know from Theorem \ref{Theorem: Main Theorem} that as $K\to\infty$ the exponents of the stochastic system converge under suitable rescaling of time towards the functions $\beta_{m,n}$. However, in reality the carrying capacity will be finite and thus we may ask how large $K$ needs to be, such that the limiting functions $\beta_{m,n}$ give a good description of the stochastic system, more precisely $\beta_{m,n}^K$. For this we conducted simulations but came to the conclusion, that explicitly simulating the Markov process is not feasible for $K>10^6$. The reason is twofold: On the one side, we need to increase the time horizon for the simulations as $K$ increases (since we are working on the $\log K$ time scale) and on the other side, the time steps between events become smaller as the population size increases.

\begin{figure}[htbp]
	\centering
	\begin{minipage}[b]{0.49\linewidth}
		\includegraphics[width=1\textwidth]{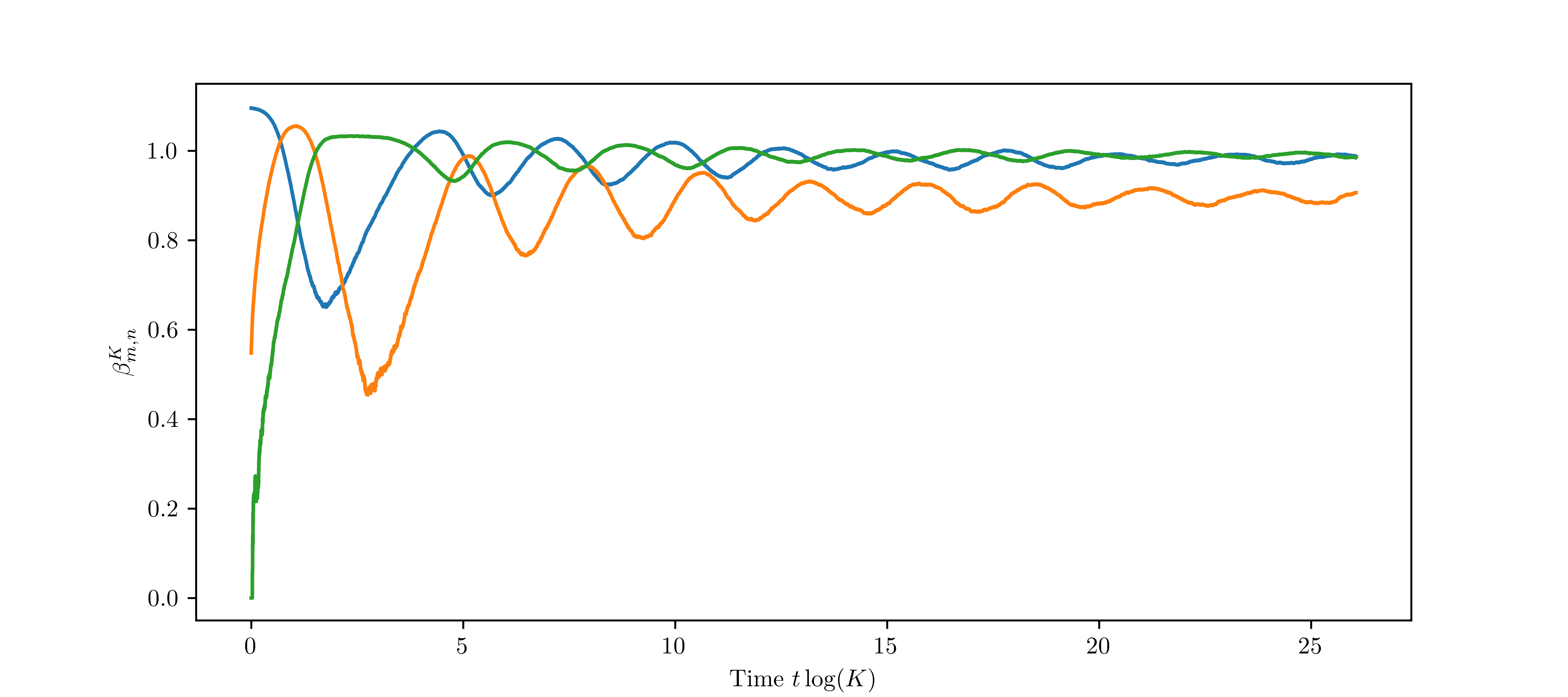}
	\end{minipage}
	\begin{minipage}[b]{0.49\linewidth}
		\includegraphics[width=1\textwidth]{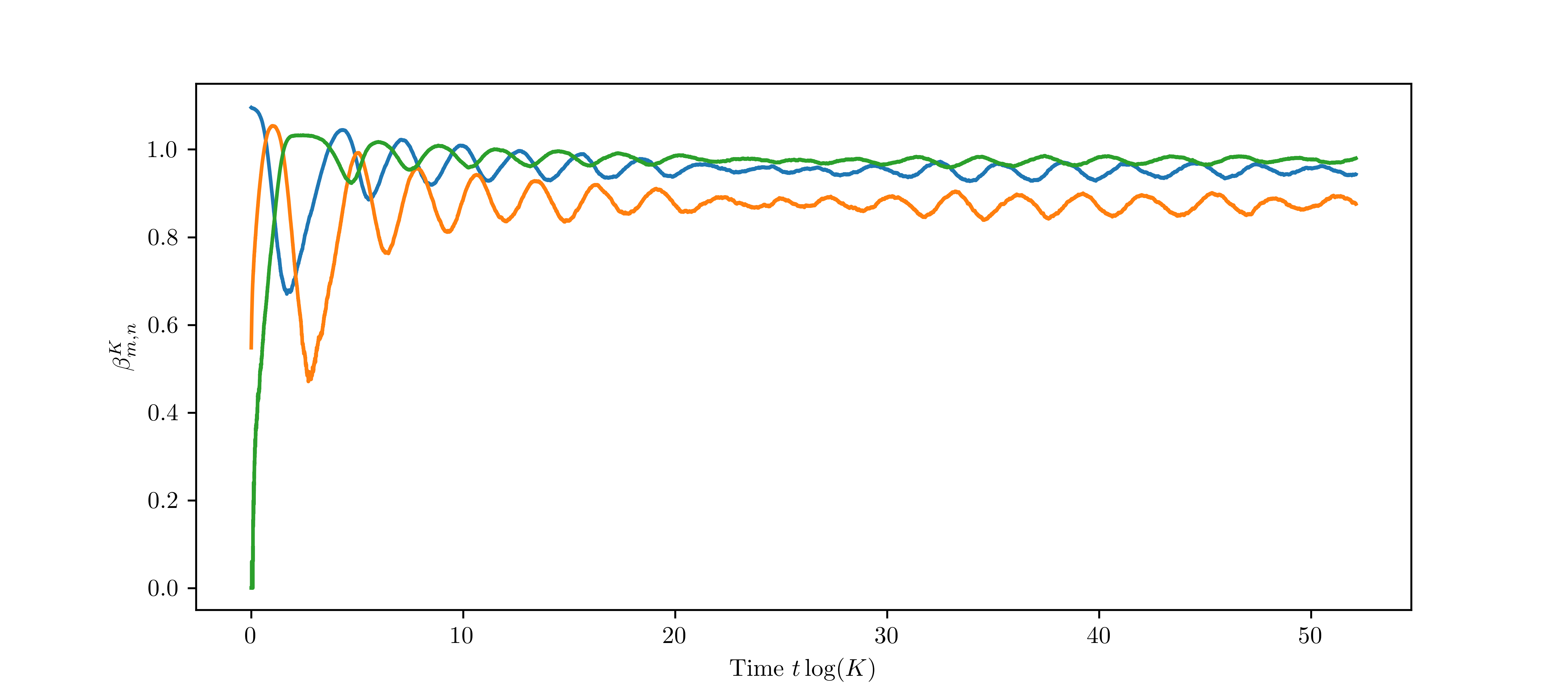}
	\end{minipage}
	\begin{minipage}[b]{0.49\linewidth}
		\includegraphics[width=1\textwidth]{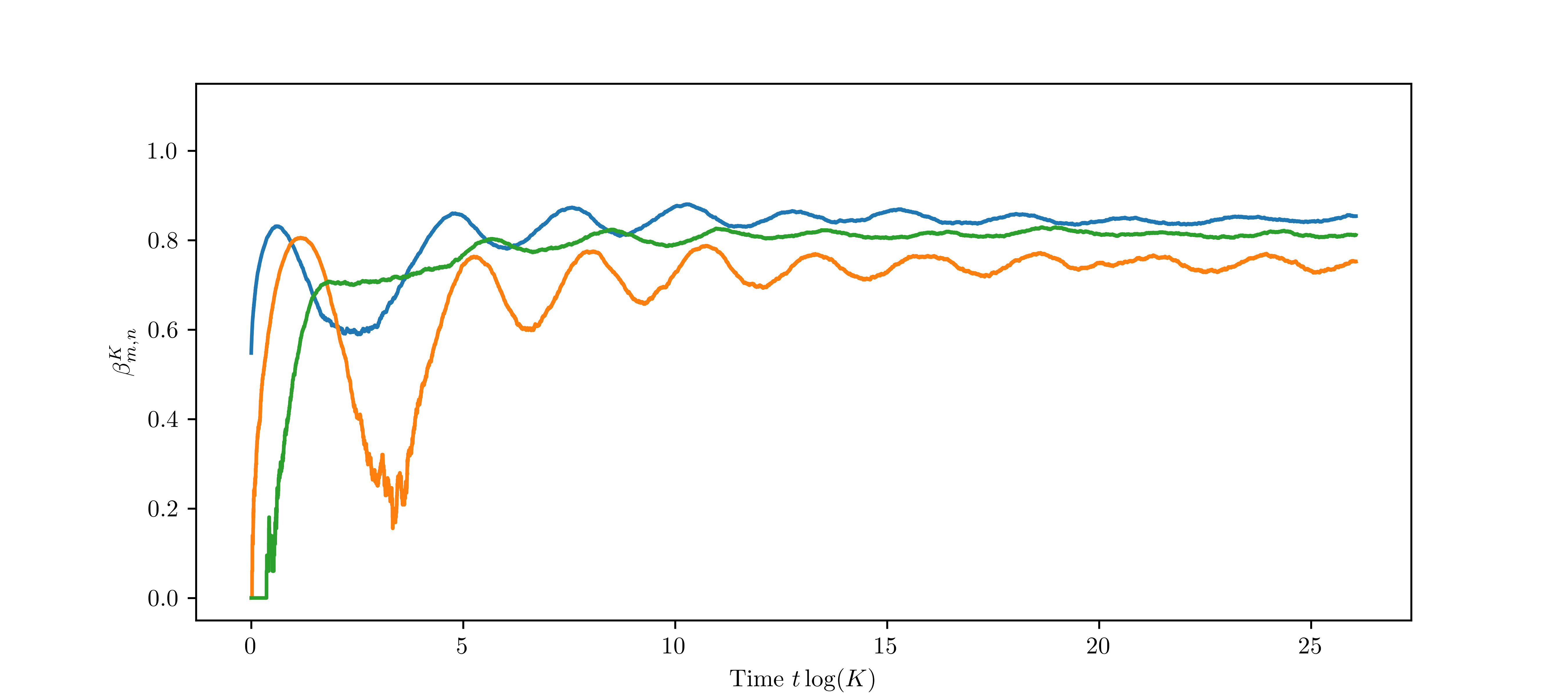}
	\end{minipage}
	\begin{minipage}[b]{0.49\linewidth}
		\includegraphics[width=1\textwidth]{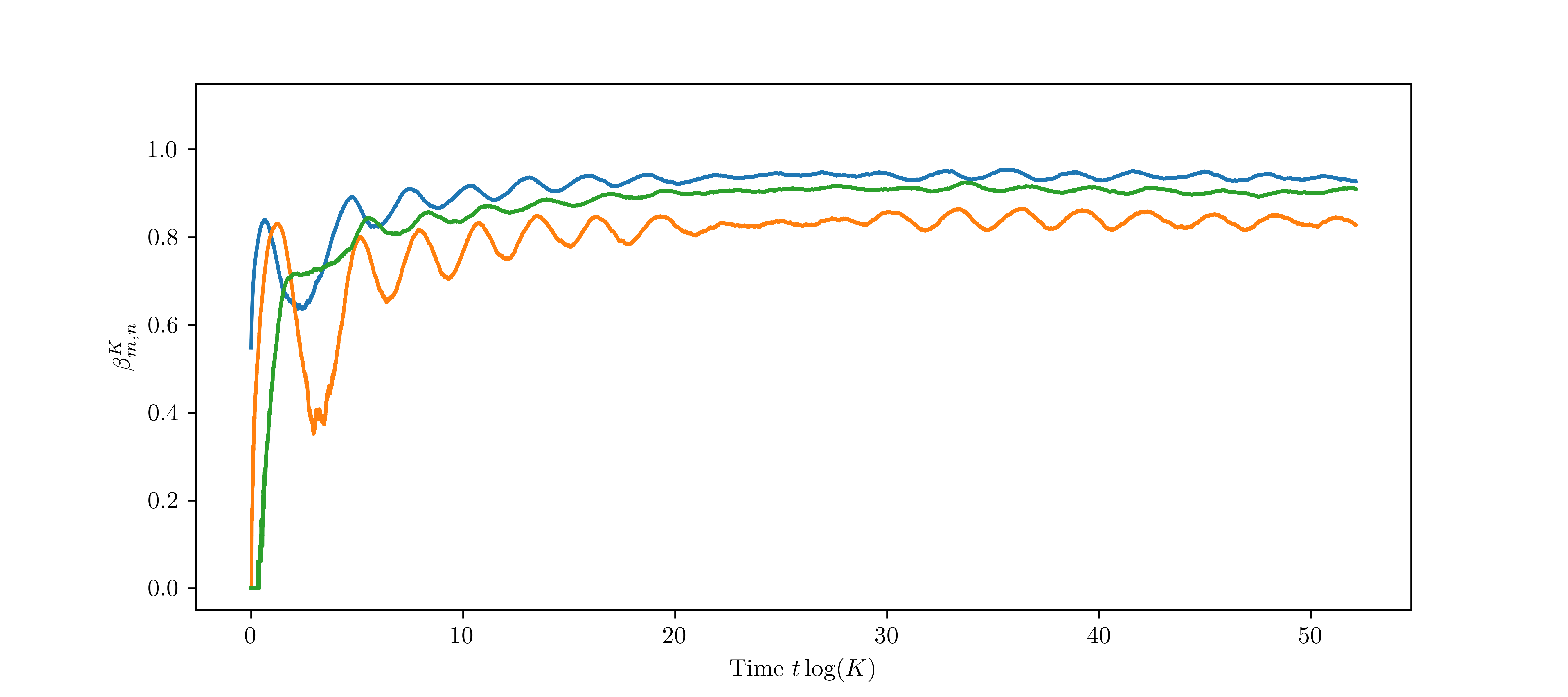}
	\end{minipage}
	\begin{minipage}[b]{0.49\linewidth}
		\includegraphics[width=1\textwidth]{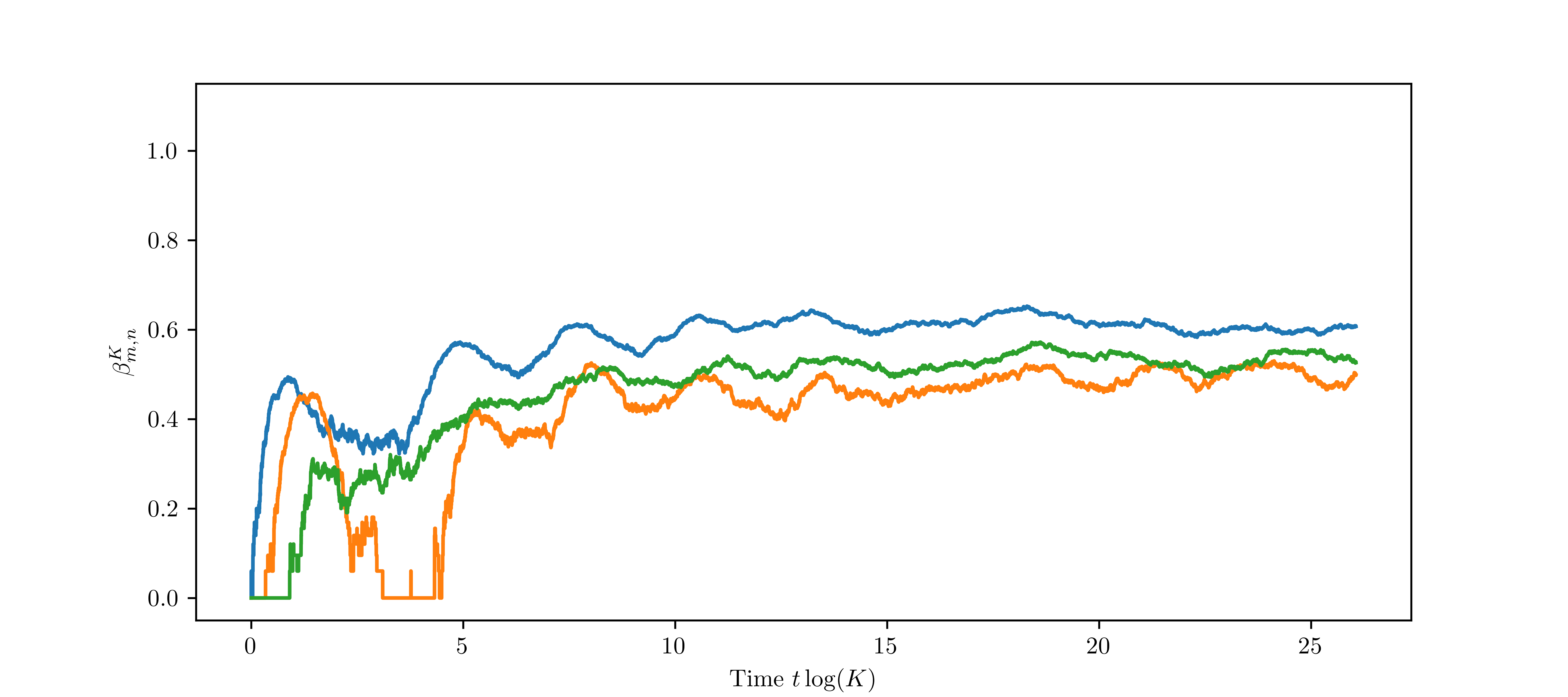}
	\end{minipage}
	\begin{minipage}[b]{0.49\linewidth}
		\includegraphics[width=1\textwidth]{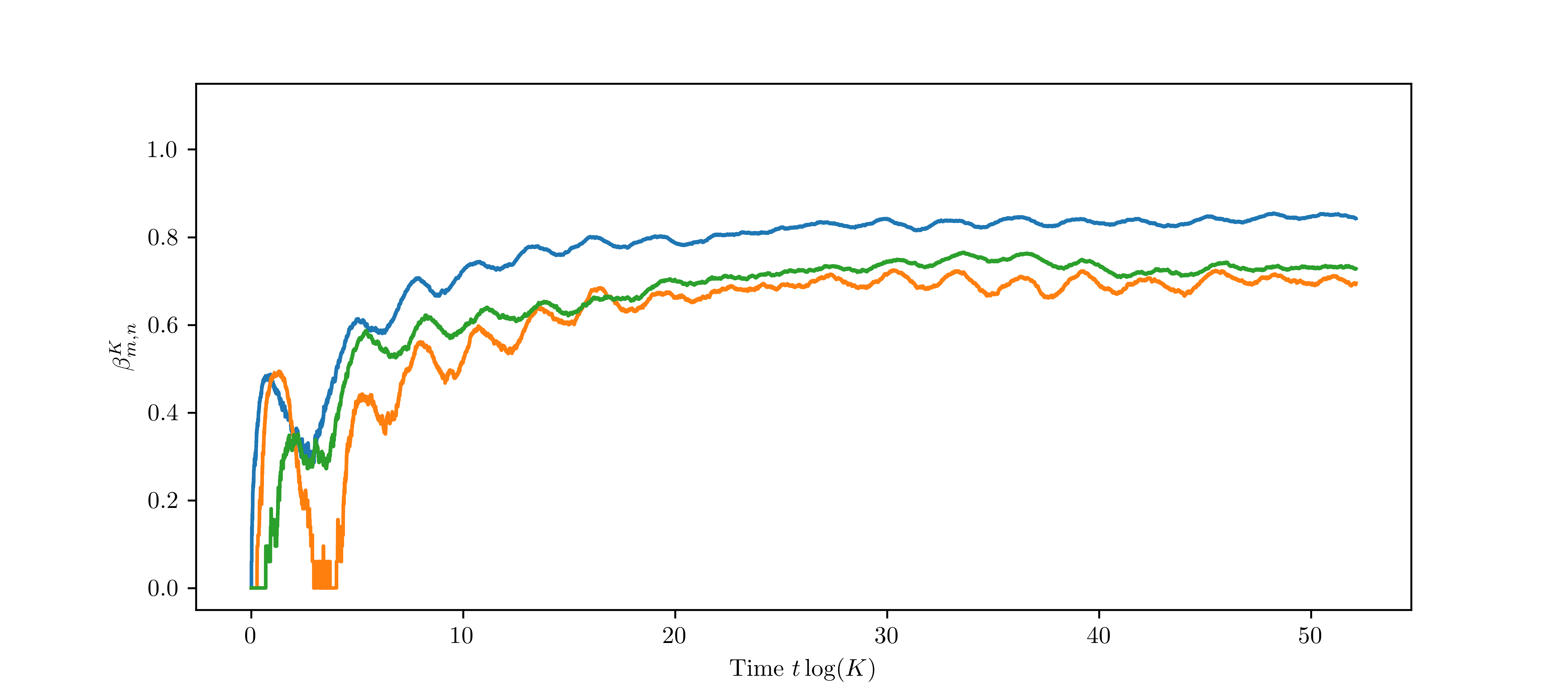}
	\end{minipage}
	\caption{Left: Simulations with the parameters as in Example \ref{Example: p=0.21} and $C=1$, $K=10^5$. Right: Simulations with parameters as in Example \ref{Example: Coexistence} and $C=1$, $K=10^5$. From top to bottom we are increasing the index $m$ of $\beta_{m,n}^K$ by $1$ and in each plot $n=0$ is blue, $n=1$ is orange and $n=2$ is green.}\label{fig:10e5}
\end{figure}

From our simulations with $K=10^5$ in Figure \ref{fig:10e5} we are able to see, that the stochastic process resembles very little spontaneous jumps when the population size is large. Note that the images on the bottom of Figure~\ref{fig:10e5} appear to be filled with jumps visible to the eye, which is due to the fact that $\sqrt{K}\approx 316$, so having an exponent of size $\tfrac{1}{2}$ means in terms of the population that around $316$ individuals are alive. Therefore, a single event causes a relatively large change in the population. Otherwise, the curves appear to be smooth, which leads us to a more efficient way of simulating the dynamics. We know, that on compact intervals the dynamics of $(\tfrac{N_{m,n}^{K,a}}{K},\tfrac{N_{m,n}^{K,d}}{K})$ without migration can be approximated by the solution of the differential equation \begin{align*}
	\dot{x}_{m,n}^a(t)&=\sigma x_{m,n}^d(t)\\
	&\hspace{-1cm}+x_{m,n}^a(t)\left[3-\frac{(m+n)\delta}{2}-C\sum_{m',n'=0}^{L}x_{m',n'}^a(t)+\tau\frac{\sum_{m'=0}^{L}\lr{\sum_{n'=0}^{n-1} x_{m',n'}^a(t)-\sum_{n'=n+1}^{L} x_{m',n'}^a(t)}}{\sum_{m',n'=0}^{L}x_{m',n'}^a(t)}\right]\\
	\dot{x}_{m,n}^d(t)&=pm\delta\cdot Cx_{m,n}^a(t)\sum_{m',n'=0}^{L}x_{m',n'}^a(t)-(\sigma+\kappa)x_{m,n}^d(t).
	\end{align*}
	
We also need to take into account the mutations which occur at birth with probability $K^{-\alpha}$. Since this probability tends to $0$ as $K\to\infty$, we do not have a mutation term in the differential equation on its own. However, as we are more interested in simulating the dynamics for some fixed $K$, we alter the derivative of the active component to be \begin{align*}
\dot{x}_{m,n}^a(t)\leftarrow\dot{x}_{m,n}^a(t)+\lr{4-\tfrac{(m+n-1)\delta}{2}}K^{-\alpha}(x_{m-1,n}(t)+x_{m,n-1}(t)),
\end{align*}
which leads to a mixed approximation of the stochastic system. Now, choosing $K$ fixed, we have on one side the usual approximation via an ODE and on the other side we have a non-zero mutation probability which is in accordance with the model. Determining the solution to these systems is numerically very efficient compared to a direct simulation and allows us to simulate the behaviour for large $K$. We refer to the exponents of the population sizes determined by solving the system as $\gamma_{m,n}^K$ However, we need to choose time steps $\Delta t$ for solving the ODE, which leads to complications: The process $N_{m,n}^K$ is only taking integer values, so in particular, if the rescaled process satisfies $\tfrac{N_{m,n}^K}{K}<\tfrac{1}{K}$, then the population should be extinct. Now, if $\Delta t$ is too small compared with $1/K$, then it may happen that the immigration during a time step of length $\Delta t$ is not sufficiently strong to start the population. Another numerical issue is the time horizon, on which we need to solve the differential equation. After rescaling, we need to solve until time $T\log K$, which in our cases would usually have $T\in[50,200]$ and thus may lead to some numerical instabilities. In particular, systems such as in Example 3.4 are sensitive to small deviations.

\begin{figure}[htbp]
	\centering
	\begin{minipage}[b]{0.49\linewidth}
		\includegraphics[width=1\textwidth]{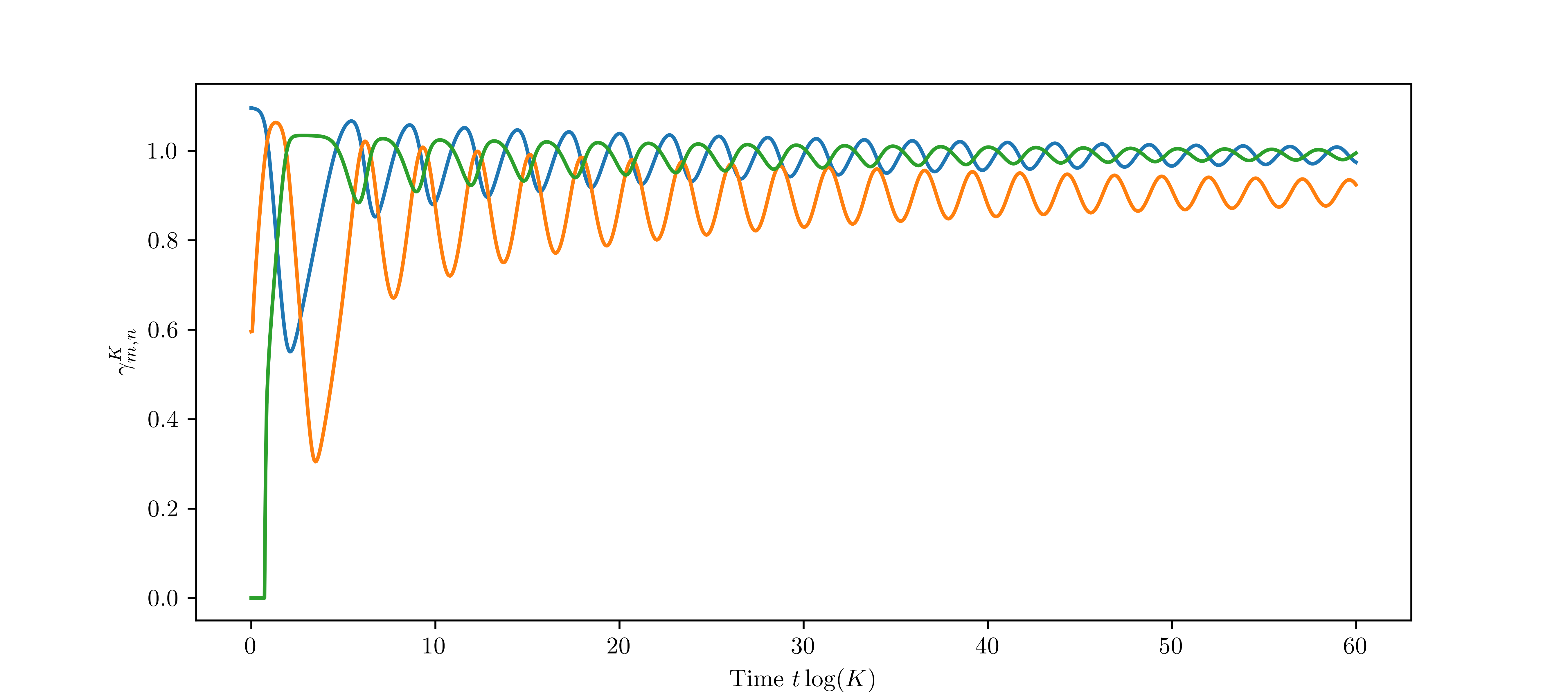}
	\end{minipage}
	\begin{minipage}[b]{0.49\linewidth}
		\includegraphics[width=1\textwidth]{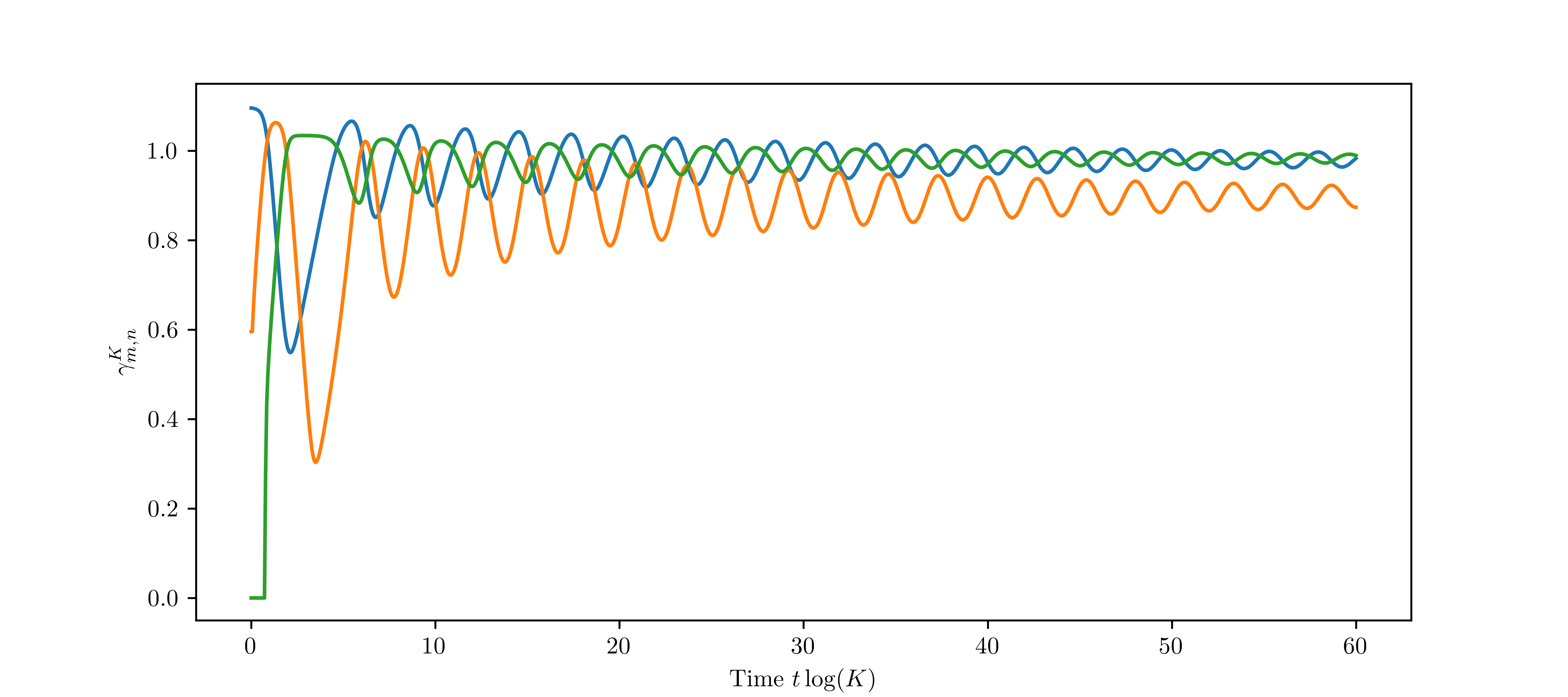}
	\end{minipage}
	\begin{minipage}[b]{0.49\linewidth}
		\includegraphics[width=1\textwidth]{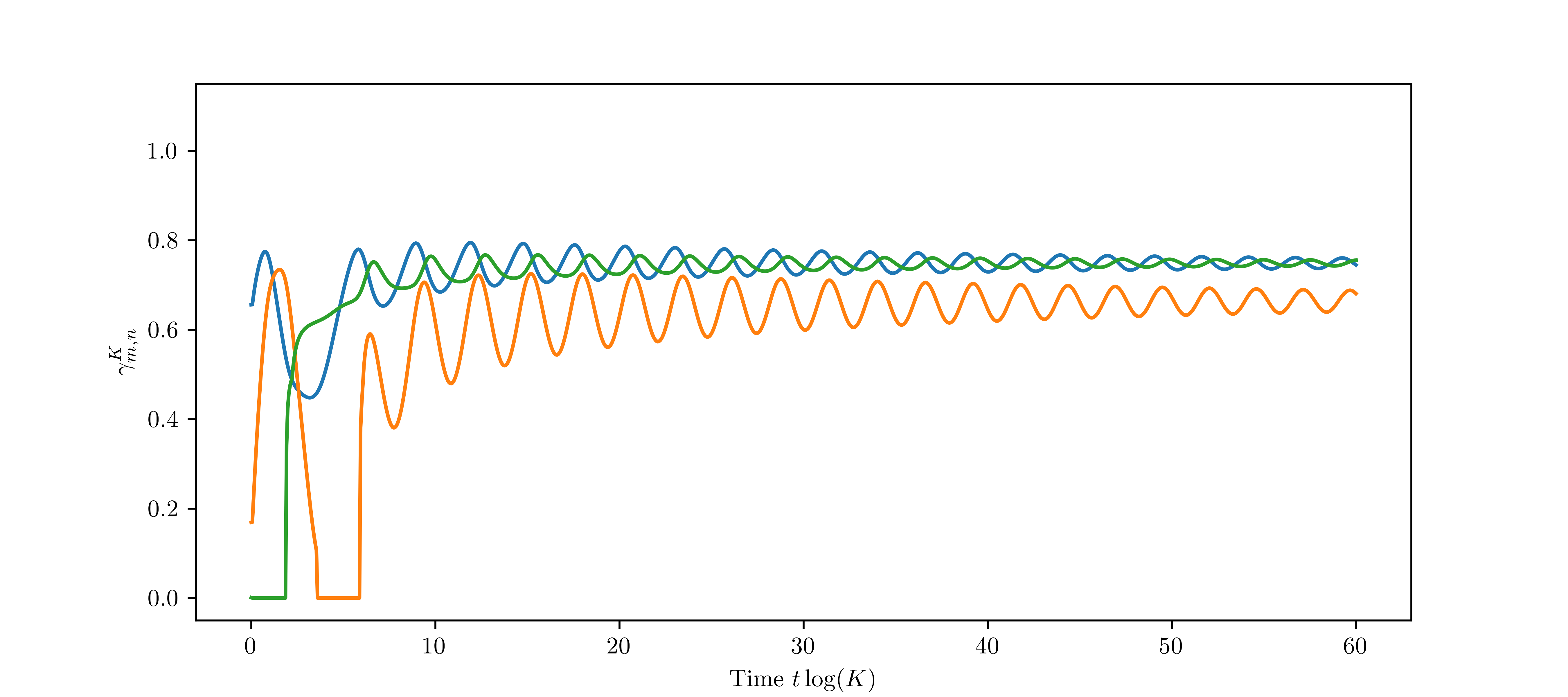}
	\end{minipage}
	\begin{minipage}[b]{0.49\linewidth}
		\includegraphics[width=1\textwidth]{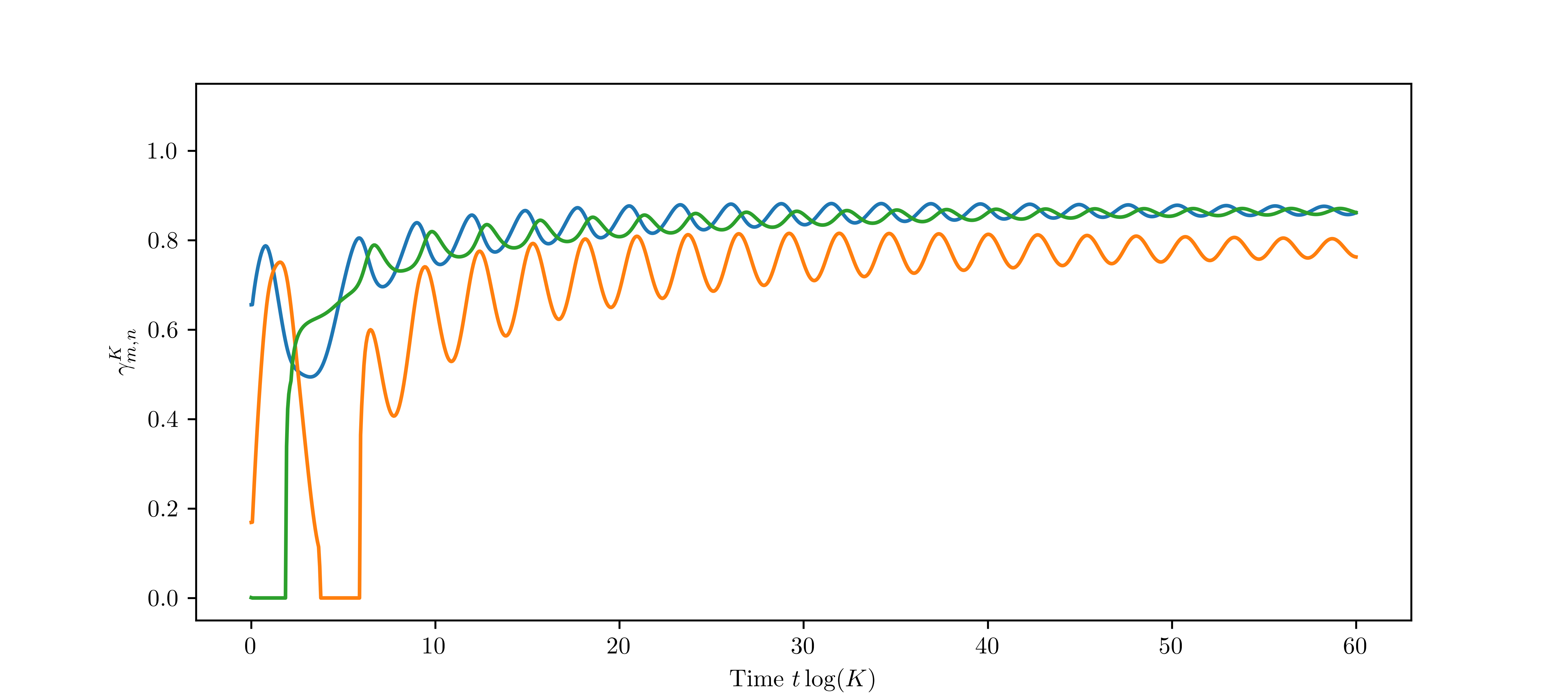}
	\end{minipage}
	\caption{Left: Solving the ODE with Euler scheme, $K=10^5$ and $\Delta t=T\log(K)K^{-1}$ with parameters as in Example \ref{Example: p=0.21} and $C=1$. Right: Solving the ODE with Euler scheme, $K=10^5$ and $\Delta t=T\log(K)K^{-1}$ with parameters as in Example \ref{Example: Coexistence} and $C=1$. From top to bottom we have the usual arrangement of the Exponents $\gamma_{m,n}^K$. We have omitted the plots $\gamma_{2,n}^K$ due to the lack of incoming mutations.}
	\label{Fig: 10^5}
\end{figure}

Comparing Figure \ref{Fig: 10^5} with the stochastic simulations, the ODE approach gives us a similar behaviour. Hence, we are confident that the solution to the differential equation will be similar to the stochastic system if we increase $K$.
Obviously, these plots (stochastic simulation and ODE solution) have very little in common with the limits which we have discussed in the corresponding examples. However, when thinking of bacterial populations, $K=10^5$ is still very small.

\begin{figure}[htbp]
	\centering
	\begin{minipage}[b]{0.49\linewidth}
		\includegraphics[width=1\textwidth]{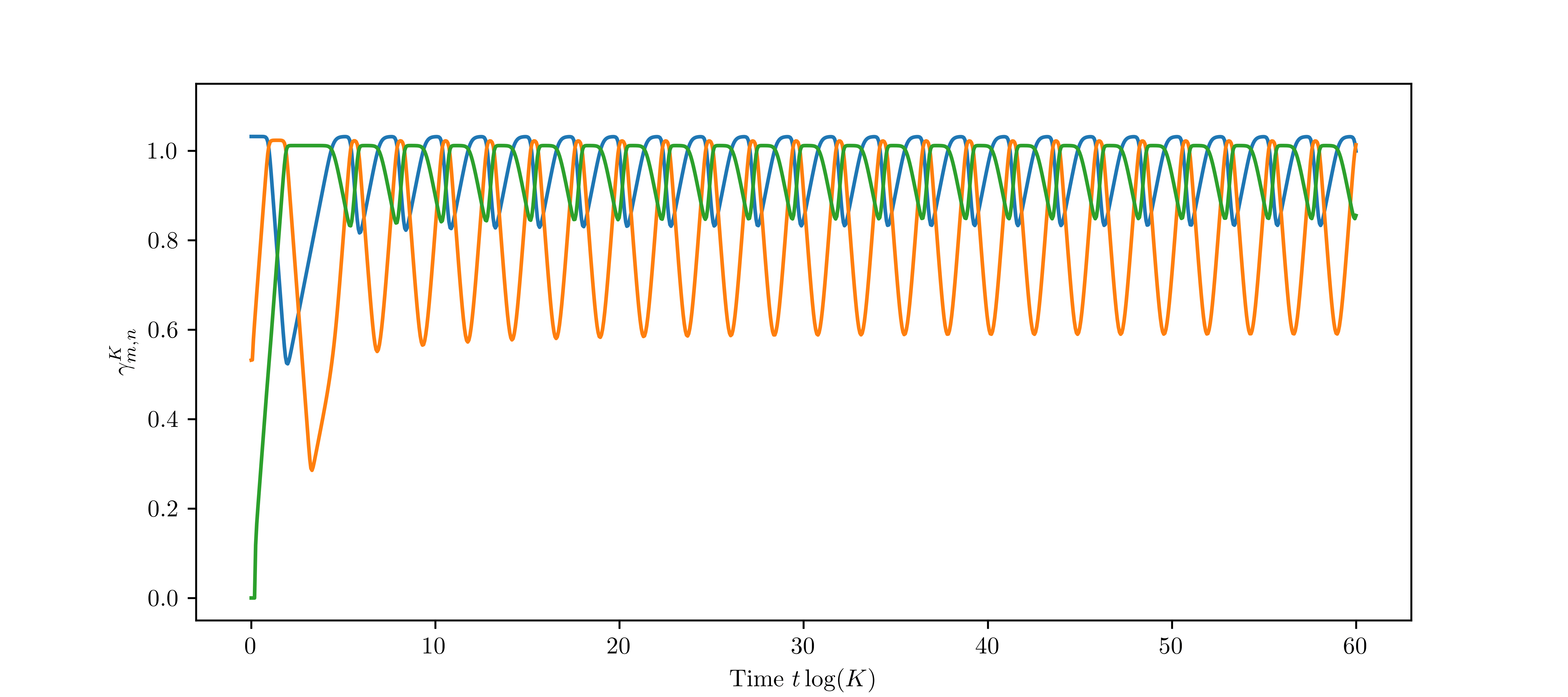}
	\end{minipage}
	\begin{minipage}[b]{0.49\linewidth}
		\includegraphics[width=1\textwidth]{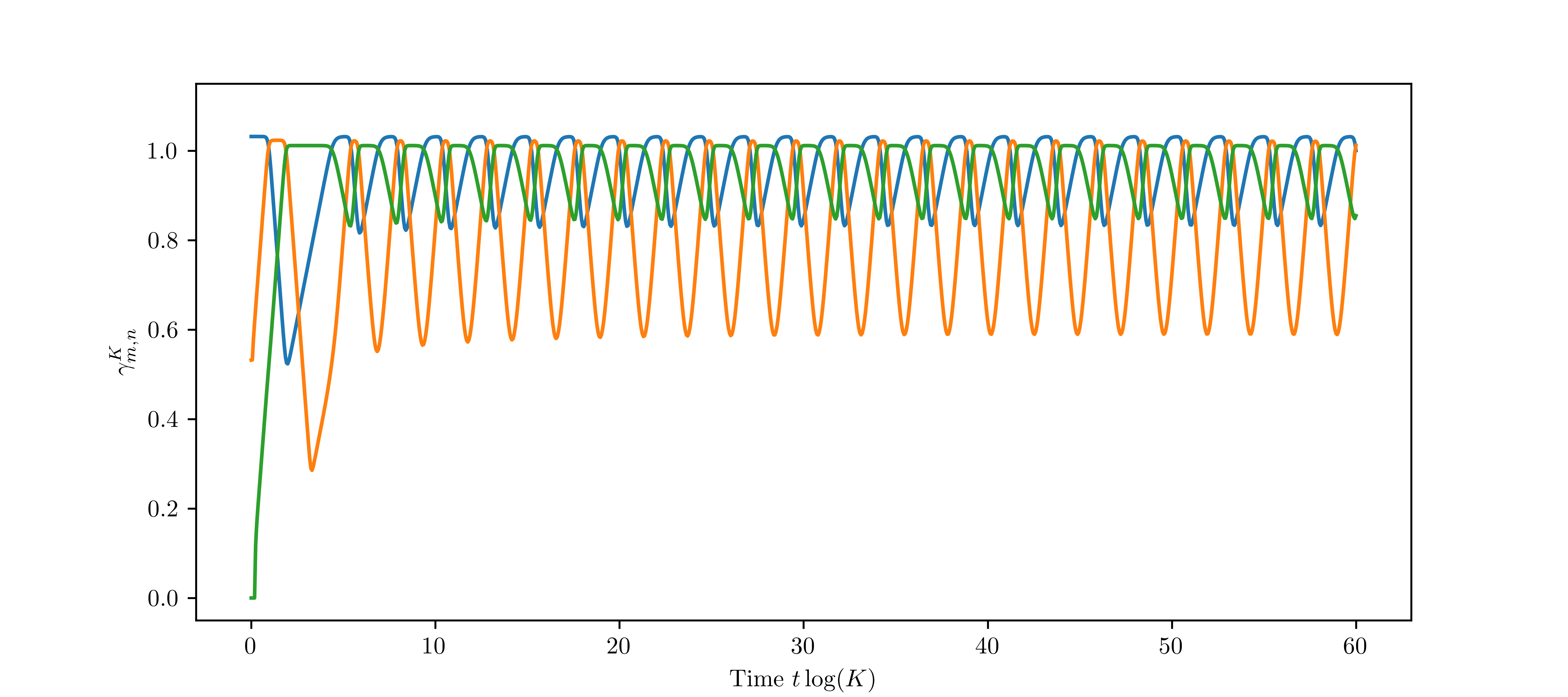}
	\end{minipage}
	\begin{minipage}[b]{0.49\linewidth}
		\includegraphics[width=1\textwidth]{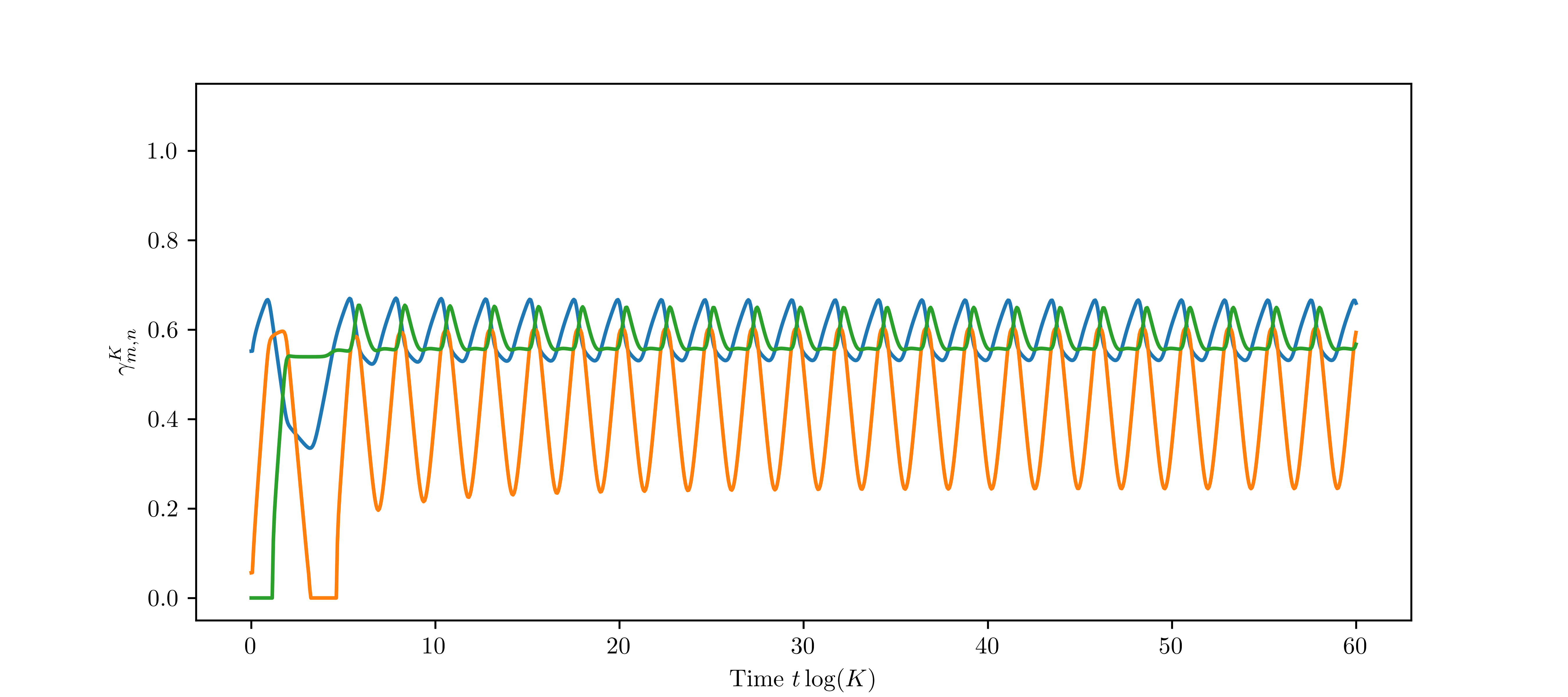}
	\end{minipage}
	\begin{minipage}[b]{0.49\linewidth}
		\includegraphics[width=1\textwidth]{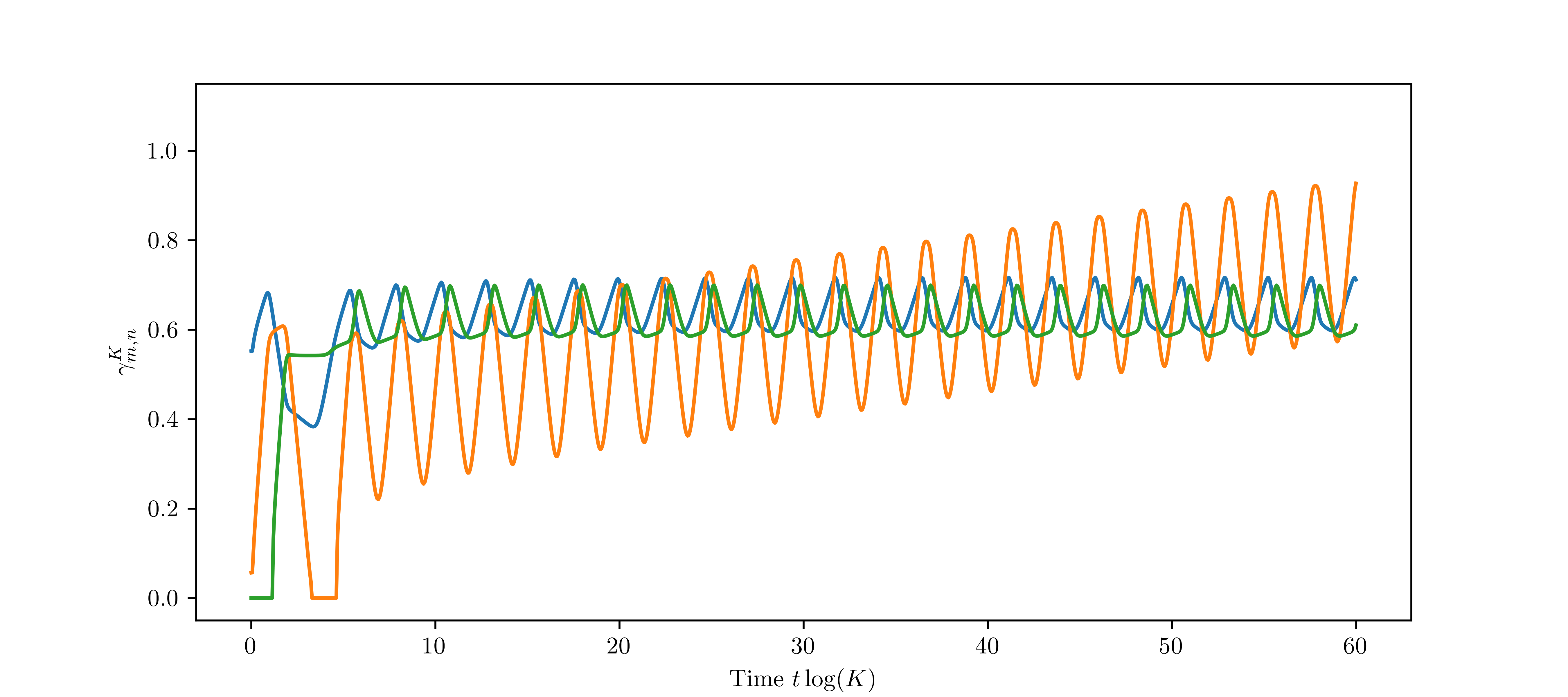}
	\end{minipage}
\begin{minipage}[b]{0.49\linewidth}
	\includegraphics[width=1\textwidth]{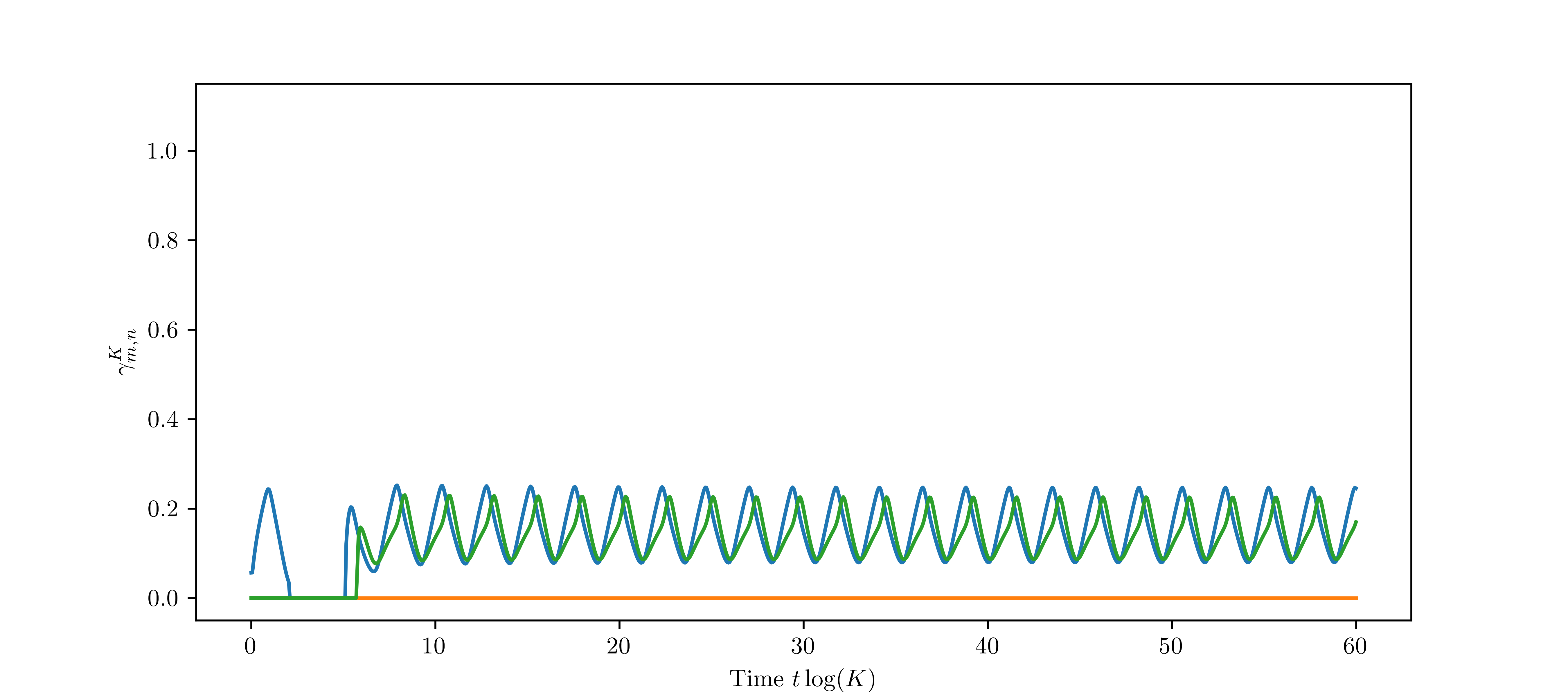}
\end{minipage}
\begin{minipage}[b]{0.49\linewidth}
	\includegraphics[width=1\textwidth]{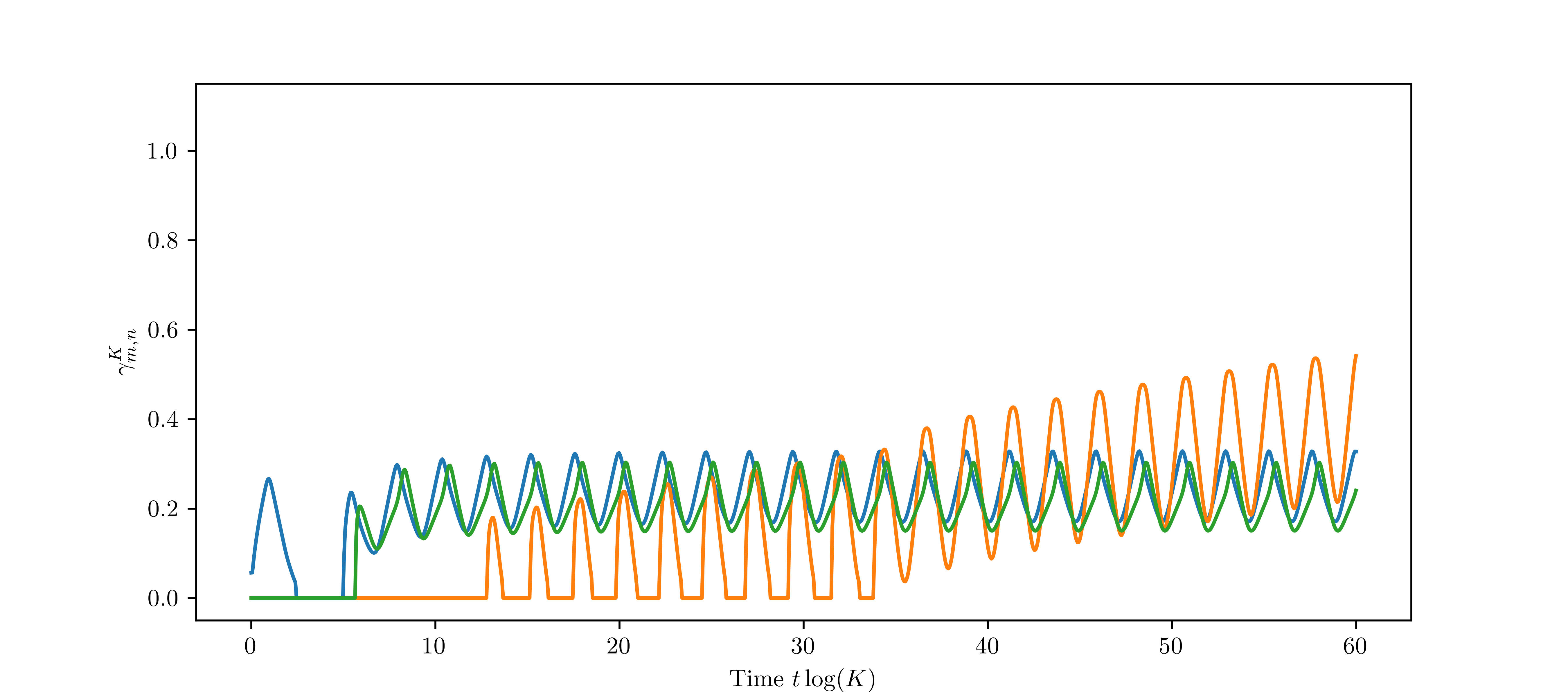}
\end{minipage}
	\caption{Left: Solving the ODE with Euler scheme, $K=10^{15}$ and $\Delta t=T\log(K)10^{-5}$ with parameters as in Example \ref{Example: p=0.21} and $C=1$. Right: Solving the ODE with Euler scheme, $K=10^{15}$ and $\Delta t=T\log(K)10^{-5}$ with parameters as in Example \ref{Example: Coexistence} and $C=1$. From top to bottom we have the usual arrangement of the Exponents $\gamma_{m,n}^K$.}
	\label{Figure: Fig 10^15}
\end{figure}

In Figure \ref{Figure: Fig 10^15} with $K=10^{15}$, the limiting functions are a much better approximation of the exponents, although $\Delta t$ is still too small in relation to $K$ for us to see any mutations arriving in $\gamma_{2,n}^K$ when $p=0.21$. Also, note that the coexistence, which we observed for $p=0.22$ in the limit, is in fact a normal cycle of residency between $(0,0)$, $(\delta,\delta)$ and $(0,2\delta)$. Only when letting $K\to\infty$, these cycles become ever shorter and lead to coexistence. Another interesting effect of finite populations is the prolonged duration it takes for the trait $(\delta,\delta)$ to become resident in the population. As $K$ increases, this duration becomes shorter on the $\log K$ time scale. Although we cannot be certain about the reason for this mechanism, we think that it may be due to the competition phases which vanish on the $\log K$ time scale as $K\to\infty$ but take up a non-negligible amount of time for fixed $K$. In particular, the competition against traits with dormancy takes longer due to the dormancy component and hence the convergence is slower in $K$ compared with systems with only HGT and no dormancy.

%% file: Proof.tex
\section{Proof of Theorem \ref{Theorem: Main Theorem}}\label{Section: Proof}

We will give a short sketch of the proof, which is very similar to \cite[Theorem 2.1]{champagnat2019stochastic}. The idea is to decompose the time scale into two different kinds of phases: First there are long phases $[\sigma_k^K\log K, \theta_k^K\log K]$ which then are followed by short intermediate phases $[\theta_k^K\log K,\sigma_{k+1}^K\log K]$. During the long phases, there is exactly one trait, whose population size is close to its equilibrium and all other traits are of size $o(K)$. During the short phases, another trait emerges and becomes significant for competitive events and due to competition the initially resident trait is replaced by the emerging trait. We will show that \[\lim\limits_{K\to\infty}\sigma_{k+1}^K=\lim\limits_{K\to\infty}\theta_k^K=s_k\] with probability converging to $1$ and hence on the $\log K$ timescale the intermediate phases vanish.

Since we only want to show this theorem in the case where only fit individuals (with a positive active equilibrium size) can become resident, we do not need to distinguish these cases, so our proof is simplified in this aspect compared to \cite{champagnat2019stochastic}. However, during the intermediate phases we need to observe whether none, one or both of the involved traits can become dormant and in which way the horizontal transfer is acting, if at all.

Thus the proof will be performed by induction on $k$. During the long phases, we will make heavy use of coupling arguments to show the convergence $\beta_{m,n}^{K}\to \beta_{m,n}$. This will again be done by induction on the traits, where we need a nested induction, since the horizontal transfer can be exerted onto all traits with a lower second component. For these phases, we will make extensive use of Theorem \ref{Theorem: Main Convergence} and Theorem \ref{Theorem: Onedimensionalconvergence}, so we refer to Appendix \ref{Section: Appendix B}. During the intermediate phases, we need the corresponding competition results, which can be found in Appendix \ref{Section: Appendix C}.

To make the structure of the induction more obvious, we give the general idea here: The trait space $\calX$ can be visualized as the $\delta$-grid on $[0,4]^2$ and we first show the convergence on a time interval for the trait $(0,0)$ as the base case. Then we advance our induction in the direction of dormancy to the trait $(\delta,0)$, where we make another base case in order to highlight the differences in the bi-type case. This is then followed by the induction step for traits $(m\delta,0)$. In this fashion, we can then assume the result to hold for all traits $(\tilde{m}\delta,\tilde{n}\delta)$ with $\tilde{n}\leq n$ and $\tilde{m}\in\{0,\ldots,L\}$ for some fixed $n\in\{0,\ldots,L-1\}$. Then we can show the result for traits $(m\delta, (n+1)\delta)$ via an induction on $m$ as for the case of $(m\delta,0)$.
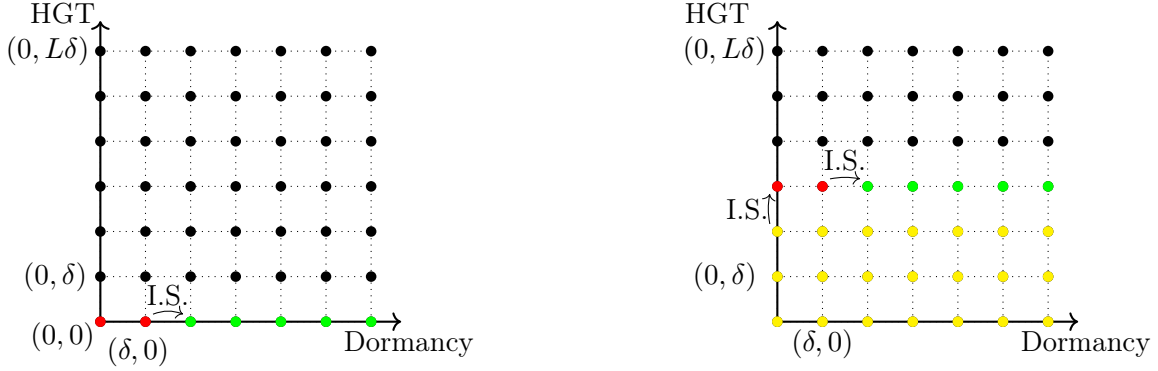
\begin{figure}[htbp]
	\begin{tikzpicture}
	\draw[->,thick](0,0)--(4,0);
	\draw[->,thick](0,0)--(0,4);
	\draw[dotted, step=.6cm](0,0) grid (3.6,3.6);
	\node at (-0.5,-0.2) {$(0,0)$};
	\node at (0.5,-0.4) {$(\delta,0)$};
	\node at (4.1,-0.3) {Dormancy};
	\node at (-0.6,0.6) {$(0,\delta)$};
	\node at (-0.5,4.1) {HGT};
	\node at (-0.7,3.6) {$(0,L\delta)$};
	\foreach \x in {0,0.6,1.2,1.8,2.4,3,3.6}{
		\foreach \y in {0,0.6,1.2,1.8,2.4,3,3.6}
		{\node[circle, inner sep=0pt, minimum size=4pt, fill=black] at (\x,\y) {};
	}};
	\node[circle, inner sep=0pt, minimum size=4pt, fill=red] at (0,0) {};
	\node[circle, inner sep=0pt, minimum size=4pt, fill=red] at (0.6,0) {};
	
	\node[circle, inner sep=0pt, minimum size=4pt, fill=green] at (2.4,0) {};
	\node[circle, inner sep=0pt, minimum size=4pt, fill=green] at (1.2,0) {};
	\node[circle, inner sep=0pt, minimum size=4pt, fill=green] at (1.8,0) {};
	\node[circle, inner sep=0pt, minimum size=4pt, fill=green] at (3,0) {};
	\node[circle, inner sep=0pt, minimum size=4pt, fill=green] at (3.6,0) {};
	\draw[->] (0.7,0.1) to [bend left=20] (1.1,0.1);
	\node at (0.9,0.35) {I.S.};

	\draw[->,thick](9,0)--(13,0);
	\draw[->,thick](9,0)--(9,4);
	\draw[dotted, step=0.6cm](9,0) grid (12.6,3.6);

	\node at (9.6,-0.3) {$(\delta,0)$};
	\node at (13.1,-0.3) {Dormancy};
	\node at (8.3,0.6) {$(0,\delta)$};
	\node at (8.2,4.1) {HGT};
	\node at (8.3,3.6) {$(0,L\delta)$};
	\foreach \x in {9,9.6,10.2,10.8,11.4,12,12.6}{
		\foreach \y in {0,0.6,1.2,1.8,2.4,3,3.6}
		{\node[circle, inner sep=0pt, minimum size=4pt, fill=black] at (\x,\y) {};
	}}
\foreach \y in {0, 0.6, 1.2}{
	\foreach \x in {9,9.6,10.2,10.8,11.4,12,12.6}
	{\node[circle, inner sep=0pt, minimum size=4pt, fill=yellow] at (\x,\y) {};
}}

	\node[circle, inner sep=0pt, minimum size=4pt, fill=red] at (9,1.8) {};
\node[circle, inner sep=0pt, minimum size=4pt, fill=red] at (9.6,1.8) {};

\node[circle, inner sep=0pt, minimum size=4pt, fill=green] at (10.2,1.8) {};
\node[circle, inner sep=0pt, minimum size=4pt, fill=green] at (10.8,1.8) {};
\node[circle, inner sep=0pt, minimum size=4pt, fill=green] at (11.4,1.8) {};
\node[circle, inner sep=0pt, minimum size=4pt, fill=green] at (12,1.8) {};
\node[circle, inner sep=0pt, minimum size=4pt, fill=green] at (12.6,1.8) {};

\draw[->] (9.7,1.9) to [bend left=10] (10.1,1.9);
\draw[->] (8.9,1.3) to [bend left=10] (8.9,1.7);
\node at (9.9, 2.15) {I.S.};
\node at (8.6, 1.5) {I.S.};

	\end{tikzpicture}
	\caption{The schematic induction: In the left picture, red are base cases, and the convergence for the remaining green traits are proven via the induction step. This is Step 1a) in the proof. Then in Step 1b) we make the induction step in the direction of HGT which is shown in the second image. For this we assume the result to be shown for all yellow traits. Then we have the base cases in red and subsequently another induction step again in green.}
\end{figure}

Throughout the proof we will use various kinds of branching processes. We denote by $BP_K(b,d,\beta)$ a one-dimensional branching process with birth rate $b$, death rate $d$ and initial condition $\lfloor K^{\beta}-1\rfloor$. Also, we denote by $BPI_K(b,d,a,c,\beta)$ a one-dimensional branching process with birth rate $b$, death rate $d$, immigration at rate $K^ce^{at}$ and initial condition $\lfloor K^{\beta}-1\rfloor $. We refer to \cite[Appendix A,B]{champagnat2019stochastic} for results concerning these processes. With $BBPI_K(b_1,b_2,d_1,d_2,\sigma_1,\sigma_2,a,c,\beta,\gamma)$ we denote a two-dimensional branching process with birth rates $b_1,b_2$, death rates $d_1,d_2$, switching rates $\sigma_1,\sigma_2$, immigration into the first coordinate at rate $K^ce^{at}$ and initial condition $\lfloor(K^{\beta}-1,K^\gamma-1)\rfloor$. We refer to Appendix \ref{Section: Appendix B}.

Further, we denote by $LBDI_K(b,d,C,\gamma)$ a one-dimensional logistic birth and death process with birth rate $b$, death rate $d+\tfrac{CN}{K}$, where $N$ denotes the population size, and immigration at a predictable rate $\gamma(t)$ at time $t\geq 0$. We refer to \cite[Appendix C]{champagnat2019stochastic}. Also $LBBI_K(b_1,d_1,d_2,\sigma_2,p,C,\gamma_1)$ denotes the distribution of a two-dimensional logistic birth and death process with birth rates $b_1, 0$, death rates $d_1+\tfrac{(1-p)CN}{K}, d_2$, where $N$ denotes the population size of the first component, switching rates $\tfrac{pCN}{K}, \sigma_2$ and immigration into the first component at a predictable rate $\gamma_1(t)$ at time $t\geq 0$. We refer to Appendix \ref{Section: Appendix C}.

\begin{proof}[Proof of Theorem \ref{Theorem: Main Theorem}]
	We distinguish two cases: Either there is only one phase (that is, all traits are unfit against the initially resident trait $(0,0)$) or there are at least two phases. In either case we now consider a fixed time $T>0$. 
	
	\subsection{Proof of Theorem \ref{Theorem: Main Theorem}, Case 1}
	\begin{description}
		\item[Case 1: $S((x,y),(0,0))\leq0$ for all $(x,y)\in\calX$] 
		As in \cite{champagnat2019stochastic}, we define a time $\theta_1^K$, during which trait $(0,0)$ is resident. Let $\veps_1>0$ and $\rho>0$ and define the time \begin{align*}
		\theta_1^K\coloneqq\inf\Bigg\{ t\geq 0\ \Bigg\vert N_{0,0}^K(t\log K)\notin\Bigg[\lr{\frac{3}{C}-3\veps_1}&K,\lr{\frac{3}{C}+3\veps_1}K\Bigg]\\
		&\text{or}\sum_{(m,n)\neq(0,0)}N_{m,n}^K(t\log K)\geq \rho\veps_1K \Bigg\}.
		\end{align*}
	\end{description}
Then we easily calculate that $\beta_{m,n}(t)=(1-(n+m)\alpha)\vee 0$ for $t\leq \theta_1^K\wedge T$.
\subsubsection*{Step 0: Deriving bounds on the rates}
Since all of the following couplings in Step 1 will need some bounds on the rates, we derive them here for our process up to time $\theta_1^K$.  In particular, the arrivals in the active component of trait $(m,n)$ due to reproduction or transfer at time $t\leq\theta_1^K\wedge T$ occur at rate \begin{align*}
N_{m,n}^{K,a}(t\log K)\lr{\lr{4-\frac{(m+n)\delta}{2}}\lr{1-K^{-\alpha}}+\tau\frac{\sum_{m'=0, n'<n}^{L}N_{m',n'}^{K,a}(t\log K)}{\sum_{m',n'=0}^LN_{m',n'}^{K,a}(t\log K)}},
\end{align*}
which satisfies for $K$ large enough and $n>0$ the inequality \begin{align*}
&N_{m,n}^{K,a}(t\log K)\lr{4-\frac{(m+n)\delta}{2}-\veps_1+\tau\frac{3-3C\veps_1}{3+C(3+\rho)\veps_1}}\\
\leq& N_{m,n}^{K,a}(t\log K)\lr{\lr{4-\frac{(m+n)\delta}{2}}\lr{1-K^{-\alpha}}+\tau\frac{\sum_{m'=0, n'<n}^{L}N_{m',n'}^{K,a}(t\log K)}{\sum_{m',n'=0}^LN_{m',n'}^{K,a}(t\log K)}}\\
\leq & N_{m,n}^{K,a}(t\log K)\lr{4-\frac{(m+n)\delta}{2}+\tau}.
\end{align*}
These bounds are true for $n=0$ when we remove the terms involving $\tau$.
The arrivals due to incoming mutations occur at rate \[
(N_{m-1,n}^{K,a}(t\log K)+N_{m,n-1}^{K,a}(t\log K))\lr{4-\frac{(m+n-1)\delta}{2}}\frac{K^{-\alpha}}{2}.
\]
Further, the departures from the active population due to transfer or death occur at rate \[
N_{m,n}^{K,a}(t\log K)\lr{1+\frac{C(1-pm\delta)}{K}\sum_{m',n'=0}^{L}N_{m',n'}^{K,a}(t\log K)+\tau\frac{\sum_{m'=0,n'>n}^{L}N_{m',n'}^{K,a}(t\log K)}{\sum_{m',n'=0}^L N_{m',n'}^{K,a}(t\log K)}},
\]
which can be bounded for $\rho$ small enough by \begin{align*}
&N_{m,n}^{K,a}(t\log K)\lr{4-3pm\delta-3C(1-pm\delta)\veps_1}\\
\leq& N_{m,n}^{K,a}(t\log K)\lr{1+\frac{C(1-pm\delta)}{K}\sum_{m',n'=0}^{L}N_{m',n'}^{K,a}(t\log K)+\tau\frac{\sum_{m'=0,n'>n}^{L}N_{m',n'}^{K,a}(t\log K)}{\sum_{m',n'=0}^L N_{m',n'}^{K,a}(t\log K)}}\\
\leq & N_{m,n}^{K,a}(t\log K)\lr{4-3pm\delta+(3C+\rho)(1-pm\delta)\veps_1+\tau\frac{C\rho\veps_1}{3-3\veps_1}}.
\end{align*}
Similarly, the active to dormant transfer rate is given by \[
N_{m,n}^{K,a}(t\log K)\frac{Cpm\delta}{K}\sum_{m',n'=0}^{L}N_{m',n'}^{K,a}(t\log K),
\]
which satisfies the bounds \begin{align*}
&N_{m,n}^{K,a}(t\log K)\lr{3pm\delta-3Cpm\delta\veps_1}\\
\leq& N_{m,n}^{K,a}(t\log K)\frac{Cpm\delta}{K}\sum_{m',n'=0}^{L}N_{m',n'}^{K,a}(t\log K)\\
\leq & N_{m,n}^{K,a}(t\log K)\lr{3pm\delta+Cpm\delta(3+\rho)\veps_1}.
\end{align*}
\subsubsection*{Step 1: Induction on the traits}
We will now show by induction on $m$ and $n$ that the bounds \begin{align}
K^{\beta_{m,n}(t)-(m+n+1)\veps_1}-1\leq N_{m,n}^K(t\log K)\leq K^{\beta_{m,n}(t)+(m+n+1)\veps_1}-1\label{eq:claim1}
\end{align}
hold true for $t\leq\theta_1^K\wedge T$. In this situation, the condition reads as \[
K^{((1-(m+n)\alpha)\vee 0)-(m+n+1)\veps_1}-1\leq N_{m,n}^K(t\log K)\leq K^{((1-(m+n)\alpha)\vee 0)+(m+n+1)\veps_1}-1
\]
\subsubsection*{Step 1a): Traits $(0,y)$}
 For $m=n=0$ this is obviously satisfied by definition of $\theta_1^K$.\\
 \textbf{Base case: $n=0$.}\\
 \textbf{Base case: $m=1$.} For $m=1$ and $n=0$ we couple the process $N_{1,0}^K$ with processes $\tilde{Z}_{1,0}^K$ and $\hat{Z}_{1,0}^K$, such that component-wise \[
\tilde{Z}_{1,0}^K(t\log K)\leq N_{1,0}^K(t\log K)\leq \hat{Z}_{1,0}^K(t\log K),
\]
where $\tilde{Z}_{1,0}^K$ is a $BBPI_K(4-\tfrac{\delta}{2}-\bar{C}\veps_1,0,4-3p\delta+2\bar{C}\veps_1,\kappa,3p\delta-\bar{C}\veps_1,\sigma,0,1-\alpha-\veps_1,1-\alpha-\veps_1)$ and $\hat{Z}_{1,0}^K$ is a $BBPI_K(4-\tfrac{\delta}{2}+\bar{C}\veps_1,0,4-3p\delta-2\bar{C}\veps_1,\kappa,3p\delta+\bar{C}\veps_1,\sigma,0,1-\alpha+\veps_1,1-\alpha+\veps_1)$, where \[
\bar{C}\coloneqq 1+(1\vee\tau)(1+C)(6+\rho).
\]
Indeed, this coupling is justified by the bounds on the rates derived in Step 0  of this proof. Hence, applying Theorem \ref{Theorem: Main Convergence} (i), we see that \begin{align*}
&\lim\limits_{K\to\infty}\frac{\log(1+\tilde{Z}_{1,0}^K(t\log K))}{\log K}\\
=&\ (1-\alpha-\veps_1)+\lr{0\vee\frac{-\frac{\delta}{2}-2\bar{C}\veps_1-(\kappa+\sigma)+\sqrt{(-\frac{\delta}{2}-2\bar{C}\veps_1+\kappa+\sigma)^2+4(3p\delta-\bar{C}\veps_1)\sigma}}{2}\cdot t}\\
\geq&\ 1-\alpha-\veps_1\\
\end{align*}
and similarly, since $S((\delta,0),(0,0))<0$, we have for $\veps_1$ sufficiently small that 
\begin{align*}
&\lim\limits_{K\to\infty}\frac{\log(1+\hat{Z}_{1,0}^K(t\log K))}{\log K}\\
=&\ (1-\alpha+\veps_1)+\lr{0\vee\frac{-\frac{\delta}{2}+2\bar{C}\veps_1-(\kappa+\sigma)+\sqrt{(-\frac{\delta}{2}+2\bar{C}\veps_1+\kappa+\sigma)^2+4(3p\delta+\bar{C}\veps_1)\sigma}}{2}\cdot t}\\
\leq&\ 1-\alpha+\veps_1.
\end{align*}
Thus, the claim is shown for $N_{1,0}^K$. For the remainder of the proof, we will use the shorthand notation $S((\tilde{x},\tilde{y}),(x,y))\pm C_*\veps_1$ to indicate the rate of growth of a bi-type branching process whose birth, death, switching and transfer rates are modified by some factor of $\veps_1$ and otherwise coincide with those of a bi-type branching process whose growth rate is given by $S((\tilde{x},\tilde{y}),(x,y))$.

\textbf{Induction step for $m-1\to m$, $n=0$:} Now assume that it has been shown that \begin{align}
K^{((1-(m-1)\alpha)\vee 0)-m\veps_1}-1\leq  N_{m-1,0}^K(t\log K)\leq K^{((1-(m-1)\alpha)\vee 0)+m\veps_1}-1.\label{assumption: (m-1,0)}
\end{align}

Then, we can couple the process $N_{m,0}^K(t\log K)$ with different processes $\tilde{Z}_{m,0}^K$ and $\hat{Z}_{m,0}^K$ such that  \[
\tilde{Z}_{m,0}^K(t\log K)\leq N_{m,0}^K(t\log K)\leq \hat{Z}_{m,0}^K(t\log K),
\]
where the distribution of $\tilde{Z}_{m,0}^K$ is determined by  $BBPI_K(4-\tfrac{m\delta}{2}-\bar{C}\veps_1,0,4-3pm\delta+2\bar{C}\veps_1,\kappa,\linebreak3pm\delta-\bar{C}\veps_1,\sigma,0,(1-(m-1)\alpha)_+-\alpha-m\veps_1,(1-m\alpha-m\veps_1)_+)$ and the distribution of $\hat{Z}_{m,0}^K$ is $BBPI_K(4-\tfrac{m\delta}{2}+\bar{C}\veps_1,0,4-3pm\delta-2\bar{C}\veps_1,\kappa,3pm\delta+\bar{C}\veps_1,\sigma,0,(1-(m-1)\alpha)_+-\alpha+m\veps_1,\linebreak(1-m\alpha+m\veps_1)_+)$.\\

Then, we see from Theorem \ref{Theorem: Main Convergence} (i) or (iii) applied accordingly that for $\veps_1>0$ small enough \begin{align*}
&\lim\limits_{K\to\infty}\frac{\log(1+\tilde{Z}_{m,0}^K(t\log K))}{\log K}\\
=&\begin{cases}
((1-(m-1)\alpha)_+-\alpha-m\veps_1)\vee (1-m\alpha-m\veps_1)_++(S((m\delta,0),(0,0))-C_*\veps_1)t),&\hspace{-0.2cm}\text{if } m\alpha<1\\
0,&\hspace{-0.2cm}\text{otherwise}
\end{cases}\\
\geq&\ \beta_{m,0}(t)-m\veps_1
\end{align*}
and similarly
\begin{align*}
&\lim\limits_{K\to\infty}\frac{\log(1+\hat{Z}_{m,0}^K(t\log K))}{\log K}\\
=&\begin{cases}
((1-(m-1)\alpha)_+-\alpha+m\veps_1)\vee ((1-m\alpha+m\veps_1)_++(S((m\delta,0),(0,0))+C_*\veps_1)t),&\hspace{-0.2cm}\text{if } m\alpha<1\\
0,&\hspace{-0.2cm}\text{otherwise},
\end{cases}\\
\leq&\ \beta_{m,0}(t)+m\veps_1.
\end{align*}

\subsubsection*{Step 1b): Traits $(m\delta,n\delta)$}
\textbf{Induction step for $n-1\to n$:} Assume now that for all $m\in\{0,\ldots,L\}$ and all $n'\leq n-1$ the bounds \[
K^{((1-(m+n')\alpha)\vee 0)-(m+n'+1)\veps_1}-1\leq  N_{m,n'}^K(t\log K)\leq K^{((1-(m+n')\alpha)\vee 0)+(m+n'+1)\veps_1}-1
\]
hold. \\

\textbf{Base case $m=0$:} For the process $N_{0,n}^K$ we only have incoming migration from $N_{0,n-1}^K$. Hence, we can couple \[
\tilde{Z}_{0,n}^K(t\log K)\leq N_{0,n}^K(t\log K)\leq \hat{Z}_{0,n}^K(t\log K),
\]
where  $\tilde{Z}_{0,n}^K$ is a $BPI_K(4-\tfrac{n\delta}{2}+\tau-\bar{C}\veps_1,4+\bar{C}\veps_1,0,(1-(n-1)\alpha)_+-\alpha-n\veps_1,(1-n\alpha-n\veps_1)_+)$ and $\hat{Z}_{0,n}^K$ is a $BPI_K(4-\tfrac{n\delta}{2}+\tau+\bar{C}\veps_1,4-\bar{C}\veps_1,0,(1-(n-1)\alpha)_+-\alpha+\veps_1,(1-n\alpha+n\veps_1)_+)$, which as before gives the sought bounds from (\ref{eq:claim1}) by applying Theorem \ref{Theorem: Onedimensionalconvergence} (i) or (iii).\\

\textbf{Base case $m=1$:} Here, we have incoming mutations as mentioned in the beginning of the proof from two different populations, for which we already have suitable bounds. Thus, we can couple the process $N_{1,n}^K$ as usual with \[
\tilde{Z}_{1,n}^K(t\log K)\leq N_{1,n}^K(t\log K)\leq \hat{Z}_{1,n}^K(t\log K),
\]
where the distribution of $\tilde{Z}_{1,n}^K$ is  $BBPI_K(4-\tfrac{(n+1)\delta}{2}+\tau-\bar{C}\veps_1,0,4-3p\delta+2\bar{C}\veps_1,\kappa,3p\delta-\bar{C}\veps_1,\sigma,0,\linebreak(1-n\alpha)_+-\alpha-(n+1)\veps_1,(1-(n+1)\alpha-(n+1)\veps_1)_+)$ and the law of $\hat{Z}_{1,n}^K$ is determined by a $BBPI_K(4-\tfrac{(n+1)\delta}{2}+\tau+\bar{C}\veps_1,0,4-3p\delta-2\bar{C}\veps_1,\kappa, 3p\delta+\bar{C}\veps_1,\sigma,0, (1-n\alpha)_+-\alpha+(n+1)\veps_1,\linebreak(1-(n+1)\alpha+(n+1)\veps_1)_+)$. Applying Theorem \ref{Theorem: Main Convergence} (i) or (iii) as in the case $m=1$, $n=0$ yields the claim (\ref{eq:claim1}).\\

\textbf{Induction step $m-1\to m$:} Assume that we have shown the bounds \[
K^{((1-(m'+n)\alpha)\vee 0)-(m'+n+1)\veps_1}-1\leq  N_{m',n}^K(t\log K)\leq K^{((1-(m'+n)\alpha)\vee 0)+(m'+n+1)\veps_1}-1
\]
for all $m'\leq m-1$. Then, we couple the process $N_{m,n}^K$ as usual with \[
\tilde{Z}_{m,n}^K(t\log K)\leq N_{m,n}^K(t\log K)\leq \hat{Z}_{m,n}^K(t\log K),
\]
where the distribution of $\tilde{Z}_{m,n}^K$ is  $BBPI_K(4-\tfrac{(m+n)\delta}{2}+\tau-\bar{C}\veps_1,0,4-3pm\delta+2\bar{C}\veps_1,\kappa,\linebreak3pm\delta-\bar{C}\veps_1,\sigma,0,(1-(m+n-1)\alpha)_+-\alpha-(m+n)\veps_1,(1-(m+n)\alpha-(m+n)\veps_1)_+)$ and the law of $\hat{Z}_{1,n}^K$ is  $BBPI_K(4-\tfrac{(m+n)\delta}{2}+\tau+\bar{C}\veps_1,0,4-3pm\delta-2\bar{C}\veps_1,\kappa, 3pm\delta+\bar{C}\veps_1,\sigma,0,\linebreak (1-(m+n-1)\alpha)_+-\alpha+(m+n)\veps_1,(1-(m+n)\alpha+(m+n)\veps_1)_+)$. Again, applying Theorem \ref{Theorem: Main Convergence} (i) or (iii) as before gives the claim (\ref{eq:claim1}).\\

\subsubsection*{Step 2: Showing $\theta_1^K\geq T$}

In Step 1 we have shown that the process $\tfrac{\log(1+N_{m,n}^K(t\log K))}{\log K}$ converges in probability in the space $L^\infty([0,\theta_1^K\wedge T])$ towards $\beta_{n,m}(t)=(1-(n+m)\alpha)_+$. It now suffices to show that $\theta_1^K\geq T$, which can be done in the same manner as in \cite{champagnat2019stochastic}. As we have computed above, for all $t\leq\theta_1^K\wedge T$ we have with high probability \[
\sum_{(m,n)\neq (0,0)}N_{m,n}^K(t\log K)\leq K^{\max_{(m,n)\neq (0,0)}\beta_{m,n}(t)+\tfrac{\alpha}{2}}=K^{1-\tfrac{\alpha}{2}}.
\]
In particular at time $t=\theta_1^K\wedge T$, we have $\textstyle\sum_{(m,n)\neq(0,0)}N_{m,n}^K(t\log K)<\rho\veps_1K$ with high probability. Hence, for $K$ large enough, we see that up to time $\theta_1^K\wedge 2T$ we can couple \[
Z_{0,0,1}^K(t\log K)\leq N_{0,0}^K(t\log K)\leq Z_{0,0,2}^K(t\log K),
\]
where $Z_{0,0,1}^K$ is a $LBDI_K(4(1-\veps),1+C\veps,C,0)$ and $Z_{0,0,2}^K$ is a $LBDI_K(4,1,C,0)$. Applying \cite[Lemma C.1 (i)]{champagnat2019stochastic} to both processes shows that at time $t=\theta_1^K\wedge T$ the process $N_{0,0}^K(t\log K)$ is still close to its equilibrium size with high probability. Thus $\theta_1^K>T$ with probability converging to $1$ as $K\to\infty$ and the proof is completed in this case.

\subsection{Proof of Theorem \ref{Theorem: Main Theorem}, Case 2}
\begin{description}
	\item[Case 2] In the second case, we consider $S((x,y),(0,0))>0$ for some $(x,y)\in\calX$.
\end{description}
\subsubsection{Phase 1}
Note that the bounds on the arrival, departure and migration rates derived in Step 0 of Case 1 remain true. Unfortunately, we are not able to give a closed form for the limiting function $\beta$ as in \cite[Section 4.2.2]{champagnat2019stochastic} without a significant number of cases to be distinguished. However, in this case there exists some time $s_1\leq T_0$, at which for the first time for some $m_2^*,n_2^*\in \{0,\ldots, L\}$ we have $\beta_{m_2^*,n_2^*}(s_1)=\beta_{0,0}(s_1)=1$. Due to our assumptions in the Theorem, $m_2^*,n_2^*$ are unique. Now, we split the time interval $[0,s_1]$ into subintervals, on which all $\beta_{m,n}$ are affine functions. That is, there exists a finite number of times $0=t_0<t_1<\ldots< t_\ell\leq s_1$ such that on the interval $[t_{i-1},t_{i}]$ all functions $\beta_{m,n}$ are of the form \[
\beta_{m,n}(t)=\beta_{m,n}(t_{i-1})+a_{m,n}(t-t_{i-1}),\quad t\in[t_{i-1},t_{i}],
\]
for some constants $a_{m,n}$ which may depend on the time interval. This representation as an affine linear function can be seen from Theorem \ref{Theorem: Main Theorem} (iii). We will now show by induction, first on $n$ and then on $m$, that $\beta_{m,n}^K\to \beta_{m,n}$ as $K\to\infty$ on the interval $[0,t_1\wedge \theta_1^K\wedge T]$. Showing the convergence on the other intervals $[t_{i-1},t_{i}\wedge\theta_1^K\wedge T]$ can be done in the same way.

\subsubsection*{Step 1: The induction on the traits}
Recalling the time $t_{(m,n),1}$ from Theorem \ref{Theorem: Main Theorem} (iii), we want to show that for all $t\in[0,t_1\wedge \theta_1^K\wedge T]$ the bounds \begin{align}
\beta_{m,n}(t)-C_*\veps_1\leq\frac{\log\lr{1+N_{m,n}^K(t\log K)}}{\log K}\label{eq: inequality}\leq \beta_{m,n}(t)+C_*\veps_1
\end{align}
hold with probability converging to $1$ as $K\to\infty$ for some constant $C_*$, which may depend on $m$ and $n$. In particular, the notation $C_*$ does not necessarily refer to any particular constant, but more to a suitable constant, which is sufficiently large. For $m=n=0$, the bounds hold trivially by definition of $\theta_1^K$.\\
In the following, we use the notation \[(\beta_{m,n}(t)+ C_*\veps_1)_\times\coloneqq\begin{cases}
	\beta_{m,n}(t)+ C_*\veps_1,&\quad\text{ if } \beta_{m,n}(t)>0\\
	0,&\quad\text{ otherwise}.
	\end{cases}	
	\]
	
\subsubsection*{Step 1a): Traits $(0,y)$}	
 \textbf{Base case: $n=0$.}\\
\textbf{Base case: $m=1$.} We can couple as in Case 1 the process $N_{1,0}^K$ with processes $\tilde{Z}_{1,0}^K$ and $\hat{Z}_{1,0}^K$, such that component-wise \[
\tilde{Z}_{1,0}^K(t\log K)\leq N_{1,0}^K(t\log K)\leq \hat{Z}_{1,0}^K(t\log K),
\]
where the distribution of $\tilde{Z}_{1,0}^K$ is determined by $BBPI_K(4-\tfrac{\delta}{2}-\bar{C}\veps_1,0,4-3p\delta+2\bar{C}\veps_1,\kappa,\linebreak3p\delta-\bar{C}\veps_1,\sigma,0,\beta_{0,0}(0)-\alpha-C_*\veps_1,\beta_{1,0}(0)-C_*\veps_1)$ and $\hat{Z}_{1,0}^K$ is a $BBPI_K(4-\tfrac{\delta}{2}+\bar{C}\veps_1,0,\linebreak4-3p\delta-2\bar{C}\veps_1,\kappa,3p\delta+\bar{C}\veps_1,\sigma,0,\beta_{0,0}(0)-\alpha+C_*\veps_1,\beta_{1,0}(0)+C_*\veps_1)$, where again \[
\bar{C}\coloneqq 1+(1\vee\tau)(1+C)(6+\rho).
\]
Obviously, we obtain the same convergence as before, but the inequalities derived may not apply anymore. By Theorem \ref{Theorem: Main Convergence} (i), we have the convergence \begin{align*}
\lim\limits_{K\to\infty}\frac{\log(1+\tilde{Z}_{1,0}^K(t\log K))}{\log K}&= (\beta_{0,0}(0)-\alpha-C_*\veps_1)\vee\lr{\beta_{1,0}(0)-C_*\veps_1+S((\delta,0),(0,0))t}\\
&=(\beta_{1,0}(0)+S((\delta,0),(0,0))t)\vee (\beta_{0,0}(t)-\alpha)-C_*\veps_1
\end{align*}
and \begin{align*}
\lim\limits_{K\to\infty}\frac{\log(1+\hat{Z}_{1,0}^K(t\log K))}{\log K}&= (\beta_{0,0}(0)-\alpha+C_*\veps_1)\vee\lr{\beta_{1,0}(0)+C_*\veps_1+S((\delta,0),(0,0))t}\\
&= (\beta_{1,0}(0)+S((\delta,0),(0,0))t)\vee (\beta_{0,0}(t)-\alpha)+C_*\veps_1,
\end{align*}
by using $\beta_{1,0}(0)=\beta_{0,0}(0)-\alpha$.\\

\textbf{Induction step for }$m-1\sra m$\textbf{,} $n=0$\textbf{:} Now, assume that (\ref{eq: inequality}) has been shown for all $m'\leq m-1$ and $n=0$. Our goal is to show that (\ref{eq: inequality}) also holds for $m$ and $n=0$. For this purpose, we couple with $\tilde{Z}_{m,0}^K$ and $\hat{Z}_{m,0}^K$ such that  \[
\tilde{Z}_{m,0}^K(t\log K)\leq N_{m,0}^K(t\log K)\leq \hat{Z}_{m,0}^K(t\log K),
\]
where the distribution of $\tilde{Z}_{m,0}^K$ is determined by  $BBPI_K(4-\tfrac{m\delta}{2}-\bar{C}\veps_1,0,4-3pm\delta+2\bar{C}\veps_1,\kappa,\linebreak3pm\delta-\bar{C}\veps_1,\sigma,a_{m-1,0},\beta_{m-1,0}(0)-\alpha-C_*\veps_1,(\beta_{m,0}(0)-C_*\veps_1)_+)$ and the law of $\hat{Z}_{m,0}^K$ is $BBPI_K\linebreak (4-\tfrac{m\delta}{2}+\bar{C}\veps_1,0,4-3pm\delta-2\bar{C}\veps_1,\kappa,3pm\delta+\bar{C}\veps_1,\sigma,a_{m-1,0},\beta_{m-1,0}(0)-\alpha+C_*\veps_1,(\beta_{m,0}(0)+C_*\veps_1)_\times)$.\\

We distinguish the cases where $\beta_{m,0}(0)-C_*\veps_1$ is strictly positive (to apply Theorem \ref{Theorem: Main Convergence} (i)) or $(\beta_{m,0}(0)-C_*\veps_1)_+=0$ to apply Theorem \ref{Theorem: Main Convergence} (ii) or (iii), depending on $a_{m-1,0}$ being strictly positive or non-positive, which yields the convergence \begin{align*}
&\lim\limits_{K\to\infty}\frac{\log(1+\tilde{Z}_{m,0}^K(t\log K))}{\log K}\\
=& \begin{cases}
(\beta_{m-1,0}(0)-\alpha-C_*\veps_1+a_{m-1,0}t)&\\
\quad\vee (\beta_{m,0}(0)-C_*\veps_1+S((m\delta,0),(0,0))t)\vee 0, &\text{ if } \beta_{m,0}(0)-C_*\veps_1>0\\[5pt]
(S((m\delta,0),(0,0))\vee a_{m-1,0})\lr{t-\frac{\abs{\beta_{m-1,0}(0)-\alpha-C_*\veps_1}}{a_{m-1,0}}}&\\
\quad\vee 0, &\text{ if }\beta_{m,0}(0)-C_*\veps_1\leq 0< a_{m-1,0}\\[5pt]
0,&\text{ if }\beta_{m,0}(0)-C_*\veps_1, a_{m-1,0}\leq 0.
\end{cases}
\end{align*}
Note that in the second case it holds $\tfrac{\abs{\beta_{m-1,0}(0)-\alpha-C_*\veps_1}}{a_{m-1,0}}=t_{(m,0),1}+C_*\veps_1$. Even though the functions $\beta_{m,n}$ do not have a change in slope by our choice of $t_1$, this limit of the coupled process may. Also, if the maximum in the second case is attained by $a_{m-1,0}$, then the second case reads as \[
a_{m-1,0}t+\beta_{m-1,0}(0)-\alpha-C_*\veps_1
\] by using that $\beta_{m-1,0}(0)-\alpha-C_*\veps_1\leq 0$. In particular, using $\beta_{m-1,0}(t)=\beta_{m-1,0}(0)+a_{m-1,0}t$, we obtain in each case \begin{align*}
&\lim\limits_{K\to\infty}\frac{\log(1+\tilde{Z}_{m,0}^K(t\log K))}{\log K}\\
\geq&\left[(\beta_{m,0}(0)+S((m\delta,0),(0,0))(t-(t\wedge t_{(m,0),1})))\vee (\beta_{m-1,0}(t)-\alpha)\vee 0\right]-C_*\veps_1.
\end{align*}
A similar application of Theorem \ref{Theorem: Main Convergence} for $\hat{Z}_{m,0}^K$ entails \begin{align*}
&\lim\limits_{K\to\infty}\frac{\log(1+\hat{Z}_{m,0}^K(t\log K))}{\log K}\\
\leq&\left[(\beta_{m,0}(0)+S((m\delta,0),(0,0))(t-(t\wedge t_{(m,0),1})))\vee (\beta_{m-1,0}(t)-\alpha)\vee 0\right]+C_*\veps_1,
\end{align*}
which finishes the induction for $n=0$.\\
\subsubsection*{Step 1b): Traits $(m\delta,n\delta)$}
\textbf{Induction step: $n-1\to n$.} We assume the bounds in (\ref{eq: inequality}) have been shown for all $m'\in\{0,\ldots,L\}$ and $n'\leq n-1$.\\

\textbf{Base case: $m=0$.} Here, the immigration is only coming from $N_{0,n-1}^K$ and hence we can couple with processes $\tilde{Z}_{0,n}^K$ and $\hat{Z}_{0,n}^K$ such that \[
\tilde{Z}_{0,n}^K(t\log K)\leq N_{0,n}^K(t\log K)\leq \hat{Z}_{0,n}^K(t\log K),
\]
where $\tilde{Z}_{0,n}^K$ is a  $BPI_K(4-\tfrac{n\delta}{2}+\tau-\bar{C}\veps_1,4+2\bar{C}\veps_1,a_{0,n-1},\beta_{0,n-1}(0)-\alpha-C_*\veps_1,(\beta_{0,n}(0)-C_*\veps_1)_+)$ and the law of $\hat{Z}_{0,n}^K$ is  $BPI_K(4-\tfrac{n\delta}{2}+\tau+\bar{C}\veps_1,4-2\bar{C}\veps_1,a_{0,n-1},\beta_{0,n-1}(0)-\alpha+C_*\veps_1,(\beta_{0,n}(0)+C_*\veps_1)_\times)$. Now applying Theorem \ref{Theorem: Onedimensionalconvergence} shows (\ref{eq: inequality}) in this case.\\

\textbf{Base case: $m=1$.} This case can be treated as the induction step below. \\

\textbf{Induction step: $m-1\to m$.} Now, we assume that for all $m'\leq m-1$ we have shown the inequality (\ref{eq: inequality}). Then, it also holds for $m$ since we can again distinguish the immigration from outside to be dominated either from $N_{m-1,n}^K$ or from $N_{m,n-1}^K$ and then we can couple as usual (in the case that the immigration is dominated by $N_{m-1,n}^K$) with processes $\tilde{Z}_{m,n}^K$ and $\hat{Z}_{m,n}^K$ such that \[
\tilde{Z}_{m,n}^K(t\log K)\leq N_{m,n}^K(t\log K)\leq \hat{Z}_{m,n}^K(t\log K),
\]
where $\tilde{Z}_{m,n}^K$ is given by a  $BBPI_K(4-\tfrac{(m+n)\delta}{2}+\tau-\bar{C}\veps_1,0,4-3pm\delta+2\bar{C}\veps_1,\kappa,3pm\delta-\bar{C}\veps_1,\linebreak\sigma,a_{m-1,n},\beta_{m-1,n}(0)-\alpha-C_*\veps_1,(\beta_{m,n}(0)-C_*\veps_1)_+)$ and the law of $\hat{Z}_{m,n}^K$ is determined by $BBPI_K\linebreak(4-\tfrac{(m+n)\delta}{2}+\tau+\bar{C}\veps_1,0,4-3pm\delta-2\bar{C}\veps_1,\kappa,3pm\delta+\bar{C}\veps_1,\sigma,a_{m-1,n},\beta_{m-1,n}(0)-\alpha+C_*\veps_1,\linebreak(\beta_{m,n}(0)+C_*\veps_1)_\times)$. As before, applying Theorem \ref{Theorem: Main Convergence} in each case shows the claimed inequality (\ref{eq: inequality}) with probability converging to $1$ as $K\to\infty$. This finishes the induction for the first phase.\\

Performing the induction in $n$ and $m$ as above also for the remaining intervals $[t_{i-1},t_i\wedge\theta_1^K\wedge T]$ shows the bounds (\ref{eq: inequality}) on the entire interval $[0,s_1\wedge\theta_1^K\wedge T]$ with probability converging to $1$ due to the Markov property. The only changes that need to be made are in the starting conditions of the coupled processes, where we replace $\beta_{m,n}(0)$ by $\beta_{m,n}(t_{i-1})$.
\subsubsection*{Step 2: Deriving a lower bound for $\theta_1^K$ }
Next, we will show that $(s_1-\eta)\wedge T<\theta_1^K$ for any $\eta>0$ with high probability. Assume for now $T>s_1-\eta$. By definition of $s_1$, all functions $\beta_{m,n}$ are bounded away from $1$ on the interval $[0,s_1-\eta]$ except for $\beta_{0,0}$. Hence for all $t\leq \theta_1^K\wedge (s_1-\eta)$ we have \[
\sum_{(m,n)\neq(0,0)}N_{m,n}^K(t\log K)\leq K^{\max_{(m,n)\neq (0,0)}\beta_{m,n}(t)+\tilde{\veps}}\leq K^{1-\tilde{\veps}}
\]
for $\tilde{\veps}>0$ sufficiently small with probability converging to $1$. Hence, to show $s_1-\eta<\theta_1^K$, we also need to exclude the possibility of $N_{0,0}^K$ exiting a neighbourhood of its equilibrium size. Indeed, we can couple the process $N_{0,0}^K$ with processes \[
Z_{0,0,1}^K(t\log K)\leq N_{0,0}^K(t\log K)\leq Z_{0,0,2}^K(t\log K),
\]
where $Z_{0,0,1}^K$ is a $LBDI_K(4(1-\veps),1+C\veps,C,0)$ and $Z_{0,0,2}^K$ is a $LBDI_K(4,1,C,0)$ and $\epsi$. As in case 1, applying \cite[Lemma C.1 (i)]{champagnat2019stochastic} to both processes shows that at time $t=\theta_1^K\wedge (s_1-\eta)$ the process $N_{0,0}^K(t\log K)$ is close to its equilibrium size with high probability. Thus, recalling $T>s_1-\eta$,  $\theta_1^K>(s_1-\eta)\wedge T$ with probability converging to $1$ as $K\to\infty$. In particular, for $T<s_1-\eta$ it holds $\theta_1^K>(s_1-\eta)\wedge T$ with high probability.

Therefore, we can conclude by letting $\veps_1\downto 0$ that the convergence $\beta_{m,n}^K\to \beta_{m,n}$ in probability as $K\to\infty$ on the interval $[0,(s_1-\eta)\wedge T]$ holds true. 

\subsubsection{Intermediate Phase 1}\label{SubsubSection: Intermediate Phase 1}
In this intermediate phase, we will show that the resident trait $(0,0)$ experiences competition with an invasive trait $(m^*_2\delta,n^*_2\delta)$, which we will show to be of order $K$ at the end of the first phase $[0,\theta_1^K\log K]$. Hence our goal is twofold: Firstly we want to show that $\theta_1^K\to s_1$ as $K\to\infty$ in probability. Secondly, we want to show that at some time $\sigma_2^K\log K=\theta_1^K\log K+ T(\veps_1,\rho)$ the competition leads to the invasive trait becoming resident and its size being close to its equilibrium size. At the same time $N_{0,0}^K$ becomes smaller than $\rho\veps_1 K$.
\subsubsection*{Step 1: Convergence of $\theta_1^K\to s_1$}
We know from the end of the previous section where we proved $\beta_{m,n}^K\to\beta_{m,n}$ on the interval $[0,(s_1-\eta)\wedge T]$ that $s_1-\eta<\theta_1^K$ with high probability. Thus, to show $\theta_1^K\to s_1$, it suffices to show $\theta_1^K< s_1+\eta$ for any $\eta>0$ with probability converging to $1$ as $K\to\infty$.

Towards a contradiction, assume that $\theta_1^K\geq s_1+\eta$. Then, the couplings on the interval $[t_{\ell-1},t_\ell\wedge\theta_1^K]$ with $t_\ell=s_1$ can be extended until time $t_\ell^*=s_1+\eta$ since the couplings are valid as long as the time $t$ satisfies $t\leq \theta_1^K$. In particular, for the coupling of $N_{m^*_2,n^*_2}^K$ we obtain the lower bound \begin{align*}
&\lim\limits_{K\to\infty}\frac{\log(1+N_{m^*_2,n^*_2}^K(t\log K))}{\log K}\\
\geq &\left[(\beta_{m^*_2,n^*_2}(t_{\ell-1})+S((m^*_2\delta,n^*_2\delta),(0,0))((t-t_{\ell-1})-((t-t_{\ell-1})\wedge t_{(m^*_2,n^*_2),1})))\right.\\
&\ \vee (\beta_{m^*_2-1,n^*_2}(t_{\ell-1})+a_{m^*_2-1,n^*_2}(t-t_{\ell-1})-\alpha)\\
&\ \vee(\beta_{m^*_2,n^*_2-1}(t_{\ell-1})+a_{m^*_2,n^*_2-1}(t-t_{\ell-1})-\alpha)\vee 0\Big]-C_*\veps_1.
\end{align*}
We know however that at time $s_1$ the last expression converges to $1$ as $\veps_1\downto 0$ and by definition of $s_1$ the lower bound is strictly increasing on some interval $[s_1-\eta,s_1]$. Since the lower bound is the maximum of different affine functions, it remains strictly increasing on the interval $[s_1,s_1+\eta]$. In particular, for $\veps_1$ small enough, at time $s_1+\eta$ the lower bound becomes larger than $1$, which is a contradiction since \[
\lim\limits_{K\to\infty}\frac{\log(1+N_{m^*_2,n^*_2}^K(t\log K))}{\log K}\\
\leq 1
\]
for all $t\geq 0$.
Hence with probability converging to $1$ we have $\theta_1^K< s_1+\eta$. We conclude $\theta_1^K\to s_1$ in probability.

\subsubsection*{Step 2: Emergence of a new population}

For our second goal, we need to show that $N_{0,0}^K$ does not exit a neighbourhood of its equilibrium, so that at time $\theta_1^K$ the population  of trait $(m_2^*\delta,n_2^*\delta)$ emerges. This part of the proof is identical to \cite[Section 4.2.3]{champagnat2019stochastic}, but is repeated here for the reader's convenience. Unfortunately, we cannot use the coupling from the previous phase anymore since $\theta_1^K\to s_1$ and therefore $\beta_{m,n}(\theta_1^K)$ are not bounded away from $1$. However, we do know that for $K$ sufficiently large, the emigration from trait $(0,0)$, which occurs at rate $4K^{-\alpha}$, can be bounded by $C\rho\veps_1$. Then, on the time interval $[0,\theta_1^K\wedge T]$, we can couple \[
Z_{0,0,1}^K(t\log K)\leq N_{0,0}^K(t\log K)\leq Z_{0,0,2}^K(t\log K),
\]
where $Z_{0,0,1}^K$ is a $LBDI_K(4-C\rho\veps_1,1+\tfrac{\tau\rho\veps_1}{3/C-3\veps_1}+C\rho\veps_1,C,0)$ and $Z_{0,0,2}^K$ is a $LBDI_K(4,1,C,0)$. We easily identify the equilibria \[
\bar{z}_{0,0,1}=\frac{3}{C}-\veps_1\lr{2\rho+\frac{\tau\rho}{3-3C\veps_1}}\quad\text{ and }\quad \bar{z}_{0,0,2}=\frac{3}{C}.
\]
Now, we choose $\rho$ sufficiently small such that $\bar{z}_{0,0,1}$ is contained in the chosen domain around the equilibrium of $N_{0,0}^K$, that is $\bar{z}_{0,0,1}\in[\tfrac{3}{C}-3\veps_1,\tfrac{3}{C}+3\veps_1]$, which holds as soon as \[
2\rho+\frac{\tau\rho}{3-3C\veps_1}<3.
\]
Applying \cite[Lemma C.1.]{champagnat2019stochastic} to $Z_{0,0,1}^K$ and $Z_{0,0,2}^K$ shows that \[
\lim\limits_{K\to\infty}\P\lr{\forall t\in[0,s_1+\eta]\colon\frac{Z_{0,0,1}^K(t\log K)}{K}\geq \frac{3}{C}-3\veps_1}=1\]
and similarly
\[\lim\limits_{K\to\infty}\P\lr{\forall t\in[0,s_1+\eta]\colon\frac{Z_{0,0,2}^K(t\log K)}{K}\leq \frac{3}{C}+3\veps_1}=1.
\]
Note that the coupling above is only true until time $\theta_1^K<s_1+\eta$, but the bounds for the processes $Z_{0,0,*}^K$ with $*\in\{1,2\}$ are still true for any later times. In particular, we obtain for the time $\theta_1^K$ that \[
\lim\limits_{K\to\infty}\P\lr{\frac{Z_{0,0,1}^K(\theta_1^K\log K)}{\log K}\geq \frac{3}{C}-3\veps_1}=1=\lim\limits_{K\to\infty}\P\lr{\frac{Z_{0,0,2}^K(\theta_1^K\log K)}{\log K}\leq \frac{3}{C}+3\veps_1}.
\]
Since at time $\theta_1^K$ the coupling still holds, we see \[
\lim\limits_{K\to\infty}\P\lr{\frac{N_{0,0}^K(\theta_1^K\log K)}{\log K}\in\left[\frac{3}{C}-3\veps_1,\frac{3}{C}+3\veps_1\right]}=1.
\]
Hence, by definition of $\theta_1^K$ we must have \[
\sum_{(m,n)\neq (0,0)}N_{m,n}^K(\theta_1^K\log K)\geq \rho\veps_1 K
\]
with probability converging to $1$ as $K\to\infty$. Since we have assumed that at any given time at most two of the limiting exponents $\beta_{m,n}$ may be $1$ and we already know from above that $\beta_{m_2^*,n_2^*}(s_1)=1$, it must hold for some $\tilde{\veps}>0$ that \[
\max_{(m,n)\notin \{(0,0),(m_2^*,n_2^*)\}}\beta_{m,n}(s_1)\leq 1-\tilde{\veps}.
\]
Since we have shown the convergences $\beta_{m,n}^K\to \beta_{m,n}$ on $[0,s_1-\eta]$, by the continuity of the exponents (see Lemma \ref{lemmacontinuity}) and the convergence $\theta_1^K\to s_1$ we conclude \[
\sum_{(m,n)\notin \{(0,0),(m_2^*,n_2^*)\}}N_{m,n}^K(\theta_1^K\log K)\leq K^{1-\tfrac{\tilde{\veps}}{2}}
\]
with high probability for $\tilde{\veps}>0$ sufficiently small. Hence, it must hold $N_{m_2^*,n_2^*}^K(\theta_1^K\log K)\geq \rho\veps_1 K/2$ with probability converging to $1$. It is important to note that by definition of $\theta_1^K$, this is equivalent to demanding \[
N_{m_2^*,n_2^*}^K(\theta_1^K\log K)\in\left[\frac{\rho\veps_1 K}{2}, \rho\veps_1 K\right]
\]
which enables us to apply the Propositions from Appendix \ref{Section: Appendix C} in combination with Remark \ref{Remark: Important Proposition Remark}.

\subsubsection*{Step 3: Competition}

Now that we have established the emergence of the invasive trait $(m_2^*\delta,n_2^*\delta)$, we need to distinguish the two cases $m_2^*=0$ and $m_2^*>0$.\\

\textbf{Case(a): $m_2^*=0$.} In this case we can proceed as in \cite{champagnat2019stochastic}, as the invading trait is a one-dimensional process which necessarily performs horizontal transfer. Firstly, we note again due to continuity of the exponent that \[
\sum_{(m,n)\notin \{(0,0),(m_2^*,n_2^*)\}}N_{m,n}^K(t\log K)\leq K^{1-\tfrac{\tilde{\veps}}{4}}
\]
for all $t\in[\theta_1^K,\theta_1^K+s]$ for $s>0$ sufficiently small with probability converging to $1$. Being consistent with the notation in \cite[Section C.2.2]{champagnat2019stochastic}, we define for any time $t$ the functions \[
b_1^K(t)=4(1-K^{-\alpha}), \quad b_2^K(t)=\lr{4-\frac{(m_2^*+n_2^*)\delta}{2}}(1-K^{-\alpha}),
\]
\[
d_1^K(t)=d_2^K(t)=1+\lr{\frac{C}{K}+\frac{\tau}{\sum_{m=0, n>0}^{L}N_{m,n}^{K,a}(t)}}\sum_{(m,n)\notin\{(0,0),(m_2^*,n_2^*)\}}N_{m,n}^{K,a}(t),
\]
\[
\tau^K(t)=\tau\cdot\frac{N_{0,0}^{K,a}(t)+N_{m_2^*,n_2^*}^{K,a}(t)}{\sum_{m,n=0}^L N_{m,n}^{K,a}(t)},\quad \gamma_1^K(t)=0,\quad \gamma_2^K(t)\leq 4K^{-\alpha}N_{0,0}^{K,a}(t).
\]
Note that for the immigration rate $\gamma_2^K$ we would need to consider the incoming immigration from the neighbouring traits. However, if $(0,0)$ is not one of them, those traits are of size of order strictly less than $K$, so the upper bound for $\gamma_2^K(t)$ is justified. The above functions except for $\gamma_2^K$ converge on the interval $[\theta_1^K\log K, (\theta_1^K+s)\log K]$ to $b_1=4$, $b_2=4-\tfrac{(m_2^*+n_2^*)\delta}{2}$, $d_1=d_2=1$, $\tau$ and $0$ respectively in order of appearance. For $\gamma_2^K(t)$ we obtain the convergence $\tfrac{\gamma_2^K(t)}{K}\to 0$ as $K\to\infty$.

Now, we can apply the Markov property at time $\theta_1^K$ and subsequently Lemma C.3 from \cite{champagnat2019stochastic}, which gives us the existence of a finite time $T(\rho,\veps_1)$ such that with probability converging to $1$ we have \[
N_{0,0}^K(\theta_1^K\log K+T(\rho,\veps_1))\leq \rho\veps_1 K\]
and\[
\frac{N_{m_2^*,n_2^*}^{K,a}(\theta_1^K\log K+T(\rho,\veps_1))}{K}\in\left[\bar{z}_{m_2^*,n_2^*}^a-\veps_2,\bar{z}_{m_2^*,n_2^*}^a+\veps_2 \right],
\]
where $\bar{z}_{m_2^*,n_2^*}^a$ denotes the active equilibrium population size (which in this case coincides with the total equilibrium population size) of the rescaled process $\tfrac{N_{m_2^*,n_2^*}^K}{K}$.

Hence, we can define the end of the first intermediate phase as \[
\sigma_2^K\log K=\theta_1^K\log K+ T(\rho,\veps_1).
\]
In particular, we have $\sigma_2^K\to s_1$ in probability as $K\to\infty$. At time $\sigma_2^K\log K$ we can use the continuity of the exponent and are left with the following bounds on our populations \[
N_{0,0}^K(\sigma_2^K\log K)\in[K^{1-\veps_1}, \rho\veps_1 K],\quad \frac{N_{m_2^*,n_2^*}^K(\theta_1^K\log K+T(\rho,\veps_1))}{K}\in\left[\bar{z}_{m_2^*,n_2^*}^a-\veps_2,\bar{z}_{m_2^*,n_2^*}^a+\veps_2 \right]
\]
and for all $(m,n)\notin \{(0,0),(m_2^*,n_2^*)\}$ we have, again using the continuity argument from Lemma \ref{lemmacontinuity},
\[
\frac{\log(1+N_{m,n}^K(\sigma_2^K\log K))}{\log K}\in[\beta_{m,n}(s_1)-\veps_2,\beta_{m,n}(s_1)+\veps_2].
\]
Note that populations for which $\beta_{m,n}(s_1)=0$ are actually extinct at time $\sigma_2^K$. This is due to our assumption that in this case we must have $\beta_{m,n}(t)=0$ on an interval $[s_1-\veps,s_1]$, which due to our starting condition implies negative fitness and weak immigration and hence by Lemma \ref{lemmaextinction} extinction of the population. Then, applying Lemma \ref{lemmanomigration} shows for $K$ sufficiently large that there is no immediate resurrection of the population after time $s_1$.\\

\textbf{Case(b): $m_2^*>0$.} Now, the individuals of the invading trait are able to become dormant. Hence, we have competition between a resident one-dimensional process and an invading two-dimensional process. Note that we may or may not have horizontal transfer exhibited from the invading trait. As in Case(a) we define a number of functions and apply the corresponding result on competition. The functions to be defined are \[
a_1^K(t)=4(1-K^{-\alpha}),\quad b_1^K(t)=\lr{4-\frac{(m_2^*+n_2^*)\delta}{2}}(1-K^{-\alpha}),
\]
\[
d_1^K(t)=1+\lr{\frac{C}{K}+\frac{\tau}{\sum_{m=0, n>0}^{L}N_{m,n}^{K,a}(t)}}\sum_{(m,n)\notin\{(0,0),(m_2^*,n_2^*)\}}N_{m,n}^{K,a}(t),\quad d_2^K(t)\equiv \kappa,
\]
\[
\gamma_1^K(t)=0,\quad \gamma_2^K(t)\leq 4K^{-\alpha}N_{0,0}^{K,a}(t).
\]
If there is no horizontal transfer (that is $n_2^*=0$), we set $\tau^K(t)\equiv 0$. Otherwise we set \[
\tau^K(t)=\tau\cdot\frac{N_{0,0}^{K,a}(t)+N_{m_2^*,n_2^*}^{K,a}(t)}{\sum_{m,n=0}^L N_{m,n}^{K,a}(t)}.
\]
Since we have non-negative horizontal transfer exerted from the invading trait onto $(0,0)$ and we have dormancy for the invading but not for the initially resident trait, we can apply Proposition \ref{Proposition: 3Dpositive} together with Remark \ref{Remark: Important Proposition Remark} due to the same convergence arguments made in Case(a). Hence, there exists some finite time $T(\rho,\veps_1)$ such that with probability larger than $1-o_{\veps_1}(1)$ we have \[
N_{0,0}^K(\sigma_2^K\log K)\in[K^{1-\veps_1}, \rho\veps_1 K]\]
 \[\frac{N_{m_2^*,n_2^*}^{K,a}(\theta_1^K\log K+T(\rho,\veps_1))}{K}\in\left[\bar{z}_{m_2^*,n_2^*}^a-\veps_2,\bar{z}_{m_2^*,n_2^*}^a+\veps_2 \right],\]
 and \[\frac{N_{m_2^*,n_2^*}^{K,d}(\theta_1^K\log K+T(\rho,\veps_1))}{K}\in\left[\bar{z}_{m_2^*,n_2^*}^d-\veps_2,\bar{z}_{m_2^*,n_2^*}^d+\veps_2 \right]
\]
as $K\to\infty$, where $\bar{z}_{m_2^*,n_2^*}^d$ is the equilibrium size of the dormant component of the rescaled process. Then, at time \[
\sigma_2^K\log K=\theta_1^K\log K+ T(\rho,\veps_1)
\]
we have the same bounds as in Case(a) with the only difference in the equilibrium size of the process $N_{m_2^*,n_2^*}^K$. 

\subsubsection{Phase $k$}
We will now consider a time interval $[\sigma_k^K\log K, \theta_k^K\log K]$, where $\sigma_k^K\to s_{k-1}$ and $\theta_k^K\to s_k$ in probability. Thus, we consider $k\geq 2$ and assume that we have already completed step $k-1$. In particular, we assume that we have defined a stopping time $\sigma_k^K$ with the convergence property mentioned above such that for the resident population of trait $(m_k^*\delta,n_k^*\delta)$ the bounds \[
\frac{N_{m_k^*,n_k^*}^{K,a}(\sigma_k^K\log K)}{K}\in\left[\bar{z}_{m_k^*,n_k^*}^a-\veps_k,\bar{z}_{m_k^*,n_k^*}^a+\veps_k \right]
\]
and \[
\frac{N_{m_k^*,n_k^*}^{K,d}(\sigma_k^K\log K)}{K}\in\left[\bar{z}_{m_k^*,n_k^*}^d-\veps_k,\bar{z}_{m_k^*,n_k^*}^d+\veps_k \right]
\]
hold. Furthermore, we assume that for the previously resident trait we have \[K^{1-\veps_k}\leq N_{m_{k-1}^*,n_{k-1}^*}^K(\sigma_k^K\log K)\leq \rho\veps_k K.\] For all remaining traits $(m\delta,n\delta)\notin\lrset{(m_{k-1}^*\delta,n_{k-1}^*\delta),(m_{k}^*\delta,n_{k}^*\delta)}$, we assume $N_{m,n}^K(\sigma_k^K\log K)=0$ if $\beta_{m,n}(s_{k-1})=0$ and otherwise we assume \[
\frac{\log(1+N_{m,n}^K(\sigma_k^K\log K))}{\log K}\in[\beta_{m,n}(s_{k-1})-\veps_k,\beta_{m,n}(s_{k-1})+\veps_k].
\]

As in the base case, we introduce the time $\theta_k^K$, which is the time until the active part of the resident trait leaves a neighbourhood of its equilibrium or a new trait emerges, that is \begin{align*}
\theta_k^K\coloneqq\inf\Bigg\{ t\geq \sigma_k^K\ \Bigg\vert\  N_{m_k^*,n_k^*}^{K,a}(t\log K)\notin\Bigg[\lr{\bar{z}_{m_k^*,n_k^*}^a-3\veps_k}&K,\lr{\bar{z}_{m_k^*,n_k^*}^a+3\veps_k}K\Bigg]\\
&\text{or}\sum_{(m,n)\neq(m_k^*,n_k^*)}N_{m,n}^K(t\log K)\geq \rho\veps_kK \Bigg\}.
\end{align*}

\subsubsection*{Step 0: Deriving bounds on the rates}

Similarly to Step 0 in Case 1 of the proof, we can derive similar bounds on the birth, death and migration rates on the time interval $[\sigma_k^K,\theta_k^K]$. The bounds for the birth and arrival due to horizontal transfer rates are \[
4-\frac{(m+n)\delta}{2}+\tau\1_{n>n_{k}^*}\pm C_*\veps_k,
\]
and for the death and emigration due to horizontal transfer we obtain the bounds \[
1+\bar{z}_{m_k^*,n_k^*}^a(1-pm\delta)+\tau\1_{n_{k}^*>n}\pm C_*\veps_k.
\]
The immigration rates stay the same as in the base case, since they do not depend on the resident trait population size. The active to dormant switching rate then satisfy the bounds \[
pm\delta\bar{z}_{m_k^*,n_k^*}^a\pm C_*\veps_k.
\]
\subsubsection*{Step 1: Induction on the traits}
As in the base case, we want to use the bounds given above to couple our processes accordingly and show by induction on the traits the upper and lower bounds on $\beta_{m,n}^K$. For this, we may again decompose the time interval $[s_{k-1},s_k]$ into sections on which all $\beta_{m,n}$ are affine. On the first such subinterval which is of the form $t\in[s_{k-1}, t_1\wedge \theta_k^K\wedge T]$, we can write \[
\beta_{m,n}=\beta_{m,n}(s_{k-1})+a_{m,n}(t-s_{k-1}).
\]
for some constants $a_{m,n}\in\R$.
We will not fully carry out the induction, but give a broad idea, since it is very similar to the base case. If $(m_k^*,n_k^*)=(0,0)$, we are in the same situation as in the base case, so we can use the Markov property at time $\sigma_k^K\log K$ and obtain the same results where in the couplings we need to replace $\beta_{m,n}(0)$ with $\beta_{m,n}(s_{k-1})$.\\

In the case where $(m_k^*,n_k^*)\neq(0,0)$ and $\beta_{0,0}(s_{k-1})>0$, there is no incoming immigration into the trait $(0,0)$ and hence we can use the coupling \[
\tilde{Z}_{0,0}^K(t\log K)\leq N_{0,0}^K(t\log K)\leq \hat{Z}_{0,0}^K(t\log K),
\]
where $\tilde{Z}_{0,0}^K$ is a $BP_K(4-C_*\veps_k,1+\bar{z}_{m_k^*,n_k^*}^a+\tau\1_{n_k^*>0}+C_*\veps_k,\beta_{0,0}(s_{k-1})-C_*\veps_k)$ and $\hat{Z}_{0,0}^K$ is given as $BP_K(4+C_*\veps_k,1+\bar{z}_{m_k^*,n_k^*}^a+\tau\1_{n_k^*>0}-C_*\veps_k,\beta_{0,0}(s_{k-1})+C_*\veps_k)$. For our coupled processes, the convergence theorem \cite[Lemma A.1]{champagnat2019stochastic} implies the bounds \begin{align*}
&\beta_{0,0}(s_{k-1})+S((0,0),(m_k^*\delta,n_k^*\delta))(t-s_{k-1})-C_*\veps_k\\
\leq&\ \frac{\log(1+N_{0,0}^K(t\log K))}{\log K}\\
\leq&\ \beta_{0,0}(s_{k-1})+S((0,0),(m_k^*\delta,n_k^*\delta))(t-s_{k-1})+C_*\veps_k.
\end{align*}
If $\beta_{0,0}(s_{k-1})=0$, then due to the lack of immigration and our observation that populations with $\beta_{m,n}(s_{k-1})=0$ are actually extinct we have $N_{0,0}^K(t\log K)=0$ for all $t\geq\sigma_k^K$.\\

As mentioned, we abbreviate the induction and assume that the bounds 
\begin{align}
\beta_{m,n}(t)-C_*\veps_k\leq\frac{\log\lr{1+N_{m,n}^K(t\log K)}}{\log K}\label{eq: inequality2}
\leq\beta_{m,n}(t)+C_*\veps_k
\end{align}
have been shown up to the neighbouring traits of $(m\delta,n\delta)$ for all $t\in [s_{k-1}, t_1\wedge \theta_k^K\wedge T]$. Then, we need to distinguish the cases where $m=0$ and $m>0$ as well as $n\geq n_k^*$ and $n<n_k^*$. The first distinction corresponds to the question of the ability to become dormant, whereas the second one dictates the way that horizontal transfer influences the dynamics. Furthermore, we need to distinguish whether $N_{m-1,n}^K$ or $N_{m,n-1}^K$ is larger (in terms of orders of powers of $K$) to determine which population is responsible for the immigration rate. Also, we need to separate the cases where $\beta_{m,n}(s_{k-1})=0$ or strictly larger than $0$. In the first case, we need to couple with processes whose initial population size is also $0$. Without loss of generality we assume $N_{m,n-1}^K$ to be of larger order than $N_{m-1,n}^K$ - the other case can be done by switching the corresponding indices. Then, we can couple \[
\tilde{Z}_{m,n}^K(t\log K)\leq N_{m,n}^K(t\log K)\leq \hat{Z}_{m,n}^K(t\log K)
\]
where $\tilde{Z}_{m,n}^K$ and $\hat{Z}_{m,n}^K$ are $BPI_K(4-\tfrac{(m+n)\delta}{2}+\tau\1_{n>n_k^*}\mp C_*\veps_k, 1+\bar{z}_{m_k^*,n_k^*}^a+\tau\1_{n<n_k^*}\pm 2 C_*\veps_k, a_{m,n-1},\linebreak \beta_{m,n-1}(s_{k-1})-\alpha\mp C_*\veps_k, (\beta_{m,n}(s_{k-1})\mp C_*\veps_k)_\times)$ in the case where $m=0$ and otherwise they are determined by $BBPI_K(4-\tfrac{(m+n)\delta}{2}+\tau\1_{n>n_k^*}\mp C_*\veps_k,0, 1+(1-pm\delta)\bar{z}_{m_k^*,n_k^*}^a+\tau\1_{n<n_k^*}\pm 2 C_*\veps_k,\kappa,\linebreak pm\delta\mp C_*\veps_k,\sigma, a_{m,n-1}, \beta_{m,n-1}(s_{k-1})-\alpha\mp C_*\veps_k, (\beta_{m,n}(s_{k-1})\mp C_*\veps_k)_\times)$.

Applying Theorem \ref{Theorem: Onedimensionalconvergence} or \ref{Theorem: Main Convergence} accordingly shows the bounds (\ref{eq: inequality2}) by definition of our fitness function. Continuing this process for all time intervals on which all $\beta_{m,n}$ are affine shows the bounds (\ref{eq: inequality2}) on the entire interval $[s_{k-1},s_k\wedge\theta_k^K\wedge T]$ with probability converging to $1$.

\subsubsection*{Step 2: Deriving a lower bound for $\theta_k^K$}

As in Step 2 of the base case, we want to show that $(s_k-\eta)\wedge T<\theta_k^K$ with probability converging to $1$. Again due to our assumption, we know that all functions $\beta_{m,n}$ except for $\beta_{m_k^*,n_k^*}$ are bounded away from $1$ on the interval $[s_{k-1}+\eta,s_k-\eta]$. Therefore, it again suffices for showing $s_1-\eta<\theta_k^K$ that $N_{m_k^*,n_k^*}^{K,a}$ does not exit a neighbourhood of its equilibrium until time $s_k-\eta$. For this purpose, we can couple with processes \[
Z_{m_k^*,n_k^*,1}^K(t\log K)\leq N_{m_k^*,n_k^*}^K(t\log K)\leq Z_{m_k^*,n_k^*,2}^K(t\log K)
\]
up to time $\theta_k^K$.
Again we need to distinguish between the possibility of becoming dormant or not. If $m_k^*=0$, we can choose $Z_{m_k^*,n_k^*,1}$ as a $LBDI_K((4-\tfrac{(m_k^*+n_k^*)\delta}{2})(1-\veps),1+C\veps,C,K^{-\alpha}N_{m_k^*,n_k^*-1})$ and $Z_{m_k^*,n_k^*,1}$ as a $LBDI_K((4-\tfrac{(m_k^*+n_k^*)\delta}{2}),1,C,K^{-\alpha}N_{m_k^*,n_k^*-1})$. If, on the other hand, we have $m_k^*>0$, we need to distinguish where the immigration is coming from and can choose the process $Z_{m_k^*,n_k^*,1}$ to be determined by a $LBBI_K((4-\tfrac{(m_k^*+n_k^*)\delta}{2})(1-\veps),1+C\veps,\kappa,\sigma,p,C,K^{-\alpha}N_{m_k^*,n_k^*-1})$ if we assume the immigration to be dominated by $N_{m_k^*,n_k^*-1}$. Then we can choose $Z_{m_k^*,n_k^*,2}$ as a $LBBI_K((4-\tfrac{(m_k^*+n_k^*)\delta}{2}),1,\kappa,\sigma,p,C,K^{-\alpha}N_{m_k^*,n_k^*-1})$. Now, applying \cite[Lemma C.1]{champagnat2019stochastic} to the first case and Corollary \ref{Corollary: Equilibriumcloseness} in the case of bi-type processes, we see that at time $s_k-\eta$ the process $N_{m_k^*,n_k^*}^K$ has not exited a neighbourhood of its equilibrium size with probability converging to $1$. In particular, we must have $s_k-\eta<\theta_k^K$ with high probability.

\subsubsection{Intermediate Phase $k$}

The structure of this intermediate phase remains the same as in Section \ref{SubsubSection: Intermediate Phase 1}. 

\subsubsection*{Step 1: Convergence of $\theta_k^K\to s_k$}
This part of the proof can be taken from Step 1 in Intermediate Phase 1 with minor changes in the times and the resident trait and is not repeated here.

\subsubsection*{Step 2: Emergence of a new population}

This part is also very similar. However, we may need to couple with logistic bi-type branching processes instead of single type. Since this is a straightforward adaptation similar to Step 2 of Phase $k$, we do not carry it out here. We do obtain however that \[
N_{m_{k+1}^*,n_{k+1}^*}^K(\theta_k^K\log K)\in\left[\frac{\rho\veps_k K}{2}, \rho\veps_k K\right]
\]
and \[
\sum_{(m,n)\notin \{(m_k^*,n_k^*),(m_{k+1}^*,n_{k+1}^*)\}}N_{m,n}^K(\theta_k^K\log K)\leq K^{1-\tfrac{\tilde{\veps}}{2}}.
\]

\subsubsection*{Step 3: Competition}
By assumption of the theorem, there is competition between the resident and the emerging trait. Distinguishing the cases, we can proceed as in Intermediate Phase 1 and define the corresponding birth, death, migration, switching and horizontal transfer rates which then allow us to apply one of the Propositions from \ref{Proposition: 4Dpositive}, \ref{Proposition: 4Dnegative}, \ref{Proposition: 3Dpositive}, \ref{Proposition: 3Dnegative}, \ref{Proposition: 3Dpositive2} and \ref{Proposition: 3Dnegative2} in conjunction with Remark \ref{Remark: Important Proposition Remark} or \cite[Lemma C.3]{champagnat2019stochastic}, which in each case give us a finite time $T(\rho,\veps_k)$ such that with probability larger than $1-o_{\veps_k}(1)$ we have, as $K\to\infty$, the bounds \[
N_{m_k^*,n_k^*}^K(\theta_k^K\log K+T(\rho,\veps_k))\in[K^{1-\veps_k}, \rho\veps_k K],\]
\[\frac{N_{m_{k+1}^*,n_{k+1}^*}^{K,a}(\theta_k^K\log K+T(\rho,\veps_k))}{K}\in\left[\bar{z}_{m_{k+1}^*,n_{k+1}^*}^a-\veps_k,\bar{z}_{m_{k+1}^*,n_{k+1}^*}^a+\veps_k \right],\]
and \[\frac{N_{m_{k+1}^*,n_{k+1}^*}^{K,d}(\theta_k^K\log K+T(\rho,\veps_k))}{K}\in\left[\bar{z}_{m_{k+1}^*,n_{k+1}^*}^d-\veps_k,\bar{z}_{m_{k+1}^*,n_{k+1}^*}^d+\veps_k \right].
\]
Thus, we can define the time $\sigma_{k+1}^K\log K\coloneqq\theta_k^K\log K+T(\rho,\veps_k)$, at which time the stated properties in the beginning of Step $k$ are satisfied with high probability. That is, for $(m,n)\notin \{(0,0),(m_2^*,n_2^*)\}$ we have again using the continuity argument from Lemma \ref{lemmacontinuity}
\[
\frac{\log(1+N_{m,n}^K(\sigma_{k+1}^K\log K))}{\log K}\in[\beta_{m,n}(s_k)-\veps_{k+1},\beta_{m,n}(s_k)+\veps_{k+1}],
\]
if $\beta_{m,n}(s_k)>0$ and $N_{m,n}^K(\sigma_{k+1}^K\log K)=0$ otherwise. To see the latter part, the argument from the end of Case(a) in Step 3 of Section \ref{SubsubSection: Intermediate Phase 1} still applies. Thus, we have proven Theorem \ref{Theorem: Main Theorem}.
\end{proof}
\newpage

%% file: BitypeProcesses.tex
\section{Results on Bi-Type Branching Processes with Immigration}\label{Section: Appendix B}

In this section, we derive a general convergence result for special bi-type branching processes. More specifically, we want to generalize the following theorem from \cite{champagnat2019stochastic}.

We denote the law of a one-dimensional branching process $(Z^K)_{t\geq 0}$ with birth rate $b\geq 0$, death rate $d\geq 0$ and time dependent immigration at rate $K^ce^{at}$ at time $t\geq 0$ with $a,c\in\R$ by $BPI_K(b,d,a,c,\beta)$, where $Z_0^K=\lfloor K^{\beta}-1\rfloor$.
\begin{theorem}\label{Theorem: Onedimensionalconvergence}
	Let $Z^K$ be a $BPI_K(b,d,a,c,\beta)$ with $c\leq\beta$ and assume either $\beta>0$ or $c\neq 0$. Then the process $\tfrac{\log(1+Z_{t\log K}^K)}{\log K}$ converges when $K$ tends to infinity in probability in $L^\infty([0,T])$ for all $T>0$ to the continuous, deterministic function $\bar{\beta}$ given by \begin{enumerate}[label=\emph{(\roman*)}]
		\item if $\beta>0$, $\bar{\beta}\colon t\mapsto (\beta+rt)\vee(c+at)\vee 0$;
		\item if $\beta=0$, $c<0$ and $a>0$, $\bar{\beta}\colon t\mapsto ((r\vee a)(t-\tfrac{\abs{c}}{a}))\vee 0$;
		\item if $\beta=0$, $c<0$ and $a\leq 0$, $\bar{\beta}\colon t\mapsto 0$;
	\end{enumerate}
	where $r=b-d$.
\end{theorem}
\begin{proof}
	This is Theorem B.5 from \cite{champagnat2019stochastic}.
\end{proof}

 In the spirit of the above theorem, we consider the process $Z_t^K=(X_t^K,Y_t^K)$ with initial  population $(X_0^K,Y_0^K)=(\lfloor K^{\beta}-1\rfloor , \lfloor K^{\gamma}-1\rfloor)$ and transition rates \begin{align*}
(n,m)\mapsto \begin{cases}
(n+1,m),&\quad\text{at rate }b_1n+K^ce^{at}\\
(n,m+1),&\quad\text{at rate }b_2m\\
(n-1,m+1),&\quad\text{at rate }\sigma_1n\\
(n+1,m-1),&\quad\text{at rate }\sigma_2m\\
(n-1,m),&\quad\text{at rate }d_1n\\
(n,m-1),&\quad\text{at rate }d_2m
\end{cases}.
\end{align*}
We refer to the rates $b_1, b_2\geq 0$ as birth rates of $X_t^K$ and $Y_t^K$ respectively, $d_1,d_2\geq 0$ as their respective death rates and $\sigma_1,\sigma_2>0$ are the switching rates. The additional $K^ce^{at}$ represents the immigration into the population from the outside, where $a,c\in\R$.

\begin{notation}\label{notationbitypebranching}
	We denote the distribution of a bi-type branching process as introduced above by $BBPI_K(b_1,b_2,d_1,d_2,\sigma_1,\sigma_2,a,c,\beta,\gamma)$. If the initial condition satisfies $\beta=\gamma$, we use the shorthand notation $BBPI_K(b_1,b_2,d_1,d_2,\sigma_1,\sigma_2,a,c,\beta)$.
\end{notation}

We are now interested in finding some convergence results for the total population size $X_t^K+Y_t^K$ similar to those from Appendix B in \cite{champagnat2019stochastic}. We will show the following theorem.

\begin{theorem}\label{Theorem: Main Convergence}
	Let $Z_t^K=(X_t^K,Y_t^K)$ be a $BBPI_K(b_1,b_2,d_1,d_2,\sigma_1,\sigma_2,a,c,\beta,\gamma)$ as introduced in Notation \ref{notationbitypebranching}. Further assume that $c\leq\beta\vee\gamma$ and $\beta\vee\gamma>0$ or $c\neq 0$ and let $\la$ as in (\ref{Eigenvalue}). Then for all $T\geq 0$ the process \[
	s\mapsto\frac{\log(1+X_{s\log K}^K+Y_{s\log K}^K)}{\log K}
	\]
	converges in probability in $L^\infty([0,T])$ as $K\to\infty$ towards a deterministic function $\bar{\beta}\colon[0,T]\to\R$, which we describe in each case:\\
	\begin{enumerate}[label=\emph{(\roman*)}]
		\item If $\beta\vee\gamma >0$, then $\bar{\beta}(t)=((\beta\vee\gamma)+\la t)\vee(c+at)\vee 0$.
		\item If $\beta\vee\gamma=0$ and $c<0$ and $a>0$, then $\bar{\beta}(t)=(\la\vee a)(t-\tfrac{\abs{c}}{a})\vee 0$.
		\item If $\beta\vee\gamma=0$ and $c<0$ and $a\leq 0$, then $\bar{\beta}(t)=0$.\\
	\end{enumerate}
\end{theorem}
The proof of the theorem will rely partly on Markov's, Chebyshev's and Doob's inequalities, so we first need to derive some bounds for the expected value and variance of our process.
\subsection{Bounds on the Expectation and Variance}
 Our first step is to find the semimartingale decomposition of $X_t^K$ and $Y_t^K$. In order to do so, we introduce some notation.
\begin{notation}\label{Notation: Growth}
	In the following we write $r_1\coloneqq b_1-d_1-\sigma_1$ and $r_2\coloneqq b_2-d_2-\sigma_2$.
\end{notation}

\begin{lemma}\label{Semimartingaledecomposition}
	Consider the process $Z_t^K=(X_t^K,Y_t^K)$ as introduced above. Then there exist càdlàg martingales $M_t^K,N_t^K$ starting at $0$, such that \[
	\begin{pmatrix}
	X_t^K\\
	Y_t^K\\
	\end{pmatrix}=\begin{pmatrix}
	X_0^K\\ Y_0^K
	\end{pmatrix}+\begin{pmatrix}
	M_t^K\\N_t^K
	\end{pmatrix}+\int_{0}^{t}\begin{pmatrix}
	r_1X_s^K+\sigma_2Y_s^K+K^ce^{as}\\
	r_2Y_s^K+\sigma_1 X_s^K
	\end{pmatrix}\ \mathrm{d}s.
	\]
\end{lemma}
\begin{proof}
	This decomposition follows from Dynkin's formula.
	\end{proof}

Our next goal is to identify the rate of growth of our population, which is directly linked to determining the expected value of the population size. In order to do so, we calculate the expected value for our process up to some constants.

\begin{lemma}\label{expectedvalue}
	The expected value $(x_t^K,y_t^K)$ of $(X_t^K,Y_t^K)$ solves the  ordinary differential equation \begin{align}
	\begin{pmatrix}
	\dot{x}_t^K\\\dot{y}_t^K
	\end{pmatrix}=\begin{pmatrix}
	r_1 & \sigma_2\\
	\sigma_1 & r_2
	\end{pmatrix}\begin{pmatrix}
	x_t^K\\y_t^K
	\end{pmatrix}+\begin{pmatrix}
	K^c e^{at}\\0
	\end{pmatrix}\quad\text{and}\quad\begin{pmatrix}
	x_0^K\\y_0^K
	\end{pmatrix}=\begin{pmatrix}
	K^\beta-1\\ K^\gamma-1
	\end{pmatrix}.\label{expectedode}
	\end{align}
\end{lemma}
\begin{proof}
	This is a direct consequence of Lemma \ref{Semimartingaledecomposition}, where we can apply the expected value on both sides. Interchanging the expected value and integral on the right hand side by Fubini shows that $(x_t^K,y_t^K)$ is absolutely continuous. Differentiating both sides gives the differential equation (\ref{expectedode}).
\end{proof}
Note that this differential equation can be solved easily: The matrix \begin{align}
\begin{pmatrix}
r_1 & \sigma_2\\
\sigma_1 & r_2
\end{pmatrix}=SDS^{-1}\label{diagonalization}
\end{align}
can be diagonalised, because its eigenvalues $\la$ and $\tilde{\la}$ can be written as \begin{align}
\la=\frac{r_1+r_2+\Delta}{2}\quad\text{and}\quad \tilde{\la}=\frac{r_1+r_2-\Delta}{2},\label{Eigenvalue}
\end{align}
where $\Delta=\sqrt{(r_1-r_2)^2+4\sigma_1\sigma_2}\neq 0$ and hence $\la>\tilde{\la}$. 
In particular we are now able to give a characterization of the expected values for $X_t^K$ and $Y_t^K$.
\begin{lemma}\label{explicitexpectation}
	The expected values $(x_t^K,y_t^K)$ of $(X_t^K,Y_t^K)$ satisfy for $t>0$ the asymptotic relation \begin{align*}
	x_t^K,y_t^K=\begin{cases}
	\Theta(K^ce^{at}),&\quad\text{if } a>\la\\
	\Theta((x_0^K+y_0^K+K^c)e^{\la t}),&\quad\text{if }\la>a\\
	\Theta((x_0^K+y_0^K+(1+t)K^c)e^{\la t}),&\quad\text{if }\la=a,
	\end{cases}
	\end{align*}
	where we use the notation $f^K=\Theta(g^K)$ for two families of functions $f^K,g^K\colon[0,\infty)\to\R$ if there exists some finite constant $C>0$ such that for all $t\geq 0$ we have \[
	\lim\limits_{K\to\infty}\frac{f^K(t)}{g^K(t)}= C.
	\]
	In fact, there exists a constant $\tilde{C}>0$ sufficiently large such that for all $K\geq 0$ and all $t\geq 0$ we have \[
	x_t^K,y_t^K\leq \tilde{C}\lr{[(x_0^K+y_0^K+(1+t)K^c)e^{\la t}]\vee [K^ce^{at}]}.
	\]
\end{lemma}
\begin{proof}
	
	The solution to the differential equation (\ref{expectedode}) is known to be \[
	\begin{pmatrix}
	x_t^K\\ y_t^K
	\end{pmatrix}=Se^{Dt}S^{-1} \begin{pmatrix}
	x_0^K\\y_0^K
	\end{pmatrix}+\int_{0}^{t}Se^{D(t-s)}S^{-1}\begin{pmatrix}
	K^ce^{as}\\ 0
	\end{pmatrix}\ \mathrm{d}s.
	\]
	An explicit computation shows the claim.
\end{proof}

In the following we will also need some bounds on the variation of $X_t^K$ and $Y_t^K$. In order to derive them, we need some more preparation. In particular, we need to compute the quadratic variation.
The purpose here is twofold: We need these variation terms once for finding an upper bound of the variance of $X^K$ and $Y^K$. Secondly, we will later, in the proof of our convergence result, make use of Doob's inequality and hence need to calculate the expected value of some quadratic variation.

\begin{lemma}\label{quadraticvariation}
	The quadratic variation of the martingales $M^K$, $N^K$ and $M^K+N^K$ as well as the quadratic covariation $[M^K,N^K]$ of $M^K$ and $N^K$ are given by \begin{align*}
	[M^K]_t&=\int_0^t(b_1+d_1+\sigma_1)X_s^K+\sigma_2Y_s^K+K^ce^{as}\ \mathrm{d}s,\\
	[N^K]_t&=\int_0^t (b_2+d_2+\sigma_2)Y_s^K+\sigma_1X_s^K\ \mathrm{d}s,\\
	[M^K+N^K]_t&=3\int_{0}^{t}(b_1+d_1)X_s^K+(b_2+d_2)Y_s^K+K^ce^{as}\ \mathrm{d}s,\\
	[M^K,N^K]_t&=\int_{0}^{t}(r_1+2d_1)X_s^K+(r_2+2d_2)Y_s^K+K^ce^{as}\ \mathrm{d}s.
	\end{align*}
\end{lemma}

\begin{proof}
	We only carry out the calculations for $M^K$. In an analogous fashion we can calculate the quadratic variation of $N^K$ and of $M^K+N^K$. For the covariation $[M^K,N^K]$ we can use the polarization identity \[
	[M^K,N^K]_t=\reci{2}\lr{[M^K+N^K]_t-[M^K]_t-[N^K]_t}.
	\]
	Applying Itô's formula to $(X_t^K)^2$ and Dynkin's formula with $f(x,y)=x^2$ shows that \begin{align*}
	\hat{M}^K_t&=(X_t^K)^2-(X_0^K)^2-\int_{0}^{t}(2X_s^K+1)(b_1X_s^K+\sigma_2Y_s^K+K^ce^{as})+(\sigma_1+d_1)(-2X_s^K+1)X_s^K\ \mathrm{d}s,\\
	\tilde{M}^K_t&=(X_t^K)^2-(X_0^K)^2-\int_{0}^{t}2X_s^K((b_1-d_1-\sigma_1)X_s^K+\sigma_2Y_s^K+K^ce^{as})\ \mathrm{d}s-[M^K]_t,
	\end{align*}
	for some martingales $\hat{M}^K$ and $\tilde{M}^K$ starting at $0$.
	By the uniqueness of the Doob-Meyer decomposition of $(X_t^K)^2-(X_0^K)^2$ we see that $\hat{M}^K_t=\tilde{M}^K_t$ and hence \[
	[M^K]_t=\int_{0}^{t}(b_1+d_1+\sigma_1)X_s^K+\sigma_2Y_s^K+K^ce^{as}\ \mathrm{d}s.
	\]
\end{proof}
Now, we can make use of the quadratic variations to derive our bounds for the variance.

\begin{lemma}\label{varianceestimate}
	There exists a constant $C_*\geq 0$ independent of $K$ such that \[
	\V(X_t^K),\V(Y_t^K)\leq C_*(1+t^2)\lr{(e^{2\la t}+e^{\la t})(x_0^K+y_0^K+K^c)+K^ce^{at}}\quad\text{ for all }t\geq 0
	\]
\end{lemma}
\begin{proof}
	We denote $u_t^K\coloneqq \V(X_t^K)$, $v_t^K\coloneqq \V(Y_t^K)$ and $w_t^K\coloneqq \operatorname{cov}(X_t^K,Y_t^K)$.
	Applying Itô's formula and Lemma \ref{quadraticvariation} to $(X_t^K-x_t^K)^2$ and $(Y_t^K-y_t^K)^2$ as well as using Integration by Parts for $(X_t^K-x_t^K)(Y_t^K-y_t^K)$ gives the differential equation \begin{align}
		\begin{pmatrix}
		\dot u_t^K\\ \dot v_t^K\\ \dot{w}_t^K
		\end{pmatrix}=\begin{pmatrix}
			2r_1 & 0 & 2\sigma_2\\
			0 & 2r_2 & 2\sigma_1\\
			\sigma_1 & \sigma_2 & r_1+r_2\\
		\end{pmatrix}\begin{pmatrix}
			u_t^K\\ v_t^K\\w_t^K
		\end{pmatrix}+\begin{pmatrix}
			(b_1+d_1+\sigma_1)x_t^K+\sigma_2y_t^K+K^ce^{at}\\
			(b_2+d_2+\sigma_2)y_t^K+\sigma_1 x_t^K\\
			(r_1+2d_1)x_t^K+(r_2+2d_2)y_t^K+K^ce^{at}
		\end{pmatrix}\label{dominatingode}
\end{align}
with initial condition $(u_0^K,v_0^K,w_0^K)=(0,0,0)$. Now we can proceed as in Lemma \ref{explicitexpectation}. The eigenvalues of the coefficient matrix are $2\tilde{\lambda}<r_1+r_2<2\lambda$, so it is diagonalisable with matrices $S,S^{-1}$ such that \[
\begin{pmatrix}
2r_1 &0 & 2\sigma_2\\
0  & 2r_2& 2\sigma_1\\
\sigma_1 & \sigma_2 & r_1+r_2
\end{pmatrix}=SDS^{-1},
\]
where $D=\operatorname{diag}(2\tilde{\la}, r_1+r_2, 2\la)$. The solution to the differential equation (\ref{dominatingode}) is given by \begin{align*}
\begin{pmatrix}
u_t^K\\ v_t^K\\ w_t^K
\end{pmatrix}&=\int_{0}^{t}Se^{D(t-s)}S^{-1}\begin{pmatrix}
(b_1+d_1+\sigma_1)x_s^K+\sigma_2 y_s^K+K^ce^{as}\\
(b_2+d_2+\sigma_2)y_s^K+\sigma_1x_s^K\\
(r_1+2d_1)x_s^K+(r_2+2d_2)y_s^K+K^ce^{as}
\end{pmatrix}\ \mathrm{d}s.\\
&\leq C_* \int_{0}^{t} e^{2\la(t-s)}\begin{pmatrix}
x_s^K+y_s^K+K^ce^{as}\\
y_s^K+x_s^K\\
x_s^K+y_s^K+K^ce^{as}
\end{pmatrix}\ \mathrm{d}s,
\end{align*}
where $C_*\geq0$ is a suitable constant independent of $K$ and the inequality holds for each component.  By Lemma \ref{explicitexpectation} we can further estimate the expected values $x_s^K$, $y_s^K$ with the constant $C_*$ changing from line to line by \begin{align*}
u_t^K,v_t^K,w_t^K&\leq C_*e^{2\la t}\int_{0}^{t}e^{-\la s}(x_0^K+y_0^K+(1+s)K^c)+K^ce^{(a-2\la)s}\ \mathrm{d}s\\
&\leq C_*e^{2\la t}(1+t+t^2)\left[e^{-\la t}(x_0^K+y_0^K+K^c)+K^ce^{(a-2\la)t}+(x_0^K+y_0^K+K^c)\right]\\
&\leq C_*(1+t^2)\left[(e^{2\la t}+e^{\la t})(x_0^K+y_0^K+K^c)+K^ce^{at}\right],	\end{align*}
which we have claimed.
\end{proof}

\subsection{A Special Case of Theorem \ref{Theorem: Main Convergence}}

With these preparations we are well situated to show a first convergence result for general bi-type branching processes, which is easily seen to be a special case of Theorem \ref{Theorem: Main Convergence} (i).

\begin{theorem}\label{firstconvergencetheorem}
	Let $Z_t^K=(X_t^K,Y_t^K)$ be a bi-type branching process whose distribution is given by $BBPI_K(b_1,b_2,d_1,d_2,\sigma_1,\sigma_2,a,c,\beta,\gamma)$ with $\beta\vee\gamma>0$. Assume that $c\leq\beta\vee\gamma$ and that $T>0$ is such that \begin{align}
	\inf_{t\in[0,T]}((\beta\vee\gamma)+\la t)\vee(c+at)>0,\label{positivitycondition}
	\end{align}
	where we recall $\la$ from \emph{(\ref{Eigenvalue})}.
	Then the following convergence in probability holds in $L^\infty([0,T])$: \[
	\lrset{s\mapsto\frac{\log(1+X_{s\log K}^K+Y_{s\log K}^K)}{\log K}}\xrightarrow{K\to\infty}\lrset{s\mapsto((\beta\vee\gamma)+\la s)\vee(c+as)}.
	\]
\end{theorem}
\begin{remark}\label{firstconvergenceremark}
	Note that due to the strictly positive switching rates, the same convergence also holds for the processes \[
	\lrset{s\mapsto\frac{\log (1+X_{s\log K}^K)}{\log K}}\quad\text{and}\quad\lrset{s\mapsto\frac{\log (1+Y_{s\log K}^K)}{\log K}}
	\]
	on the interval $(0,T]$ if $\beta\neq\gamma$ and on $[0,T]$ if $\beta=\gamma$. Intuitively, if these processes were of different sizes, the switching would immediately fill the difference. For a formal proof, a straightforward adaptation of the proof of Theorem \ref{firstconvergencetheorem} is possible.
\end{remark}
\begin{proof}
	For the proof we make extensive use of ideas from \cite[Theorem B.1]{champagnat2019stochastic}. We define $\bar{\beta}_t\coloneqq ((\beta\vee\gamma)+\la t)\vee(c+at)$.
	\subsubsection*{Step 1: Semimartingale Arguments}
	 For $\eta>0$ to be determined, we define the set \[
		\Omega_1^K\coloneqq\lrset{\sup_{t\in[0,T\log K]}\abs{e^{-\la t}(X_t^K+Y_t^K-(x_t^K+y_t^K))}\leq K^\eta}.
		\]
		Our first goal is to identify a set of parameters $\eta$ such that $\P(\Omega_1^K)\to 1$ as $K\to\infty$. For this, in \cite[Lemma B.3]{champagnat2019stochastic} it is shown, that the process of which the absolute value is taken in $\Omega_1^K$ is a martingale. Here however, due to the switching between $X_t^K$ and $Y_t^K$ we do not have a martingale. Instead we use Integration by Parts as well as Lemmata \ref{Semimartingaledecomposition} and \ref{expectedvalue} to get the decomposition \begin{align*}
		&e^{-\la t}(X_t^K+Y_t^K-(x_t^K+y_t^K))\\=&\int_{0}^{t}-\la e^{-\la s}(X_s^K+Y_s^K-(x_s^K+y_s^K))\ \mathrm{d}s+\int_{0}^{t}e^{-\la s}\ \mathrm{d}(X_s^K+Y_s^K)-\int_{0}^{t}e^{-\la s}\ \mathrm{d}(x_s^K+y_s^K)\\
		=&\int_{0}^{t}e^{-\la s}\ \mathrm{d}(M_s^K+N_s^K) +\int_{0}^{t} e^{-\la s}((r_1+\sigma_1-\la)(X_s^K-x_s^K)+(r_2+\sigma_2-\la)(Y_s^K-y_s^K))\ \mathrm{d}s.
		\end{align*}
		 We denote the martingale $\textstyle\int_{0}^{t}e^{-\la s}\ \mathrm{d}(M_s^K+N_s^K)$ by $\tilde{M}_t^K$. Hence, by Doob's inequality we have \begin{align*}
		&\P\lr{\sup_{t\leq T\log K}\abs{e^{-\la t}(X_t^K+Y_t^K-(x_t^K+y_t^K))}\geq K^\eta}\\
		\leq&\P\lr{\sup_{t\leq T\log K} \abs{\tilde{M}_t^K}+\int_{0}^{t}\abs{r_1+\sigma_1-\la}e^{-\la s}\abs{X_s^K-x_s^K}+\abs{r_2+\sigma_2-\la}e^{-\la s}\abs{Y_s^K-y_s^K}\ \mathrm{d}s\geq K^\eta}\\
		\leq&C_*K^{-2\eta}\E\left[\lr{\abs{\tilde{M}_{T\log K}^K}+\int_{0}^{T\log K}\hspace{-0.5cm}e^{-\la s}\abs{X_s^K-x_s^K}+e^{-\la s}\abs{Y_s^K-y_s^K}\ \mathrm{d}s}^2\right],
		\end{align*}
		where $C_*>0$ is a suitable constant, which may change in the following from line to line. Now, successively using $(a+b+c)^2\leq 3(a^2+b^2+c^2)$, Hölder's Inequality and Fubini's Theorem, we see that 
		\begin{align}
		&\P\lr{\sup_{t\leq T\log K}\abs{e^{-\la t}(X_t^K+Y_t^K-(x_t^K+y_t^K))}\geq K^\eta}\nonumber\\
		\leq
		&C_*K^{-2\eta}\E\left[\lr{\abs{\tilde{M}_{T\log K}^K}+\int_{0}^{T\log K}\hspace{-0.5cm}e^{-\la s}\abs{X_s^K-x_s^K}+e^{-\la s}\abs{Y_s^K-y_s^K}\ \mathrm{d}s}^2\right]\nonumber\\
		\leq& C_*K^{-2\eta}\lr{\E[\tilde{M}_{T\log K}^2]+\E\left[\lr{\int_{0}^{T\log K}\hspace{-0.5cm}e^{-\la s}\abs{X_s^K-x_s^K}\ \mathrm{d}s}^2+\lr{\int_{0}^{T\log K}\hspace{-0.5cm}e^{-\la s}\abs{Y_s^K-y_s^K}\ \mathrm{d}s}^2\right]}\nonumber\\
		\leq& C_*K^{-2\eta}\E[(\tilde{M}_{T\log K}^K)^2]\nonumber\\
		&\quad +C_*K^{-2\eta}T\log K\E\left[\int_{0}^{T\log K}\hspace{-0.5cm}e^{-2\la s}\abs{X_s^K-x_s^K}^2\ \mathrm{d}s+\int_{0}^{T\log K}\hspace{-0.5cm}e^{-2\la s}\abs{Y_s^K-y_s^K}^2\ \mathrm{d}s\right]\nonumber\\
		\leq & C_*K^{-2\eta}\lr{\E[(\tilde{M}_{T\log K}^K)^2]+T\log K\int_{0}^{T\log K} e^{-2\la s}(\V(X_s^K)+\V(Y_s^K))\ \mathrm{d}s}.\label{estimateprobability}
		\end{align}
		
		We will now estimate the expectation and the integral separately. Firstly, using the definition of $\tilde{M}^K$, we easily see using  Itô's Isometry and Lemma \ref{quadraticvariation} that \begin{align*}
		\E[(\tilde{M}_{T\log K}^K)^2]&=\E\left[\lr{\int_{0}^{T\log K}e^{-\la s}\ \mathrm{d}(M_s^K+N_s^K)}^2\right]=\E\left[\int_{0}^{T\log K}e^{-2\la s}\ \mathrm{d}[M^K+N^K]_s\right]\\
		&=3\int_{0}^{T\log K}e^{-2\la s}((b_1+d_1)x_s^K+(b_2+d_2)y_s^K+K^ce^{as})\ds.
		\end{align*}
		Furthermore, using the calculation of the expected values $x_s^K$ and $y_s^K$ up to constants from Lemma \ref{explicitexpectation}, we can estimate \begin{align*}
		\E[(\tilde{M}_{T\log K}^K)^2]
		\leq C_*(1+T^2)\log^2(K)\cdot (K^\beta+K^\gamma+K^c+K^{\beta-\la T}+K^{\gamma-\la T}+K^{c-\la T}+K^{c+(a-2\la)T})
		\end{align*}
		again for some suitable constant $C_*>0$, which may change from line to line and can without loss of generality be chosen sufficiently large such that (\ref{estimateprobability}) holds as well. Since we assume $c\leq\beta\vee\gamma$, we obtain the estimate \begin{align}
		\E[(\tilde{M}_{T\log K}^K)^2]\leq C_*(1+T^2)\log^2(K)\cdot K^{(\beta\vee\gamma)\vee((\beta\vee\gamma)-\la T)\vee (c+(a-2\la)T)}.\label{martingaleestimate}
		\end{align}
		We now turn to the integral. Using Lemma \ref{varianceestimate} and \ref{explicitexpectation} again, as well as $c\leq\beta\vee\gamma$, we see that 
		\begin{align}
		&\int_{0}^{T\log K} e^{-2\la s}(\V(X_s^K)+\V(Y_s^K))\ \mathrm{d}s\nonumber\\
		\leq & C_*K^{(\beta\vee\gamma)\vee((\beta\vee\gamma)-\la T)\vee (c+(a-2\la)T)}(1+T^3)\log^3(K)\label{integralvarianceestimate}
		\end{align}
	    for some constant $C_*\geq0$ sufficiently large.
		Hence, plugging the estimates (\ref{martingaleestimate}) and (\ref{integralvarianceestimate}) into (\ref{estimateprobability}), we see that \begin{align*}
		1-\P(\Omega_1^K)&\leq\P\lr{\sup_{t\leq T\log K}\abs{e^{-\la t}(X_t^K+Y_t^K-(x_t^K+y_t^K))}\geq K^\eta}\\
		&\leq C_*K^{-2\eta}K^{(\beta\vee\gamma)\vee((\beta\vee\gamma)-\la T)\vee (c+(a-2\la)T)}(1+T^3)\log^3(K).
		\end{align*}
		From now on, we will consider $\eta$ such that \begin{align}
		\frac{(\beta\vee\gamma)\vee((\beta\vee\gamma)-\la T)\vee (c+(a-2\la)T)}{2}<\eta<\beta\vee\gamma.\label{assumptioneta}
		\end{align}
		This condition ensures as shown above that $\textstyle\lim_{K\to\infty}\P(\Omega_1^K)=1$. On the set $\Omega_1^K$, we can obtain \begin{align}
		&\sup_{t\leq T}\abs{\frac{\log(1+X_{t\log K}^K+Y_{t\log K}^K)}{\log K}-\bar{\beta}_t}\nonumber\\
		=&\sup_{t\leq T}\reci{\log K}\abs{\log\lr{\frac{1+X_{t\log K}^K+Y_{t\log K}^K}{1+x_{t\log K}^K+y_{t\log K}^K}}+\log\lr{\frac{1+x_{t\log K}^K+y_{t\log K}^K}{ K^{\bar{\beta}_t}}}}\nonumber\\
		\leq &\sup_{t\leq T}\reci{\log K}\cdot \frac{K^{-\la t}\abs{X_{t\log K}^K+Y_{t\log K}^K-(x_{t\log K}^K+y_{t\log K}^K)}}{K^{-\la t}(x_{t\log K}^K+y_{t\log K}^K)\wedge K^{-\la t}(X_{t\log K}^K+Y_{t\log K}^K)}+\frac{C_*}{\log K}\nonumber\\
		\leq & \reci{\log K}\sup_{t\leq T}\frac{K^{\eta+\la t}}{x_{t\log K}^K+y_{t\log K}^K-K^{\eta+\la t}}+\frac{C_*}{\log K}\label{denominator}\\
		\leq&\frac{C_*}{\log K}\lr{K^{\eta-(\beta\vee\gamma)}+1}\xrightarrow{K\to\infty}0,\label{betaconvergence}
		\end{align}
		where again $C_*$ is a sufficiently large constant, which may change from line to line. Note that the denominator in (\ref{denominator}) is well defined for $K$ large enough since $\eta<\beta\vee\gamma$. Also, the first inequality holds due to our choice of $\bar{\beta}_t$, which gives in combination with Lemma \ref{explicitexpectation} that \[
		K^{-\bar{\beta}_t}(1+x_{t\log K}^K+y_{t\log K}^K)\leq C_*(K^{-\bar{\beta}_t}+1)\leq 2C_*
		\]
		for $K$ large enough since by assumption $\bar{\beta}_t>0$ for all $t\in[0,T]$. Thus, we are in the same situation as in Step 1 of the proof of Theorem B.1 in \cite{champagnat2019stochastic} with $\beta\vee\gamma$ instead of $\beta$ and $\la$ instead of $r$. For completeness we distinguish the same cases:
		\begin{description}
			\item[Case 1(a): $\la\geq 0$ and $a\leq 2\la$] In this case the condition (\ref{assumptioneta}) reduces to $\tfrac{\beta\vee\gamma}{2}<\eta<\beta\vee\gamma$, so choosing $\eta=\tfrac{3(\beta\vee\gamma)}{4}$ shows the claim from (\ref{assumptioneta}) and (\ref{betaconvergence}).
			\item[Case 1(b): $\la<0$ and $a\leq \la$] Here the assumption (\ref{positivitycondition}) becomes $\textstyle\inf_{t\in[0,T]}(\beta\vee\gamma)+\la t>0$, which gives $T<\tfrac{\beta\vee\gamma}{\abs{\la}}$. Now condition (\ref{assumptioneta}) becomes \begin{align*}
			\frac{((\beta\vee\gamma)-\la T)\vee(c+(a-2\la)T)}{2}=\frac{(\beta\vee\gamma)-\la T}{2}<\eta<\beta\vee\gamma.
			\end{align*}
			By the condition on $T$, such $\eta$ exists and we can again conclude.
			\item[Case 1(c): $\la\geq 0$ and $a>2\la$] In this case the restriction (\ref{assumptioneta}) can be satisfied as long as $T$ is such that $c+(a-2\la)T<2(\beta\vee\gamma)$, that is $T<T^*\coloneqq\tfrac{2(\beta\vee\gamma)}{a-2\la}$. Hence we obtain the convergence on all intervals $[0,T]$ such that $T<T^*$. In Step 3 we will consider $T>T^*$ which still may satisfy (\ref{positivitycondition}).
			\item[Case 1(d): $\la< 0$, $a>\la$ and $c+\tfrac{a(\beta\vee\gamma)}{\abs{\la}}\leq 0$] Here we easily see that $(\beta\vee\gamma)+\la t\geq c+at$ for all $t\leq\tfrac{\beta\vee\gamma}{\abs{\la}}$. Thus assumption (\ref{positivitycondition}) is satisfied if and only if $T<\tfrac{\beta\vee\gamma}{\abs{\la}}$. For these times $T$ the condition (\ref{assumptioneta}) can be satisfied for suitable $\eta$ since \[
			((\beta\vee\gamma)-\la T)<(\beta\vee\gamma-\la\frac{\beta\vee\gamma}{\abs{\la}})=2(\beta\vee\gamma)
			\]
			and \[
			(c+(a-2\la)T)<c+(a-2\la)\frac{\beta\vee\gamma}{\abs{\la}}=2(\beta\vee\gamma)+c+a\frac{\beta\vee\gamma}{\abs{\la}}\leq 2(\beta\vee\gamma).
			\]
			\item[Case 1(e): $\la< 0$, $a>\la$ and $c+\tfrac{a(\beta\vee\gamma)}{\abs{\la}}\geq 0$] For these parameters, the condition (\ref{assumptioneta}) can be satisfied as long as $T<T^*=\tfrac{\beta\vee\gamma}{\abs{\la}}\wedge \tfrac{2(\beta\vee\gamma)-c}{a-2\la}$. Note that \begin{align*}
			\frac{\beta\vee\gamma}{\abs{\la}}>\frac{2(\beta\vee\gamma)-c}{a-2\la}&\quad\Llra\quad (a-2\la)(\beta\vee\gamma)>c\la-2(\beta\vee\gamma)\la\\
			&\quad\Llra\quad a(\beta\vee\gamma)>c\la,
			\end{align*}
			which is true since $a>\la$ and $(\beta\vee\gamma)>c$. Hence $T^*=\tfrac{2(\beta\vee\gamma)-c}{a-2\la}$ as in Case 1(c). As before we get the convergence for all $T<T^*$ and we will show in Step 3 how to obtain convergence for $T\geq T^*$ which satisfy (\ref{positivitycondition}).
		\end{description}
	\subsubsection*{Step 2: Strong Immigration}
	 Here, we will consider only the case $\beta\vee\gamma=c$ and $a>\la$. Similarly to Step 1, for $\eta>0$ to be determined later, we consider the set \[
		\Omega_2^K\coloneqq\lrset{\sup_{t\in[0,T\log K]}\abs{e^{-a t}(X_t^K+Y_t^K-(x_t^K+y_t^K))}\leq K^\eta}.
		\]
		We experience the same difficulties as in Step 1: We are not able to use a supermartingale inequality, since the switching between $X^K$ and $Y^K$ complicates our process. However, proceeding in the same manner as in Step 1, we get the inequality \begin{align*}
		&\P\lr{\sup_{t\leq T\log K}\abs{e^{-a t}(X_t^K+Y_t^K-(x_t^K+y_t^K))}\geq K^\eta}\\
		\leq& C_*K^{-2\eta}\lr{\E[(\hat{M}^K_{T\log K})^2]+T\log K\int_{0}^{T\log K} e^{-2 a s}(\V(X_s^K)+\V(Y_s^K))\ \mathrm{d}s},
		\end{align*}
		where $\textstyle\hat{M}_t^K\coloneqq\int_{0}^{t}e^{-as}\ \mathrm{d}(M_s^K+N_s^K)$ is a martingale. Applying our estimates and the Itô Isometry from above gives with another calculation similar to the corresponding part in step 1 that \begin{align*}
		1-\P(\Omega_2^K)&\leq\P\lr{\sup_{t\leq T\log K}\abs{e^{-a t}(X_t^K+Y_t^K-(x_t^K+y_t^K))}\geq K^\eta}\\
		&\leq C_*K^{-2\eta}K^{\beta\vee\gamma\vee((\beta\vee\gamma)+(2(\la-a)\vee (\la-2a))T)\vee(c-aT) }(1+T^3)\log^3(K)\\
		&=C_*K^{-2\eta}K^{\beta\vee\gamma\vee((\beta\vee\gamma)-aT)}(1+T^3)\log^3(K),
		\end{align*}
		where we used $a>\la$ and $c=\beta\vee\gamma$ in the last equality. Indeed the exponent $2(\la-a)T$ is always negative and can therefore be omitted. On the other hand $\la-2a<-a$ and thus $(\la-2a)T$ may be replaced by $-aT$, which is accounted for in the last equality. Therefore, we now consider $\eta$ such that \begin{align}
		\frac{(\beta\vee\gamma)\vee((\beta\vee\gamma)-aT)}{2}<\eta<\beta\vee\gamma,\label{assumptioneta2}
		\end{align}
		which ensures that $\P(\Omega_2^K)\to 1$ as $K\to\infty$.
		Again a calculation similar to step 1 shows that for this choice of $\eta$ we have \begin{align}
		&\sup_{t\leq T}\abs{\frac{\log(1+X_{t\log K}^K+Y_{t\log K}^K)}{\log K}-\bar{\beta}_t}\nonumber\\
		\leq &\sup_{t\leq T}\reci{\log K}\cdot \frac{K^{-a t}\abs{X_{t\log K}^K+Y_{t\log K}^K-(x_{t\log K}^K+y_{t\log K}^K)}}{K^{-a t}(x_{t\log K}^K+y_{t\log K}^K)\wedge K^{-\la t}(X_{t\log K}^K+Y_{t\log K}^K)}+\frac{C_*}{\log K}\nonumber\\
		\leq & \reci{\log K}\sup_{t\leq T}\frac{K^{\eta+\la t}}{x_{t\log K}^K+y_{t\log K}^K-K^{\eta+a t}}+\frac{C_*}{\log K}\nonumber\\
		\leq&\frac{C_*}{\log K}\lr{K^{\eta-(\beta\vee\gamma)}+1}\xrightarrow{K\to\infty}0.\nonumber
		\end{align}
		This computation allows us to show our convergence result for two more possible cases.
		\begin{description}
			\item[Case 2(a): $c=\beta\vee\gamma$, $a>\la$ and $a\geq 0$.] As in Case 1(a) we may choose $\eta=\tfrac{3(\beta\vee\gamma)}{4}$ and have shown convergence for this case.
			\item[Case 2(b): $c=\beta\vee\gamma$, $a>\la$ and $a<0$.] Here, condition (\ref{positivitycondition}) on the final time $T$ is satisfied if and only if $T<\tfrac{\beta\vee\gamma}{\abs{a}}$. Hence \[
			\frac{(\beta\vee\gamma)-aT}{2}<\frac{(\beta\vee\gamma)-a\frac{\beta\vee\gamma}{\abs{a}}}{2}=\beta\vee\gamma
			\]
			and thus we can find $\eta$ such that (\ref{assumptioneta2}) is satisfied.
		\end{description}
		
	\subsubsection*{Step 3: Completion of Step 1}
		It remains to extend the following two cases to $T>T^*=\tfrac{2(\beta\vee\gamma)-c}{a-2\la}$:

	\begin{itemize}
		\item $\la\geq 0$, $a>2\la$ and $c<\beta\vee\gamma$,
		\item $\la<0$, $a>\la$, $c<\beta\vee\gamma$ and $c+\tfrac{a(\beta\vee\gamma)}{\abs{\la}}\geq 0$.
	\end{itemize}	
	This can be done exactly as in \cite{champagnat2019stochastic} in Step 3 of the proof of Theorem B.1. In order to do so, we note that at time $t^*\coloneqq \tfrac{(\beta\vee\gamma)-c}{a-\la}$ the lines $(\beta\vee\gamma)+\la t$ and $c+at$ intersect. Furthermore, we see that in both of the above cases $t^*<T^*$ since in the first case we may assume without loss of generality that $a>\la$ (otherwise $t^*$ is negative) and therefore \begin{align*}
	\frac{2(\beta\vee\gamma)-c}{a-2\la}>\frac{(\beta\vee\gamma)-c}{a-\la}&\quad\Llra\quad 2(\beta\vee\gamma)(a-\la)-ac+\la c>(\beta\vee\gamma)(a-2\la)-ac+2\la c\\
	&\quad\Llra\quad(\beta\vee\gamma)a>\la c, \end{align*}
	which is true. In the second case we can perform a similar computation. Therefore we may apply Case 1(c) or Case 1(e) to our process in each case up to a time $T_1\in(t^*,T^*)$. Note that at this time the limiting function satisfies $\bar{\beta}_{T_1}=c+aT_1$. Hence for all $\epsi$ on a set $\Omega_3^K$ with $\P(\Omega_3^K)\to 1$ as $K\to\infty$ we have \[
	K^{c+aT_1-\veps}\leq X_{T_1\log K}^K+Y_{T_1\log K}^K\leq K^{c+aT_1+\veps}.
	\] 
	We now couple the process $Z_{T_1\log K+t}^K=(X_{T_1\log K+t}^K,Y_{T_1\log K+t}^K)$ in the following manner: Let $\hat{Z}_t^K=(\hat{X}_t^K,\hat{Y}_t^K)$ be a $BBPI_K(b_1,b_2,d_1,d_2,\sigma_1,\sigma_2,a,c+aT_1-\veps,c+aT_1-\veps,c+aT_1-\veps)$ and let $\bar{Z}_t^K=(\bar{X}_t^K,\bar{Y}_t^K)$ be a $BBPI_K(b_1,b_2,d_1,d_2,\sigma_1,\sigma_2,a,c+aT_1+\veps,c+aT_1+\veps,c+aT_1+\veps)$ such that \[
	\hat{X}_t^K+\hat{Y}_t^K\leq X_{T_1\log K+t}^K+Y_{T_1\log K+t}^K\leq\bar{X}_t^K+\bar{Y}_t^K.
	\]
	Indeed, the starting conditions of the bounding processes are justified by Remark \ref{firstconvergenceremark}. Then, we can apply the convergence from Step 2 to $\hat{Z}_t^K$ and $\bar{Z}_t^K$ to show that \[
	\frac{\log(1+\hat{X}_{t\log K}^K+\hat{Y}_{t\log K}^K)}{\log K}\xrightarrow{K\to\infty}c+aT_1-\veps+((\la\vee a)t)=c-\veps+a(T_1+t)
	\]
	and 
	\[
	\frac{\log(1+\bar{X}_{t\log K}^K+\bar{Y}_{t\log K}^K)}{\log K}\xrightarrow{K\to\infty}c+aT_1+\veps+((\la\vee a)t)=c+\veps+a(T_1+t)
	\]
	where $t\in[0,T-T_1]$. Note that for the second case, $a<0$ and therefore the condition (\ref{positivitycondition}) is satisfied only for $T<\tfrac{c}{\abs{a}}$. In particular, for $\veps$ small enough, we even have $T<\tfrac{c-\veps}{\abs{a}}$ and therefore $T-T_1<\tfrac{c+aT_1-\veps}{\abs{a}}$, so we can indeed apply case 2(b). Using the Markov property at time $T_1$ and letting $\veps\to 0$ finishes the proof.
	\end{proof}

As in \cite{champagnat2019stochastic}, we want to extend Theorem \ref{firstconvergencetheorem} to further cases without needing the assumption $c\leq\beta\vee\gamma$ or the positivity condition (\ref{positivitycondition}). To obtain the convergence result for $c>\beta\vee\gamma$, we will use the next lemma. This, combined with using the Markov property and our previous Theorem \ref{firstconvergencetheorem}, already extends the convergence to all processes such that $c\in\R$, $\beta\vee\gamma>0$ and the terminal time $T$ satisfies condition (\ref{positivitycondition}) \[
\inf_{t\in[0,T]}((\beta\vee\gamma)+\la t)\vee (c+at)>0.
\]

\begin{lemma}\label{Lemma: Strong mutation}
	Let $0\leq\beta\vee\gamma<c$. Then for all $0<\veps<\tfrac{c}{4(\abs{\la}\vee\abs{a})}$ and all $\bar{a}>\abs{\la}\vee\abs{a}$, the convergence \[
	\lim\limits_{K\to\infty}\P(X_{\veps\log K}^K+Y_{\veps\log K}^K\in[K^{c-\bar{a}\veps},K^{c+\bar{a}\veps}])=1
	\]
	holds true.
\end{lemma}
\begin{proof}
	Let $\epsi$. Then we see from Lemma \ref{explicitexpectation} that for $K$ sufficiently large \[
	C_*^{-1}K^{c-(\abs{\la}\vee\abs{a})\veps}\leq\E(X_{\veps\log K}^K+Y_{\veps\log K}^K)\leq C_*\log (K)K^{c+(\abs{\la}\vee\abs{a})\veps}
	\]
	and from Lemma \ref{varianceestimate} \[
	\V(X_{\veps\log K}^K+Y_{\veps\log K}^K)\leq 2(\V(X_{\veps\log K}^K)+\V(Y_{\veps\log K}^K))\leq C_*\log^2(K)K^{c+(2\abs{\la}\vee\abs{a})\veps}
	\]
	for a suitable constant $C_*>0$. Hence, using Markov's inequality for $K$ large enough, we see that \begin{align*}
	\P(X_{\veps\log K}^K+Y_{\veps\log K}^K\leq K^{c+\bar{a}\veps})&=1-\P(X_{\veps\log K}^K+Y_{\veps\log K}^K> K^{c+\bar{a}\veps})\\
	&\geq 1-\frac{\E(X_{\veps\log K}^K+Y_{\veps\log K}^K)}{K^{c+\bar{a}\veps}}\xrightarrow{K\to\infty}1.
	\end{align*}
	For the lower bound we can use Chebyshev's inequality to obtain \begin{align*}
	&\P(X_{\veps\log K}^K+Y_{\veps\log K}^K\leq K^{c-\bar{a}\veps})\\
	\leq&\P(X_{\veps\log K}^K+Y_{\veps\log K}^K-(x_{\veps\log K}^K+y_{\veps\log K}^K)\leq K^{c-\bar{a}\veps}-C_*^{-1}K^{c-(\abs{\la}\vee\abs{a})\veps})\\
	\leq&\P(\abs{X_{\veps\log K}^K+Y_{\veps\log K}^K-(x_{\veps\log K}^K+y_{\veps\log K}^K)}\geq C_*'K^{c-(\abs{\la}\vee \abs{a})\veps})\\
	\leq& C_*' \frac{\V(X_{\veps\log K}^K+Y_{\veps\log K}^K)}{K^{2c-2(\abs{\la}\vee\abs{a})\veps}}\\
	\leq& C_*'\log^2(K)K^{-c+4(\abs{\la}\vee\abs{a})\veps}\xrightarrow{K\to\infty}0
	\end{align*}
	for some suitable changing constant $C_*'>0$.
\end{proof}

\subsection{Proof of Theorem \ref{Theorem: Main Convergence}}
Henceforth, thanks to the previous Lemma, we are able to only consider the case $c\leq\beta\vee\gamma$ for the remainder of this section. In order to extend the convergence result Theorem \ref{firstconvergencetheorem} to times $T$ that do not satisfy the condition (\ref{positivitycondition}), we need another series of Lemmata.

\begin{lemma}\label{lemmanomigration}
	Let $\beta\vee\gamma=0$ (that is, initially there are no individuals) and $c<0$. Then, for all $T>0$ with $c+aT<0$ we have \[
	\lim\limits_{K\to\infty}\P(X_{t}^K+Y_t^K=0 \text{ for all }t\leq T\log K)=1.
	\]
\end{lemma}
\begin{proof}
	This proof is identical to the one of \cite[Lemma B.7]{champagnat2019stochastic}.
\end{proof}

\begin{lemma}\label{lemmareemergence}
	Let $\beta\vee\gamma=0$ and $c=-\veps$ for some $\epsi$ and let $a>0$. Then for all $\eta>(1\vee \tfrac{2\la}{a})\veps$, the convergence \[
	\lim\limits_{K\to\infty}\P(K^{\frac{\veps}{2}}-1\leq X_{\frac{2\veps}{a}\log K}^K+Y_{\frac{2\veps}{a}\log K}^K\leq K^{\eta}-1)=1
	\]
	holds.
\end{lemma}
\begin{proof}
	We consider the one-dimensional branching process $\tilde{X}_t^K$ with birth rate $b_1$, death rate $d_1+\sigma_1$, immigration at rate $K^ce^{at}$ and starting condition $X_0^K=K^{\beta}-1=0$. Then, from the proof of Lemma B.8 from \cite{champagnat2019stochastic}, we have the convergence \[
	\lim\limits_{K\to\infty}\P(K^{\frac{\veps}{2}}-1\leq\tilde{X}_{\frac{2\veps}{a}\log K}^K)=1.
	\]
	Using a suitable coupling, we also have that $\tilde{X}_t^K\leq X_t^K+Y_t^K$ for all $t\geq 0$. Hence, it holds \[
	\lim\limits_{K\to\infty}\P(K^{\frac{\veps}{2}}-1\leq X_{\frac{2\veps}{a}\log K}^K+Y_{\frac{2\veps}{a}\log K}^K)=1.
	\]
	For the upper bound, we know from Lemma \ref{explicitexpectation} with $x_0=y_0=0$, that \[
	\E(X_{\frac{2\veps}{a}\log K}^K+Y_{\frac{2\veps}{a}\log K}^K)\leq \begin{cases}
	C_*K^{\veps},&\quad\text{if }a>\la\\
	C_*K^{\veps}\log K, &\quad\text{if } a=\la\\
	C_*K^{\frac{2\la}{a}\veps}, &\quad\text{if } a<\la.
	\end{cases}
	\]
	With our choice of $\eta>(1\vee\tfrac{2\la}{a})\veps$ and Markov's inequality, we have \[
	\lim\limits_{K\to\infty}\P(X_{\frac{2\veps}{a}\log K}+Y_{\frac{2\veps}{a}\log K}\geq K^\eta)\leq \lim\limits_{K\to\infty}K^{-\eta}\E(X_{\frac{2\veps}{a}\log K}+Y_{\frac{2\veps}{a}\log K})=0.
	\]
	Thus the lemma is proven.
\end{proof}

\begin{lemma}\label{lemmacontinuity}
	There exists a constant $\bar{c}=\bar{c}(b_1,b_2,d_1,d_2,\sigma_1,\sigma_2,a)$, such that for all $\epsi$ we have the convergence \[
	\lim\limits_{K\to\infty}\P(K^{(\beta\vee\gamma)-\bar{c}\veps}-1\leq X_t^K+Y_t^K\leq K^{(\beta\vee\gamma)+\bar{c}\veps} \text{ for all } t\in[0,\veps\log K])=1.
	\]
\end{lemma}
\begin{proof}
	This result is very similar to \cite[Lemma B.9]{champagnat2019stochastic} and can be proven analogously. 
\end{proof}

\begin{lemma}\label{lemmaextinction}
	Suppose $\la<0$, where $\la$ is taken from (\ref{Eigenvalue}). \begin{enumerate}[label=\emph{(\roman*)}]
		\item In addition, let $c<0$ and $c+a\tfrac{\beta\vee\gamma}{\abs{\la}}<0$. Then, for all sufficiently small $\eta>0$ it holds \[
		\lim\limits_{K\to\infty}\P\lr{\forall t\in\left[\lr{\frac{\beta\vee\gamma}{\abs{\la}}+\eta}\log K,\lr{\frac{\beta\vee\gamma}{\abs{\la}}+2\eta}\log K\right]\colon X_t^K+Y_t^K=0}=1.
		\]
		\item If in addition (independent of (i)) $a<0$ and $c+a\tfrac{\beta\vee\gamma}{\abs{\la}}>0$, then for all $\eta>0$ and all $T>\eta$, we have \[
		\lim\limits_{K\to\infty}\P\lr{\forall t\in\left[\lr{\frac{c}{\abs{a}}+\eta}\log K,\lr{\frac{c}{\abs{a}}+T}\log K\right]\colon X_t^K+Y_t^K=0}=1.
		\]
	\end{enumerate}
\end{lemma}
\begin{proof}
	This result is the bi-type analogue of \cite[Lemma B.10]{champagnat2019stochastic}. We start by proving part (i). Let $\eta$ be small enough such that $c+a(\tfrac{\beta\vee\gamma}{\abs{\la}}+2\eta)<0$. Define $T_1\coloneqq \tfrac{\beta\vee\gamma}{\abs{\la}}+\eta$ and $T_2\coloneqq T_1+\eta$. Then the probability of a migrant arriving during the interval $[0,T_2\log K]$ converges to $0$ as $K\to\infty$. This can be seen from the probability of immigration being bounded by $K^{c\vee (c+aT_2)}T_2\log K$, which converges to $0$ as $K\to\infty$ (cf.~Lemma \ref{lemmanomigration}). Hence, it suffices to show that, assuming no immigration occurs, the extinction time $T_{\text{ext}}\coloneqq\inf\{t\geq 0\mid X_{t}^K+Y_{t}^K=0\}$ is asymptotically almost surely less than $T_1\log K$, that is, our population itself is extinct before time $T_1\log K$, and since no migrant arrives during the time interval $[T_{\text{ext}},T_2\log K]$, which would potentially resurrect the population, the claim follows.  
	
	Denote the event that a migrant arrives in the population during the time interval $[0,T_1\log K]$ by $\Gamma$. Then, on the complement, the process $Z_t^K=(X_t^K,Y_t^K)$ behaves as a bi-type branching process $\tilde{Z}_t^K=(\tilde{X}_t^K,\tilde{Y}_t^K)$ with birth rates $b_1,b_2$, death rates $d_1,d_2$, switching rates $\sigma_1,\sigma_2$, starting condition $\tilde{Z}_0^K=Z_0^K$ and no immigration. Hence, we obtain using Markov's inequality \begin{align}
	\P(T_{\text{ext}}>T_1\log K,\Gamma^c)&=\P(\tilde{X}_{T_1\log K}^K+\tilde{Y}_{T_1\log K}^K\geq 1)\nonumber\\
	&\leq\E(\tilde{X}_{T_1\log K}^K+\tilde{Y}_{T_1\log K}^K)\nonumber\\
	&\leq C_*e^{\la T_1\log K}K^{\beta\vee\gamma}\nonumber\\
	&= C_*K^{(\beta\vee\gamma)+\la T_1},\label{extinctiontimebound}
	\end{align}
	where we can find the bound on the expected value for some constant $C_*>0$ from the homogeneous solution to (\ref{expectedode}). Therefore, by our choice of $T_1$, we obtain \[
	\P(T_{\text{ext}}>T_1\log K)\leq \P(\Gamma)+C_*K^{(\beta\vee\gamma)+\la T_1}\xrightarrow{K\to\infty} 0.
	\]
	For part (ii) we may assume without loss of generality, that $c\leq\beta\vee\gamma<\tfrac{\eta\abs{\la}}{4}$. To justify this, we apply Theorem \ref{firstconvergencetheorem} until time $T\geq \tfrac{\beta\vee\gamma}{\abs{\la}}$ which satisfies $r\coloneqq ((\beta\vee\gamma)+\la T)\vee(c+aT)<\tfrac{\eta\abs{\la}}{4}$. This allows us to consider the process $Z^K\sim BBPI_K(b_1,b_2,d_1,d_2,\sigma_1,\sigma_2,a,c+aT,r,r)$ instead by using the Markov property at time $T$, which satisfies our assumption. Indeed, by definition of $T$ and $r$, we have $c+aT=r$.\\
	Under the assumption $c\leq\beta\vee\gamma<\tfrac{\eta\abs{\la}}{4}$, the condition $c+a\tfrac{\beta\vee\gamma}{\abs{\la}}>0$ implies $a>\la$. Now, let $\eta>0$ and $T>\eta$ be arbitrary. As in case (i), we can prove that as $K\to\infty$ there is almost surely no immigrant arriving in the population on the time interval $[(\tfrac{c}{\abs{\la}}+\tfrac{\eta}{2})\log K, (\tfrac{c}{\abs{\la}}+T)\log K]$. Denote the event of a migrant arriving during this time by $\Gamma$. From now on, we only consider the event $\Gamma^c$. Then we only need to show that on $\Gamma^c$, the process becomes extinct before time $(\tfrac{c}{\abs{\la}}+\eta)\log K$. Note that the number of total families initially present in the population and those families started due to an immigration event up to time $(\tfrac{c}{\abs{a}}+\tfrac{\eta}{2})\log K$ is given by $K^{\beta\vee\gamma}$ (representing the families initially present) plus a Poisson random variable whose parameter is bounded from above by \[
	\int_{0}^{\tfrac{c}{\abs{a}}+\tfrac{\eta}{2}\log K}K^ce^{as}\ds\leq \frac{K^c}{\abs{a}},
	\]
	which represents the number of families coming from immigration.
	In particular, the total number of families is less than $K^{\eta\la/3}$ with probability converging to $1$. The size of such a family at time $t$ is bounded from above by the size of a bi-type branching process $\tilde{Z}=(\tilde{X},\tilde{Y})$ with birth rates $b_1,b_2$, death rates $d_1,d_2$, switching rates $\sigma_1,\sigma_2$, no immigration and starting population $\tilde{Z}_0=(1,1)$. As in part (i), we see that the probability of such a process surviving for a time longer than $\tfrac{\eta}{2}\log K$ is dominated by \[
	\P(\tilde{X}_{\frac{\eta}{2}\log K}+\tilde{Y}_{\frac{\eta}{2}\log K}\geq 1)\leq \E(\tilde{X}_{\frac{\eta}{2}\log K}+\tilde{Y}_{\frac{\eta}{2}\log K})\leq C_*e^{\la \frac{\eta}{2}\log K }=C_*K^\frac{\la\eta}{2},
	\]
	where we get the bound on the expectation from solving the homogeneous equation of (\ref{expectedode}) with initial condition $(\tilde{X}_0,\tilde{Y}_0)=(1,1)$. Therefore, the probability of having one family alive at time $(\tfrac{c}{\abs{a}}+\eta)\log K$ is given by the probability of at least one family alive at time $(\tfrac{c}{\abs{a}}+\tfrac{\eta}{2})\log K$ surviving for longer than $\tfrac{\eta}{2}\log K$, which is dominated by \[
	(1-C_*K^{\frac{\la\eta}{2}})^{K^{\frac{\eta\la}{3}}}\xrightarrow{K\to\infty}0.
	\]
	Therefore, the overall probability of having an individual alive at time $(\tfrac{c}{\abs{a}}+\eta)\log K$ is dominated by \[
	\P(\Gamma)+(1-C_*K^{\frac{\la\eta}{2}})^{K^{\frac{\eta\la}{3}}}\xrightarrow{K\to\infty}0.
	\]
	Since the process is extinct at time $(\tfrac{c}{\abs{a}}+\eta)\log K$ with probability converging to $1$ and also with high probability there is no migrant arriving in the population after this time, the lemma is proven.
\end{proof}

To end this section, we can now prove the general convergence from Theorem \ref{Theorem: Main Convergence} which we were looking for.

\begin{proof}[Proof of Theorem \ref{Theorem: Main Convergence}]This proof is taken from \cite[Theorem B.5]{champagnat2019stochastic} and adapted to our case.
	\begin{enumerate}
		\item[(iii)] This is a direct consequence of Lemma \ref{lemmanomigration}.
		\item[(ii)] Let $\veps>0$. We can apply Lemma \ref{lemmanomigration} up to time $T_1=\tfrac{\abs{c}}{a}-\veps$. This shows \[
		\lim\limits_{K\to\infty}\frac{\log(1+X_{t\log K}^K+Y_{t\log K}^K)}{\log K}=0\ \quad\text{ for all }t\in[0,T_1]\ \text{almost surely}.
		\]
		Applying the Markov property at time $T_1\log K$, we can apply Lemma \ref{lemmareemergence} and Lemma \ref{lemmacontinuity} to see that \[
		\frac{\log(1+X_{(T_1+\delta\veps)\log K}^K+Y_{(T_1+\delta\veps)\log K}^K)}{\log K}\in (\underline{c}\veps,\bar{c}\veps)
		\]
		and 
		\[
		\limsup\limits_{K\to\infty}\sup_{t\in[T_1,T_1+\delta\veps]}\frac{\log(1+X_{t\log K}^K+Y_{t\log K}^K)}{\log K}\leq c'\veps,
		\]
		where $\delta,\underline{c}>0$ and $\bar{c}<c'<\infty$. Now, applying Theorem \ref{firstconvergencetheorem} to a coupling from time $(T_1+\delta\veps)\log K$ with a lower bounding process $BBPI_K(b_1,b_2,d_1,d_2,\sigma_1,\sigma_2,a,\underline{c}\veps,\underline{c}\veps,\underline{c}\veps)$ and an upper bounding process $BBPI_K(b_1,b_2,d_1,d_2,\sigma_1,\sigma_2,a,\bar{c}\veps,\bar{c}\veps,\bar{c}\veps)$ shows that for times $t\in[T_1+\delta\veps,T]$ it holds \begin{align*}
		\underline{c}\veps+(\la\vee a)(t-t_1-\delta\veps)&\leq\liminf_{K\to\infty}\frac{\log(1+X_{t\log K}^K+Y_{t\log K}^K)}{\log K}\\
		&\leq \limsup_{K\to\infty}\frac{\log(1+X_{t\log K}^K+Y_{t\log K}^K)}{\log K}\\
		&\leq \bar{c}\veps+(\la\vee a)(t-t_1-\delta\veps).
		\end{align*}
		Now letting $\veps\to 0$ proves the claim.
		\item[(i)] Note that this part follows immediately from Theorem \ref{firstconvergencetheorem} in the case where $\la\geq 0$ or if $\la<0$, $a\geq 0$ and $c+a\tfrac{\beta\vee\gamma}{\abs{\la}}>0$. The remaining cases are as follows.
		\begin{description}
			\item[Case(a): $\la<0$ and $a<0$.] Assume for now that $c+a\tfrac{\beta\vee\gamma}{\abs{\la}}<0$. Then, we can apply Theorem \ref{firstconvergencetheorem} on the interval $[0,\tfrac{\beta\vee\gamma}{\abs{\la}}-\veps]$ for $\epsi$. Using Lemma \ref{lemmaextinction} (i), we also see that the process converges on the interval $[\tfrac{\beta\vee\gamma}{\abs{\la}}+\veps,\tfrac{\beta\vee\gamma}{\abs{\la}}+2\veps]$. Using Lemma \ref{lemmacontinuity} together with the coupling argument from (ii) and letting $\veps\to 0$ shows the convergence of the process on the interval $[0,T_1]$ for $T_1>\tfrac{\beta\vee\gamma}{\abs{\la}}$ sufficiently small against $\bar{\beta}$. Using the Markov property at time $T_1$ we can apply Lemma \ref{lemmanomigration} to obtain convergence on the entire interval $[0,T]$ towards $\bar{\beta}$.\\
			
			If $c+a\tfrac{\beta\vee\gamma}{\abs{\la}}>0$, we obtain convergence on the interval $[0,\tfrac{c}{\abs{a}}-\veps]$ from Theorem \ref{firstconvergencetheorem} and use Lemma \ref{lemmaextinction} (ii) instead. The remainder of the argument is still valid. In the case where $c+a\tfrac{\beta\vee\gamma}{\abs{\la}}=0$, we can use a coupling  argument where $\tilde{Z}^K=(\tilde{X}^K,\tilde{Y}^K)$ is a $BBPI_K(b_1,b_2,d_1,d_2,\sigma_1,\sigma_2,a,c-\veps,\beta,\gamma)$ and $\hat{Z}^K=(\hat{X}^K,\hat{Y}^K)$ is a $BBPI_K(b_1,b_2,d_1,d_2,\sigma_1,\sigma_2,a,c+\veps,\beta,\gamma)$ such that \[
			\tilde{X}_t^K+\tilde{Y}_t^K\leq X_t^K+Y_t^K\leq \hat{X}_t^K+\hat{Y}_t^K,
			\] and let $\veps\to 0$.
			
			\item[Case(b): $\la<0$, $a=0$ and $c<0$.] This case can be argued as in the first paragraph of the proof of Case (a).
			
			\item[Case(c): $\la<0$, $a>0$, $c<0$ and $\tfrac{\beta\vee\gamma}{\abs{\la}}<\tfrac{\abs{c}}{a}$.]  In this case, the population becomes extinct at first and is then revived due to immigration. Hence, we apply Theorem \ref{firstconvergencetheorem} up to time $\tfrac{\beta\vee\gamma}{\abs{\la}}-\veps$ and Lemma \ref{lemmaextinction} (i) on the interval $[\tfrac{\beta\vee\gamma}{\abs{\la}}+\veps,\tfrac{\beta\vee\gamma}{\abs{\la}}+2\veps]$ for $\epsi$. As in Case (a), using Lemma \ref{lemmacontinuity} and letting $\veps\to 0$ gives again convergence on the interval $[0,T_1]$ for $T_1>\tfrac{\beta\vee\gamma}{\abs{\la}}$ sufficiently small. Then, using the Markov property allows the application of Lemma \ref{lemmanomigration} until time $T_2=\tfrac{\abs{c}}{a}$. Again, we can use the Markov property and Lemma \ref{lemmacontinuity} to apply Lemma \ref{lemmareemergence} on a sufficiently small interval $[T_2,T_3]$. Then, using the Markov property again and using another coupling argument as in Case (a), letting $T_3\to T_2$, we can apply Theorem \ref{firstconvergencetheorem} to obtain convergence towards $\bar{\beta}$ on the entire interval $[0,T]$.
			\item[Case(d): $\la<0$, $a>0$, $c<0$ and $\tfrac{\beta\vee\gamma}{\abs{\la}}=\tfrac{\abs{c}}{a}$.] Here, we can couple as before with a lower bound process $BBPI_K(b_1,b_2,d_1,d_2,\sigma_1,\sigma_2,a,c-\veps,\beta,\gamma)$, for which we can apply Case (c), and the upper bound process $BBPI_K(b_1,b_2,d_1,d_2,\sigma_1,\sigma_2,a,c+\veps,\beta+\veps,\gamma+\veps)$, which satisfies $c+\veps+a\tfrac{(\beta\vee\gamma)+\veps}{\abs{\la}}>0$ and thus has already been treated as one of the trivial cases above. Letting $\veps\to 0$ yields the claim.
			\item[Case(e): $\la<0$, $c=a=0$.] Here, we use a similar coupling argument as in Case (d). The lower bound process has distribution $BBPI_K(b_1,b_2,d_1,d_2,\sigma_1,\sigma_2,0,-\veps,\beta,\gamma)$, which satisfies the assumption of Case (b), and the upper bound process has distribution $BBPI_K(b_1,b_2,d_1,d_2,\sigma_1,\sigma_2,0,\veps,\beta\vee\veps,\gamma\vee\veps)$, which satisfies $\veps+0\cdot \tfrac{\beta\vee\gamma\vee\veps}{\abs{\la}}>0$, so this case is treated as the second trivial case above. The claim follows by letting $\veps\to 0$.
		\end{description}
	\end{enumerate}
\end{proof}

%% file: Competition.tex
\section{Results on Logistic Processes}\label{Section: Appendix C}

In this section, we consider the bi-type logistic birth and death process $Z_t^K=(X_t^K,Y_t^K)$ where the transitions are given through \[
(n,m)\mapsto\begin{cases}
(n+1,m)&\text{ at rate }nb_1^K(\omega,t)+\gamma_1^K(\omega,t)\\
(n-1,m)&\text{ at rate }n(d_1^K(\omega,t)+\tfrac{(1-p)C}{K} n) \\
(n,m-1)&\text{ at rate }md_2^K(\omega,t)\\
(n-1,m+1)&\text{ at rate }\tfrac{pC}{K}n^2\\
(n+1,m-1)&\text{ at rate }m\sigma_2\\
\end{cases}
\]
with predictable, non-negative functions $b_1^K,d_1^K,d_2^K,\gamma_1^K\colon\Omega\times[0,\infty)\to \R$ and constants $C,\sigma_2>0$, $p\in(0,1)$.

\begin{lemma}\label{Lemma: EthierKurtz}
	Suppose that there are constants $b_1,d_1,d_2\geq 0$ such that \begin{align}
	\sup_{0\leq t\leq s\log K}\norm{b_1^K(t)-b_1}+\norm{d_1^K(t)-d_1}+\norm{d_2^K(t)-d_2}+\norm{\tfrac{\gamma_1^K(t)}{K}}\xrightarrow{K\to\infty}0\label{Convergence1}
	\end{align}
	in probability for some $s>0$.
	If we have $\tfrac{Z_0^K}{K}\to(\veps_1,\veps_2)$ as $K\to\infty$ for fixed $\veps_1,\veps_2>0$, then the process $\tfrac{Z^K_t}{K}$ converges uniformly on compact intervals in probability towards the solution $(x(t),y(t))$ of the ordinary differential equation \begin{align}
	\begin{aligned}
	\dot{x}(t)&=(b_1-d_1)x(t)-Cx^2(t)+\sigma_2y(t)\\
	\dot{y}(t)&=-(d_2+\sigma_2)y(t)+pCx^2(t)
	\end{aligned}\label{ODEAppendixC}
	\end{align}
	with initial condition $(x(0),y(0))=(\veps_1,\veps_2)$ as $K\to\infty$.
\end{lemma}
\begin{proof}
	This is similar to \cite[Theorem 11.2.1]{ethier2009markov}.
\end{proof}
\begin{notation}
	We denote processes $Z^K$ as introduced above by $LBBI_K(b_1^K,d_1^K,d_2^K,\sigma_2,p,C,\gamma_1^K)$. In the case where the functions $b_1^K,d_1^K$ and $d_2^K$ are all constant and $\gamma_1^K\equiv 0$, we may refer to the process as a $LBBI_K(b_1,d_1,d_2,\sigma_2,p,C)$.
\end{notation}
We are interested in calculating a coordinatewise positive equilibrium of the system (\ref{ODEAppendixC}). Assume that the equilibrium $(\bar{x},\bar{y})$ is positive, such that we can divide both sides of (\ref{ODEAppendixC}) by $x$ to obtain \[
\frac{\bar{y}}{\bar{x}}=-\frac{b_1-d_1-C\bar{x}}{\sigma_2}=\frac{pC\bar{x}}{d_2+\sigma_2}
\]
Hence, we see that \begin{align}
\bar{x}=\frac{(b_1-d_1)(d_2+\sigma_2)}{C(d_2+(1-p)\sigma_2)}
\quad\text{ and }\quad
\bar{y}=\frac{p(b_1-d_1)^2(d_2+\sigma_2)}{C(d_2+(1-p)\sigma_2)^2}.\label{eq:equilibrium}
\end{align}
Thus, the assumption $b_1>d_1$ is sufficient for obtaining a coordinatewise positive equilibrium. As is shown in \cite[Section 2.2]{blath2020invasion}, this equilibrium is the only stable equilibrium and in fact the system converges towards this equilibrium for any initial condition $(\veps_1,\veps_2)\in (0,\infty)^2$. This can be seen from Lemma 4.6 in \cite{blath2020invasion}.\\

\subsection{Problem of Exit and Entry of a Domain}

We will now concern ourselves with estimating the length of time until a logistic bi-type branching process exits a neighbourhood of its equilibrium population size; and with the existence of a time such that the process enters for the first time a neighbourhood of its equilibrium. The important results in this section are Lemma \ref{Lemma: convergence to equilibrium} for the entry into a domain around the equilibrium and Corollary \ref{Corollary: Equilibriumcloseness} for the exit of such a domain. Our first step is to generalize Lemma \ref{Lemma: EthierKurtz}.

\begin{lemma}\label{Lemma: Generalized EthierKurz}
	Let $T>0$ and $b_1,b_2,d_1,d_2\geq 0$, $\sigma_2>0$. Further let $C>0$, $p\in(0,1)$ and let $\tilde{C}$ be a compact subset of $(0,\infty)^2$. Denote the solution of the differential equation \begin{align}\begin{aligned}
	\dot{\vphi}_1&=(b_1-d_1-C\vphi_1)\vphi_1+\sigma_2\vphi_2\\
	\dot{\vphi}_2&=(b_2-d_2-\sigma_2)\vphi_2 +Cp(\vphi_1)^2
	\end{aligned}\label{ODELemmaC}
	\end{align}
	with initial condition $(z_1,z_2)\in \tilde{C}$ by $\vphi_{z_1,z_2}$. Then for any $T>0$, \[
	r\coloneqq \inf_{z\in \tilde{C}}\inf_{t\in[0,T]}\norm{\vphi_{z_1,z_2}(t)}>0\quad\text{ and }\quad R\coloneqq\sup_{z\in \tilde{C}}\sup_{t\in[0,T]}\norm{\vphi_{z_1,z_2}(t)}<\infty.
	\]
	Denote the distribution of the Markov process with transition rates \[
	(\tfrac{n}{K},\tfrac{m}{K})\mapsto\begin{cases}
	(\tfrac{n+1}{K},\tfrac{m}{K})&\text{ at rate }nb_1\\
	(\tfrac{n}{K},\tfrac{m+1}{K})&\text{ at rate }mb_2\\
	(\tfrac{n-1}{K},\tfrac{m}{K})&\text{ at rate }n(d_1+\tfrac{(1-p)C}{K} n) \\
	(\tfrac{n}{K},\tfrac{m-1}{K})&\text{ at rate }md_2\\
	(\tfrac{n-1}{K},\tfrac{m+1}{K})&\text{ at rate }\tfrac{pC}{K}n^2\\
	(\tfrac{n+1}{K},\tfrac{m-1}{K})&\text{ at rate }m\sigma_2\\
	\end{cases}
	\]
	by $\P_z^K$. Then for any $0<\delta<r$, we have \[
	\lim\limits_{K\to\infty}\sup_{z\in \tilde{C}}\P_z^K\lr{\sup_{t\in[0,T]}\norm{w_t-\vphi_{z_1,z_2}(t)}_{}\geq\delta}=0,
	\]
	where $w_t$ is the canonical process on $D([0,\infty),\R^2)$, the space of càdlàg paths from $[0,\infty)$ to $\R^2$.
\end{lemma}

\begin{remark}
	Note that this result holds also for processes where $b_2>0$. The rescaled process $Z^K$ mentioned in the beginning of this section is a special case with $b_2=0$.
\end{remark}
\begin{proof}
	We use the techniques from \cite[Theorem 3]{champagnat2005microscopic}. Let $T>0$. To show the boundedness properties, we first show that $\vphi_2$ is strictly larger than $0$ for any positive initial condition. Indeed, we easily see that \[
	\dot{\vphi}_2\geq (b_2-d_2-\sigma_2)\vphi_2
	\]
	which we can integrate directly to obtain \[
	\vphi_2(t)\geq \vphi_2(0)\exp\lr{\int_{0}^{t}(b_2-d_2-\sigma_2))\ds}=\vphi_2(0)\exp\lr{(b_2-d_2-\sigma_2)t},
	\]
	which is positive for all times $t$ as soon as $\vphi_2(0)>0$. In particular, the solution $(\vphi_{z_1,z_2})_1$ is bounded from below by the solution $\tilde{\vphi}$ of the differential equation \[
	\dot{\tilde{\vphi}}=(b_1-d_1-C\tilde{\vphi})\tilde{\vphi}.
	\] The function $\tilde{\vphi}$ is bounded from above by some constant $\tilde{R}$ for any positive initial condition, since $\dot{\tilde{\vphi}}<0$ as soon as $\tilde{\vphi}>(b_1-d_1)/C$. Hence, by integrating the differential equation, we obtain for any positive initial condition $\tilde{\vphi}(0)>0$ the inequality \[
	\tilde{\vphi}(t)=\tilde{\vphi}(0)\exp\lr{\int_{0}^{t}(b_1-d_1-C\tilde{\vphi}(s))\ds}\geq \tilde{\vphi}(0)\exp\lr{(b_1-d_1-C\tilde{R})t}>0.
	\]
	
	Thus far we have shown $(\vphi_{z_1,z_2})_1(t),(\vphi_{z_1,z_2})_2(t)>0$ for any time $t\geq 0$ and any initial condition $(z_1,z_2)\in\tilde{C}$. In particular, due to the continuity of $(z_1,z_2)\mapsto \vphi_{z_1,z_2}$ we have \[
	\inf_{z\in \tilde{C}}\inf_{t\in[0,T]}\norm{\vphi_{z_1,z_2}(t)}>0.
	\] For the upper bound, adding the two equations in (\ref{ODELemmaC}) yields \[
	\dot{\vphi}_1+\dot{\vphi}_2=(b_1-d_1)\vphi_1+(b_2-d_2)\vphi_2-C(1-p)(\vphi_1)^2\leq 2\max(b_1-d_1,b_2-d_2)(\vphi_1+\vphi_2),
	\]
	where we used $\vphi_1,\vphi_2\geq 0$ for any initial condition $(z_1,z_2)\in\tilde{C}$.
	Thus, Gronwall's inequality implies\begin{align*}
	\sup_{z\in \tilde{C}}\sup_{t\in[0,T]}\norm{\vphi_{z_1,z_2}(t)}_1&=\sup_{z\in \tilde{C}}\sup_{t\in[0,T]}(\vphi_{z_1,z_2})_1(t)+(\vphi_{z_1,z_2})_2(t)\\
	&\leq \sup_{z\in \tilde{C}} (z_1+z_2)\exp\lr{\int_{0}^{T}2\max(b_1-d_1,b_2-d_2)\dt}<\infty.
	\end{align*}
	By equivalence of norms on $\R^2$, for any given norm, these bounds on the infimum and supremum can be chosen to hold.
	
	Now, we can define a family of Markov processes with transitions \begin{align*}
	(\tfrac{i}{K},\tfrac{j}{K})\to\begin{cases}
	(\tfrac{i+1}{K},\tfrac{j}{K}),&\quad\text{ at rate }Kp_1(\tfrac{i}{K},\tfrac{j}{K})\\
	(\tfrac{i}{K},\tfrac{j+1}{K})&\quad\text{ at rate }Kp_2(\tfrac{i}{K},\tfrac{j}{K})\\
	(\tfrac{i-1}{K},\tfrac{j}{K})&\quad\text{ at rate }Kq_1(\tfrac{i}{K},\tfrac{j}{K})\\
	(\tfrac{i}{K},\tfrac{j-1}{K})&\quad\text{ at rate }Kq_2(\tfrac{i}{K},\tfrac{j}{K})\\
	(\tfrac{i-1}{K},\tfrac{j+1}{K})&\quad\text{ at rate }Kr_1(\tfrac{i}{K},\tfrac{j}{K})\\
	(\tfrac{i+1}{K},\tfrac{j-1}{K})&\quad\text{ at rate }Kr_2(\tfrac{i}{K},\tfrac{j}{K}),\\
	\end{cases}
	\end{align*} 
	where $p_1,p_2,q_1,q_2,r_1,r_2\colon\R^2\to\R$ are positive, bounded and Lipschitz functions. We denote the law of such a process by $\Q_z^K$, when the initial condition is given by $z\in\tfrac{1}{K}\N_0^2$. This choice of transition rates corresponds (using the notation from \cite{dupuis1997weak} in equation (10.1)) to the choice of $\veps=\tfrac{1}{K}$ and the measure $\nu_x$ being given by \begin{align*}
	\nu_{(\frac{i}{K},\frac{j}{K})}(\{(1,0)\})=p_1(\tfrac{i}{K}),&\quad  \nu_{(\frac{i}{K},\frac{j}{K})}(\{(0,1)\})=p_2(\tfrac{j}{K}), & \nu_{(\frac{i}{K},\frac{j}{K})}(\{(-1,0)\})=q_1(\tfrac{i}{K}),\\
	\nu_{(\frac{i}{K},\frac{j}{K})}(\{(0,-1)\})=q_2(\tfrac{j}{K}), &\quad  \nu_{(\frac{i}{K},\frac{j}{K})}(\{(-1,1)\})=r_1(\tfrac{i}{K}), &\nu_{(\frac{i}{K},\frac{j}{K})}(\{(1,-1)\})=r_2(\tfrac{j}{K}).
	\end{align*}
	The extension of $\nu_{x,y}$ for any vector $(x,y)\in\R^2$ is straightforward by replacing either $\tfrac{i}{K}$ by $x$ or $\tfrac{j}{K}$ by $y$ respectively. In addition, we choose the functions $b_1=p_1-q_1-r_1+r_2$, $b_2=p_2-q_2-r_2+r_1$ and $a\equiv 0$.
	
	Since the functions involved are all bounded and continuous, Condition 10.2.2 from \cite{dupuis1997weak} is satisfied. In order to apply Theorem 10.2.6 of \cite{dupuis1997weak}, we do not need any additional conditions. However, we only obtain the upper bound of the Laplace principle as can be seen from the remark preceding the Theorem. The good rate function $I_T$ appearing in the Laplace principle writes for functions $\vphi\colon[0,T]\to\R^2$ as \[
	I_T(\vphi)=\begin{cases}
	\int_{0}^{T}L(\vphi(t),\dot{\vphi}(t))\dt &\quad\text{if }\vphi \text{ is absolutely continuous}\\
	\infty&\quad\text{ otherwise},
	\end{cases}
	\]
	where for $y,z\in\R^2$ we define the function $L(y,z)=\sup_{\alpha\in\R^2}\lr{\SP{\alpha}{z}-H(y,\alpha)}$ with \[
	H(y,\alpha)=\int_{\R^2}\lr{\exp(\SP{\alpha}{x})-1}\nu_{y}(\mathrm{d}x).
	\]
	Hence, calculating the gradient of the function in the supremum with respect to $\alpha$ shows that $L(y,z)=0$ if and only if $z_1=p_1(y)-q_1(y)-r_1(y)+r_2(y)$ and $z_2=p_2(y)-q_2(y)-r_2(y)+r_1(y)$. Therefore, our rate function $I_T$ satisfies \begin{align}
	I_T(\vphi)=0\quad\Llra\quad \begin{pmatrix}
	\dot{\vphi}_1\\
	\dot{\vphi}_2
	\end{pmatrix}=\begin{pmatrix}
	p_1(\vphi)-q_1(\vphi)-r_1(\vphi)+r_2(\vphi)\\
	p_2(\vphi)-q_2(\vphi)-r_2(\vphi)+r_1(\vphi)
	\end{pmatrix}.\label{ODERateFunction}
	\end{align}
	Since the upper bound of the Laplace principle is by \cite[Corollary 1.2.5]{dupuis1997weak} equivalent to the upper bound in the large deviation principle, we obtain \[
	\limsup_{K\to\infty}\reci{K}\log\lr{\sup_{z\in \tilde{C}}\Q_z^K(F)}\leq -\inf_{\psi\in F,\psi(0)\in \tilde{C}}I_T(\psi)
	\]
	for any compact set $\tilde{C}\sse\R^2$ and any closed set $F\sse D([0,T],\R^2)$, that is the set of cádlág functions on $[0,T]$ into $\R^2$.\\
	
	We define the cut-off function $\chi$ as the orthogonal projection from $\R^2$ onto $[r-\delta,R+\delta]^2$ with respect to the Euclidean norm. Then, we can define our functions \begin{align*}
	&p_1(z)=b_1\chi_1(z),\quad p_2(z)=b_2\chi_2(z),\quad q_1(z)=d_1\chi_1(z)+C(1-p)\chi_1^2(z),\\ &q_2(z)=d_2\chi_2(z),\quad
	r_1(z)=pC\chi_1^2(z),\quad r_2(z)=\sigma_2\chi_2(z).
	\end{align*}
	Now, the laws $\P_z^K$ and $\Q_z^K$ coincide as long as $w_t^1$ and $w_t^2$ are still inside the interval $[r-\delta,R+\delta]$. Thus, we have \begin{align*}
	&\limsup_{K\to\infty}\reci{K}\log\sup_{z\in \tilde{C}}\P_z^K\lr{\sup_{t\in[0,T]}\norm{w_t-\vphi_{z_1,z_2}(t)}\geq\delta}\\
	=&\limsup_{K\to\infty}\reci{K}\log\sup_{z\in \tilde{C}}\Q_z^K\lr{\sup_{t\in[0,T]}\norm{w_t-\vphi_{z_1,z_2}(t)}\geq\delta}\leq -\inf_{\psi\in F}I_T(\psi),
	\end{align*}
	where \[
	F\coloneqq \lrset{\psi\in D([0,T],\R^2)\mid \psi(0)\in \tilde{C}\text{ and } \exists t\in[0,T]: \norm{\psi(t)-\vphi_{\psi(0)}}\geq\delta}.
	\]
	Note that due to the continuity of $\vphi_z$ the set $F$ is closed with respect to the supremum norm, which allows us to use the large deviation principle. Furthermore, it is easy to see that absolutely continuous functions $\psi\in F$ cannot satisfy (\ref{ODERateFunction}), as otherwise the distance between $\vphi$ and $\psi$ cannot become large. Since $I_T$ is a good rate function, the infimum is attained for some function $\psi\in F$ and thus by (\ref{ODERateFunction}) is non-zero.
\end{proof}

This lemma has put us in a good position to show that the considered logistic bi-type branching process converges in finite time into a neighbourhood of its equilibrium.
\begin{lemma}\label{Lemma: convergence to equilibrium}
	Let $Z^K$ be a $LBBI_K(b_1^K,d_1^K,d_2^K,\sigma_2,p,C,\gamma_1^K)$ and assume that the convergence (\ref{Convergence1}) holds with $b_1>d_1$. Then, for all $\veps_1,\veps_2,\veps_1',\veps_2'>0$ there exists a finite time $T(\veps_1,\veps_2,\veps_1',\veps_2')$ such that for all initial starting conditions $(2\bar{x},2\bar{y})\geq\tfrac{Z_0^K}{K}=\tfrac{(X_0^K,Y_0^K)}{K}\geq (\veps_1,\veps_2)$, we have \[
	\lim\limits_{K\to\infty}\P\lr{\norm{\frac{Z_{T(\veps_1,\veps_2,\veps_1',\veps_2')}^K}{K}-(\bar{x},\bar{y})}\leq \veps_1'+\veps_2'}=1,
	\]
	where $(\bar{x},\bar{y})$ is given in (\ref{eq:equilibrium}).
\end{lemma}
\begin{proof}
	For the proof of this claim, we use for $K$ large enough the following coordinatewise coupling \[
	\tilde{Z}^K\leq Z^K\leq \hat{Z}^K,
	\]
	where $\tilde{Z}^K$ is a $LBBI_K(b_1-\veps,d_1+\veps,d_2+\veps,\sigma_2,p,C)$ and the distribution of $\hat{Z}^K$ is given by $LBBI_K(b_1+\veps,(d_1-\veps)_+,(d_2-\veps)_+,\sigma_2,p,C)$, and the initial conditions are $\tilde{Z}^K_0=Z^K_0=\hat{Z}^K_0$ with $\epsi$. We need to justify the coupling, specifically why we can increase the birth rate and at this cost neglect the immigration from outside in the upper bounding process. For this purpose, we need to show that ${Z}^K$ is bounded component-wise from below in probability by $K\veps$ for some $\epsi$. Indeed, this can already be seen from the lower bounding process: Firstly there exists $\delta>0$ and a time $T>0$ such that the corresponding solutions $\tilde\vphi$ and $\hat{\vphi}$ to the differential equation for the processes $\tilde{Z}^K$ and $\hat{Z}^K$ satisfy \[
	\norm{\tilde{\vphi}(T)-(x,y)_{\tilde{\vphi}}}<\delta\quad\text{ and }\quad \norm{\hat{\vphi}(T)-(x,y)_{\hat{\vphi}}}<\delta,
	\]
	where $(x,y)_{\vphi}$ denotes the unique coordinatewise positive stable equilibrium of the differential equation corresponding to $\vphi$. Note that we can use the same times, as we have convergence towards the equilibria as $T\to\infty$ from any starting condition and hence we can choose $T$ sufficiently large so that both systems are close to their equilibrium. For this, we again refer the reader to \cite[Lemma 4.6]{blath2020invasion}. Now, Lemma \ref{Lemma: Generalized EthierKurz} implies that with $\tilde{C}\coloneqq [\veps_1,2\bar{x}]\times[\veps_2,2\bar{y}]$ and $\tilde{Z}_0^K=z\in\tilde{C}$, we have for any $\delta>0$ small enough \[
	\sup_{z\in \tilde{C}}\P\lr{\sup_{t\in[0,T]}\norm{\frac{\tilde{Z}_t^K}{K}-\tilde{\vphi}(t)}\geq\delta}\xrightarrow{K\to\infty}0.
	\]
	Note, that \[\inf_{z\in\tilde{C}}\inf_{t\in[0,T]}\norm{\tilde{\vphi}(t)}>0\quad\text{ and }\quad\inf_{z\in\tilde{C}}\inf_{t\in[0,T]}\norm{\tilde{\vphi}_1(t)}>0,\] where the second inequality is due to the positive switching between components which can be seen from the proof of Lemma \ref{Lemma: Generalized EthierKurz}. In particular, for $\epsi$ small enough it holds $\textstyle\inf_{t\in[0,T]}\abs{\tfrac{(\tilde{Z}_t^K)_1}{K}}>\veps$ with probability converging to $1$ as $K\to\infty$. Thus for $K$ large enough, we have \[\inf_{t\in[0,T]}\frac{\veps}{2}\cdot(\tilde{Z}_t^K)_1>\sup_{t\in[0,T]}\gamma_1^K(\w,t)\] with probability converging to $1$. Hence, we may dismiss the immigration component in the coupling and replace it by an increase in the birth rate as done above. 
	Now, choosing $\veps$ small enough, we can achieve for any given $\delta>0$ that \[
	\norm{(x,y)_{\tilde{\vphi}}-(x,y)_{\hat{\vphi}}}<\delta,\quad\norm{(x,y)_{\tilde{\vphi}}-(\bar{x},\bar{y})}<\delta\quad\text{ and }\quad \norm{(x,y)_{\hat{\vphi}}-(\bar{x},\bar{y})}<\delta,
	\]
	which can be seen from computing the equilibria similarly to (\ref{equilibrium}). Hence, by Lemma \ref{Lemma: Generalized EthierKurz} we have with high probability, that \begin{align*}
	\norm{\frac{\hat{Z}^K_T-\tilde{Z}^K_T}{K}}&\leq \norm{\frac{\hat{Z}^K_T}{K}-\hat{\vphi}(T)}+\norm{\hat{\vphi}(T)-\tilde{\vphi}(T)}+\norm{\frac{\tilde{Z}^K_T}{K}-\tilde{\vphi}(T)}\\
	&\leq\norm{\hat{\vphi}(T)-\tilde{\vphi}(T)}+\delta\leq 4\delta.
	\end{align*}
	Because of Lemma \ref{Lemma: Generalized EthierKurz} for any starting condition the inequality \[
	\norm{\frac{\hat{Z}_T^K}{K}-(\bar{x},\bar{y})}\leq\norm{\frac{\hat{Z}_T^K}{K}-\hat{\vphi}(T)}+\norm{\hat{\vphi}(T)-(x,y)_{\hat{\vphi}}}+\norm{(x,y)_{\hat{\vphi}}-(\bar{x},\bar{y})}\leq 3\delta
	\]
	is satisfied with probability converging to $1$ as $K\to\infty$, and thus we obtain from the component-wise coupling that \[
	\norm{\frac{Z^K_T}{K}-(\bar{x},\bar{y})}\leq\norm{\frac{Z^K_T-\hat{Z}^K_T}{K}}+\norm{\frac{\hat{Z}^K_T}{K}-(\bar{x},\bar{y})}\leq \norm{\frac{\hat{Z}^K_T-\tilde{Z}^K_T}{K}}+3\delta\leq 7\delta
	\]
	with high probability. Letting $\delta>0$ small enough yields the claim.
\end{proof}

Thus far, we have been considering the behaviour of our process when it is initially not close to its equilibrium size. We now turn to the question, how long it takes for a logistic process to exit a neighbourhood of its equilibrium. For this we first consider a process with constant rates.

\begin{lemma}\label{Lemma:exponentialcloseness}
	In the situation of Lemma \ref{Lemma: Generalized EthierKurz} let $b_1>d_1$. Then, the unique asymptotically stable equilibrium of the system \begin{align*}
	\dot{\vphi}_1&=(b_1-d_1-C\vphi_1)\vphi_1+\sigma_2\vphi_2\\
	\dot{\vphi}_2&=(-d_2-\sigma_2)\vphi_2 +Cp(\vphi_1)^2
	\end{align*}
	is given by $(\bar{\vphi}_1,\bar{\vphi}_2)=(\bar{x},\bar{y})$ in (\ref{eq:equilibrium}). Let $\eta_1,\eta_2>0$ and set \[
	T^K\coloneqq\inf\lrset{t\geq 0\mid w_t^1\notin[\bar{x}-\eta_1,\bar{x}+\eta_1]\text{ or } w_t^2\notin[\bar{y}-\eta_2,\bar{y}+\eta_2]},
	\]
	where $w_t=(w_t^1,w_t^2)$ is the canonical process on $D([0,\infty),\R^2)$. Then there exists a constant $V>0$ such that for all compact subsets $\tilde{C}$ of $[\bar{x}-\tfrac{\eta_1}{2},\bar{x}+\tfrac{\eta_1}{2}]\times[\bar{y}-\tfrac{\eta_2}{2},\bar{y}+\tfrac{\eta_2}{2}]$ we have \[
	\lim\limits_{K\to\infty}\sup_{z\in \tilde{C}}\P_z^K(T^K<e^{KV})=0.
	\]
\end{lemma}
\begin{proof}
	As in the proof of Lemma \ref{Lemma: Generalized EthierKurz} and similarly to \cite{champagnat2005microscopic} we define the cut-off function $\chi$ as the orthogonal projection onto $C_1\coloneqq[\bar{x}-\eta_1,\bar{x}+\eta_1]\times[\bar{y}-\eta_2,\bar{y}+\eta_2]$. Also, we can similarly construct the functions $p_1,q_1,q_2,r_1,r_2$, by which we obtain a family of laws $\Q_z^K$ which coincide with $\P_z^K$ on the time interval $[0,T^K]$. Furthermore, we obtain in a similar manner the good rate function $I_T(\vphi)$, where now  \begin{align}
	I_T(\vphi)=0\quad\Llra\quad \begin{pmatrix}
	\dot{\vphi}_1\\
	\dot{\vphi}_2
	\end{pmatrix}=\begin{pmatrix}
	p_1(\vphi)-q_1(\vphi)-r_1(\vphi)+r_2(\vphi)\\
    -q_2(\vphi)-r_2(\vphi)+r_1(\vphi)
	\end{pmatrix}.\label{secondODE}
	\end{align}
	Now, \cite[p.157]{freidlin1998random} and Theorem 4.4.2 therein imply, that there exists a constant $\bar{V}\geq 0$, such that for any $\delta>0$ it holds \[
	\lim\limits_{K\to\infty}\inf_{z\in\tilde{C}}\Q_z^K(e^{K(\bar{V}-\delta)}<T^K<e^{K(\bar{V}+\delta)})=1.
	\]
	Hence, proving that $\bar{V}>0$ shows our claim. For this purpose, we note that $\bar{V}$ is defined in \cite[Theorem 4.4.1]{freidlin1998random} by \[
	\bar{V}\coloneqq\inf_{(x,y)\in \partial C_1} V((\bar{x},\bar{y}),(x,y))=\min_{(x,y)\in \partial C_1} V((\bar{x},\bar{y}),(x,y)),
	\]
	where \[
	V((x_1,x_2),(y_1,y_2))\coloneqq\inf_{t>0,\ \vphi(0)=(x_1,x_2),\ \vphi(t)=(y_1,y_2)}I_t(\vphi),
	\]
	that is $V((x_1,x_2),(y_1,y_2))$ is the infimum of our rate function over all possible paths connecting $(x_1,x_2)$ and $(y_1,y_2)$. Since the function $(x,y)\mapsto V((\bar{x},\bar{y}),(x,y))$ is continuous and the set $\partial C_1$ is compact, we can indeed replace the infimum by a minimum. In particular, there exists some $(x_0,y_0)\in\partial C_1$ where the infimum is attained. Theorem 5.4.3 from \cite{freidlin1998random} shows that there exists an absolutely continuous function $\vphi$, which attains the infimum for $V((\bar{x},\bar{y}),(x_0,y_0))$ over the rate function in the sense that for some $T>0$ the function $\vphi\colon [0,T]\to\R^2$ satisfies the conditions $\vphi(0)=(\bar{x},\bar{y})$, $\vphi(T)=(x_0,y_0)$ and $V((\bar{x},\bar{y}),(x_0,y_0))=I_T(\vphi)$ or that for some $T>-\infty$ the function $\vphi\colon(-\infty,T]\to\R^2$ satisfies $\textstyle\lim_{t\to-\infty}\vphi(t)=(\bar{x},\bar{y})$, $\vphi(T)=(x_0,y_0)$ and $\textstyle V((\bar{x},\bar{y}),(x_0,y_0))=\int_{-\infty}^{T}L(\vphi(t),\dot{\vphi}(t))\dt$. \\
	
	As long as $\vphi$ satisfies the differential equation in (\ref{secondODE}) with initial condition $\vphi(0)=(\bar{x},\bar{y})$, we have $\vphi(t)=(\bar{x},\bar{y})\neq(x_0,y_0)$. Thus, for the case where $\vphi$ is defined on $[0,T]$, we have $I_T(\vphi)\neq 0$, as the differential equation in (\ref{secondODE}) must be violated at some time in order to leave the equilibrium state. For the second case, we already know that any solution to (\ref{secondODE}) started close to $(\bar{x},\bar{y})$ stays close, since the equilibrium is asymptotically stable. Hence, choosing $0<\veps<\eta_1\wedge\eta_2$ small enough and $T_1\coloneqq \sup\{t\leq T\mid \norm{\vphi(t)-(\bar{x},\bar{y})}<\tfrac{\veps}{2}\}$, for all $t\geq T_1$ we have $\norm{\vphi(t)-(\bar{x},\bar{y})}<\veps$. Therefore in this case connecting $(\bar{x},\bar{y})$ to $(x_0,y_0)$ requires violating (\ref{secondODE}). Hence we have $\bar{V}>0$.
\end{proof}

Using our usual coupling arguments to make the leap from fixed rates to some non-fixed rates, we can formulate and prove a very similar result for logistic processes with non-constant rates.
\begin{lemma}\label{Lemma: exponentialclosenessapproximation}
	Let $Z^K$ be a $LBBI_K(b_1^K,d_1^K,d_2^K,\sigma_2,p,C,\gamma_1^K)$ and assume that the convergence (\ref{Convergence1}) holds with $b_1>d_1$. Let $\eta_1,\eta_2>0$, $Z_0^K\in[\bar{x}-\tfrac{\eta_1}{2},\bar{x}+\tfrac{\eta_1}{2}]\times[\bar{y}-\tfrac{\eta_2}{2},\bar{y}+\tfrac{\eta_2}{2}]$ and set \[
	T^K\coloneqq\inf\lrset{t\geq 0\ \Big\vert\  \frac{(Z_t^K)_1}{K}\notin[\bar{x}-\eta_1,\bar{x}+\eta_1]\text{ or } \frac{(Z_t^K)_2}{K}\notin[\bar{y}-\eta_2,\bar{y}+\eta_2]}.
	\] Then there exists a constant $V>0$ such that  \[
	\lim\limits_{K\to\infty}\P(T^K<e^{KV})=0.
	\]
\end{lemma}
\begin{proof}
	As in the proof of Lemma \ref{Lemma: convergence to equilibrium}, we couple our process with \[
	\tilde{Z}^K\leq Z^K\leq \hat{Z}^K,
	\]
	where $\tilde{Z}^K$ is given as a $LBBI_K(b_1-\veps,d_1+\veps,d_2+\veps,\sigma_2,p,C)$ and the law of $\hat{Z}$ is $LBBI_K(b_1+\veps,(d_1-\veps)_+,(d_2-\veps)_+,\sigma_2,p,C)$, where the initial conditions are $\tilde{Z}^K_0=Z^K_0=\hat{Z}^K_0$ and $\epsi$. Choosing $\veps$ small enough  and applying Lemma \ref{Lemma:exponentialcloseness} to $\tilde{Z}^K$ and $\hat{Z}^K$ yields the claim.
\end{proof}

A very simple consequence is now the result that our process remains at least a time of order $\log K$ in a neighbourhood of its equilibrium. 
\begin{corollary}\label{Corollary: Equilibriumcloseness}
	Under the assumptions of Lemma \ref{Lemma: exponentialclosenessapproximation} it holds  \[
	\lim\limits_{K\to\infty}\P\lr{\forall t\in[0,T\log K], \frac{(Z_t^K)_1}{K}\in[\bar{x}-\eta_1,\bar{x}+\eta_1]\text{ and } \frac{(Z_t^K)_2}{K}\in[\bar{y}-\eta_2,\bar{y}+\eta_2]}=1.
	\]
	for any $T>0$.
\end{corollary}

\subsection{Competition Between two Bi-Type Processes with Transfer}\label{SubSection: Competition}

Now, we consider a four-dimensional logistic branching process, which we will interpret as competition between two bi-type logistic branching processes each with active (`$a$') individuals and dormant (`$d$') individuals. Our main results in this section will be Propositions \ref{Proposition: 4Dpositive} and \ref{Proposition: 4Dnegative}, where we show under suitable assumptions that the initially resident process declines below a small threshold, while the invading process reaches a neighbourhood of its equilibrium. The transfer rates of this process $(X_a^K,X_d^K,Y_a^K,Y_d^K)$ are \begin{align*} (i,j,k,\ell)\to\begin{cases}
(i+1,j,k,\ell) &\text{ at rate }ia_1^K(\omega,t)+\gamma_1^K(\omega,t)\\
(i,j,k+1,\ell) &\text{ at rate }kb_1^K(\omega,t)+\gamma_2^K(\omega,t)\\
(i-1,j,k,\ell) &\text{ at rate }i(d_1^K(\w,t)+\tfrac{(1-p)C}{K}(i+k))\\
(i,j,k-1,\ell) &\text{ at rate }k(d_1^K(\w,t)+\tfrac{(1-q)C}{K}(i+k))\\
(i,j-1,k,\ell) &\text{ at rate }jd_2^K(\omega,t)\\
 (i,j,k,\ell-1) &\text{ at rate }\ell d_2^K(\omega,t)\\
(i-1,j+1,k,\ell) &\text{ at rate }i\tfrac{pC}{K}(i+k) \\ (i,j,k-1,\ell+1) &\text{ at rate }k\tfrac{qC}{K}(i+k)\\
(i+1,j-1,k,\ell)&\text{ at rate }j\sigma_2\\ (i,j,k+1,\ell-1)&\text{ at rate }\ell\sigma_2\\
(i-1,j,k+1,\ell)&\text{ at rate }\tau^K(\w,t)\tfrac{ik}{i+k},
\end{cases}
\end{align*}
with predictable, non-negative functions $a_1^K,b_1^K,d_1^K,d_2^K,\tau^K,\gamma_1^K,\gamma_2^K\colon\Omega\times[0,\infty)\to\R$ and constants $C,\sigma_2>0$, $p,q\in(0,1)$. For now, we assume that the transition rates are constant with $a_1^K\equiv a_1$, $b_1^K\equiv b_1$, $d_1^K\equiv d_1$, $d_2^K\equiv d_2$, $\tau^K\equiv\tau$ and $\gamma_1^K=\gamma_2^K=0$ and that $a_1,b_1>d_1$.
Then the process $\tfrac{1}{K}(X_a^K,X_d^K,Y_a^K,Y_d^K)$ converges towards the unique solution of the differential equation
\begin{align} 
\begin{aligned}
\dot{x}_a&=x_a(a_1-d_1-C(x_a+y_a))+x_d\sigma_2-\tau\frac{x_ay_a}{x_a+y_a}\\
\dot{x}_d&=pCx_a(x_a+y_a)-(d_2+\sigma_2)x_d\\
\dot{y}_a&=y_a(b_1-d_1-C(x_a+y_a))+y_d\sigma_2+\tau\frac{x_ay_a}{x_a+y_a}\\
\dot{y}_d&=qCy_a(x_a+y_a)-(d_2+\sigma_2)y_d.
\end{aligned}\label{eq: Dynamical System}
\end{align}
We are now interested in finding a suitable criterion for invasion of the process $(Y_a^K,Y_d^K)$ into the initially resident population $(X_a^K,X_d^K)$. More specifically, we assume that initially the size of $(\tfrac{X_a^K}{K},\tfrac{X_d^K}{K})$ is close to its equilibrium $(\bar{x}_a,\bar{x}_d)$ with \begin{align}
\bar{x}_a=\frac{(a_1-d_1)(d_2+\sigma_2)}{C(d_2+(1-p)\sigma_2)}
\quad\text{ and }\quad
\bar{x}_d=\frac{p(a_1-d_1)^2(d_2+\sigma_2)}{C(d_2+(1-p)\sigma_2)^2}\label{eq: Equilibrium}
\end{align}
and the total size of the invasive species is $Y_a^K+Y_d^K=\lfloor\veps K\rfloor$. Note that for such small population sizes we may approximate the transfer rate between $X_a^K$ and $Y_a^K$ by \[\frac{X_a^KY_a^K}{X_a^K+Y_a^K}\approx\frac{X_a^KY_a^K}{X_a^K}=Y_a^K.\] Further, we assume that the mean matrix \begin{align}
J=\begin{pmatrix}
b_1+\tau-d_1-C\bar{x}_a & qC\bar{x}_a\\
\sigma_2 & -d_2-\sigma_2
\end{pmatrix}\label{eq: meanmatrix}
\end{align} of the approximating process $(\hat{Y}_a,\hat{Y}_d)$ of $(Y_a^K,Y_d^K)$ given by the transitions \[
(n,m)\mapsto\begin{cases}
(n+1,m)&\text{ at rate }n(b_1+\tau)\\
(n-1,m)&\text{ at rate }n(d_1+(1-q)C \bar{x}_a) \\
(n,m-1)&\text{ at rate }md_2\\
(n-1,m+1)&\text{ at rate }nqC\bar{x}_a\\
(n+1,m-1)&\text{ at rate }m\sigma_2\\
\end{cases}
\]
has a positive eigenvalue $\la$, which means that the invasion fitness is positive and the process is supercritical. For one eigenvalue to be positive, the determinant must be negative, which in our case is equivalent to the inequality \begin{align}
-(b_1+\tau-d_1-C\bar{x}_a)<\frac{\sigma_2qC\bar{x}_a}{d_2+\sigma_2}.\label{eq: sufficientpositiveeigenvalue}
\end{align}
In addition, we will assume that in a population, where $(Y_a^K,Y_d^K)$ is resident and $(X_a^K,X_d^K)$ is invasive, the approximating process $(\hat{X}_a,\hat{X}_d)$ given by the transitions \[
(n,m)\mapsto\begin{cases}
(n+1,m)&\text{ at rate }na_1\\
(n-1,m)&\text{ at rate }n(d_1+\tau+(1-p)C \bar{y}_a) \\
(n,m-1)&\text{ at rate }md_2\\
(n-1,m+1)&\text{ at rate }npC\bar{y}_a\\
(n+1,m-1)&\text{ at rate }m\sigma_2\\
\end{cases}
\] is sub-critical. This is the case if and only if both eigenvalues of the mean-matrix \begin{align}
\tilde{J}=\begin{pmatrix}
a_1-\tau-d_1-C\bar{y}_a & pC\bar{y}_a\\
\sigma_2 & -d_2-\sigma_2
\end{pmatrix}\label{eq: meanmatrix2}
\end{align}
are strictly negative. In particular, we must have a positive determinant, which is equivalent to \begin{align}
-(a_1-\tau-d_1-C\bar{y}_a)>\frac{\sigma_2 p C\bar{y}_a}{d_2+\sigma_2}.\label{eq: sufficientnegativeeigenvalue}
\end{align}
Our first result is concerned with finding the equilibria of the dynamical system (\ref{eq: Dynamical System}).
\begin{lemma}\label{Lemma: OnlyEquilibria}
	Consider the system (\ref{eq: Dynamical System}) and assume the matrix $J$ has a positive eigenvalue and the matrix $\tilde{J}$ only has negative eigenvalues. Then the systems only non-negative equilibria are $(0,0,0,0)$, $(\bar{x}_a,\bar{x}_d,0,0)$ and $(0,0,\bar{y}_a,\bar{y}_d)$, the latter of which is asymptotically stable, where \[
	\bar{y}_a=\frac{(b_1-d_1)(d_2+\sigma_2)}{C(d_2+(1-q)\sigma_2)}
	\quad\text{ and }\quad
	\bar{y}_d=\frac{q(b_1-d_1)^2(d_2+\sigma_2)}{C(d_2+(1-q)\sigma_2)^2}
	\]
	and $(\bar{x}_a,\bar{x}_d)$ are as in \eqref{eq: Equilibrium}.
\end{lemma}
\begin{proof}
	It is easy to verify that the claimed vectors are indeed equilibria of the system. Also a quick calculation for all cases shows that any non-negative equilibrium with a zero component must coincide with one of the vectors above. Hence, it remains to show that there is no coordinatewise strictly positive equilibrium in this system. Towards a contradiction, let $(x_a,x_d,y_a,y_d)$ be such an equilibrium. Then, rearranging the second line in (\ref{eq: Dynamical System}) yields \[
	x_d=\frac{pCx_a(x_a+y_a)}{d_2+\sigma_2}.
	\]
	Hence, the first line in (\ref{eq: Dynamical System}) gives \begin{align*}
	0=x_a(a_1-d_1-C(x_a+y_a))+\frac{pCx_a(x_a+y_a)\sigma_2}{d_2+\sigma_2}-\tau\frac{x_ay_a}{x_a+y_a}.
	\end{align*}
	Since we assumed $x_a>0$, we may divide by $x_a$ to give \begin{align}
	0&=a_1-d_1-C(x_a+y_a)+\frac{pC(x_a+y_a)\sigma_2}{d_2+\sigma_2}-\tau\frac{y_a}{x_a+y_a}\nonumber\\
	\Llra\ x_a+y_a&=\bar{x}_a-\tau\frac{y_a(d_2+\sigma_2)}{C(x_a+y_a)(d_2+(1-p)\sigma_2)}.\label{eq:firstx_a+y_aexpression}
	\end{align}
	Similarly, we obtain \[
	y_d=\frac{qCy_a(x_a+y_a)}{d_2+\sigma_2}\quad\text{ and }\quad 0=y_a(b_1-d_1-C(x_a+y_a))+\sigma_2y_d+\tau\frac{x_ay_a}{x_a+y_a}
	\]
	as well as \begin{align}
	x_a+y_a=\bar{y}_a+\tau\frac{x_a(d_2+\sigma_2)}{C(x_a+y_a)(d_2+(1-q)\sigma_2)}.\label{eq:secondx_a+y_aexpression}
	\end{align}
	Furthermore, this shows \[
	\frac{y_d}{y_a}=-\frac{b_1-d_1-C(x_a+y_a)+\tau\frac{x_a}{x_a+y_a}}{\sigma_2}=\frac{qC(x_a+y_a)}{d_2+\sigma_2}
	\]
	and substituting the expression (\ref{eq:firstx_a+y_aexpression}) for $x_a+y_a$ shows that \[
	-\frac{b_1-d_1-C\bar{x}_a}{\sigma_2}-\tau\lr{\frac{y_a(d_2+\sigma_2)+x_a(d_2+(1-p)\sigma_2)}{\sigma_2(x_a+y_a)(d_2+(1-p)\sigma_2)}}=\frac{qC\bar{x}_a}{d_2+\sigma_2}-\tau\frac{qy_a}{(x_a+y_a)(d_2+(1-p)\sigma_2)},
	\]
	which can be rearranged to read \begin{align*}
	&-(b_1+\tau-d_1-C\bar{x}_a)-\tau\lr{\frac{y_a(d_2+\sigma_2)+x_a(d_2+(1-p)\sigma_2)}{(x_a+y_a)(d_2+(1-p)\sigma_2)}-1}\\
	=&\ \frac{\sigma_2 qC\bar{x}_a}{d_2+\sigma_2}-\tau\frac{\sigma_2 qy_a}{(x_a+y_a)(d_2+(1-p)\sigma_2)}.
	\end{align*}
	If $\tau=0$, this violates condition (\ref{eq: sufficientpositiveeigenvalue}), and thus there is no such equilibrium. Assuming $\tau>0$ and (\ref{eq: sufficientpositiveeigenvalue}) the previous equality implies the inequality \begin{align*}
	\frac{y_a(d_2+\sigma_2)+x_a(d_2+(1-p)\sigma_2)}{(x_a+y_a)(d_2+(1-p)\sigma_2)}-1&<\frac{\sigma_2qy_a}{(x_a+y_a)(d_2+(1-p)\sigma_2)}\\
	\Llra\ \sigma_2 py_a&<\sigma_2 q y_a.
	\end{align*}
	Hence, in the case $p\geq q$ such a coordinatewise positive equilibrium cannot exist. \\
	Now, consider the case $q>p$. We find in a similar manner\[
	\frac{x_d}{x_a}=-\frac{a_1-d_1-C(x_a+y_a)-\tau\frac{y_a}{x_a+y_a}}{\sigma_2}=\frac{pC(x_a+y_a)}{d_2+\sigma_2}
	\]
	and, substituting $x_a+y_a$ with the right hand side of (\ref{eq:secondx_a+y_aexpression}), we obtain \[
	-(a_1-d_1-C\bar{y}_a)+\tau\lr{\frac{x_a(d_2+\sigma_2)+y_a(d_2+(1-q)\sigma_2)}{(x_a+y_a)(d_2+(1-q)\sigma_2)}}=\frac{\sigma_2pC\bar{y}_a}{d_2+\sigma_2}+\tau\frac{\sigma_2px_a}{(x_a+y_a)(d_2+(1-q)\sigma_2)}.
	\]
	In particular, condition (\ref{eq: sufficientnegativeeigenvalue}) implies that \begin{align*}
	\frac{x_a(d_2+\sigma_2)+y_a(d_2+(1-q)\sigma_2)}{(x_a+y_a)(d_2+(1-q)\sigma_2)}-1&<\frac{\sigma_2px_a}{(x_a+y_a)(d_2+(1-q)\sigma_2)}\\
	\Llra\ \sigma_2 q x_a&<\sigma_2 p x_a.
	\end{align*}
	This contradicts the assumption $q>p$. Hence, there cannot be a coordinatewise positive equilibrium.\\
	
	Now we turn towards the stability claim. For this we consider the Jacobian at $(0,0,\bar{y}_a,\bar{y}_d)$, which is given by \[
	A=\begin{pmatrix}
	a_1-d_1-C\bar{y}_a-\tau & \sigma_2 & 0 & 0\\
	 pC\bar{y}_a & -d_2-\sigma_2 & 0 & 0\\
	 -C\bar{y}_a+\tau & 0 & b_1-d_1-2C\bar{y}_a & \sigma_2\\
	 qC\bar{y}_a & 0 & 2qC\bar{y}_a & -d_2-\sigma_2
	\end{pmatrix}.
	\]
	It remains to show that all eigenvalues have negative real part. Notice that we recover the matrix $\tilde{J}^T$ in the upper left corner. Since the matrix $A$ is a block matrix, the eigenvalues of $\tilde{J}^T$ are also eigenvalues of $A$. In particular, it suffices to show that \[
	B=\begin{pmatrix}
	b_1-d_1-2C\bar{y}_a & \sigma_2 \\
	2qC\bar{y}_a & -d_2-\sigma_2
	\end{pmatrix}
	\]
	has only negative eigenvalues. By the same argument as for the matrix $J$ in \eqref{eq: sufficientpositiveeigenvalue} it suffices to show that the determinant is positive. Using the definition of $\bar{y}_a$ we compute the determinant to be $\det(B)=(d_2+\sigma_2)(b_1-d_1)$ (cf.~\cite[Section 2.2]{blath2020invasion}), which is strictly positive by the assumption $b_1>d_1$. Hence all eigenvalues of $A$ are negative and thus $(0,0,\bar{y}_a,\bar{y}_d)$ is an asymptotically stable equilibrium.
\end{proof}

The next goal is to find a sufficient criterion for the convergence of our dynamical system towards the equilibrium of the process $(Y_a^K,Y_d^K)$. For this purpose, we need a suitable initial condition for the system (\ref{eq: Dynamical System}). Indeed, the following lemmata give a first step towards this direction.

\begin{lemma}\label{Lemma: Convergence Lemma}
	Consider the system (\ref{eq: Dynamical System}) and assume that the matrix $J$ from (\ref{eq: meanmatrix}) has a positive eigenvalue and the matrix $\tilde{J}$ from (\ref{eq: meanmatrix2}) only has negative eigenvalues. If the initial condition $(x_a,x_d,y_a,y_d)=(x_a(0),x_d(0),y_a(0),y_d(0))$ satisfies \begin{align}
	\frac{qC(x_a+y_a)}{d_2+\sigma_2}>\frac{y_d}{y_a}>\frac{d_1-b_1+C(x_a+y_a)-\tau\frac{x_a}{x_a+y_a}}{\sigma_2},\label{eq: Growth condition}
	\end{align}
	then \[
	\lim\limits_{t\to\infty}(x_a(t),x_d(t),y_a(t),y_d(t))=(0,0,\bar{y}_a,\bar{y}_d).
	\]
	
\end{lemma}
\begin{proof}
	The proof of this claim can be easily adapted from \cite[Lemma 4.7]{blath2020invasion}. 
\end{proof}

Now, we are interested in finding a suitable condition such that the inequality (\ref{eq: Growth condition}) is satisfied. For this purpose, observe that the approximating process $(\hat{Y}_a,\hat{Y}_d)$ is supercritical and hence by the Kesten-Stigum Theorem \cite[Theorem 2.1]{10.2307/1428501} we have \[
\lr{\frac{\hat{Y}_{a,t}}{\hat{Y}_{a,t}+\hat{Y}_{d,t}},\frac{\hat{Y}_{d,t}}{\hat{Y}_{a,t}+\hat{Y}_{d,t}}}\xrightarrow{t\to\infty}(\pi_a,\pi_d),
\]
where $(\pi_a,\pi_d)$ is the unique left eigenvector of the matrix $J$ from (\ref{eq: meanmatrix}) for the principal eigenvalue $\lambda$ with $\pi_a+\pi_d=1$.
\begin{lemma}\label{Lemma: RefLemma}
	Suppose that the initial condition $(x_a,x_d,y_a,y_d)$ of the dynamical system (\ref{eq: Dynamical System}) satisfies $x_a\in[\bar{x}_a-A\sqrt{\veps},\bar{x}_a+A\sqrt{\veps}]$, $x_d\in[\bar{x}_d-A\sqrt{\veps},\bar{x}_d+A\sqrt{\veps}]$ for some constant $A>0$ and $y_a+y_d\in(0,\sqrt{\veps})$ with $\tfrac{y_d}{y_a}=\tfrac{\pi_d}{\pi_a}$. Then for $\veps$ sufficiently small, $(x_a,x_d,y_a,y_d)$ satisfies (\ref{eq: Growth condition}).
\end{lemma}
\begin{proof}
	The proof follows the proof of \cite[Lemma 4.8]{blath2020invasion}. Since $(\pi_a,\pi_d)$ is an eigenvector for $J$, we easily see that \[
	b_1+\tau-d_1-C\bar{x}_a+\sigma_2\frac{\pi_d}{\pi_a}=\la=qC\bar{x}_a\frac{\pi_a}{\pi_d}-(d_2+\sigma_2).
	\]
	Since $\la>0$, we obtain for $\epsi$ small enough from the first of the two equalities that \begin{align*}
	\frac{\pi_d}{\pi_a}&=\frac{\la-b_1-\tau+d_1+C\bar{x}_a}{\sigma_2}\\
	&>\frac{-b_1+d_1+C(\bar{x}_a+2(A+1)\sqrt{\veps})-\tau\frac{\bar{x}_a-A\sqrt{\veps}}{\bar{x}_a+(A+1)\sqrt{\veps}}}{\sigma_2}\\
	&> \frac{d_1-b_1+C(x_a+y_a)-\tau\frac{x_a}{x_a+y_a}}{\sigma_2}
	\end{align*}
	and similarly from the second equality we see \[
	\frac{\pi_d}{\pi_a}=\frac{qC\bar{x}_a}{\la+d_2+\sigma_2}<\frac{qC(\bar{x}_a-2(A+1)\sqrt{\veps})}{d_2+\sigma_2}<\frac{qC(x_a+y_a)}{d_2+\sigma_2}.
	\]
\end{proof}
The next series of lemmata shows that this initial condition is satisfied with high probability as $K\to\infty$ and $\veps\to 0$. In order to show this, we need a couple of preparatory results on the hitting times of a certain population size. 

\begin{lemma}\label{Lemma: veps xi growth lemma}
	Assume that the matrix $J$ from (\ref{eq: meanmatrix}) has a positive eigenvalue. Let $K\mapsto (m_1^K,m_2^K)$ be a function from $(0,\infty)$ to $[0,\infty)$ such that $(m_1^K,m_2^K)\in\tfrac{1}{K}\N_0^2$ and $\lim\limits_{K\to\infty}(m_1^K,m_2^K)=(\bar{x}_a,\bar{x}_d)$. Define the stopping times \[
	R_\veps^K\coloneqq\inf\lrset{t\geq 0\ \left\vert\ \abs{\frac{X_{a,t}^K}{K}-\bar{x}_a}>\veps\text{ or }\abs{\frac{X_{d,t}^K}{K}-\bar{x}_d}>\veps\right.}
	\]
	and
	\[
	T_x^K\coloneqq\inf\lrset{t>0\mid Y_{a,t}^K+Y_{d,t}^K=\lfloor xK\rfloor}.
	\]
	Then for any $\xi\in[\tfrac{1}{2},1]$ there exist some positive constants $A,\veps_0>0$ such that for all $0<\veps\leq\veps_0$ and any starting condition $\tfrac{1}{K}(X_{a,0}^K,X_{d,0}^K,Y_{a,0}^K,Y_{d,0}^K)=(m_1^K,m_2^K,\tfrac{n_1^K}{K},\tfrac{n_2^K}{K})$ with $n_1^K$ and $n_2^K$ such that $0<n_1^K+n_2^K<K\veps^\xi$, the convergence \[
	\lim\limits_{K\to\infty}\P(R_{2A\veps^\xi}^K\leq T_{\veps^\xi}^K\wedge T_0^K)=0
	\]
	holds.
\end{lemma}
\begin{proof}
	The proof of \cite[Lemma 4.2]{blath2020invasion} can be modified to encompass our situation in a straightforward manner.  
\end{proof}
Our next goal is to show that for any given initial population size $Y_{a,0}^K+Y_{d,0}^K=\lfloor\veps K\rfloor$ of $(Y_a^K,Y_d^K)$, we have the limit \[
\lim\limits_{K\to\infty}\P(T_{\veps^\xi}^K< T_0^K)=1.
\]

\begin{lemma}\label{Lemma: Fast growth}
	Under the assumptions of Lemma \ref{Lemma: veps xi growth lemma}, for the initial population size given by $\mathbf{X}_0^K\coloneqq(X_{a,0}^K,X_{d,0}^K,Y_{a,0}^K,Y_{d,0}^K)$, it holds that \[
	\lim_{K\to\infty}\P\lr{T_{\veps^\xi}^K<T_0^K\wedge R_{2A\veps^\xi}^K\ \Bigg\vert\ \mathbf{X}_0^K=K(m_1^K,m_2^K,\veps_1^K,\veps_2^K)}=1,
	\]
	where $\veps_1^K\to \veps_1$ and $\veps_2^K\to\veps_2$ such that $\veps_1+\veps_2=\veps$ with $\veps,\veps_1,\veps_2>0$.
\end{lemma}
\begin{proof}
	The proof follows the ideas of \cite[Proposition 4.1]{blath2020invasion} and \cite[Proposition 3.1]{camille2019emergence}. We will consider the process on the event that the invading population is extinct or reaches a sufficient population size before the resident population exits a neighbourhood of its equilibrium, that is \[
	E_\veps\coloneqq\lrset{T_0^K\wedge T_{\veps^\xi}^K<R_{2A\veps^\xi}^K},
	\]
	whose probability converges to $1$ for $\epsi$ small enough as $K\to\infty$ by Lemma \ref{Lemma: veps xi growth lemma}. Thus, it suffices to show $T_{\veps^\xi}^K<T_0^K$ with probability converging to $1$ on the event $E_\veps$. On this event, we can couple the process $(Y_{a}^K,Y_d^K)$ up to the time $t_\veps\coloneqq T_0^K\wedge T_{\veps^\xi}^K\wedge R_{2A\veps^\xi}^K$ such that \begin{align*}
	Y_v^{\veps,-}&\leq \hat{Y}_v\\
	Y_v^{\veps,-}&\leq Y_v^K
	\end{align*}
	where $v\in\{a,d\}$ by defining the transition rates as follows: For the process $(Y_a^{\veps,-},Y_d^{\veps,-})$ we choose the transition rates \[
	(n,m)\mapsto\begin{cases}
	(n+1,m)&\text{ at rate }n\lr{b_1+\tau\frac{\bar{x}_a-2A\veps^\xi}{\bar{x}_a+(2A+1)\veps^\xi}}\\
	(n-1,m)&\text{ at rate }n(d_1+(1-q)C (\bar{x}_a+2A\veps^\xi+\veps^\xi)+qC(4A\veps^\xi+\veps^\xi)) \\
	(n,m-1)&\text{ at rate }md_2\\
	(n-1,m+1)&\text{ at rate }nqC(\bar{x}_a-2A\veps^\xi)\\
	(n+1,m-1)&\text{ at rate }m\sigma_2.\\
	\end{cases}
	\]
	Intuitively, the coupling is correct due to switching from active to dormant being more favourable over death, but not being better than not experiencing any competition at all. The precise reasoning for this coupling is identical to the one in the proof of \cite[Proposition 4.1]{blath2020invasion}. Also, we obtain the inequality $q_a\leq q^{\veps,-}_a<1$ for the extinction probabilities of each individual family started at an active individual in a similar manner. When starting the populations from one dormant individual, we obtain the extinction probabilities $q_d\leq q^{\veps,-}_d<1$. Note that $q^{\veps,-}_{\diamond}$ is indeed strictly less than $1$ for $\diamond\in\{a,d\}$, since the process $(Y_a^{\veps,-},Y_d^{\veps,-})$ is supercritical for $\veps>0$ sufficiently small. Since the extinction probabilities are continuous in the transition rates -- see \cite[Lemma A.3]{camille2019emergence} -- we obtain \[
	0\leq \liminf\limits_{\veps\downto0}\abs{q^{\veps,-}-q}\leq\limsup\limits_{\veps\downto0}\abs{q^{\veps,-}-q}=0.
	\]
	Next, we define the corresponding hitting time for the coupled process. More specifically, for $x\geq 0$, let \[
	T_x^{\veps,-}\coloneqq\inf\{t\geq 0\mid Y_{a,t}^{\veps,-}+Y_{d,t}^{\veps,-}=\lfloor Kx\rfloor \}.
	\]
	Then, due to the coupling on $E_\veps$, we see that \[
	\P(T_{\veps^\xi}^{\veps,-}\leq T_0^{\veps,-}, E_\veps)\leq \P(T_{\veps^\xi}^K\leq T_0^K, E_\veps).
	\]
	Note that since $\P(E_\veps)\to 1$ as $K\to\infty$, for any $\veps>0$ sufficiently small, we have for $\delta$ small enough \begin{align*}
	\liminf_{K\to\infty}\P(T_{\veps^\xi}^{\veps,-}\leq T_0^{\veps,-}, E_\veps)&\geq\liminf_{K\to\infty}\P(T_{\veps^\xi}^{\veps,-}\leq T_0^{\veps,-})-\delta\geq \liminf_{K\to\infty}\P(T_0^{\veps,-}=\infty)-\delta\\
	&\geq\liminf_{K\to\infty}\lr{1-\lr{q^{\veps,-}}^{K(\veps_1^K+\veps_2^K)}}-\delta\\
	&\geq\liminf_{K\to\infty}\lr{1-\lr{q^{\veps,-}}^{K(\veps-\delta)}}-\delta\\
	&=1-\delta\xrightarrow{\delta\to 0}1,
	\end{align*}
	where $q^{\veps,-}\coloneqq q^{\veps,-}_a\vee q^{\veps,-}_d$.
	Also due to $\P(E_\veps)\to 1$ as $K\to\infty$, it therefore follows that \[
	\lim_{K\to\infty}\P(T_{\veps^\xi}^K\leq T_0^K, E_\veps)=\lim_{K\to\infty}\P(T_{\veps^\xi}^K\leq T_0^K)=1.
	\]

\end{proof}

Finally, we are able to show that with high probability as $\veps\downto 0$, the assumptions of Lemma \ref{Lemma: RefLemma} are satisfied.

\begin{lemma}\label{Lemma: Sufficient condition}
	Assume that the matrix $J$ from (\ref{eq: meanmatrix}) has a positive eigenvalue and that the initial condition of the process satisfies $\tfrac{1}{K}(X_{a,0}^K,X_{d,0}^K,Y_{a,0}^K,Y_{d,0}^K)=(m_1^K,m_2^K,\veps_1,\veps_2)$ with $\veps_1+\veps_2\leq\veps$ for some $\veps,\veps_1,\veps_2>0$ and $m_1^K,m_2^K$ as in Lemma \ref{Lemma: veps xi growth lemma}. Then for any $\delta>0$ with $\pi_a\pm\delta\in(0,1)$, it holds that \[
	\liminf_{K\to\infty}\P\lr{\exists t\in [T_\veps^K,T_{\sqrt{\veps}}^K]\colon \pi_a-\delta< \frac{Y_{a,t}^K}{Y_{a,t}^K+Y_{d,t}^K}<\pi_a+\delta}\geq 1-o_{\veps}(1),
	\]
	where $o_{\veps}(1)\to 0$ as $\veps\to 0$.
	In fact, the bounds on the frequency process $\tfrac{Y_a^K}{Y_{a}^K+Y_d^K}$ will be satisfied by the time $T_\veps^K+\log\log(1/\veps)$ with probability converging to $1$ as $K\to\infty$.
\end{lemma}
\begin{proof}
	This proof again is a simple adaptation of the proofs from \cite[Proposition 4.4.]{blath2020invasion} or \cite[Proposition 3.2]{camille2019emergence}. 
\end{proof}

Now, we can show our general result on the competition with non-negative transfer. We do not assume the transition rates to be constant anymore, but instead assume that there are constants $a_1,b_1,d_1,d_2>0$, $\tau\geq 0$ with $a_1,b_1>d_1$ such that for some $s>0$ \begin{align}
\sup_{0\leq t\leq s\log K}&\abs{a_1^K(t)-a_1}+\abs{b_1^K(t)-b_1}+\abs{d_1^K(t)-d_1}\nonumber\\
&+\abs{d_2^K(t)-d_2}+\abs{\tau^K(t)-\tau}+\abs{\tfrac{\gamma_1^K(t)}{K}}+\abs{\tfrac{\gamma_2^K(t)}{K}}\xrightarrow{K\to\infty}0\label{eq: Convergence}
\end{align}
in probability. If $\tau=0$, then we assume $\tau^K\equiv 0$ for all $K$.
\begin{proposition}\label{Proposition: 4Dpositive}
	Assume that the conditions (\ref{eq: sufficientpositiveeigenvalue}), (\ref{eq: sufficientnegativeeigenvalue}) and (\ref{eq: Convergence}) are true. Consider the process $(X_a^K,X_d^K,Y_a^K,Y_d^K)$ with initial condition $\tfrac{1}{K}(X_{a,0}^K,X_{d,0}^K)\in[\bar{x}_a-\veps,\bar{x}_a+\veps]\times[\bar{x}_d-\veps,\bar{x}_d+\veps]$ and $\tfrac{1}{K}(Y_{a,0}^K+Y_{d,0}^K)=m\veps$ for some $\epsi$ and $m>0$ sufficiently small. Then, for any $\veps'>0$ there exists a finite time $T=T(m,\veps,\veps')$ such that \[
	\lim\limits_{K\to\infty}\P\lr{X_{a,T}^K+X_{d,T}^K\leq \veps'K,\ \frac{Y_{a,T}^K}{K}\in [\bar{y}_a-\veps',\bar{y}_a+\veps'],\ \frac{Y_{d,T}^K}{K}\in [\bar{y}_d-\veps',\bar{y}_d+\veps']}\geq 1-o_\veps(1).
	\]
\end{proposition}
\begin{remark}\label{Remark: Important Proposition Remark}
	The choice on the initial condition for $(Y_a^K,Y_d^K)$ can be generalized to hold for an entire interval of initial conditions. With the proposition as stated, for $\veps,\veps',m>0$ sufficiently small, it holds \[
	Y_{a,T}^K+Y_{d,T}^K\geq K(\bar{y}_a+\bar{y}_d-2\veps')>2m\veps K
	\] at time $T=T(m,\veps,\veps')$ with high probability as  $\veps\to 0$ and $K\to\infty$. Thus, with high probability we have $T_{2m\veps}^K<T(m,\veps,\veps')$. Hence, with high probability, for all initial conditions from an interval $Y_{a,0}^K+Y_{d,0}^K\in[\tfrac{Km\veps}{2},Km\veps]$, the time $T_{2m\veps}^K$ is bounded, even as $K\to\infty$. In particular, we can adapt the proof of Proposition \ref{Proposition: 4Dpositive} such that the claim holds for all initial conditions $Y_{a,0}^K+Y_{d,0}^K\in[\tfrac{Km\veps}{2},Km\veps]$, where the only change is in the application of Lemma \ref{Lemma: Sufficient condition}, which now yields that with high probability the good initial condition is satisfied after a time shorter than $T_{2m\veps}^K+\log\log(1/(2m\veps))$. This time however is finite with high probability by the arguments made above.
\end{remark}

\begin{proof}
	The proof is based on a suitable coupling of the process in combination with the above lemmata. Let $\delta>0$. We want to couple in such a way that coordinatewise both bi-type branching processes are bounded. That is, we want to find a coupling such that \[
	(X_a^{\delta,1},X_d^{\delta,1},Y_{a}^{\delta,1},Y_d^{\delta,1})\leq (X_a^K,X_d^K,Y_a^K,Y_d^K)\leq (X_a^{\delta,2},X_d^{\delta,2},Y_{a}^{\delta,2},Y_d^{\delta,2})
	\]
	coordinatewise. As in Lemma \ref{Lemma: convergence to equilibrium}, this coupling can be achieved by subtracting and adding $\delta$ to the birth and death rates for $K$ large enough, which we can do by the convergence criterion above. For now, we do not alter the horizontal transfer rate $\tau^K$. Note that indeed we are allowed to neglect the immigration rate for to the same reason as in Lemma \ref{Lemma: convergence to equilibrium}. For $\delta>0$ small enough, the corresponding equilibria $(\bar{x}_a^{\delta,*},\bar{x}_d^{\delta,*})$ and $(\bar{y}_a^{\delta,*},\bar{y}_d^{\delta,*})$ with $*\in\{1,2\}$ are closer than $\tfrac{\veps\wedge\veps'}{2}$ to $(\bar{x}_a,\bar{x}_d)$ and $(\bar{y}_a,\bar{y}_d)$ respectively. If $\tau>0$, then we further couple the processes $(X_a^{\delta,*},X_d^{\delta,*},Y_{a}^{\delta,*},Y_d^{\delta,*})$ with $*\in\{1,2\}$ with processes $(X_a^{\delta,*,\diamond},X_d^{\delta,*,\diamond},Y_{a}^{\delta,*,\diamond},Y_d^{\delta,*,\diamond})$ where $\diamond\in\{+,-\}$.  In the case $\diamond=-$ we set the horizontal transfer rate to be $\tau-\delta$ and for $\diamond=+$ it is set to be $\tau+\delta$. This definition yields the inequalities as displayed in Table \ref{tab: inequalities}.  \begin{table}[htbp] \[\begin{array}{ccccc|ccccc}
		X_a^{\delta,1,-} & \geq & X_a^{\delta,1} & \geq & X_a^{\delta,1,+} & X_a^{\delta,2,-} & \geq & X_a^{\delta,2} & \geq & X_a^{\delta,2,+}\\[0.7em]
		X_d^{\delta,1,-} & \geq & X_d^{\delta,1} & \geq & X_d^{\delta,1,+} & X_d^{\delta,2,-} & \geq & X_d^{\delta,2} & \geq & X_d^{\delta,2,+}\\[0.7em]
		Y_a^{\delta,1,-} & \leq & Y_a^{\delta,1} & \leq & Y_a^{\delta,1,+} & Y_a^{\delta,2,-} & \leq & Y_a^{\delta,2} & \leq & Y_a^{\delta,2,+}\\[0.7em]
		Y_d^{\delta,1,-} & \leq & Y_d^{\delta,1} & \leq & Y_d^{\delta,1,+} & Y_d^{\delta,2,-} & \leq & Y_d^{\delta,2} & \leq & Y_d^{\delta,2,+}\\
		\end{array} 
		\]
		\caption{An overview of the almost sure inequalities that we obtain by coupling.}
		\label{tab: inequalities}
	\end{table} Now arguments analogous to Lemma \ref{Lemma: Generalized EthierKurz} show that for any initial condition from a compact set, the processes $(X_a^{\delta,*,\diamond},X_d^{\delta,*,\diamond},Y_{a}^{\delta,*,\diamond},Y_d^{\delta,*,\diamond})$ converge in probability to the solutions of the respective differential equations with $K\to\infty$. Furthermore, Lemma \ref{Lemma: Sufficient condition} implies that with high probability in $m\veps$ the criterion for a good initial condition (\ref{eq: Growth condition}) is satisfied for each of the coupled processes after a time shorter than $\log\log(1/(m\veps))$. Indeed, when applying the lemma, notice that we can substitute $T_\veps^K$ by $0$ due to our choice of the starting condition. Hence, by Lemma \ref{Lemma: Convergence Lemma} the solutions of the differential equations converge towards the equilibria $(0,0,\bar{y}_a^{\delta,*,\diamond},\bar{y}_d^{\delta,*,\diamond})$. In particular, there exists a finite time such that for all initial conditions from a compact set as in the proposition, the process is in a neighbourhood of $(0,0,\bar{y}_a^{\delta,*,\diamond},\bar{y}_d^{\delta,*,\diamond})$ with high probability in $m\veps$. Thus, the claim follows.
\end{proof}

With a similar proof we also obtain the same result for negative transfer. In this situation we assume the process $(Y_a^K,Y_d^K)$ to be initially resident and the process $(X_a^K,X_d^K)$ to be invading.
\begin{proposition}\label{Proposition: 4Dnegative}
	Assume that the reverse inequalities of (\ref{eq: sufficientpositiveeigenvalue}) and (\ref{eq: sufficientnegativeeigenvalue}) are true that is \begin{align}
	-(b_1+\tau-d_1-C\bar{x}_a)>\frac{\sigma_2qC\bar{x}_a}{d_2+\sigma_2}\quad\text{ and }\quad
	-(a_1-\tau-d_1-C\bar{y}_a)>\frac{\sigma_2 p C\bar{y}_a}{d_2+\sigma_2},\label{eq: Inverseinequalities}
	\end{align} which indicates the approximating process $(\hat{Y}_a,\hat{Y}_d)$ to be subcritical and the approximating process $(\hat{X}_a,\hat{X}_d)$ to be supercritical. Further assume (\ref{eq: Convergence}). Consider the process $(X_a^K,X_d^K,Y_a^K,Y_d^K)$ with initial condition $\tfrac{1}{K}(Y_{a,0}^K,Y_{d,0}^K)\in[\bar{y}_a-\veps,\bar{y}_a+\veps]\times[\bar{y}_d-\veps,\bar{y}_d+\veps]$ and $\tfrac{1}{K}(X_{a,0}^K+X_{d,0}^K)=m\veps$ for some $\epsi$ and $m>0$ sufficiently small. Then for any $\veps'>0$, there exists a finite time $T=T(m,\veps,\veps')$ such that \[
	\lim\limits_{K\to\infty}\P\lr{Y_{a,T}^K+Y_{d,T}^K\leq \veps'K,\ \frac{X_{a,T}^K}{K}\in [\bar{x}_a-\veps',\bar{x}_a+\veps'],\ \frac{X_{d,T}^K}{K}\in [\bar{x}_d-\veps',\bar{x}_d+\veps']}\geq 1-o_\veps(1).
	\]
\end{proposition}
The proof follows a very similar structure of the previous proposition. We give an outline of the necessary results. 
	\begin{lemma}\label{Lemma: Appendix 1}
		Consider the system (\ref{eq: Dynamical System}) and assume (\ref{eq: Inverseinequalities}). If the initial condition $(x_a,x_d,y_a,y_d)$ satsfies \begin{align}
		\frac{pC(x_a+y_a)}{d_2+\sigma_2}>\frac{x_d}{x_a}>\frac{d_1-a_1+C(x_a+y_a)+\tau\frac{y_a}{x_a+y_a}}{\sigma_2},\label{eq: Growth condition2}
		\end{align}
		then \[
		\lim\limits_{t\to\infty}(x_a(t),x_d(t),y_a(t),y_d(t))=(\bar{x}_a,\bar{x}_d,0,0).
		\]
	\end{lemma}
	\begin{proof}
		The proof of this lemma is identical with the proof of Lemma \ref{Lemma: Convergence Lemma} where we reversed the roles of $(x_a,x_d)$ and $(y_a,y_d)$. The validity holds due to the fact that we never explicitly use the term involving $\tau$.
	\end{proof}
	Proceeding in the same manner, we define $(\pi_a,\pi_d)$ to be the unique normed left eigenvector of the matrix \begin{align}
	J=\begin{pmatrix}
	a_1-\tau-d_1-C\bar{y}_a & pC\bar{y}_a\\
	\sigma_2 & -d_2-\sigma_2
	\end{pmatrix}\label{eq:meanmatrix 2}
	\end{align}
	corresponding to the principal eigenvalue $\la>0$, which exists due to our assumption (\ref{eq: Inverseinequalities}). 
	\begin{lemma}\label{Lemma: AppendixA Lemma}
		Suppose that the initial condition $(x_a,x_d,y_a,y_d)$ of the dynamical system (\ref{eq: Dynamical System}) satisfies $y_a\in[\bar{y}_a-A\sqrt{\veps},\bar{y}_a+A\sqrt{\veps}]$, $y_d\in[\bar{y}_d-A\sqrt{\veps},\bar{y}_d+A\sqrt{\veps}]$ for some constant $A>0$ large enough and $x_a+x_d\in(0,\sqrt{\veps})$ with $\tfrac{x_d}{x_a}=\tfrac{\pi_d}{\pi_a}$. Then, for $\veps$ sufficiently small, $(x_a,x_d,y_a,y_d)$ satisfies (\ref{eq: Growth condition2}).
	\end{lemma}
	\begin{proof}
		The proof in the case of positive transfer from Lemma \ref{Lemma: RefLemma} is easily adapted to this case. 
	\end{proof}
	Next in our series of Lemmata, we had shown bounds on some exit times in Lemma \ref{Lemma: veps xi growth lemma} in order to show that the assumptions from Lemma \ref{Lemma: AppendixA Lemma} are satisfied with high probability. Here, there will be a major difference in the proof, as at one point we made use of the positive transfer.
	\begin{lemma}\label{Lemma: AppendixA 2}
		Let $K\mapsto (m_1^K,m_2^K)$ be a function from $(0,\infty)$ to $[0,\infty)^2$ such that $(m_1^K,m_2^K)\in\tfrac{1}{K}\N_0^2$ and $\textstyle\lim_{K\to\infty}(m_1^K,m_2^K)=(\bar{y}_a,\bar{y}_d)$. Define the stopping times \[
		R_\veps^K\coloneqq\inf\lrset{t\geq 0\ \Big\vert\ \abs{\frac{Y_{a,t}^K}{K}-\bar{y}_a}>\veps\text{ or }\abs{\frac{Y_{d,t}^K}{K}-\bar{y}_d}>\veps}
		\]
		and
		\[
		T_x^K\coloneqq\inf\lrset{t>0\mid X_{a,t}^K+X_{d,t}^K=\lfloor xK\rfloor}.
		\]
		Then, for any $\xi\in[\tfrac{1}{2},1]$ and starting condition $\tfrac{1}{K}(X_{a,0}^K,X_{d,0}^K,Y_{a,0}^K,Y_{d,0}^K)=(\tfrac{n_1^K}{K},\tfrac{n_2^K}{K},m_1^K,m_2^K)$ with $n_1^K$ and $n_2^K$ such that $0<n_1^K+n_2^K<K\veps^\xi$, there exist some positive constants $ A,\veps_0>0$ such that for all $0<\veps\leq \veps_0$ \[
		\lim\limits_{K\to\infty}\P(R_{2A\veps^\xi}^K\leq T_{\veps^\xi}^K\wedge T_0^K)=0.
		\]
	\end{lemma}
	\begin{proof}
		This proof can be adapted from the proof of \cite[Lemma C.2]{blath2020interplay}. 
	\end{proof}
	
	From this lemma, we can show that with probability converging to $1$, the invasive trait reaches a critical population size.
	
	\begin{lemma}
		Under the assumptions of Lemma \ref{Lemma: AppendixA 2}, denoting $\mathbf{X}_0^K\coloneqq(X_{a,0}^K,X_{d,0}^K,Y_{a,0}^K,Y_{d,0}^K)$, it holds \[
		\lim_{K\to\infty}\P\lr{T_{\veps^\xi}^K<T_0^K\wedge R_{2A\veps^\xi}^K\ \Bigg\vert\ \mathbf{X}_0^K=K(\veps_1^K,\veps_2^K,m_1^K,m_2^K)}=1,
		\]
		where $\xi\in[\tfrac{1}{2},1]$ and $\veps_1^K\to\veps_1$, $\veps_2^K\to\veps_2$ are such that $\veps_1+\veps_2=\veps$ with $\veps,\veps_1,\veps_2>0$.
	\end{lemma}
	\begin{proof}
		The proof now is almost identical with the one of Lemma \ref{Lemma: Fast growth}. The only difference arises in the coupling, where we now couple such that \begin{align*}
		X_v^{\veps,-}\leq \hat{X}_v\quad\text{ and }\quad	X_v^{\veps,-}\leq X_v^K
		\end{align*}
		where $v\in\{a,d\}$ by defining the transition rates as follows: For the process $(X_a^{\veps,-},X_d^{\veps,-})$, we choose the transition rates \[
		(n,m)\mapsto\begin{cases}
		(n+1,m)&\text{ at rate }na_1\\
		(n-1,m)&\text{ at rate }n\left(d_1+\tau\frac{\bar{x}_a+2A\veps^\xi}{\bar{x}_a-(2A+1)\veps^\xi}+(1-q)C (\bar{x}_a+2A\veps^\xi+\veps^\xi)\right.\\
		&\qquad\qquad\left.+qC(4A\veps^\xi+\veps^\xi)\right) \\
		(n,m-1)&\text{ at rate }md_2\\
		(n-1,m+1)&\text{ at rate }nqC(\bar{x}_a-2A\veps^\xi)\\
		(n+1,m-1)&\text{ at rate }m\sigma_2,\\
		\end{cases}.
		\]
		
		It is  only important to note that for $\veps>0$ sufficiently small, the process $(X_a^{\veps,-},X_d^{\veps,-})$ is again supercritical and thus the claim follows as in Lemma \ref{Lemma: Fast growth}
	\end{proof}
	
	Finally, we need to show that with high probability the assumptions of Lemma \ref{Lemma: AppendixA Lemma} are satisfied, which is the analogous version of Lemma \ref{Lemma: Sufficient condition}.
	
	\begin{lemma}\label{Lemma: Sufficient condition2}
		Assume (\ref{eq: Inverseinequalities}) and assume that  $\tfrac{1}{K}(X_{a,0}^K,X_{d,0}^K,Y_{a,0}^K,Y_{d,0}^K)=(\veps_1,\veps_2,m_1^K,m_2^K)$  is the initial condition with $\veps_1+\veps_2\leq\veps$ for some $\veps,\veps_1,\veps_2>0$ and $m_1^K,m_2^K$ as in Lemma \ref{Lemma: AppendixA 2}. Then for any $\delta>0$ with $\pi_a\pm\delta\in(0,1)$, it holds \[
		\liminf_{K\to\infty}\P\lr{\exists t\in [T_\veps^K,T_{\sqrt{\veps}}^K]\colon \pi_a-\delta< \frac{X_{a,t}^K}{X_{a,t}^K+X_{d,t}^K}<\pi_a+\delta}\geq 1-o_{\veps}(1).
		\]
		In fact, the bound on the frequency process $\tfrac{X_a^K}{X_a^K+X_d^K}$ will be satisfied by the time $T_\veps^K+\log\log(1/\veps)$ with probability larger than $1-o_\veps(1)$ as $K\to\infty$.
	\end{lemma}
	\begin{proof}
		The proof is largely analogous to the one of \cite[Proposition 4.4]{blath2020invasion}, but there are subtle differences which we discuss here. We assume \[
		\frac{X_{a,T_{\veps}^K}^K}{X_{a,T_{\veps}^K}^K+X_{d,T_{\veps}^K}^K}\leq \pi_a-\delta.
		\]
		Then, we introduce the event \[
		\tilde{E}_\veps\coloneqq\lrset{T_{\sqrt{\veps}}^K<T_0^K\wedge R_{2A\sqrt{\veps}}^K},
		\]
		whose probability again converges to $1$ as $K\to\infty$ by the previous lemma. Also, we  define the stopping time \[
		T_{\veps,\veps/M}^K\coloneqq\inf\lrset{t\geq T_\veps^K\mid X_{a,t}^K+X_{d,t}^K\leq\tfrac{\veps K}{M}}.
		\]
		On the event $\tilde{E}_\veps$, we can bound our process from below on the time interval $[T_\veps^K,T_{\sqrt{\veps}}^K]$ by a pure death process $Z^K$, with initial condition  $Z_{T_\veps^K}^K=\veps K$ and death rate $(d_1+\tau\tfrac{\sqrt{\veps}(\bar{y}_a+(2A+1)\veps)}{\bar{y}_a-2A\veps})\vee d_2$. We are interested in $(\veps-\tfrac{\veps}{M})K$ individuals dying from the process $Z^K$. This takes longer than $\log\log(1/\veps)$ with probability converging to $1$ for $M>1$. Therefore, we have \[
		\lim_{K\to\infty}\P(T_{\veps,\veps/M}^K<T_\veps^K+\log\log(1/\veps)\mid T_\veps^K+\log\log(1/\veps)<T_{\sqrt{\veps}}^K, \tilde{E}_\veps )=0.
		\]
		Furthermore, the population size of $X_{a}^K+X_d^K$ can be bounded from above by a Yule process with birth rate $a_1$. Hence, \cite[Lemma A.2]{camille2019emergence} implies \begin{align*}
		\lim\limits_{K\to\infty}\P(T_{\sqrt{\veps}}^K<T_\veps^K+\log\log(\tfrac{1}{\veps})|\tilde{E}_\veps)\leq \tilde{M}\sqrt{\veps}(\log(\tfrac{1}{\veps}))^{a_1}.
		\end{align*}
		Since the probability of $\tilde{E}_\veps$ converges to $1$, we obtain \begin{align}
		\lim_{K\to\infty}\P(T_{\veps,\veps/M}^K< T_\veps^K+\log\log(1/\veps))\leq o_\veps(1)\quad\text{ and }\quad \lim\limits_{K\to\infty}\P(T_{\sqrt{\veps}}^K<T_\veps^K+\log\log(\tfrac{1}{\veps}))\leq o_\veps(1).\label{eq: weak bound}
		\end{align}
		
		The next step in this proof is the semimartingale decomposition, for which we will satisfy ourselves with the abstract representation \[
		\frac{X_{a,t}^K}{X_{a,t}^K+X_{d,t}^K}=\frac{X_{a,T_\veps^K}^K}{X_{a,T_\veps^K}^K+X_{d,T_\veps^K}^K}+M^K(t)+V^K(t)\quad\text{ for }t\geq T_\veps^K,
		\]
		where $M^K$ is a martingale and $V^K$ a process of finite variation. The remainder of the proof can be adapted from \cite[Proposition 4.4]{blath2020invasion}, with the slight exception of now only obtaining \begin{align*}
		&\lim\limits_{K\to\infty}\P\Big(\sup_{T_\veps^K\leq s\leq T_\veps^K+\log\log(1/\veps)}\abs{M^K(t)}\geq\veps\Big)\nonumber\\
		&\leq\lim\limits_{K\to\infty}\P\Big(\sup_{T_\veps^K\leq s\leq (T_\veps^K+\log\log(1/\veps))\wedge T_{\veps,\veps/M}}\abs{M^K(t)}\geq\veps\Big)+\P\lr{T_{\veps,\veps/M}<T_\veps^K+\log\log(\tfrac{1}{\veps})}\\
		&\leq \lim\limits_{K\to\infty}\reci{\veps^2}\E[\langle M^K\rangle_{(T_\veps^K+\log\log(1/\veps))\wedge T_{\veps,\veps/M}}]+o_\veps(1)=o_\veps(1),
		\end{align*}
		due to our weaker bound (\ref{eq: weak bound}) and thus afterwards only having \[
		\pi_a-\frac{\delta}{2}\geq\frac{X_{a,t}^K}{X_{a,t}^K+X_{d,t}^K}\geq\frac{\theta}{2}\lr{\log\log(1/\veps)\wedge (t_a^\veps-T_\veps^K)}-\veps
		\]
		with probability $1-o_\veps(1)$ as $K\to\infty$ for the corresponding time $t_a^\veps$, which does not change the claim, since we are considering only high probabilities for small $\veps$.
	\end{proof}
	We can now discuss the proof of Proposition \ref{Proposition: 4Dnegative}, where the transition rates are not assumed to be constant.
	\begin{proof}[Proof of Proposition \ref{Proposition: 4Dnegative}]
		We can copy the proof of Proposition \ref{Proposition: 4Dpositive} until the point, where we have found the couplings $(X_a^{\delta,*,\diamond},X_d^{\delta,*,\diamond},Y_a^{\delta,*,\diamond},Y_d^{\delta,*,\diamond}).$ Now, the only difference is that we need to use our Lemmata applying to this case. The general argument is not changed. That is, from Lemma \ref{Lemma: Sufficient condition2} we see that with high probability in $m\veps$ the convergence condition from Lemma \ref{Lemma: AppendixA Lemma} is satisfied and therefore Lemma \ref{Lemma: Appendix 1} implies convergence to the equilibrium $(\bar{x}_a^{\delta,*,\diamond},\bar{x}_d^{\delta,*,\diamond},0,0)$. In particular, there exists a finite time such that for all initial conditions from a compact set we have entered a neighbourhood of this equilibrium, which as before implies the claim.
	\end{proof}

\subsection{Competition Between Bi-Type and Single-Type Processes with Transfer}
We consider a three-dimensional branching process very similar to the one from the previous section. The transfer rates of this process $(X_a^K,X_d^K,Y^K)$ are  \begin{align*} (i,j,k)\to\begin{cases}
(i+1,j,k) &\text{ at rate }ia_1^K(\omega,t)+\gamma_1^K(\omega,t)\\
(i,j,k+1) &\text{ at rate }kb_1^K(\omega,t)+\gamma_2^K(\omega,t)\\
(i-1,j,k) &\text{ at rate }i(d_1^K(\w,t)+\tfrac{(1-p)C}{K}(i+k))\\
(i,j,k-1) &\text{ at rate }k(d_1^K(\w,t)+\tfrac{C}{K}(i+k))\\
(i,j-1,k) &\text{ at rate }jd_2^K(\omega,t) \\
(i-1,j+1,k) &\text{ at rate }i\tfrac{pC}{K}(i+k) \\
(i+1,j-1,k)&\text{ at rate }j\sigma_2\\
(i-1,j,k+1)&\text{ at rate }\tau^K(\w,t)\tfrac{ik}{i+k}
\end{cases}
\end{align*}
with predictable, non-negative functions $a_1^K,b_1^K,d_1^K,d_2^K,\tau^K,\gamma_1^K,\gamma_2^K\colon\Omega\times[0,\infty)\to\R$ and constants $C,\sigma_2>0$, $p\in(0,1)$. As before, we assume for some $s>0$ the convergence \begin{align}
\sup_{0\leq t\leq s\log K}&\abs{a_1^K(t)-a_1}+\abs{b_1^K(t)-b_1}+\abs{d_1^K(t)-d_1}\nonumber\\
&+\abs{d_2^K(t)-d_2}+\abs{\tau^K(t)-\tau}+\abs{\tfrac{\gamma_1^K(t)}{K}}+\abs{\tfrac{\gamma_2^K(t)}{K}}\xrightarrow{K\to\infty}0\label{eq: Convergence2}
\end{align}
in probability, where $a_1,b_1,d_1,d_2>0$, $\tau\geq 0$ and $a_1,b_1>d_1$. If $\tau=0$ we again assume $\tau^K\equiv 0$. Then, we can approximate the process $Y^K$ in a population close to the equilibrium size of the process $(X_a^K,X_d^K)$ which as in the previous section is given by $K(\bar{x}_a,\bar{x}_d)$ by the process $\hat{Y}$, which has the transitions \[
n\mapsto\begin{cases}
n+1,&\quad\text{ at rate } nb_1+\tau\\
n-1,&\quad\text{ at rate } n(d_1+C\bar{x}_a).
\end{cases}
\]
We want the growth rate to be strictly positive, so that we have invasion of this trait. In the case of one dimensional branching processes, this is equivalent to \begin{align}
b_1+\tau-d_1-C\bar{x}_a>0,\label{sufficientpositiveeigenvalue2}
\end{align}
which also coincides with (\ref{eq: sufficientpositiveeigenvalue}) in the case $q=0$. The approximation of the process $(X_a^K,X_d^K)$ in a population, where $Y^K$ is close to its equilibrium size $K\bar{y}$ with \[
\bar{y}=\frac{b_1-d_1}{C},
\]
can be done as before. That is, we approximate using the process $(\hat{X}_a,\hat{X}_d)$ with transitions \[
(n,m)\mapsto\begin{cases}
(n+1,m)&\text{ at rate }na_1\\
(n-1,m)&\text{ at rate }n(d_1+\tau+(1-p)C \bar{y}) \\
(n,m-1)&\text{ at rate }md_2\\
(n-1,m+1)&\text{ at rate }npC\bar{y}\\
(n+1,m-1)&\text{ at rate }m\sigma_2.
\end{cases}
\]
In order to guarantee a successful invasion, we want this process to be subcritical, which as before coincides with the criterion (\ref{eq: sufficientnegativeeigenvalue}), that is \begin{align}
-(a_1-\tau-d_1-C\bar{y})>\frac{\sigma_2 p C\bar{y}}{d_2+\sigma_2}. \label{eq:sufficientnegativeeigenvalue2}
\end{align}

We are now in a position to formulate our competition results for this case. Note, that a generalized form, in the sense that the invading process is initially of size contained in the interval $[\tfrac{m\veps K}{2},m\veps K]$, can be proven as indicated in Remark \ref{Remark: Important Proposition Remark}.

\begin{proposition}\label{Proposition: 3Dpositive}
	Assume that conditions (\ref{eq: Convergence2}), (\ref{sufficientpositiveeigenvalue2}) and (\ref{eq:sufficientnegativeeigenvalue2}) are true. Consider the process $(X_a^K,X_d^K,Y^K)$ with initial condition $\tfrac{1}{K}(X_{a,0}^K,X_{d,0}^K)\in[\bar{x}_a-\veps,\bar{x}_a+\veps]\times[\bar{x}_d-\veps,\bar{x}_d+\veps]$ and $\tfrac{Y^K}{K}=m\veps$ for some $\epsi$ and $m>0$ sufficiently small. Then for any $\veps'>0$, there exists a finite time $T=T(m,\veps,\veps')$ such that \[
	\lim\limits_{K\to\infty}\P\lr{X_{a,T}^K+X_{d,T}^K\leq \veps'K,\ \frac{Y_T^K}{K}\in [\bar{y}-\veps',\bar{y}+\veps']}=1.
	\]
\end{proposition}

For the proof of this proposition, we use arguments from \cite[Section 5]{blath2020interplay}. The structure of the proof is as usual: First, we assume the population to satisfy a suitable initial condition and then approximate the dynamics of the system by a differential equation, whose solutions will converge to the corresponding fixed points.

Assume that the initial condition of the process $\tfrac{1}{K}(X_a^K,X_d^K,Y^K)$ is contained in the set \[
\calA_\veps^3\coloneqq [\bar{x}_a-2A\veps^\xi,\bar{x}_a+2A\veps^\xi]\times [\bar{x}_d-2A\veps^\xi,\bar{x}_d+2A\veps^\xi]\times[\veps,\sqrt{\veps}]
\]
for fixed $\epsi$.
We want to show that the solution of the dynamical system 
\begin{align}\begin{aligned}
\dot{x}_a&=x_a(a_1-d_1-C(x_a+y))+x_d\sigma_2-\tau\frac{x_a y}{x_a+y}\\
\dot{x}_d&=pCx_a(x_a+y)-(d_2+\sigma_2)x_d\\
\dot{y}&=y(b_1-d_1-C(x_a+y))+\tau\frac{x_ay}{x_a+y}
\end{aligned}\label{eq: DynamicalSystem2}
\end{align}
converges towards the equilibrium $(0,0,\bar{y})$ for any starting condition from $\calA_\veps^3$.

\begin{lemma}\label{Lemma: AppendixA convergence}
	Consider the dynamical system (\ref{eq: DynamicalSystem2}). If the initial condition $(x_a(0),x_d(0),y)$ is contained in the set $\calA_\veps^3$ and the inequalities (\ref{sufficientpositiveeigenvalue2}) and (\ref{eq:sufficientnegativeeigenvalue2}) are satisfied, then \[
	\lim\limits_{t\to\infty}(x_a(t),x_d(t),y(t))=(0,0,\bar{y}).
	\]
\end{lemma}
\begin{proof}
	This proof is taken from \cite[Proposition 5.4]{blath2020interplay} and adapted to our case. Firstly, we notice that $y(t)$ is strictly increasing as long as \[
	C(x_a+y)-\tau\frac{x_a}{x_a+y}<b_1-d_1.
	\]
	By our choice of the starting conditions, we see that \[
	C(x_a+y)-\tau\frac{x_a}{x_a+y}\leq C(\bar{x}_a+(2A+1)\sqrt{\veps})-\tau\frac{\bar{x}_a-2A\sqrt{\veps}}{\bar{x}_a+(2A+1)\sqrt{\veps}}\leq C\bar{x}_a-\tau+C_*\sqrt{\veps}< b_1-d_1,
	\]
	where we used (\ref{sufficientpositiveeigenvalue2}) and $C_*$ is a sufficiently large constant. Hence, initially $y(t)$ is strictly increasing and will do so until \begin{align}
	C(x_a+y)-\tau\frac{x_a}{x_a+y}=b_1-d_1.\label{eq: Perturbtion assumption}
	\end{align}
	We now consider the resident population dynamics $(x_a,x_d)$ in the case where this equality is true and then we perform a perturbation argument. Assuming the equality, we easily compute \begin{align*}
	C(x_a+y)+\tau\frac{y}{x_a+y}=C(x_a+y)-\tau\frac{x_a}{x_a+y}+\tau=b_1-d_1+\tau
	\end{align*}
	and similarly\begin{align*}
	pCx_a(x_a+y)=px_a\lr{C(x_a+y_a)-\tau\frac{x_a}{x_a+y}+\tau\frac{x_a}{x_a+y}}\leq px_a(b_1-d_1+\tau).
	\end{align*}
	Hence, we can dominate the dynamics of $(x_a,x_d)$ under the assumption (\ref{eq: Perturbtion assumption}) by the solutions to the system
	\begin{align}
	\begin{aligned}
	\dot{x}_a&=x_a(a_1-b_1-\tau)+\sigma_2 x_d\\
	\dot{x}_d&=px_a(b_1-d_1+\tau)-(d_2+\sigma_2)x_d.
	\end{aligned}\label{eq: Approximation}
	\end{align}
	This system has the unique equilibrium $(0,0)$ if the coefficient matrix of the corresponding linear system \[
	A=\begin{pmatrix}
	a_1-b_1-\tau & \sigma_2\\
	p(b_1-d_1+\tau) & -(d_2+\sigma_2)
	\end{pmatrix}
	\] has non-zero determinant. Indeed, the determinant is $0$ if and only if \begin{align*}
	a_1&=b_1+\tau-\frac{p\sigma_2(C\bar{y}+\tau)}{d_2+\sigma_2}.
	\end{align*}
	However, this choice of $a_1$ would imply \[
	\bar{x}_a=\frac{(a_1-d_1)(d_2+\sigma_2)}{C(d_2+(1-p)\sigma_2)}=\frac{(b_1+\tau-d_1)(d_2+\sigma_2)-p\sigma_2(b_1-d_1+\tau)}{C(d_2+(1-p)\sigma_2)}=\frac{b_1+\tau-d_1}{C},
	\]
	so \[
	b_1+\tau-d_1-C\bar{x}_a=0,
	\]
	which contradicts (\ref{sufficientpositiveeigenvalue2})
	and hence the only equilibrium is $(0,0)$.\\
	
	Next, we will show that the system converges towards the equilibrium $(0,0)$. For this, we need to show that both eigenvalues of the matrix $A$ are negative. Note that for the existence of a positive eigenvalue we need a positive trace and hence $a_1-b_1-\tau>0$. Then, we obtain \[
	a_1-\tau-d_1-C\bar{y}=a_1-\tau-b_1>0,
	\]
	which contradicts the condition (\ref{eq:sufficientnegativeeigenvalue2}) since now \[
	0>-(a_1-\tau-d_1-C\bar{y})>\frac{\sigma_2pC\bar{y}}{d_2+\sigma_2}\geq 0.
	\] Hence, there cannot be a positive eigenvalue. Moreover, due to the determinant of $A$ being non-zero, both eigenvalues must be strictly negative. Since the dynamics of the solutions are determined by the eigenvalues, the solution will converge for any positive starting condition to $(0,0)$.\\
	
	Now we turn again to the general dynamical system (\ref{eq: DynamicalSystem2}). We distinguish two cases:\\
	
	\textbf{Case 1: Monotonicity.} In this case we assume that $y(t)$ eventually will be a monotone function. By boundedness of $y(t)$ this implies that $\textstyle\lim_{t\to\infty}\dot{y}(t)=0$ and since $y(t)$ is bounded away from $0$, we must have \[
	C(x_a(t)+y(t))-\tau\frac{x_a(t) }{x_a(t)+y(t)}\xrightarrow{t\to\infty}b_1-d_1.
	\]
	In particular, for any $\delta>0$ we can find a time $t_1(\delta)>0$ such that for all starting conditions in $\calA_\veps^3$ we have \[
	b_1-d_1-\delta<C(x_a(t)+y(t))-\tau\frac{x_a }{x_y+a}< b_1-d_1+\delta
	\]
	for all $t>t_1$. Then, we can consider the system (\ref{eq: Approximation}) altered in the corresponding places by $\pm\delta$ to account for the bounds above. We still obtain convergence to $(0,0)$ for both systems and therefore we must have \[
	\lim\limits_{t\to\infty}(x_a(t),x_d(t))=(0,0).
	\]
	As $C(x_a(t)+y(t))-\tau\tfrac{x_a(t) }{x_a(t)+y(t)}\to b_1-d_1$ we have $y(t)\to\bar{y}$.\\
	
	\textbf{Case 2: Non-monotone convergence.} It is possible that $t\mapsto y(t)$ is not monotone on any interval $(s,\infty)$ for $s\geq 0$. In this case, due to the boundedness of $s$ there must be a countable number of local extrema of $y$, for if there were only finitely many, then $y$ would become monotone eventually. At any local minimum or maximum, since $y$ is bounded away from $0$, we must have \[
	C(x_a+y)-\tau\frac{x_a}{x_a+y}=b_1-d_1.
	\]
	Using our observations from Case 1, we see that at any maximum and at any minimum the population of $(x_a,x_d)$ is decreasing and converges to $(0,0)$ even monotonically. Since this holds independently of the size of $y(t)$, we obtain that \[\lim\limits_{t\to\infty}(x_a(t),x_d(t))=(0,0).\] As before, we derive $\textstyle\lim_{t\to\infty}y(t)=\bar{y}$.
	
\end{proof}

\begin{proof}[Proof of Proposition \ref{Proposition: 3Dpositive}]
	As usual, we couple the process $(X_a^K,X_d^K,Y^K)$ with two processes from above and below by \[
	(X_a^{\delta,-},X_d^{\delta,-},Y^{\delta,-})\leq (X_a^K,X_d^K,Y^K)\leq (X_a^{\delta,+},X_d^{\delta,+},Y^{\delta,+})
	\]
	by accordingly increasing or decreasing the birth, death and switching rates by some term involving $\delta>0$. We further couple with processes $(X_a^{\delta,*,\diamond},X_d^{\delta,*,\diamond},Y^{\delta,*,\diamond})$ as in Proposition \ref{Proposition: 4Dpositive} by increasing or decreasing the transfer rate $\tau^K$ by $\delta$. For $\epsi$ sufficiently small, the initial condition from Lemma \ref{Lemma: AppendixA convergence} is satisfied and thus after a finite time, the solutions to the differential equations converge towards their equilibria as time tends to infinity. By arguments similar to Lemma \ref{Lemma: Generalized EthierKurz}, the processes converge to the solutions of the corresponding differential equations. In particular, as in Proposition \ref{Proposition: 4Dpositive} we see that for $\delta>0$ sufficiently small, the processes are inside a neighbourhood of the equilibrium after a finite time with probability converging to $1$.
\end{proof}

A similar result holds for the inverse invasion.

\begin{proposition}\label{Proposition: 3Dnegative}
	Assume that (\ref{eq: Convergence2}) and the inverse inequalities of (\ref{sufficientpositiveeigenvalue2}) and (\ref{eq:sufficientnegativeeigenvalue2}) are true, that is \begin{align}
	b_1+\tau-d_1-C\bar{x}_a<0\quad\text{ and }\quad -(a_1-\tau-d_1-C\bar{y})<\frac{\sigma_2 p C\bar{y}}{d_2+\sigma_2}.\label{eq: Inverseinequalities2}
	\end{align} Consider the process $(X_a^K,X_d^K,Y^K)$ with initial condition $\tfrac{1}{K}(X_{a,0}^K+X_{d,0}^K)=m\veps$ and $\tfrac{Y^K}{K}\in[\bar{y}-\veps,\bar{y}+\veps]$ for some $\epsi$ and $m>0$ sufficiently small. Then for any $\veps'>0$, there exists a finite time $T=T(m,\veps,\veps')$ such that \[
	\lim\limits_{K\to\infty}\P\lr{Y_{T}^K\leq \veps'K,\ \frac{X_{a,T}^K}{K}\in [\bar{x}_a-\veps',\bar{x}_a+\veps'],\ \frac{X_{d,T}^K}{K}\in [\bar{x}_d-\veps',\bar{x}_d+\veps']}\geq 1-o_\veps(1).
	\]
\end{proposition}
\begin{proof}[Proof of Proposition \ref{Proposition: 3Dnegative}]
	Here, we are looking at the invasion of a population with dormancy which does not benefit from horizontal transfer. This is the same situation as in Proposition \ref{Proposition: 4Dnegative}, where the difference is only in the dormancy of the initially resident trait. However, since we have never used this aspect in the proof of Proposition \ref{Proposition: 4Dnegative}, the chain of arguments remains valid, where of course in the couplings we need to account for the lack of dormancy in the initially resident trait.
\end{proof}

In addition, we also need the same results, but with inverted horizontal transfer. That is, we now consider the process $(X_a^K,X_d^K,Y^K)$ which has the same transitions as in the beginning of the section except for the transition  $(i,j,k)\to (i-1,j,k+1)$ to be replaced with the transition $(i,j,k)\to (i+1,j,k-1)$. Then we can still approximate the processes $(X_a^K,X_d^K)$ and $Y^K$ as before, but we need to switch the addition of $\tau$ in the rates from the birth to the death rate and vice versa. In particular, we want the inequalities \begin{align}
b_1-\tau-d_1-C\bar{x}_a>0\quad\text{ and }\quad -(a_1+\tau-d_1-C\bar{y})>\frac{\sigma_2 pC\bar{y}}{d_2+\sigma_2}\label{eq: Moreinequalities}
\end{align}
to hold. Then we obtain the same results as above.

\begin{proposition}\label{Proposition: 3Dpositive2}
	Assume that (\ref{eq: Convergence2}) and (\ref{eq: Moreinequalities}) are true. Consider the process $(X_a^K,X_d^K,Y^K)$ with initial condition $\tfrac{1}{K}(X_{a,0}^K,X_{d,0}^K)\in[\bar{x}_a-\veps,\bar{x}_a+\veps]\times[\bar{x}_d-\veps,\bar{x}_d+\veps]$ and $\tfrac{Y^K}{K}=m\veps$ for some $\epsi$ and $m>0$ sufficiently small. Then for any $\veps'>0$, there exists a finite time $T=T(m,\veps,\veps')$ such that \[
	\lim\limits_{K\to\infty}\P\lr{X_{a,T}^K+X_{d,T}^K\leq \veps'K,\ \frac{Y_T^K}{K}\in [\bar{y}-\veps',\bar{y}+\veps']}=1.
	\]
\end{proposition}

The idea for proving Proposition \ref{Proposition: 3Dpositive2} is very similar to the one of Proposition \ref{Proposition: 3Dpositive}. The only difficulty arises in the preceding lemma, where we have used explicitly that the invading population is benefiting from horizontal transfer. As before, assume that the initial condition of the process $\tfrac{1}{K}(X_a^K,X_d^K,Y^K)$ is contained in the set \[
\calA_\veps^3\coloneqq [\bar{x}_a-2A\veps^\xi,\bar{x}_a+2A\veps^\xi]\times [\bar{x}_d-2A\veps^\xi,\bar{x}_d+2A\veps^\xi]\times[\veps,\sqrt{\veps}]
\]
for fixed $\epsi$.
We want to show that the solution to the dynamical system 
\begin{align}\begin{aligned}
\dot{x}_a&=x_a(a_1-d_1-C(x_a+y))+x_d\sigma_2+\tau\frac{x_a y}{x_a+y}\\
\dot{x}_d&=pCx_a(x_a+y)-(d_2+\sigma_2)x_d\\
\dot{y}&=y(b_1-d_1-C(x_a+y))-\tau\frac{x_ay}{x_a+y}
\end{aligned}\label{eq: DynamicalSystem3}
\end{align}
converges towards the equilibrium $(0,0,\bar{y})$ for any starting condition from $\calA_\veps^3$.
\begin{lemma}
	Consider the dynamical system (\ref{eq: DynamicalSystem3}). If the initial condition $(x_a(0),x_d(0),y)$ is contained in the set $\calA_\veps^3$ and the inequalities (\ref{eq: Moreinequalities}) are satisfied, then \[
	\lim\limits_{t\to\infty}(x_a(t),x_d(t),y(t))=(0,0,\bar{y}).
	\]
\end{lemma}
\begin{proof}
	The concept of the proof stays the same. Initially the function $y(t)$ is increasing until eventually \[
	C(x_a+y)+\tau\frac{x_a}{x_a+y}=b_1-d_1.
	\]
	We again consider the system for $(x_a,x_d)$ under this condition and need to show that this implies $(x_a(t),x_d(t))\to 0$ as $t\to \infty$. As before, we calculate \[
	C(x_a+y)-\tau\frac{y}{x_a+y}=C(x_a+y)+\tau\frac{x_a}{x_a+y}-\tau=b_1-d_1-\tau
	\]
	and \[
	pCx_a(x_a+y)=px_a\lr{C(x_a+y_a)+\tau\frac{x_a}{x_a+y}-\tau\frac{x_a}{x_a+y}}\leq px_a(b_1-d_1).
	\]
	Thus, we consider \begin{align}
	\begin{aligned}
	\dot{x}_a&=x_a(a_1-b_1+\tau)+\sigma_2 x_d\\
	\dot{x}_d&=px_a(b_1-d_1)-(d_2+\sigma_2)x_d.
	\end{aligned}\label{eq: Approximation2}
	\end{align}
	Again, we compute the determinant of the coefficient matrix \[
	A=\begin{pmatrix}
	a_1-b_1+\tau & \sigma_2 \\
	p(b_1-d_1) & -(d_2+\sigma_2)
	\end{pmatrix}
	\]
	to be $0$ if and only if \begin{align*}
	a_1=b_1-\tau -\frac{p\sigma_2 C\bar{y}}{d_2+\sigma_2}.
	\end{align*}
	From (\ref{eq: Moreinequalities}) we know that\[
	a_1<-\frac{p\sigma_2 C\bar{y}}{d_2+\sigma_2}-\tau+d_1+C\bar{y}=-\frac{p\sigma_2 C\bar{y}}{d_2+\sigma_2}-\tau+b_1,
	\]
	so $(0,0)$ is again the only equilibrium and the eigenvalues of $A$ are negative, so for any positive initial condition the solution of the system (\ref{eq: DynamicalSystem3}) converges to $(0,0)$. The remainder of the proof can be taken from the proof of Lemma \ref{Lemma: AppendixA convergence}.
\end{proof}
Now, we can prove Proposition \ref{Proposition: 3Dpositive2}.
\begin{proof}[Proof of Proposition \ref{Proposition: 3Dpositive2}]
	The proof is identical with the one of Proposition \ref{Proposition: 3Dpositive}, except for the need of using Lemma A.7 in this section instead of Lemma \ref{Lemma: AppendixA convergence}.
\end{proof}

\begin{proposition}\label{Proposition: 3Dnegative2}
	Assume that (\ref{eq: Convergence2}) and the inverse inequalities of (\ref{eq: Moreinequalities}) are true, that is \begin{align}
	b_1-\tau-d_1-C\bar{x}_a<0\quad\text{ and }\quad -(a_1+\tau-d_1-C\bar{y})<\frac{\sigma_2 p C\bar{y}}{d_2+\sigma_2}.\label{eq: Moreinequalities2}
	\end{align} Consider the process $(X_a^K,X_d^K,Y^K)$ described above before Proposition \ref{Proposition: 3Dpositive2} with initial condition $\tfrac{1}{K}(X_{a,0}^K+X_{d,0}^K)=m\veps$ and $\tfrac{Y^K}{K}\in[\bar{y}-\veps,\bar{y}+\veps]$ for some $\epsi$ and $m>0$ sufficiently small. Then for any $\veps'>0$, there exists a finite time $T=T(m,\veps,\veps')$ such that \[
	\lim\limits_{K\to\infty}\P\lr{Y_{T}^K\leq \veps'K,\ \frac{X_{a,T}^K}{K}\in [\bar{x}_a-\veps',\bar{x}_a+\veps'],\ \frac{X_{d,T}^K}{K}\in [\bar{x}_d-\veps',\bar{x}_d+\veps']}\geq 1-o_\veps(1).
	\]
\end{proposition}
\begin{proof}[Proof of Proposition \ref{Proposition: 3Dnegative2}]
	In this case, we are looking at the invasion of a population with dormancy which benefits from horizontal transfer. This is the same situation as in Proposition \ref{Proposition: 4Dpositive}, where the difference is only in the dormancy of the initially resident trait. As above, we have never used this aspect in the proof and hence the chain of arguments is still valid, where again the couplings with the initially resident population need to be slightly adapted.
\end{proof}